\documentclass[11pt,a4paper]{article}

\usepackage{amsmath}     
\usepackage{dsfont}		
\usepackage{amsfonts}		
\usepackage{amsthm}
\usepackage{graphicx}
\usepackage{amssymb}

\usepackage{empheq} 
\usepackage[thicklines]{cancel} 

\usepackage{latexsym}
\usepackage{euscript,makeidx,color,mathrsfs}
\usepackage{enumerate}
\usepackage{tabu}
\usepackage{multirow}

\usepackage{tikz}  
\usetikzlibrary{arrows.meta}
\usetikzlibrary{positioning} 
\usetikzlibrary{graphs}
\usepackage{xcolor} 
\usetikzlibrary {decorations.pathmorphing} 

\usepackage[colorlinks,linkcolor=blue,anchorcolor=green,citecolor=red]{hyperref}

\def\sqr#1#2{{\vcenter{\vbox{\hrule height.#2pt
              \hbox{\vrule width.#2pt height#1pt \kern#1pt \vrule width.#2pt}
              \hrule height.#2pt}}}}
\def\5n{\negthinspace \negthinspace \negthinspace \negthinspace \negthinspace }
\def\4n{\negthinspace \negthinspace \negthinspace \negthinspace }
\def\3n{\negthinspace \negthinspace \negthinspace }
\def\2n{\negthinspace \negthinspace }
\def\1n{\negthinspace }

\def\dbA{\mathbb{A}}

\def\dbD{\mathbb{D}}
\def\dbE{\mathbb{E}}
\def\dbF{\mathbb{F}}

\def\dbH{\mathbb{H}}

\def\dbL{\mathbb{L}}

\def\dbN{\mathbb{N}}

\def\dbP{\mathbb{P}}

\def\dbR{\mathbb{R}}
\def\dbS{\mathbb{S}}

\def\dbU{\mathbb{U}}
\def\dbV{\mathbb{V}}
\def\dbW{\mathbb{W}}
\def\dbX{\mathbb{X}}

\def\sP{\mathscr{P}}

\def\fH{\mathfrak{H}}

\def\fM{\mathfrak{M}}
\def\fN{\mathfrak{N}}

\def\={\buildrel \triangle \over =}

\def\ds{\displaystyle}
\def\nm{\noalign{\ms}}
\def\nb{\noalign{\bs}}
\def\ns{\noalign{\ss}}
\def\a{\alpha}
\def\b{\beta}
\def\g{\gamma}
\def\d{\delta}
\def\e{\varepsilon}

\def\l{\lambda}

\def\si{\sigma}
\def\t{\tau}
\def\f{\varphi}
\def\th{\theta}
\def\o{\omega}

\def\G{\Gamma}
\def\D{\Delta}

\def\O{\Omega}
\def\mf{\mathcal{F}}
\def\me{\mathbb{E}}
\def\rd{\,\mathrm d}  
\def\nb{\nabla}
\def\cA{{\cal A}}

\def\cC{{\cal C}}

\def\cF{{\cal F}}
\def\cG{{\cal G}}
\def\cH{{\cal H}}

\def\cJ{{\cal J}}
\def\cK{{\cal K}}
\def\cL{{\cal L}}

\def\cO{{\cal O}}
\def\cP{{\cal P}}
\def\cQ{{\cal Q}}
\def\cR{{\cal R}}
\def\cS{{\cal S}}
\def\cT{{\cal T}}

\def\no{\noindent}

\def\ss{\smallskip}
\def\ms{\medskip}
\def\bs{\bigskip}
\def\q{\quad}
\def\qq{\qquad}

\def\lt{\left}
\def\rt{\right}
\def\lan{\langle}
\def\ran{\rangle}

\def\rf{\eqref}

\def\h{\widehat}
\def\wt{\widetilde}

\def\cd{\cdot}
\def\cds{\cdots}

\def\dim{\hbox{\rm dim$\,$}}

\def\ae{\hbox{\rm a.e.}}
\def\as{\hbox{\rm a.s.}}

\def\deq{\triangleq}

\def\({\Big (}
\def\){\Big )}
\def\[{\Big[}
\def\]{\Big]}

\def\bde{\begin{definition}\label}
\def\ede{\end{definition}}
\def\be{\begin{equation}}
\def\bel{\begin{equation}\label}
\def\ee{\end{equation}}
\def\beq{\begin{equation*}\begin{aligned}}
\def\eeq{\end{aligned}\end{equation*}}
\def\bal{\begin{aligned}}
\def\eal{\end{aligned}}
\def\bay{\begin{array}{ll}}
\def\eay{\end{array}}
\def\beay{\begin{equation*}\begin{array}{ll}}
\def\eeay{\end{array}\end{equation*}}
\def\bt{\begin{theorem}\label}
\def\et{\end{theorem}}
\def\bc{\begin{corollary}\label}
\def\ec{\end{corollary}}
\def\bl{\begin{lemma}\label}
\def\el{\end{lemma}}
\def\bp{\begin{proposition}\label}
\def\ep{\end{proposition}}
\def\bas{\begin{assumption}\label}
\def\eas{\end{assumption}}
\def\br{\begin{remark}\label}
\def\er{\end{remark}}
\def\bex{\begin{example}\label}
\def\ex{\end{example}}
\def\ba{\begin{array}}
\def\ea{\end{array}}
\def\ed{\end{document}}

\def\square#1{\vbox{\hrule\hbox{\vrule height#1%
     \kern#1\vrule}\hrule}}
\def\rectangle#1#2{\vbox{\hrule\hbox{\vrule height#1%
     \kern#2\vrule}\hrule}}

\font\tenbb=msbm10 \font\sevenbb=msbm7 \font\fivebb=msbm5

\newfam\bbfam
\scriptscriptfont\bbfam=\fivebb \textfont\bbfam=\tenbb
\scriptfont\bbfam=\sevenbb

\newtheorem{theorem}{\hskip 1.3em Theorem}[section]
\newtheorem{definition}[theorem]{\hskip 1.3em Definition}
\newtheorem{proposition}[theorem]{\hskip 1.3em Proposition}
\newtheorem{corollary}[theorem]{\hskip 1.3em Corollary}
\newtheorem{lemma}[theorem]{\hskip 1.3em Lemma}
\newtheorem{remark}[theorem]{\hskip 1.3em Remark}
\newtheorem{example}[theorem]{\hskip 1.3em Example}
\newtheorem{algorithm}[theorem]{\hskip 1.3em Algorithm}

\newtheorem{assumption}[theorem]{\hskip 1.3em Assumption}

\allowdisplaybreaks

\makeatletter
   
   \@addtoreset{equation}{section}
\makeatother

\begin{document}

\title{ Numerical Methods for Optimal Control Problems with SPDEs
}

\author{
Andreas Prohl\thanks{
Mathematisches Institut, Universit\"at T\"ubingen, Auf der Morgenstelle 10,
D-72076 T\"ubingen, Germany.
 {\small\it
e-mail:} {\small\tt prohl@na.uni-tuebingen.de}.} \quad 
and \quad 
Yanqing Wang\thanks{
School of Mathematics and Statistics, Southwest University, Chongqing 400715, China.  {\small\it
e-mail:} {\small\tt yqwang@swu.edu.cn}. \ms}}

\date{Nov. 17, 2024}
\maketitle

\begin{abstract}

This paper investigates numerical methods for solving stochastic linear quadratic (SLQ) optimal control problems governed by stochastic partial differential equations (SPDEs). Two distinct approaches, the open-loop and closed-loop ones, are developed to ensure convergence rates in the fully discrete setting. The open-loop approach, utilizing the finite element method for spatial discretization and the Euler method for temporal discretization, addresses the complexities of coupled forward-backward SPDEs and employs a gradient descent framework suited for high-dimensional spaces. Separately, the closed-loop approach applies a feedback strategy, focusing on Riccati equation for spatio-temporal discretization. Both approaches are rigorously designed to handle the challenges of fully discrete SLQ problems, providing rigorous convergence rates and computational frameworks.

\end{abstract}

\ms

\no\bf Keywords: \rm Convergence rate, stochastic linear quadratic optimal control problem, open-loop strategy, closed-loop strategy,
 Riccati equation


\ms

\no\bf AMS 2020 subject classification: \rm  
49N10, 49J20, 93B52, 
65M60,
35Q93
 

\tableofcontents
%

\section{Introduction}\label{ch-1}

Let $D$\index{$D$}  be a bounded domain with smooth boundary or  
a convex polyhedral domain in $\dbR^d$, $d=1,2,3$, and $T>0$. 
The main goal in this work is to {\em numerically approximate}  the (${\mathbb F}$-adapted) optimal control process $U^*(\cd)$
 that minimizes the {\em quadratic} cost functional ($\alpha \geq 0$) 
%
\begin{equation}\label{intro-1a}
\setlength\abovedisplayskip{3pt}
{\mathcal J}\big(U(\cd)\big) = \frac{1}{2} {\mathbb E}\Big[ \int_0^T \big[ \| X(t) \|^2 + \| U(t)\|^2\big] \, {\rm d}t +
\alpha \| X(T) \|^2\Big]
\setlength\belowdisplayskip{3pt}
\end{equation}
subject to the {\em linear} stochastic partial differential equation (SPDE),\index{SPDE} with $\beta \in \dbR$,
\begin{equation}\label{spde}
\setlength\abovedisplayskip{3pt}
\left\{
\begin{array}{ll}
\ds {\rm d}X(t) = \big[ \Delta X(t) + U(t) \big] {\rm d}t +  \big[ \b  X(t)+\si(t)\big] {\rm d}W(t)\qq   t \in (0,T]\,,  \\
\ns \ds        X(0) = x \in {\mathbb H}^1_0 \cap {\mathbb H}^2\, ,
 \end{array}
 \right.  
 \setlength\belowdisplayskip{3pt}
\end{equation}
which is supplemented with a homogeneous Dirichlet boundary condition. The concise formulation of the problem includes a given complete filtered probability space $(\Omega, {\mathcal F}, \{ {\mathcal F}_t\}_{0 \leq t \leq T}, {\mathbb P})$ on which a $\dbR$-valued standard Brownian motion $W$ is defined, and will be referred to as\\
\no {\bf  Problem (SLQ)}:\index{Problem {\bf (SLQ)}}  `minimize (\ref{intro-1a}) subject to (\ref{spde})'. 

\ss

In this work, Problem {\bf (SLQ)} serves as a prototypic formulation, for which numerical methods and their error analysis will be developed. Most of the results
can easily be extended  to more
general stochastic optimal control problems subject to SPDEs with distributed controls.

While the existence and uniqueness of a strong solution --- which is the tuple $\big(X^*(\cdot), U^*(\cdot)\big)$ that 
minimizes Problem {\bf (SLQ)} --- is well-known by now, its explicit construction is usually not possible, and thus requires a 
numerical approximation: in fact,  to construct an ({\em implementable}) {\em spatio-temporal  discretization scheme } of a 
stochastic linear quadratic optimal control problem (SLQ problem, for short),
\index{SLQ problem} and to verify convergence of its solution is a very recent field in mathematical control theory and
the subject of this work, which 
combines tools from stochastic analysis, numerical optimization, and mathematical statistics. 
\ss

Let us motivate the main conceptional difference between the numerics of the deterministic ($\beta = 0, \si(\cd)=0$) vs.~the stochastic ($\beta \neq 0$ or $\si(\cd)\neq 0$) control problem (\ref{intro-1a})--(\ref{spde}),
which is the {\em adjoint equation} in the optimality conditions. The optimality conditions will be needed to formulate a gradient descent method, 
by which an implementable iterative algorithm can be proposed.
We start with the deterministic case, for which different optimally convergent discretization methods are available.
%
%
\begin{enumerate} [ (a)]
\item Let $\beta = 0$ and $\si(\cd)=0$. 
Then (\ref{spde}) is a {\em deterministic} PDE, and the overall problem is the {\em deterministic linear-quadratic} 
optimal control problem (LQ problem, for short).\index{LQ problem} The related optimality condition with unique solution $\big( X^*(\cdot), Y(\cdot), U^*(\cdot)\big)$ is:
\begin{subequations}\label{intro-1cd}
    \begin{empheq}[left={\empheqlbrace\,}]{align}
      &{\rm d}X^*(t) = \big[ \Delta X^*(t) + U^*(t) \big] {\rm d}t   & t \in (0,T]\,, \label{intro-1cda}\\
      &{\rm d}Y(t) = \big[ -\Delta Y(t)  + X^*(t) \big] {\rm d}t   &  t \in [0,T)\,,  \label{intro-1cdb} \\
       & X^*(0) = x \, , \qquad Y(T) = -\a X^*(T)\, ,  \label{intro-1cdc}
    \end{empheq}
\end{subequations}
together with the maximum principle 
 \begin{equation}\label{intro-1dd}
  \setlength\abovedisplayskip{5pt}
 U^*(t) - Y(t) = 0 \qquad \forall\, t \in [0,T]\, .
  \setlength\belowdisplayskip{5pt}
 \end{equation} 
 The deterministic  PDEs in (\ref{intro-1cd}) are here written in the integral form to facilitate the direct comparison with the stochastic setting in  (b) below.
 Clearly, (\ref{intro-1dd}) may be inserted into \rf{intro-1cdb}, and the  {\em adjoint equation} is a PDE running backward in time; hence, its complexity is of the same kind as the forward PDE \rf{intro-1cda}. Efficient, fast numerical methods that are based on the spatio-temporal discretization of this optimality system are available by now; in addition, 
convergence with rates in terms of time and space discretization parameters 
may be shown which base on (discrete) stability properties and space-time regularity of the solution $\big(X^*(\cdot), U^*(\cdot)\big)$; see {\em e.g.}~\cite{Hinze-Pinnau-Ulbrich-Ulbrich09}.

\item Let $\beta \neq 0$ or $\si(\cd)\neq 0$. 
The optimality condition for $\big(X^*(\cd),U^*(\cd)\big)$ constitutes a {\em coupled forward-backward SPDE system} (FBSPDE)\index{FBSPDE} with a unique solution $\big(X^*(\cd),Y(\cd),Z(\cd),U^*(\cd)\big)$ (see {\em e.g.}~\cite{Lv-Zhang21}),
\begin{subequations}\label{intro-1c}
    \begin{empheq}[left={\empheqlbrace\,}]{align}
      &{\rm d}X^*(t) = \big[ \Delta X^*(t) + U^*(t) \big] \rd t 
       + \big[ \b X^*(t) +\si (t)\big] \rd W(t)  &  t \in (0,T]\,, \label{intro-1ca}\\
      &{\rm d}Y(t) = \big[ -\Delta Y(t) - \b Z(t) + X^*(t) \big] \rd t + Z(t) \rd W(t)  &  t \in [0,T)\,,  \label{intro-1cb} \\
       & X^*(0) = x \, , \qquad Y(T) = -\a X^*(T)\, ,  \label{intro-1cc}
    \end{empheq}
\end{subequations}
together with the optimality condition 
 \begin{equation}\label{intro-1d}
  \setlength\abovedisplayskip{5pt}
 U^*(t) - Y(t) = 0 \qquad t \in [0,T]\, .
  \setlength\belowdisplayskip{5pt}
 \end{equation}  
 The adjoint equation \rf{intro-1cb}  is a {\em backward SPDE} (BSPDE, for short),\index{BSPDE} 
 whose solution now is a {\em tuple} $\big( Y(\cdot), Z(\cdot)\big)$, where the (additional) second component $Z(\cdot)$ 
 plays its role to make sure that the first component $Y(\cdot)$ is ${\mathbb F}$-adapted;  
 throughout the work, we assume that $\dbF=\{\mf_t\}_{0\leq t\leq T}$ is generated by the Brownian motion $W$. Its
 numerical approximation requires  {\em significantly} larger computational resources if compared to 
 PDE \rf{intro-1cdb}, or SPDE \rf{intro-1ca} --- as will be motivated in the following paragraph. 
 Next to a different algorithmic setup to approximate $\big(X^*(\cdot), U^*(\cdot)\big)$, it is also that the error analysis 
in Section \ref{ch-open} for a spatio-temporal discretization scheme of \rf{intro-1c}--\rf{intro-1d} differs profoundly from the deterministic case: a key step here is to bound temporal changes for (schemes of) $\big(X^*(\cd), Y(\cd), Z(\cd), U^*(\cd)\big)$ whose derivation requires tools from Malliavin calculus --- where complications again enter which are caused by the driving Brownian motion.

\end{enumerate}

\begin{table}[!htb]

\begin{center}   

\label{table:1} 
\begin{tabular}{|  c | m{5.5cm}| m{5.5cm}|}   
\hline\hline   

\parbox{3cm} {\vspace{0.25cm} 
\textbf{ $\q\,\,\,\,$Method} 
\vspace{0.15cm} 
}
& \parbox{5.5cm} {\vspace{0.2cm} Open-loop approach \vspace{0.2cm}}
& \parbox{5.5cm} {\vspace{0.2cm} Closed-loop approach \vspace{0.2cm}} \\   
\hline   

\textbf{Theory} 
& \parbox{5.5cm} {\vspace{0.2cm} \underline {Pontryagin's maximum principle}: 
FBSPDE \rf{intro-1c} with optimality  condition \rf{intro-1d} \vspace{0.2cm}}
& \parbox{5.5cm} {\vspace{0.2cm} \underline {Stochastic LQ theory}:\\
Riccati equation \rf{Riccati} and PDE \rf{varphi} \vspace{0.2cm}} \\   
\hline 

\parbox{2.8cm} {\vspace{0.2cm} 
 \textbf{Discretization\\
  strategies
 } }
& \parbox{5.5cm} {\vspace{0.2cm} 
{\bf a)} Spatial discretization by FEM for {\bf (SLQ)}: a coupled FBSDE in finite dimensional space \\
{\bf b)} Temporal discretization via the Euler method: 
             a coupled stochastic  difference equation \\
{\bf c)} Decoupling by a gradient descent method 
\vspace{0.2cm}}
& \parbox{5.5cm} {\vspace{0.2cm} 
{\bf a)} Spatial discretization by FEM for {\bf (SLQ)}: the Riccati equation in finite dimensional space \\
{\bf b)} Temporal discretization via the Euler method: 
           difference Riccati equation
 \vspace{0.2cm}} \\   
\hline  


\multirow{2}{*}{\textbf {Rates}}
&  \parbox{5.5cm} {\vspace{0.2cm} 
$\cO\big(h^2+\t^{1/2}\big)$ in Theorems  \ref{rate1}  \\ 
and \ref{rate2} 
\vspace{0.2cm} }
&  \parbox{5.5cm} {\vspace{0.2cm} 
$\cO\big(h^2\ln{\frac 1 h}+\t^{1/2}\big[|\b|+\t^{1/2}\big]\big)$ in \\
Remark \ref{w1023r1} and Theorem \ref{SLQ-pair-rate} 
\vspace{0.2cm} }  \\ 
\cline{2-3}
& \multicolumn{2}{c|}{\parbox{10cm}{\vspace{0.2cm} 
 Improvement to $\cO\big(h^2+\t\big)$ for additive noise 
by Theorem \ref{rate1} for
 space and Theorem \ref{SLQ-pair-rate} for time
\vspace{0.2cm} }}  \\  
\hline   
\end{tabular}   
\end{center}   
\caption{Construction of different numerical schemes for Problem {\bf (SLQ)}}  
\end{table}

If compared to the PDE \rf{intro-1cdb}, we already hinted at the profoundly more complex endeavor to numerically simulate the 
stochastic adjoint equation \rf{intro-1cb}, which in fact is a BSPDE: after a finite element discretization of (\ref{intro-1c})--(\ref{intro-1d}) in space (mesh width $h$), and  in time
(time step $\tau$ of a uniform mesh $I_\t=\{t_{n}\}_{n=0}^N \subset [0,T]$) for $\b=0$
we have
\begin{subequations}\label{w1212e33a}
    \begin{empheq}[left={\empheqlbrace\,}]{align}
      & [\mathds{1}_h - \tau \Delta_h]X^*_{h\t}(t_{n+1})= X^*_{h\t}(t_n)+ \tau U^*_{h\t}(t_n) +\Pi_h\si(t_n)  \D_{n+1}W\, , \label{w1212e33aa}\\
      &[\mathds{1}_h - \tau \Delta_h]Y_{h\t}(t_n) = \me^{t_n}\big[ Y_{h\t}(t_{n+1})- {\tau}  X^*_{h\t}(t_{n+1})\big]\, ,   \label{w1212e33ab} \\
       & X^*_{h\t}(0)=\Pi_h x\, ,\qquad Y_{h\t}(T)=-\a X^*_{h\t}(T)\, ,  \label{w1212e33ac}
    \end{empheq}
\end{subequations}
together with 
\bel{w1003e12a}
U^*_{h\t}(t_n) - Y_{h\t}(t_n)=0 \qquad  0 \leq n \leq N-1 \, ;
\ee
%
see Remark \ref{w829r1} in Section \ref{fbspde-num1z}.
Here, $\mathds{1}_h$ denotes the identity operator on the finite element space,\index{${\mathds 1}_h$} 
 and $\Delta_{n+1}W = W(t_{n+1})-W(t_{n})$;
the operator $\Delta_h$ denotes the discrete Laplacian, and $\Pi_h$ is a projection that is related to finite elements; we refer to Section \ref{ch-open} for a detailed discussion.

The main effort now
is to simulate the {\em conditional expectation} $\me^{t_n}[\cd]$ in \rf{w1212e33ab};
in fact, to accomplish this goal is a well-known task from regression analysis in {\em statistical learning}, where related estimators are available, including {\em e.g.}~`{\tt kNN}-', partitioning-, or kernel estimators; see {\em e.g.}~\cite[Chapter 4]{Gyorfi-Kohler-Krzyzak-Walk02}. These estimators perform well for problems where the dimension  
of the state space  where $X^*_{h\t}(\cd)$ takes its values is {\em low} --- such as stochastic control problems  with {\em small} SDE-systems.
%
%
In the present setting of Problem {\bf (SLQ)} however, it  
is as {\em large} as 
the dimension of the finite element space ${\mathbb V}_h \subset {\mathbb H}^1_0$ that contains admissible states. Therefore, to efficiently cope with the related `curse of dimensionality' we employ a `data-dependent' partitioning estimator in
Section \ref{stati-1} --- the construction of which was motivated by \cite{Dunst-Prohl16} 
--- to approximate the conditional expectation in BSPDE \rf{w1212e33ab}. 
However, the simulation is still costly,  and requires to store large data sets 
to approximate the conditional expectation; and so is an iterative numerical solver based on `gradient descent' for the whole system (\ref{w1212e33a}) which decouples simulations of the `forward' and the `backward part' per iteration. Note that this algorithmic complexity of (\ref{intro-1cd})--(\ref{intro-1dd}) does not appear in the deterministic setting (\ref{intro-1a})--(\ref{spde}) --- where to keep track of ${\mathbb F}$-adaptedness of iterates is {\em not} part of the problem. 
In Section \ref{impl-1a}, we therefore propose a conceptually different   method that avoids the simulation of conditional expectations and hence cuts simulation times to only a fraction of the one
where regression estimators are used --- but this method is only possible for {\em additive } noise ({\em i.e.}, $\b=0$).
As a consequence, to judge the effectiveness of the `open-loop' based numerics in the stochastic setting ($\beta \neq 0$ or
$\si(\cd)\neq 0$) requires to rethink established criteria from the deterministic case (such as efficiency, complexity,...), which {\em e.g.}~often favor
spatio-temporal discretization
based on
\rf{intro-1c}--\rf{intro-1d}.

\ss


An advantage of the `open-loop' based numerics in the stochastic setting is its applicability to even more general stochastic control problems than 
SLQ problems
only. Consequently, the new numerical tools to accurately approximate the solution of BSPDEs in general, and in the context of (\ref{w1212e33a})--(\ref{w1003e12a}) in particular, which are detailed in Section \ref{ch-open} are good to have. 
It is, however, that  specific problems such as Problem {\bf (SLQ)} additionally have a `closed-loop'  representation of  the optimal control with the help of a feedback operator, which may be used to set up numerical schemes 
whose complexity shrinks drastically. The key idea here is to avoid BSPDEs, and use
instead the (terminal backward) {\em stochastic Riccati equation}  
\begin{equation}\label{Riccati}
 \setlength\abovedisplayskip{4pt}
 \setlength\belowdisplayskip{4pt}
\left\{
\begin{array}{ll} 
\ds \cP'(t) + \Delta \cP(t) +  \cP(t)\Delta  +\b^2 \cP(t) + \mathds{1} - \cP^2(t)  = 0 \qquad  t \in [0,T]\,,  \\
 \ns\ds        \cP(T) = \a\mathds{1}\,.
 \end{array}
 \right.  
 \end{equation}
Here $\mathds{1} $ denotes the identity operator on $\dbL^2$.\index{${\mathds 1}$}
Note that (\ref{Riccati}) is a deterministic equation with {\em operator-valued} solution, and the adjective `stochastic' here only refers to its close relation to Problem {\bf (SLQ)}, as is by now common in the literature; see {\em e.g.}~\cite[Chapter 13]{Lv-Zhang21}. 
Once (an approximation of) $\cP(\cd)$ is available, another deterministic linear PDE needs be solved
 \begin{equation}\label{varphi}
  \setlength\abovedisplayskip{4pt}
 \lt\{
 \begin{array}{ll}
\ds \eta'(t)=-\D \eta(t)+\cP(t)\eta(t)- \b\cP(t)\si(t) &\qq t\in [0,T]\,,\\
\ns\ds \eta(T)=0\,,
 \end{array}
 \rt.
  \setlength\belowdisplayskip{4pt}
\end{equation} 
where $\cP(\cd)$ now enters. Then, we may represent the optimal control  
$U^*(\cd)$ of Problem {\bf (SLQ)} 
with the help of the feedback law
 \begin{equation}\label{feedback-0}
 \setlength\abovedisplayskip{4pt}
U^*(t) = - \cP(t) X^*(t)-\eta (t) \qquad  t \in [0,T]\,,
\setlength\belowdisplayskip{4pt}
\end{equation} 
which may be inserted into (\ref{spde}).
Hence, to solve (\ref{intro-1a})--(\ref{spde}) now only requires to solve spatio-temporal approximations of PDEs 
\rf{Riccati}, \rf{varphi}, and of the SPDE \rf{spde}. As a consequence, neither an iterative optimization algorithm (`gradient descent') is needed, nor the simulation of coupled FBSPDEs; it is, however, that a spatio-temporal discretization scheme of the stochastic Riccati equation (\ref{Riccati}) needs to be solved, which requires to solve for and store a finite sequence of large matrices $\{ \cP_n\}_{n=0}^N$.
It turns out, however, that this
is still cheaper (at least for $1d$ and $2d$ simulations) than to solve BSPDEs in the
`open-loop' approach --- as comparative computational studies will show.  

\ss

We conclude this introduction with our own appreciation of the young research field: `numerical control theory for SPDEs' --- to which we wish to contribute with this work. It is (still a small) neighbor of the established field of `numerical control theory for PDEs', whose continuing interest is nourished by practical applications. In this young area of numerical stochastic analysis, the last decade has seen only 
a few research papers, which may be due to the subject matter that combines tools from numerics, simulation, stochastic and deterministic PDEs. However,  the practical relevancy of this field is evident,
and we expect a more rapid construction of provably accurate, efficient algorithms in the future; we conjecture this
to happen also due to the related
developments in theory, {\em i.e.}: `mathematical control theory for SPDEs', which {\em e.g.}~lead to the recent publication \cite{Lv-Zhang21}. 


The work consists of two parts, each of which contains numerical schemes and their error analysis
to illustrate their complexity:
\begin{enumerate}[(a)]
\item  `open-loop' based numerics in Section \ref{ch-open}: We derive optimal rates of convergence for the  spatio-temporal discretization 
of the optimality system (\ref{intro-1c})--(\ref{intro-1d}), which is a coupled FBSPDE. 
In the second step, we decouple the computation of iterates in the course of a `gradient descent method',
which itself precedes an exact computation of conditional expectations ($\b=0$) or a regression analysis ($\b\neq 0$).
The mathematical tools needed to accomplish this goal are Malliavin calculus, (stochastic) Riccati equation analysis, (discrete) semigroup theory, and regression analysis.  

\item `closed-loop' based numerics in Section \ref{ch-closed}: the tools needed here include optimal control analysis, the (stochastic) Riccati equation, and (discrete) semigroup theory.

\end{enumerate} 

\section{Preliminaries}\label{pre}

Section \ref{not1} introduces notations that will be used throughout the work, and collects tools to discretize Problem {\bf (SLQ)}, in particular. 
Section \ref{spde-theory} recalls bounds in strong norms for the solutions to 
linear SPDEs and BSPDEs, which are needed
for the error analysis in Sections \ref{ch-open} and \ref{ch-closed}.
The numerical schemes in these two sections base on two different concepts: `open-loop approach' vs.~`closed-loop approach', which are recalled in Section \ref{se-openclosed}.

\subsection{Notations, assumptions,  methods and tools}\label{not1}

{\bf a.~Spaces and operators.} 
Let $D$\index{$D$}  be a bounded domain with smooth boundary or  
a convex polyhedral domain in $\dbR^d$, $d=1,2,3$.
By $\| \cdot \| $ resp.~$(\cdot, \cdot)_{\dbL^2} $ 
we denote the norm resp.~the scalar product in Lebesgue space ${\mathbb L}^2 \deq  L^2(D)$\index{${\mathbb L}^2$}.
By $\|\cd\|_{\dbH^{-1}}$,  $\|\cd\|_{\dbH_0^1}$, and $\|\cd\|_{\dbH^2}$, we denote norms in Sobolev spaces ${\mathbb H}^{-1}\deq H^{-1}(D)$\index{${\mathbb H}^{-1}$}, ${\mathbb H}^1_0\deq H_0^1(D)$\index{${\mathbb H}_0^1$}, and $\dbH^2\deq H^2(D)$\index{${\mathbb H}^2$} respectively. 
The duality pairing  between $\dbH^{-1}$ and $\dbH_0^1$ is denoted by $\lan \cd,\cd\ran$.

Fix $T>0$, and
let $(\Omega, {\mathcal F}, {\mathbb F}, {\mathbb P})$ be a complete filtered probability space, where 
${\mathbb F}=\{\mf_t\}_{t\in[0,T]}$ is the natural filtration generated by the ${\mathbb R}$-valued Brownian motion $W$, which is augmented by 
all the ${\mathbb P}$-null sets. 
Denote by $\me^t[\xi]$ the conditional expectation $\me\big[\xi\,\big|\,\mf_t\big]$\index{${\mathbb E}^t[\xi]$}.

For a given Hilbert space $\big(\dbH , ( \cdot, \cdot)_{\dbH }\big)$,
we set 
\begin{equation*}
\setlength\abovedisplayskip{3pt}
\dbS(\dbH)\deq \big\{\cP\in \cL(\dbH)\,|\,\cP \mbox{ is symmetric}\big\}\,,
\setlength\belowdisplayskip{3pt}
\end{equation*}
\index{${\mathbb S}(\dbH)$}
and
\begin{equation*}
\setlength\abovedisplayskip{3pt}
\setlength\belowdisplayskip{3pt}
\dbS_+(\dbH)\deq \big\{\cP\in \dbS(\dbH)\,|\, (\cP\eta,\eta)_{\dbH}\geq 0\q \forall\,\eta\in \dbH\big\}\,.
\end{equation*}
\index{${\mathbb S}_+(\dbH)$}
For any $t\in [0,T]$,  let
$L_{\mf_t}^2(\Omega;\dbH)\deq
L^2(\Omega,\mf_t,\dbP;\dbH)$\index{$L^2_{\mf_t}(\O;\dbH)$},  and
\begin{eqnarray*}
&& L^2_\dbF(0,T;\dbH) \deq \Big\{
\psi :(0,T) \!\times\! \Omega\!\to\dbH \,\Big| \,
\psi(\cd)\hbox{
is $\dbF$-adapted}\mbox{ and }\me\Big[\int_0^T\|\psi(t)\|_\dbH^2\rd t\Big] <\infty\Big\}\,,
\index{$L^2_{\dbF}(0,T;\dbH)$}\\
&& L^2_\dbF\big(\O;C([0,T];\dbH)\big) \deq \Big\{
\psi \in L^2_\dbF(0,T;\dbH) \,\Big| \,
\psi(\cd)\hbox{
admits a continuous path }\\ 
&&\qq\qq\qq\qq\qq  \mbox{ and } \sup_{t\in[0,T]}\me\big[\|\psi(t)\|_\dbH^2\big]<\infty\Big\}\,,
\index{$L^2_{\dbF}\big(\Omega;C([0,T];\dbH)\big)$}\\
&& C_\dbF\big([0,T];L^2(\O;\dbH)\big) \deq \Big\{
\psi \in L^2_\dbF(0,T;\dbH) \,\Big| \,
\psi(\cd)\hbox{
is $\dbF$-adapted }\\ 
&&\qq\qq\qq\qq\qq  \mbox{ and } \f(\cd): [0,T]\to L^2_{\mf_T}(\O;\dbH) \mbox{ is continuous}\Big\}\,,\index{$C_\dbF\big([0,T];L^2(\O;\dbH)\big)$}\\
&& D_\dbF\big([0,T];L^2(\O;\dbH)\big) \deq \Big\{
\psi \in L^2_\dbF(0,T;\dbH) \,\Big| \,
\psi(\cd)\hbox{
is $\dbF$-adapted }\\ 
&&\qq\qq\qq\qq\qq  \mbox{ and } \f(\cd): [0,T]\to L^2_{\mf_T}(\O;\dbH) \mbox{ is c\`adl\`ag}\Big\}\,.\index{$D_\dbF\big([0,T];L^2(\O;\dbH)\big)$}
\end{eqnarray*}
For any $a\,,b\in\dbR$, set
\bel{w1025e1}
\setlength\abovedisplayskip{4pt}
\setlength\belowdisplayskip{4pt}
a\vee b\deq \max\{a\,,b\}\,, \qq a\wedge b\deq \min\{a\,,b\}\,.\index{$a\vee b\,,a\wedge b$}
\ee

Below, let $\cC$ be a generic positive constant which may depend on data of the specific problems.

\ms

Let $\D: \dbH_0^1\cap\dbH^2\to \dbL^2$ be the Dirichlet realization of the  Laplace operator.  
Naturally, $\D$ generates the analytic semigroup $E(\cd) \equiv \{E(t)\,|\, t \geq 0\}$,
 with $E(t)=e^{ t\D }$ for $t\geq 0$\index{$E(t)$}. Moreover,
 there exist an increasing sequence of eigenvalues $\{\l_n\}_{n\geq 1}$ and an orthonormal basis of eigenvectors $\{\f_n\}_{n\geq 1}$ in $\dbL^2$
 such that $\D\f_n=-\l_n\f_n$ and
\begin{equation*}
\setlength\abovedisplayskip{4pt}
\setlength\belowdisplayskip{4pt}
0<\l_1\leq\l_2\leq\cds(\to\infty)\,.
\end{equation*}
By the spectral decomposition for $\D$, we can define fractional powers $(-\D)^\g$ for $\g\geq 0$ by
\begin{equation*}
\setlength\abovedisplayskip{4pt}
\setlength\belowdisplayskip{4pt}
(-\D)^\g x=\sum_{n=1}^\infty \l_n^\g(x,\f_n)_{\dbL^2}\f_n\,,
\end{equation*}
where $x\in D\big((-\D)^\g \big)\deq \big\{x\in \dbL^2\,\big|\, \sum_{n=1}^\infty \l_n^{2\g}(x,\f_n)_{\dbL^2}^2<\infty\big\}$.
Also, we introduce
$\dot {\mathbb H}^\g\deq D\big((-\D)^{\g/2} \big) $,\index{$\dot {\mathbb H}^\g$}
and have $\|x\|_{\dot\dbH^\g}=\|(-\D)^{\g/2}x\|$ for any $x\in \dot {\mathbb H}^\g$. Obviously, $\|\cd\|=\|\cd\|_{\dot\dbH^0}$.
Following \cite{Fujita-Suzuki91}, we know that
\begin{equation*}
\setlength\abovedisplayskip{4pt}
\setlength\belowdisplayskip{4pt}
\dot\dbH^1=\dbH_0^1\,,\qq \dot\dbH^2=\dbH_0^1\cap\dbH^2\,.
\end{equation*}

For $\g>0$, we consider the set
\begin{equation*}
\setlength\abovedisplayskip{4pt}
\setlength\belowdisplayskip{4pt}
\dot\dbH^{-\g}\deq\Big\{x=\sum_{n=1}^\infty x_n \f_n\,\Big|\, x_n\in\dbR, n=1,2,\cds, \mbox{ such that } 
 \|x\|_{\dot\dbH^{-\g}}^2=\sum_{n=1}^\infty \l_n^{-\g} x_n^2<\infty \Big\}\,.
\end{equation*}
For any $x=\sum_{n=1}^\infty x_n \f_n\in \dot\dbH^{-\g}$, we also can define the fractional power of $-\D$ for negative exponents by 
\begin{equation*}
\setlength\abovedisplayskip{4pt}
\setlength\belowdisplayskip{4pt}
(-\D)^{-\g/2}x=\sum_{n=1}^\infty \l_n^{-\g/2}x_n\f_n\,.
\end{equation*}
By \cite[Theorem B8]{Kruse14}, $\dot \dbH^{-\g}$ is a separable Hilbert space, which is isometrically isomorphic to $(\dot \dbH^\g)'$.

The following properties of the semigroup $E(\cd)$ 
will be utilized frequently throughout this work. 

\bl{w228l1}
For the analytic semigroup $E(\cd)$, there exists a constant $\cC$ such that
the following properties hold true:
\begin{subequations}
    \begin{empheq}[left={\empheqlbrace\,}]{align}
      &\big\|(-\D)^{\g} E(t) \big\|_{\cL(\dbL^2)}\leq \cC t^{-\g} &&\forall\,\g\geq 0\,, t> 0\,; \label{w228e2}\\
      &\big\|(-\D)^{-\g}\big(E(t)-\mathds{1}\big) \big\|_{\cL(\dbL^2)} \leq \cC t^{\g} &&\forall\,\g\in[0,1]\,, t\geq 0\,;  \label{w228e3} \\
       &\|E(t)-E(s)\|_{\cL(\dbL^2)}\leq \cC (t-s)^{\g}s^{-\g} && \forall\g\in[0,1]\,,\, s\,,t>0\,, s<t\,.  \label{w320e01}
    \end{empheq}
\end{subequations}
%
%
%
\el

The assertions \rf{w228e2} and \rf{w228e3} can be found in \cite{Pazy83}, and  then \rf{w320e01} can be 
derived directly.

\ms

In this work, we make  the following assumptions on further data that appear in problem (\ref{intro-1a})--(\ref{spde}). \\
{\bf Assumption (A)}: \label{assump} 
Let $x\in  \dbH_0^1\cap\dbH^2$, and $\si\in C( [0,T];\dbH_0^1)\cap L^2(0,T;\dbH_0^1\cap\dbH^2)$. There exists a positive constant $L_{\si,\a}$, 
for $\a=1/2\,,1$ such that
\begin{equation*}
\setlength\abovedisplayskip{4pt}
\setlength\belowdisplayskip{4pt}
\|\si(t)-\si(s)\|\leq L_{\si,\a} |t-s|^{\a} \qq \forall\, t,s \in[0,T]\,.
\end{equation*}

\ms

{\bf b.~Malliavin derivative.}  We use Malliavin derivatives of involved processes in the BSPDE \rf{fbspdeb} below to quantify 
temporal changes
of $\big(X^*(\cd), U^*(\cd)\big)$ in Problem {\bf (SLQ)}; see {\em e.g.}~Lemma \ref{w1008l2} in Section \ref{se-openclosed}.
Hence, we briefly recall here the definition and those results of the Malliavin derivative
of processes that will be needed below. For further details, we refer to \cite{Nualart06,ElKaroui-Peng-Quenez97}. 

We define
$\dbW(\cd):L^2(0,T;{\mathbb R})\to L^2_{\cF_{T}}(\Omega;{\mathbb R})$ by
\begin{equation*}
\setlength\abovedisplayskip{4pt}
\setlength\belowdisplayskip{4pt}
\dbW(g)=\int_0^{T}g(t) \, {\mathrm d}W(t)\, .
\end{equation*}
For $\ell \in {\mathbb N}$, we denote by $C_p^\infty(\dbR^\ell)$ the {space} of all smooth functions $f:\dbR^\ell \to \dbR$ such that $f(\cd)$ and all of its partial derivatives have polynomial growth.
Let  $\cQ$ be the {set} of ${\mathbb R}$-valued  random variables of the form
\begin{equation*}
\setlength\abovedisplayskip{4pt}
\setlength\belowdisplayskip{4pt}
F=f\big(\dbW(g_1),\dbW(g_2),\cdots,\dbW(g_\ell) \big) 
\end{equation*}
for some $f(\cd) \in C_p^{\infty}(\dbR^\ell)$, $\ell \in {\mathbb N}$, and   $g_1(\cd),\cds,g_\ell(\cd)\in L^2(0,T;{\mathbb R})$.
For any $F\in \cQ$ we define its  ${\mathbb R}$-valued Malliavin derivative process
$DF \equiv \{ D_\th F\,|\, 0 \leq \th \leq T\}$  via
\begin{equation*}
\setlength\abovedisplayskip{4pt}
\setlength\belowdisplayskip{4pt}
D_\th F=\sum\limits_{i=1}^\ell \frac{\partial f}{\partial x_i} \big(\dbW(g_1),\dbW(g_2),\cdots,W(g_\ell)\big)g_i(\th)\, .
\end{equation*}
In general, we define the $k$-th iterated derivative of $F$ by $D^kF=D(D^{k-1}F)$, for any
$k\in {\mathbb N}$. Note that for any $\theta \in[0,T]$, $D_\th F$ is $\mf_T$-measurable; and if $F$ is 
$\mf_t$-measurable, then $D_\th F=0$ for any $\th\in (t,T]$. 

Now we can extend the derivative operator to $\dbH$-valued random variables, 
where $\big(\dbH, (\cd,\cd)_{\dbH}\big)$ is a
real separable Hilbert space. For any $k\in \dbN$, and $u$ in the set of $\dbH$-valued variables:
\begin{equation*}
\setlength\abovedisplayskip{4pt}
\setlength\belowdisplayskip{4pt}
\cQ_\dbH= \Big\{ u=\sum_{j=1}^n F_j \phi_j \,\Big|\, F_j\in \cQ,\,\phi_j\in \dbH,\, n\in \dbN  \Big\}\, ,
\end{equation*}
we can define the $k$-th iterated derivative of $u$ by
$
D^k u=\sum_{j=1}^n D^kF_j\otimes \phi_j\, .
$
For $p \geq 1$, we define the norm $\|\cdot\|_{k,p}$  via
\begin{equation*}
\setlength\abovedisplayskip{4pt}
\setlength\belowdisplayskip{4pt}
\|u\|_{k,p} \deq \Big(\dbE \big[\|u\|_\dbH ^{p}\big] +\sum_{j=1}^k \me \big[\lt\|D^ju\rt\|_{\big({ L^2(0,T;{\mathbb R})}\big)^{\otimes j}\otimes \dbH}^p\big] \Big)^{1/ p}\, .
\end{equation*}
Then $\dbD^{k,p}(\dbH)$ is the completion of $\cQ_\dbH$ under the norm $\|\cdot\|_{k,p}$.\index{${\mathbb D}^{k,p}(\mathbb{H})$}

\ms

{\bf c.~Discretization in space -- the finite element method.} We partition the bounded domain $D \subset {\mathbb R}^d$ via a regular triangulation ${\mathcal T}_h$ into elements $K$ with maximum mesh size
$h \deq  \max \{ {\rm diam}(K)\,|\, K \in {\mathcal T}_h\}$, and consider the space
\begin{equation*}
\setlength\abovedisplayskip{4pt}
\setlength\belowdisplayskip{4pt}
{\mathbb V}_h \deq   \big\{\phi \in {\mathbb H}^1_0 \,\big|\, \phi |_K \in {\mathbb P}_1(K) \quad \forall\, K \in {\mathcal T}_h \big\}\,, \index{${\mathbb V}_h $}
\end{equation*}
where ${\mathbb P}_1(K)$ denotes the space of {polynomials} of degree $1$; see 
{\em e.g.}~\cite{Brenner-Scott08}.
We define the discrete Laplacian $\Delta_h: {\mathbb V}_h \rightarrow {\mathbb V}_h$ by $(-\Delta_h \f_h, \phi_h)_{{\mathbb L}^2} = (\nabla \f_h, \nabla \phi_h)_{{\mathbb L}^2}$\index{$\Delta_h$} for all $\f_h, \phi_h \in {\mathbb V}_h$, the generalized  $\dbL^2$-projection $\Pi_h: \dbH^{-1} \rightarrow {\mathbb V}_h$ by $(\Pi_h x, \phi_h)_{\dbL^2}= 
\lan x, \phi_h\ran$\index{$\Pi_h$} for all $x \in \dbH^{-1},\, \phi_h \in {\mathbb V}_h$, and the Ritz-projection $\cR_h:\dbH_0^1\to \dbV_h$
by $(\nb[\cR_h x-x], \nb\phi_h)_{\dbL^2}=0 $\index{${\mathcal R}_h$} for all $x \in\dbH_0^1,\,\phi_h\in \dbV_h$.  
It is evident when $x \in\dbL^2$, that $\Pi_h x$ is the standard $\dbL^2$-projection; see {\em e.g.}~\cite{Thomee06}.
Via definitions of $\Pi_h\,,\cR_h$, it is easy to get that
\bel{w1024e1}
\setlength\abovedisplayskip{4pt}
\setlength\belowdisplayskip{4pt}
\lt\{
\begin{array}{ll}
\ds \|\Pi_h x\|\leq \|x\| \qq & \forall\, x\in \dbL^2\,,\\
\ns\ds \|\nb \cR_h x\|\leq \|\nb x\|\qq\q & \forall \, x \in \dbH_0^1\,.
\end{array}
\rt.
\ee

Throughout this work, we fix a constant $h_0\in(0,1)$, such that $h\in(0,h_0]$, and use the notation
$\cC_{h_0}=\frac{\ln{T}}{\ln{\frac 1{h_0}}}$. 


\br{w1006r1}
Similar to $\D$, the finite rank operator $\D_h$ also generates a semigroup $E_h(\cd)\equiv\{E_h(t)\,|\,t\geq 0\}$, where $E_h(t)=e^{t\D_h }$\index{$E_h(t)$}.
Besides, $E_h(\cd)$ also enjoys the properties stated in Lemma \ref{w228l1}. See {\em e.g.}~\cite{Fujita-Suzuki91}.
Furthermore,
 there exist an increasing sequence of eigenvalues $\{\l_{h,n}\}_{n= 1}^{\dim(\dbV_h)}$, and an orthonormal basis
 of eigenvectors $\{\f_{h,n}\}_{n= 1}^{\dim(\dbV_h)}$ in $\dbL^2$
 such that $\D_h\f_{h,n}=-\l_{h,n}\f_{h,n}$, where
\begin{equation*}
\setlength\abovedisplayskip{4pt}
\setlength\belowdisplayskip{4pt}
0<\l_{h,1}\leq\l_{h,2}\leq\cds\leq \l_{h,\dim(\dbV_h)}\,.
\end{equation*}
By the spectral decomposition for $\D_h$, we can define fractional powers $(-\D_h)^\g$ for $\g\in\dbR$ by the following:
\begin{equation*}
\setlength\abovedisplayskip{4pt}
\setlength\belowdisplayskip{4pt}
(-\D_h)^\g x=\sum_{n=1}^{\dim(\dbV_h)} \l_{h,n}^\g(x,\f_{h,n})_{\dbL^2}\f_{h,n} \,,
\end{equation*}
where $x\in D\big((-\D_h)^\g \big)\deq \big\{x\in \dbV_h\,\big|\, \sum_{n=1}^{\dim(\dbV_h)} \l_{h,n}^{2\g}(x,\f_{h,n})_{\dbL^2}^2<\infty\big\}$.
Also, we introduce
$\dot \dbH^\g_h\deq D\big((-\D_h)^{\g/2} \big) $,\index{$\dot {\mathbb H}^\gamma_h$}
and define $\|x\|_{\dot \dbH^\g_h}=\|(-\D_h)^{\g/2}x\|$ for any $x\in \dot \dbH^\g_h$. 
Obviously for any $x\in \dbV_h$, $\|x\|_{\dot \dbH^0_h}=\|x\|$ and $\|x\|_{\dot \dbH^1_h}=\|x\|_{\dot \dbH^1}$.

\er

The following result lists some properties on $\Pi_h$ and $\cR_h$, which will be frequently applied.

\bl{w1017l1}
There exists a constant $\cC$ such that
\begin{subequations}\label{w1022e1}
    \begin{empheq}[left={\empheqlbrace\,}]{align}
      &\|x-\Pi_hx\| +h  \|x-\Pi_hx\|_{\dbH_0^1}\leq \cC h^\g \|x\|_{\dot\dbH^{\g}} & &\forall\,x\in\dot\dbH^\g\,, \g=1,2\,, \label{w1022e1a}\\
&  \|x-\cR_hx\| +h  \|x-\cR_hx\|_{\dbH_0^1}\leq \cC h^\g \|x\|_{\dot\dbH^{\g}} &&\forall\,x\in\dot\dbH^\g\,, \g=1,2\,, \label{w1022e1b}\\
 &\|x-\Pi_hx\|_{\dot \dbH^{-1}}\leq \cC h^2\|x\|_{\dbH_0^1}  && \forall\,x\in\dbH_0^1\,, \label{w1022e1c}\\
 &\|\D_h\Pi_h x\|\leq \cC \|x\|_{\dot\dbH^2 }  && \forall\,x\in \dot\dbH^2\,, \label{w1022e1d}\\
 & \|(\Pi_h\D-\D_h\Pi_h)x\|_{\dot\dbH^{-\g}_h} \leq \cC h^\g\|x\|_{\dot\dbH^2} && \forall\,x\in \dot\dbH^2\,, \g=1,2\,. \label{w1022e1e}
\end{empheq}
\end{subequations}

\el
\begin{proof}
{\bf (1)} The first two estimates are standard; see {\em e.g.}~\cite{Thomee06,Brenner-Scott08}. For the assertion \rf{w1022e1c}, we use
$\dot\dbH^{-1}=\dbH^{-1}\deq \big(\dbH_0^1\big)^*$, and conclude as follows:
\begin{equation*}
\setlength\abovedisplayskip{4pt}
\setlength\belowdisplayskip{4pt}
\bal
\|x-\Pi_hx\|_{\dot \dbH^{-1}}
\!=\! \sup_{\f\in\dbH_0^1} \frac{( x-\Pi_hx, \f-\Pi_h\f  )_{\dbL^2}}{\|\f\|_{\dbH_0^1}}
\!\leq\! \cC h^2 \|x\|_{\dbH_0^1}  \sup_{\f\in\dbH_0^1} \frac{\|\f\|_{\dbH_0^1}}{\|\f\|_{\dbH_0^1}}
\!\leq\! \cC h^2  \|x\|_{\dbH_0^1}\,.
\eal
\end{equation*}

{\bf (2)} For any  $x\in\dbH_0^1\cap \dbH^2$ and $\f\in \dbL^2$, by using
the identity
\bel{w1022e2}
\setlength\abovedisplayskip{4pt}
\setlength\belowdisplayskip{4pt}
(-\Delta_h {\mathcal R}_h x, \Pi_h \varphi)_{\dbL^2} = (\nabla {\mathcal R}_h x, \nabla \Pi_h \varphi)_{\dbL^2} =
(\nabla x, \nabla \Pi_h \varphi)_{\dbL^2} = (-\Delta x, \Pi_h \varphi)_{\dbL^2}\,,
\ee
which follows from the definition of ${\mathcal R}_h$  and integration by parts.
Then relying on 
an inverse estimate ({\em cf.}~\cite{Brenner-Scott08}), assertions \rf{w1022e1a}--\rf{w1022e1b} and
\rf{w1024e1} on ${\mathbb L}^2$-stability of $\Pi_h$, we conclude that
\begin{equation*}
\setlength\abovedisplayskip{4pt}
\setlength\belowdisplayskip{4pt}
\bal
(\D_h\Pi_h x, \Pi_h \f )_{\dbL^2}
&=(\D x, \Pi_h \f )_{\dbL^2}-(\nb  [\Pi_h-\mathds{1}]x, \nb \Pi_h \f )_{\dbL^2} +(\nb  [\cR_h-\mathds{1}]x, \nb \Pi_h \f )_{\dbL^2}\\
&\leq \|\D x\| \|\Pi_h\f \|+\cC h \|x\|_{\dot \dbH^2} \frac 1 h \|\Pi_h\f \|\\
&\leq \cC \|x\|_{\dot\dbH^2}  \|\f \|\,,
\eal
\end{equation*}
which yields \rf{w1022e1d}.

\ms
{\bf (3)} By \rf{w1022e2}, it follows that for any $x\in \dot\dbH^2$, $\D_h\cR_hx=\Pi_h\D x$. Then
\begin{equation*}
\setlength\abovedisplayskip{4pt}
\setlength\belowdisplayskip{4pt}
\bal
(-\D_h)^{-1}(\Pi_h\D-\D_h\Pi_h)x=(-\D_h)^{-1}(\D_h\cR_h-\D_h\Pi_h)x=(\cR_h-\mathds{1})x+(\mathds{1}-\Pi_h)x\,,
\eal
\end{equation*}
which, together with \rf{w1022e1a}--\rf{w1022e1b}, leads to the assertion \rf{w1022e1e} for $\g=2$.
The estimate for $\g=1$ can be derived similarly.
\end{proof}

For errors committed by 
the finite element method, we define the error operator
\bel{w1022e4}
\setlength\abovedisplayskip{4pt}
\setlength\belowdisplayskip{4pt}
G_h(\cd)\deq E(\cd)-E_h(\cd)\Pi_h\,.\index{$G_h(t)$}
\ee
The following result will be frequently used below; see {\em e.g.}~\cite[Lemma 3.1]{Tambue-Mukam19}.

\bl{error-G_h} 
There exists a constant $\cC$, such that
the following estimate for the error operator $G_h(\cd)$ holds: for any $h\in(0,1]$, $t\in(0,T]$,
\bel{w1115e1}
\setlength\abovedisplayskip{4pt}
\setlength\belowdisplayskip{4pt}
\|G_h(t)x\|\leq \cC h^\g t^{-(\g-\rho)/2}\|x\|_{\dot \dbH^\rho} \qq \forall\, x\in \dot \dbH^\rho\,, 0\leq \rho\leq \g\leq 2\,.
\ee

%
%
%
%
%

\el

\ms

{\bf d.~Discretization in time -- interpolation and discrete semigroup $E_{h,\tau}(\cdot)$.}
We denote by $I_\tau = \{ t_{n}\}_{n=0}^N \subset [0,T]$\index{$I_\tau$} a uniform time mesh with step size 
$\tau \deq T/N$\index{$\tau$} and $t_n=n\t$, and related Wiener increments by $\D_nW\deq W(t_n)-W(t_{n-1})$ for all $n=1,\cds,N$.
For a given time mesh $I_\t$, we can define piecewise constant functions $\mu(\cd),\,\nu(\cd),\,\pi(\cd)$\index{$\mu(\cd)$} \index{$\nu(\cd)$} \index{$\pi(\cd)$} by
\bel{w827e1}
\setlength\abovedisplayskip{4pt}
\setlength\belowdisplayskip{4pt}
\mu(t)= t_{n+1}  \, ,\qquad  \nu(t)= t_n\,,\qq \pi(t)= n+1\,, 
\ee
where $t \in [t_n, t_{n+1}),\,n=0,1,\cds,N-1$;
we will use the (piecewise constant) operator $\Pi_\t: C([0,T];\dbH)\to D([0,T];\dbH)$\index{$\Pi_\tau$}  defined by 
\bel{w828e1}
\setlength\abovedisplayskip{4pt}
\setlength\belowdisplayskip{4pt}
\Pi_\t f(t)=f(t_n)  \qquad \forall\, t \in [t_n,t_{n+1}) \qquad n=0,1,\cds, N-1\, .
\ee

Throughout this work, we fix some constant $\t_0\in (0,1)$ such that $\t\in (0,\t_0]$ and introduce a related
constant $\cC_{\t_0}=\frac{\ln{T}}{\ln{\frac 1 {\t_0}}}$. 

\ms

To deduce convergence rates for numerical schemes in later sections on
a uniform partition $I_\t$, we introduce
a discrete version $E_{h,\tau}(\cd) \equiv \{ E_{h,\tau}(t)\,|\, t \geq 0\}$ of the semigroup $E_h(\cd)$ via 
\begin{equation*}
\setlength\abovedisplayskip{4pt}
\setlength\belowdisplayskip{4pt}
E_{h,\t}(t)= A_0^n\qq \forall\,t\in [t_{n-1},t_n) \qq n=1,2,\cds,N\,,\index{$E_{h,\tau}(t)$}
\end{equation*}
where $A_0=(\mathds{1}_h-\t \D_h)^{-1}$,\index{$A_0$}
and set
\begin{equation*}
\setlength\abovedisplayskip{4pt}
\setlength\belowdisplayskip{4pt}
G_{\t}(\cd)\deq E_h(\cd)\Pi_h-E_{h,\t}(\cd)\Pi_h\,. \index{$G_\tau(t)$}
\end{equation*}
The following two lemmata bound the distance between $E_h(\cd)$ and $E_{h,\t}(\cd)$.

\bl{w207l1}
There exists a constant $\cC$, such that
for any $t\in (0,T]$, $h\,,\t \in (0,1]$,
\begin{subequations}\label{w1115e2}
   \begin{empheq}[left={\empheqlbrace\,}]{align}
 &\|G_{\t}(t)x\|\leq \cC \t \|\Pi_h x\|_{\dot\dbH^2_h} && \forall\,x\in \dot {\mathbb H}^2\,;\label{w1115e2a}\\
&\|G_{\t}(t)x\|\leq \cC \t^{\g/2} t^{-(\g-1)/2} \|x\|_{\dot\dbH^1} && \forall\,x\in \dot {\mathbb H}^1\,, \g\in[1,2]\,; \label{w1115e2b}\\
 &\|G_{\t}(t)x\|\leq \cC \t^{\g/2} t^{-\g/2} \|x\| && \forall\,x\in \dbL^2\,, \g\in[0,2]\,.  \label{w1115e2c}
\end{empheq}
\end{subequations}

%
%
%
%
\el

\begin{proof}

 For any $t\in [t_{n-1},t_n)$, consider the identity
\begin{equation}\label{w220e1}
\setlength\abovedisplayskip{4pt}
\setlength\belowdisplayskip{4pt}
G_\t (t)x=\big[E_h(t)\Pi_h-E_h(t_n)\Pi_h\big]x+ \big[E_h(t_n)\Pi_h-A_0^n\Pi_h\big]x\,.
\end{equation}
By applying Remark \ref{w1006r1} and \cite[Lemma 3.3 (iv)]{Tambue-Mukam19}, we have
\begin{equation*}
\setlength\abovedisplayskip{4pt}
\setlength\belowdisplayskip{4pt}
\bal
\|G_\t (t)x\|
\!\leq\! \big\|\big[E_h(t)\Pi_h-E_h(t_n)\Pi_h\big]x\big\|+ \big\|\big[E_h(t_n)\Pi_h-A_0^n\Pi_h\big]x\big\|
\!\leq\! \cC \t \|\Pi_h x\|_{\dot\dbH^2_h}\,,
\eal
\end{equation*}
which is the assertion \rf{w1115e2a}.
By Remark \ref{w1006r1} and \cite[Lemma 1]{Mukam-Tambue18}, it follows for the first term of \rf{w220e1} that
\begin{equation*}
\setlength\abovedisplayskip{4pt}
\setlength\belowdisplayskip{4pt}
\bal
\big\|\big[E_h(t)\Pi_h-E_h(t_n)\Pi_h\big]x\big\| 
&=\big \| (-\D_h)^{-\g/2} \big[\mathds{1}_h-E_h(t_n-t)\big] (-\D_h)^{(\g-1)/2}E_h(t) (-\D_h)^{1/2}\Pi_h x \big \|\\
&\leq \cC (t_n-t)^{\g/2}t^{-(\g-1)/2}\|\Pi_h x\|_{\dot \dbH^1_h}\\
&\leq \cC \t^{\g/2}t^{-(\g-1)/2}\|x\|_{\dot \dbH^1}\,.
\eal
\end{equation*}
Relying on \cite[Lemma 3.3 (iii), (iv)]{Tambue-Mukam19} and applying an interpolation argument, we have
\begin{equation*}
\setlength\abovedisplayskip{4pt}
\setlength\belowdisplayskip{4pt}
\bal
\big\| \big[E_h(t_n)\Pi_h-A_0^n\Pi_h\big]x \big\|
\leq  \cC \t^{\g/2}t_n^{-(\g-1)/2}\|x\|_{\dot \dbH^1}
\leq  \cC \t^{\g/2}t^{-(\g-1)/2}\|x\|_{\dot \dbH^1}.
\eal
\end{equation*}
A combination of these two estimates and \rf{w220e1} then leads to the assertion \rf{w1115e2b}. 

The third assertion \rf{w1115e2c} can be derived in the same way.
\end{proof}


%
%
%
%

\subsection{Bounds in strong norms for linear SPDEs and BSPDEs}\label{spde-theory} 

A relevant part of the analysis of Problem {\bf (SLQ)} is to study related 
 linear SPDEs and BSPDEs on a given filtered probability space $(\O,\mf,\dbF,\dbP)$. In the below, we present definitions of their strong and mild solutions, and show their stability properties in strong norms.
 
 \ms
 
{\bf a.~Linear SPDEs.}
We begin with the following
SPDE ($\b\in\dbR$)
\bel{spde1}
\setlength\abovedisplayskip{4pt}
\setlength\belowdisplayskip{4pt}
\lt\{
\begin{array}{ll}
\ds {\rm d}X(t)=\big[\D X(t)+f(t)\big]\rd t+\big[\b X(t)+g(t)\big] \rd W(t) \qq  t \in (0,T]\,,\\
\ns \ds X(0)=x\,,
\end{array}
\rt.
\ee
where 
$f(\cd)\,,g(\cd)$ are $\dbF$-adapted processes, and $x\in\dbL^2$.

\bde{spde-solution}
{\bf (a)} An $\dbF$-adapted, continuous stochastic process $X(\cd)$ is called a {\em strong solution} to \rf{spde1} if
\begin{enumerate}[{\rm(1)}]

\item $X(\cd)\in \dbH_0^1\cap\dbH^2$ for $\ae\, (t\,,\o)\in [0,T]\times \O$ and $\D X(\cd)\in L^1(0,T;\dbL^2)\,\,\,  \as$;

\item For all $t\in[0,T]$,
\begin{equation*}
\setlength\abovedisplayskip{4pt}
\setlength\belowdisplayskip{4pt}
X(t)=x+\int_0^t \big[\D X(s)+f(s)\big]\rd s+ \int_0^t \big[\b X(s)+g(s)\big]\rd W(s) \qq \as
\end{equation*}
\end{enumerate}

{\bf (b)} An $\dbF$-adapted, continuous stochastic process $X(\cd)$ is called a {\em mild solution} to \rf{spde1} if
 $ X(\cd)\in L^2_\dbF(0,T;\dbL^2)$, and
for all $t\in[0,T]$,
\begin{equation*}
\setlength\abovedisplayskip{4pt}
\setlength\belowdisplayskip{4pt}
X(t)\!=\!E(t)x+\int_0^t E(t-s)f(s)\rd s\!+\! \int_0^t E(t-s)\big[\b X(s)+g(s)\big]\rd W(s) \qq \as
\end{equation*}
\ede


If $X(\cd)$ is a strong solution to \rf{spde1}, then $X(\cd)$ is also a mild solution to \rf{spde1}; see {\em e.g.}~\cite{Prevot-Rockner07,Lv-Zhang21}.

\bl{reg-spde1}
Let $x\in \dbH_0^1\cap\dbH^2$,  $f(\cd)\in L^2_\dbF(0,T;\dbH_0^1)$ and $g(\cd)\in L^2_\dbF(0,T;\dbH_0^1\cap\dbH^2)$. Then \rf{spde1} admits a unique strong solution, and there exists a constant $\cC$ such that
\begin{equation}\label{reg-e1}
\setlength\abovedisplayskip{4pt}
\setlength\belowdisplayskip{4pt}
\bal
&\me\Big[\sup_{t\in[0,T]}\|X(t)\|_{\dot\dbH^\g}^2\Big]+\me\Big[\int_0^T\|X(t)\|_{\dot\dbH^{\g+1}}^2\rd t\Big]\\
&\q\leq \cC \bigg\{ \|x\|_{\dot \dbH^\g}^2+\me\Big[\int_0^T\|f(t)\|_{\dot \dbH^{\g-1}}^2+\|g(t)\|_{\dot \dbH^\g}^2\rd t\Big]\bigg\} \qq \g=-1,0,1,2\,.
\eal
\end{equation}
Furthermore, if $g(\cd)\in C_\dbF\big([0,T];L^2(\O;\dbL^2)\big)$, then for any $t\,,s\in [0,T]$ it holds that
\bel{reg-e2}
\setlength\abovedisplayskip{4pt}
\setlength\belowdisplayskip{4pt}
\bal
&\me\big[\|X(t)-X(s)\|^2\big]\\
&\q\leq\! \cC |t\!-\!s| \big[\|x\|_{\dbH_0^1}^2\!+\!\|f(\cd)\|_{L^2_\dbF(0,T;\dbH_0^1)}^2
      \!+\!\|g(\cd)\|_{L^2_\dbF(0,T;\dbH_0^1)\cap C_\dbF([0,T];L^2(\O;\dbL^2))}^2\big]  \,.
\eal
\ee
\el

\begin{proof} 
{\bf (1) Verification of \rf{reg-e1}.}  
Firstly, \cite[Theorem 3.14]{Lv-Zhang21} implies that \rf{spde1} admits a unique mild solution.
Then, we can adopt the spectral decomposition technique to derive 
the existence of a (unique) strong solution and 
validate
 assertion \rf{reg-e1}; see {\em e.g.}~\cite[Chapter 6]{Chow15}. 

\ms
{\bf (2) Verification of \rf{reg-e2}.} 
Starting with \rf{spde1},
for any $t,s\in[0,T]$ with $s\leq t$, as well as 
H\"older's inequality and the It\^o isometry (see {\em e.g.}~\cite{DaProto-Zabczyk92}),  we have
\begin{equation*}
\setlength\abovedisplayskip{4pt}
\setlength\belowdisplayskip{4pt}
\bal
\me\big[\|X(t)-X(s)\|^2\big]
&\leq \cC\Big\{ \big\| \big[E(t)-E(s)\big] x \big\|^2
  +\me\Big[\int_0^s \big\| E(s-\th) \big[E(t-s)-\mathds{1}\big]   f(\th)\big\|^2 \rd \th \Big]\\
&\qq +|t-s|\me\Big[\int_s^t \big\|E(t-\th) f(\th)\big\|^2 \rd \th \Big]\\  
&\qq +\me\Big[\int_0^s \big\| E(s-\th) \big[E(t-s)-\mathds{1}\big]   \big[\b X(\th)+g(\th)\big]\big\|^2 \rd \th \Big]\\
&\qq +\me\Big[\int_s^t \big\| E(t-\th) \big[\b X(\th)+g(\th)\big]\big\|^2 \rd \th \Big]\Big\}\\
&=:\cC \sum_{i=1}^5 I_i\,.
\eal
\end{equation*}
By \rf{w228e3} in Lemma \ref{w228l1}, 
it follows that
\begin{equation*}
\setlength\abovedisplayskip{4pt}
\setlength\belowdisplayskip{4pt}
\bal
I_1\leq \big\| E(s)(-\D)^{-1/2}\big[E(t-s)-\mathds{1}\big] (-\D)^{1/2}x \big\|^2\leq  \cC|t-s|\|x\|_{\dbH_0^1}^2\,.
\eal
\end{equation*}
By \rf{reg-e1} with $\g=0$, and the assumption on $g(\cd)$, 
we have
\begin{equation*}
\setlength\abovedisplayskip{4pt}
\setlength\belowdisplayskip{4pt}
\bal
I_3+I_5\leq \cC|t-s|\big[\|x\|^2+\|f(\cd)\|_{L^2_\dbF(0,T;\dbL^2)}^2+\|g(\cd)\|_{C_\dbF([0,T];L^2(\O;\dbL^2))}^2\big]\,.
\eal
\end{equation*}
Similar to the estimation of $I_1$, we may use
Lemma \ref{w228l1} and \rf{reg-e1} with $\g=0$ to conclude
\begin{equation*}
\setlength\abovedisplayskip{4pt}
\setlength\belowdisplayskip{4pt}
\bal
I_2+I_4
&\leq \cC|t-s| \big[\|f(\cd)\|_{L^2_\dbF(0,T;\dbH_0^1)}^2+\|X(\cd)\|_{L^2_\dbF(0,T;\dbH_0^1)}^2+\|g(\cd)\|_{L^2_\dbF(0,T;\dbH_0^1)}^2\big]\\
&\leq \cC|t-s|\big[\|x\|^2+\|f(\cd)\|_{L^2_\dbF(0,T;\dbH_0^1)}^2+\|g(\cd)\|_{L^2_\dbF(0,T;\dbH_0^1)}^2\big]\,.
\eal
\end{equation*}
That completes the proof of \rf{reg-e2}.
\end{proof}

\ms
{\bf b.~Linear BSPDEs.}
To handle the adjoint equation coming from Pontryagin's  maximum principle which is related to Problem {\bf (SLQ)}, we consider the following linear
BSPDE ($\b\in\dbR$)
\bel{bshe}
\setlength\abovedisplayskip{4pt}
\setlength\belowdisplayskip{4pt}
\lt\{
\begin{array}{ll}
{\rm d}Y(t)= \big[-\D Y(t)-\b Z(t)+f(t) \big] \rd t+Z(t) \rd W(t)  \qq \forall\, t \in [0,T)\, , \\
\ns\ds Y(T)=Y_T\,.
\end{array}
\rt.
\ee

\bde{bspde-solution}
{\bf (a)} An $\dbL^2\times\dbL^2$-valued  stochastic process tuple $\big(Y(\cd),Z(\cd)\big)$ is called a {\em strong solution} to \rf{bshe} if
\begin{enumerate}[{\rm(1)}]

\item $Y(\cd)$ is $\dbF$-adapted and continuous, and $Z(\cd)\in L^2_\dbF(0,T;\dbL^2)$;

\item $Y(\cd)\in \dbH_0^1\cap\dbH^2$ for $\ae (t\,,\o)\in [0,T]\times \O$ and $\D Y(\cd)\in L^1(0,T;\dbL^2)\,\, \as$;

\item For all $t\in[0,T]$,
\begin{equation*}
\setlength\abovedisplayskip{4pt}
\setlength\belowdisplayskip{4pt}
\bal
Y(t)=Y_T+\int_t^T \big[\D Y(s)+\b Z(s)-f(s)\big]\rd s- \int_t^T Z(s)\rd W(s) \qq \as
\eal
\end{equation*}
\end{enumerate}

{\bf (b)} An $\dbL^2\times\dbL^2$-valued  stochastic process tuple $\big(Y(\cd),Z(\cd)\big)$ is called a {\em mild solution} to \rf{bshe} if
\begin{enumerate}[{\rm(1)}]

\item $Y(\cd)$ is $\dbF$-adapted and continuous, and $Z(\cd)\in L^2_\dbF(0,T;\dbL^2)$;

\item For all $t\in[0,T]$,
\begin{equation*}
\setlength\abovedisplayskip{4pt}
\setlength\belowdisplayskip{4pt}
\bal
Y(t)&=E(T-t)Y_T+\int_t^T E(s-t)\big[\b Z(s)-f(s)\big]\rd s\\
&\q- \int_t^T E(s-t)Z(s)\rd W(s) \qq \as
\eal
\end{equation*}
\end{enumerate}
\ede

Similar to SPDEs,
if $\big(Y(\cd), Z(\cd)\big)$ is a strong solution to \rf{bshe}, then $\big(Y(\cd), Z(\cd)\big)$ is also a mild solution to \rf{bshe}; see {\em e.g.}~\cite{Al-Hussein05,Lv-Zhang21}.

%

The following lemma can be derived by a spectral decomposition argument as in the proof of Lemma \ref{reg-spde1}. 

\bl{reg-bshe}
Assume $Y_T\in \dbH_0^1\cap\dbH^2$, and $f(\cd)\in L^2_\dbF(0,T;\dbH_0^1)$. Then \rf{bshe} admits a unique strong solution $\big(Y(\cd),Z(\cd)\big) \in \big(L^2_{{\mathbb F}}\big(\Omega; C([0,T]; {\mathbb H}^1_0)\big) \cap L^2_\dbF(0,T; {\mathbb H}^1_0 \cap {\mathbb H}^2) \big)\times L^2_\dbF(0,T; {\mathbb H}^1_0)$. Furthermore, there exists a constant $\cC$ such that
\begin{equation}\label{vari-2}
\setlength\abovedisplayskip{4pt}
\setlength\belowdisplayskip{4pt}
\bal
&\me\Big[\sup_{t\in[0,T]}\|Y(t)\|_{\dot\dbH^\g}^2\Big]
+\me\Big[\int_0^T\|Y(t)\|_{\dot\dbH^{\g+1}}^2+\|Z(t)\|_{\dot\dbH^{\g}}^2\rd t\Big]\\
&\qq\qq\leq \cC\, \me\Big[\|Y_T\|_{\dot \dbH^\g}^2+\int_0^T\|f(t)\|_{\dot \dbH^{\g-1}}^2\rd t\Big] \qq \g=-1,0,1,2\,.
\eal
\end{equation}
\el

\subsection{Open and closed-loop approaches for Problem {\bf (SLQ)}}\label{se-openclosed}

{\bf a.~Pontryagin's maximum principle.}  At the center of the `open-loop approach' stands {\em Pontryagin's maximum principle} (see {\em e.g.}~\cite{Lv-Zhang21}), which involves the further adjoint processes $\big(Y(\cd),Z(\cd)\big)\in \big(L^2_\dbF\big(\O;C([0,T];\dbH_0^1)\big)\cap L^2_\dbF(0,T;\dbH_0^1\cap\dbH^2)\big)\times L^2_\dbF(0,T;\dbH_0^1)$ next to
the optimal pair $\big(X^*(\cd), U^*(\cd)\big)$. While $Y(\cd)$ is an adjoint process as in the context of necessary optimality conditions for deterministic problems, the additional $Z(\cd)$ makes sure that it is ${\mathbb F}$-adapted as well. Consequently, we look after the quadruple $\big(X^*(\cd), Y(\cd), Z(\cd), U^*(\cd)\big)$ that solves
\begin{subequations}\label{fbspde}
    \begin{empheq}[left={\empheqlbrace\,}]{align}
      &  {\rm d}X^*(t)=\big[\D X^*(t)+U^*(t)\big]\rd t+\big[\b X^*(t)+\si(t)\big] \rd W(t) &  t \in (0,T]\,, \label{fbspdea}\\
& {\rm d}Y(t)= \big[-\D Y(t)-\b Z(t)+ X^*(t)  \big]{\rm d}t+Z(t) {\rm d}W(t)  &  t \in [0,T)\, , \label{fbspdeb}\\
 & X^*(0)=x\,, \q Y(T)=- \a X^*(T)\,,  \label{fbspdec}
\end{empheq}
\end{subequations}
together with the optimality condition
\bel{op-condition1}
\setlength\abovedisplayskip{4pt}
\setlength\belowdisplayskip{4pt}
 U^*(t) - Y(t) =0  \qquad  t \in [0,T]\, .
\ee

\ms

{\bf b.~Stochastic Riccati equation and feedback control.}
By  LQ theory, Problem {\bf (SLQ)} admits a unique optimal pair $\big(X^*(\cd),U^*(\cd)\big)$, which is linked by 
an operator-valued
function $\cP(\cd)$, which solves the {\em stochastic Riccati equation}
\bel{Riccati1}
\setlength\abovedisplayskip{4pt}
\setlength\belowdisplayskip{4pt}
\lt\{
\begin{array}{ll} 
\ds\cP'(t)+\D \cP(t)+\cP(t)\D+\b^2\cP(t)+\mathds{1}-\cP^2(t)=0\qq t\in [0,T]\,,\\
\ns\ds\cP(T)=\a\mathds{1}\,.
\end{array}
\rt.
\ee
The following definition clarifies the notion of a solution to this deterministic terminal value problem.
\bde{Riccati-solution}
We call $\cP(\cd) \in C\big([0,T];\dbS(\dbL^2)\big)$
a {\em mild solution} to \rf{Riccati1} if for any $x \in\dbL^2$ and any $s\in [0,T]$,
\bel{Riccati-sol}
\setlength\abovedisplayskip{4pt}
\setlength\belowdisplayskip{4pt}
\cP(s)x =\a E\big(2(T-s)\big)x+\int_s^T E(\th-s) \big[ \b^2 \cP(\th)+\mathds{1}- \cP^2(\th)\big] E(\th-s) x \rd \th \,.
\ee
\ede
It is shown in \cite[Theorem 2.2]{Lv19} that \rf{Riccati1} admits a unique mild solution 
$\cP(\cd) \in C\big([0,T];\dbS(\dbL^2)\big)$;
furthermore, by 
\begin{equation*}
\setlength\abovedisplayskip{4pt}
\setlength\belowdisplayskip{4pt}
\bal
&\inf_{U(\cd)\in L^2_\dbF(t,T;\dbL^2)} {\mathbb E}\Big[ \int_t^T \big[ \| X\big(s; t,x,U(\cd)\big) \|^2 + \| U(s)\|^2\big] \rd t +
\alpha \| X\big(T; t,x,U(\cd)\big) \|^2\Big] 
= (\cP(t)x, x)_{\dbL^2}\,,
\eal
\end{equation*}
we actually deduce that $\cP(\cd) \in C\big([0,T];\dbS_+(\dbL^2)\big)$,
where $X\big(\cd; t,x,U(\cd)\big) $ solves
%
\begin{equation*}
\setlength\abovedisplayskip{4pt}
\setlength\belowdisplayskip{4pt}
\lt\{
\begin{array}{ll}
\ds {\rm d}X(s)=\big[\D X(s)+U(s)\big]\rd t+\b X(s) \rd W(s) \qq  s \in (t,T]\,,\\
\ns\ds X(t)=x\,.
\end{array}
\rt.
\end{equation*}
%
Relying on $\cP(\cd)$ and $\eta(\cd)$, where the latter solves
 the $\cP(\cd)$-dependent deterministic terminal value PDE \rf{varphi},
we may then characterize the solution $U^*(\cd)$ of Problem {\bf (SLQ)} by the following {\em feedback control law}\begin{equation}\label{feedback1}
\setlength\abovedisplayskip{4pt}
\setlength\belowdisplayskip{4pt}
\bal
U^*(t)=-\cP(t) X^*(t)- \eta(t) \qq t\in[0,T]\,.
\eal
\ee
This  feedback law \rf{feedback1}, together with the optimality condition \rf{op-condition1} and the results in Section 
\ref{spde-theory}, allows to verify improved regularity properties and bounds for 
the optimal pair $\big(X^*(\cd), U^*(\cd)\big)$; see Lemma \ref{w1008l2} in this section. 
Besides,
\rf{feedback1} is crucial to algorithmically settle the `closed-loop approach', which is proposed in Section \ref{ch-closed}.

\ss

With the help of the preceding two parts, we can now show bounds in stronger norms for the solution of FBSPDE (\ref{fbspde})--(\ref{op-condition1}) --- and hence of Problem {\bf (SLQ)}.
\bl{w1008l2}
Suppose that {\bf (A)} holds. 
Then $\big(X^*(\cd),U^*(\cd)\big)$, the optimal pair of Problem {\bf (SLQ)}, belongs to $L^2_\dbF\big(\O;C([0,T];\dbH_0^1)\big)\times L^2_\dbF\big(\O;C([0,T];\dbH_0^1)\big)$, and  there exists a constant $\cC$ such that
\begin{subequations}\label{w1008e5}
    \begin{empheq}[left={\empheqlbrace\,}]{align}
      & \sup_{t\in[0,T]}\Big(\me\big[\|U^*(t)\|^4\big]\Big)^{1/2}
\leq  \cC \big[\|x\|^2+\|\si(\cd)\|^2_{C([0,T];\dbL^2)}\big]\,, \label{w1008e5a}\\
& \me\Big[\!\sup_{t\in[0,T]}\!\|U^*(t)\|^2+\int_0^T\|U^*(t)\|_{\dbH_0^1}^2\rd t\Big]
 \leq \cC \big[\|x\|^2+\|\si(\cd)\|^2_{L^2(0,T;\dbL^2)}\big]\,, \label{w1008e5b}\\
& \me\Big[\sup_{t\in[0,T]}\|X^*(t)\|_{\dbH_0^1}^2+\int_0^T\|X^*(t)\|_{\dbH_0^1\cap\dbH^2}^2\rd t\Big] 
\leq \cC \big[\|x\|_{\dbH_0^1}^2+\|\si(\cd)\|^2_{L^2(0,T;\dbH_0^1)}\big]\,. \label{w1008e5c}
\end{empheq}
\end{subequations}
Furthermore, there exists a constant $\cC$ such that for any $s\,,t\in[0,T] $
\bel{w1008e7}
\setlength\abovedisplayskip{4pt}
\setlength\belowdisplayskip{4pt}
\me\big[\|U^*(t)-U^*(s)\|^2\big]\leq \cC|t-s| \big[\|x\|_{\dbH_0^1}^2+\|\si\|^2_{C([0,T];\dbH_0^1)}\big]\,.
\ee
\el

\begin{proof} {\bf (1) Verification of \rf{w1008e5a}.} 
Relying on the fact that $\cP(\cd) \in C\big([0,T];\dbS(\dbL^2)\big)$, we have
$\sup_{t\in[0,T]}\|P(t)\|_{\cL(\dbL^2)} \leq \cC$.
Then by equation \rf{varphi}, we may use Gronwall's inequality  to derive
\bel{w1016e1}
\setlength\abovedisplayskip{4pt}
\setlength\belowdisplayskip{4pt}
\!\sup_{t\in[0,T]}\!\|\eta(t)\|^2\!+\!\int_0^T\!\big[\|\eta(t)\|_{\dbH_0^1}^2\!+\!\big(\cP(t)\eta(t), \eta(t)\big)_{\dbL^2}\big]\rd t
\!\leq\! \cC \int_0^T\!\|\si(t)\|^2 \rd t\,.
\ee
Here we apply the fact that $\cP(\cd)\in C\big([0,T];\dbS_+(\dbL^2)\big)$.
Then, by inserting the feedback law \rf{feedback1} into SPDE \rf{fbspdea} and applying  Gronwall's 
inequality as well as the Burkholder-Davis-Gundy inequality (BDG inequality, for short),
we can see that
\bel{w1028e6}
\setlength\abovedisplayskip{4pt}
\setlength\belowdisplayskip{4pt}
\bal
&\me\Big[\sup_{t\in[0,T]}\|X^*(t)\|^2\Big]+\me\Big[\int_0^T\|X^*(t)\|_{\dbH_0^1}^2+\big(\cP(t)X^*(t), X^*(t)\big)_{\dbL^2}\rd t\Big]\\
&\q\leq \cC \Big[\|x\|^2+\int_0^T\|\eta(t)\|^2+\|\si(t)\|^2\rd t\Big]\\
&\q\leq \cC \big[\|x\|^2+\|\si(\cd)\|^2_{L^2(0,T;\dbL^2)}\big]\,.
\eal
\ee
On the other hand, we may apply It\^o's formula to $\|X^*(t)\|^4$, the feedback law \rf{feedback1}, and 
the estimate \rf{w1016e1} for $\eta(\cd)$ to get
\begin{equation*}
\setlength\abovedisplayskip{4pt}
\setlength\belowdisplayskip{4pt}
\bal
\sup_{t\in[0,T]}\me\big[\|X^*(t)\|^4\big]
&\leq \cC \Big[\|x\|^4+\int_0^T\|\eta(t)\|^4+\|\si(t)\|^4\rd t\Big]
\leq \cC  \big[\|x\|^2+\|\si(\cd)\|^2_{C([0,T];\dbL^2)}\big]^2\,.
\eal
\end{equation*}
Using \rf{feedback1} once more, we can derive the assertion \rf{w1008e5a}.

\ss
{\bf (2) Verification of \rf{w1008e5b}}
On one hand, we use the feedback law \rf{feedback1} and \rf{w1016e1}, \rf{w1028e6} to get
\begin{equation*}
\setlength\abovedisplayskip{4pt}
\setlength\belowdisplayskip{4pt}
\bal
\me\Big[\sup_{t\in[0,T]}\|U^*(t)\|^2\Big]
&\!\leq\! \cC\Big\{\sup_{t\in[0,T]}\|\cP(t)\|^2_{\cL(\dbL^2)}\me\Big[\sup_{t\in[0,T]}\|X^*(t)\|^2\Big]\!+\!\sup_{t\in[0,T]}\|\eta(t)\|^2\Big\}\\
&\!\leq \!\cC \big[\|x\|^2+\|\si(\cd)\|^2_{L^2(0,T;\dbL^2)}\big]\,.
\eal
\end{equation*}
On the other hand, the optimality condition \rf{op-condition1},  Lemma \ref{reg-bshe} and \rf{w1028e6}, 
yields
\begin{equation*}
\setlength\abovedisplayskip{4pt}
\setlength\belowdisplayskip{4pt}
\bal
\me\Big[\int_0^T\|U^*(t)\|_{\dbH_0^1}^2\rd t\Big]
\leq \cC\,\me\Big[\|X^*(T)\|^2\big]+\int_0^T\|X^*(t)\|^2\rd t\Big] 
\leq \cC \big[\|x\|^2+\|\si(\cd)\|^2_{L^2(0,T;\dbL^2)}\big]\,.
\eal
\end{equation*}
That settles assertion \rf{w1008e5b}.

\ss

{\bf (3) Verification of \rf{w1008e5c}.} 
It is straightforward by Lemma \ref{reg-spde1} with $f(\cd)=U^*(\cd)\,,g(\cd)=\si(\cd)$ and \rf{w1008e5b}.

\ss

{\bf (4) Verification of \rf{w1008e7}.} For any $t\in[0,T]$, we find that $X^*(t)\in \dbD^{1,2}(\dbL^2)$ (see, {\em e.g.}~\cite{Nualart06}), and $D_\th X^*(\cd)$ solves
\bel{d-spde}
\setlength\abovedisplayskip{4pt}
\setlength\belowdisplayskip{4pt}
\lt\{
\begin{array}{ll}
\ds {\rm d}D_\th X^*(t)=\big[\D-\cP(t)\big] D_\th X^*(t)\rd t+\b D_\th X^*(t) \rd W(t) \qq t \in (\th,T]\,,\\
\ns \ds D_\th X^*(\th)=\b X^*(\th) +\si(\th)\,.
\end{array}
\rt.
\ee
Then by applying \cite[Proposition 3.2]{Dou-Lv19}, we conclude that the Malliavin derivative $\big(D_\th Y(\cd), D_\th Z(\cd)\big)$ of 
the solution $\big(Y(\cd),Z(\cd)\big)$ to BPSDE \rf{fbspdeb} solves
\bel{d-bspde}
\setlength\abovedisplayskip{4pt}
\setlength\belowdisplayskip{4pt}
\lt\{
\begin{array}{ll}
\ns \ds {\rm d}D_\th Y(t)= \big[-\D D_\th Y(t)-\b D_\th Z(t)+D_\th X^*(t)  \big]{\rm d}t
+D_\th Z(t) {\rm d}W(t)  \qq t \in [\th,T)\, ,\\
\ns \ds D_\th Y(T)=- \a D_\th X^*(T)\,,\\
\ns \ds D_\th Y(t)=0\,, D_\th Z(t)=0 \qq t\in [0,\th)\,,
\end{array}
\rt.
\ee
and
\bel{w124e1}
\setlength\abovedisplayskip{4pt}
\setlength\belowdisplayskip{4pt}
Z(t)=D_t Y(t)\qq \ae\, t\in[0,T]\,.
\ee
Now Lemmata \ref{reg-bshe} and \ref{reg-spde1} imply that
\bel{w1008e6}
\setlength\abovedisplayskip{4pt}
\setlength\belowdisplayskip{4pt}
\bal
&\sup_{\th \in [0,T]}\bigg\{\sup_{t \in [\th,T]} \me\big[\| D_\th Y(t)\|^2\big] +
\me \Big[\int_\th^T  \| \nb D_\th Y(t)\|^2+
\|D_\th Z(t)\|^2\rd t \Big]\bigg\}\\
&\qquad \leq
\cC \sup_{\th \in [0,T]} \me \Big[ \a^2\|D_\th X^*(T)  \|^2+ \int_\th ^T \| D_\th X^*(t)\|^2\rd t\Big] \\
&\qquad \leq \cC  \sup_{\th \in [0,T]}  \me \big[ \| X^*(\th)  \|^2+  \| \si(\th)\|^2\big] \\
&\qquad \leq \cC  \big[\|x\|^2+\|\si(\cd)\|^2_{C([0,T];\dbL^2)}\big]\,.
\eal
\ee

By the optimality condition  \rf{op-condition1},  
\rf{vari-2}, \rf{w124e1} and \rf{w1008e6}, we can conclude that for $s\leq t$,
\begin{equation*}
\setlength\abovedisplayskip{4pt}
\setlength\belowdisplayskip{4pt}
\bal
&\me\big[\|U^*(t)-U^*(s)\|^2\big]\\
&\q \leq \cC (t-s)  \me\Big[\int_0^T \big[\|Y(t)\|_{\dbH_0^1\cap \dbH^2}^2 +\|Z(t)\|^2+\|X^*(t)\|^2 \big] \rd t \Big]
+\cC\,\me\Big[\int_s^t \|D_\th Y(\th)\|^2 \rd \th \Big]\\
&\q \leq \cC |t-s|  \bigg\{\me\Big[\|X^*(T)\|^2_{\dbH_0^1}+\int_0^T \|X^*(t)\|^2\rd t \Big]
+\sup_{\th\in[0,T]} \me\big[\|D_\th Y(\th)\|^2 \big]\bigg\}\\
&\q \leq \cC|t-s| \big[\|x\|_{\dbH_0^1}^2+\|\si(\cd)\|^2_{C([0,T];\dbL^2)\cap L^2(0,T;\dbH_0^1)}\big]\,.
\eal
\end{equation*}
That settles assertion \rf{w1008e7}.
\end{proof}

\section{Discretization based on the open-loop approach}\label{ch-open}

Different analytical tools from the `open-loop approach' ({\em i.e.},  
Pontryagin's maximum principle), and the `closed-loop approach' ({\em i.e.}, stochastic Riccati equation) were introduced
in Section \ref{se-openclosed} to conclude improved regularity properties of the unique optimal pair $\big(X^*(\cd), U^*(\cd)\big)$ to Problem {\bf (SLQ)}; in particular, they were used to verify Lemma \ref{w1008l2}, which contains bounds for $\big(X^*(\cd), U^*(\cd)\big)$ in strong norms.

This equivalent characterization of a minimizer of Problem {\bf (SLQ)} in discrete form via the related optimality conditions or the feedback-law formulae  
 still holds if related discretization schemes in space and time are considered: from a practical algorithmic view; however, we will see that their usability is very different. In this section,
we focus on the `open-loop approach', and the construction of a related numerical method.

For the numerical approach in this section, 
numerical schemes to approximate FBSPDE \rf{fbspde} with
the optimality condition \rf{op-condition1} are the  crucial strategy point;
once a spatio-temporal discrete formulation of \rf{fbspde}--\rf{op-condition1} has been established (see \rf{w1212e3}--\rf{w1003e12}) which also serves as the optimality condition for a proper discretization of Problem {\bf (SLQ)} (see Problem {\bf (SLQ)$_{h\t}$} related to \rf{w1003e8}--\rf{w1003e7}), a decoupling of the BSPDE from the SPDE (after discretization)  happens through a gradient decent method to simplify computations (see Algorithm \ref{alg1}). For its simulation, however, there remains to compute conditional expectations per iteration step. For the additive noise case ($\b=0$),
an exact formulation for conditional expectations can be derived in Theorem \ref{imple-gdm}, which leads to an algorithm that converges with optimal rate --- and is free from the calculation of conditional expectations.
For the multiplicative noise case, we use (data dependent) regression for its estimation in high-dimensional spaces; see Section \ref{stati-1}.

This part starts with two preparatory  sections: in Section \ref{nume-spde1}, we recall shortly the derivation of convergence rates for spatio-temporal discretization  schemes  of SPDE \rf{spde1} which are well-known; 
Section \ref{nume-bspde1} addresses corresponding goals for BSPDE \rf{bshe}, whose numerical analysis is only 
available in recent research  articles. 
The results in both sections will then be used in the later ones which address Problem {\bf (SLQ)}.

\subsection{Discretization of linear SPDEs and rates}\label{nume-spde1}
We split the derivation of convergence rates for the spatio-temporal discretization of SPDE (\ref{spde1}) into two steps: we begin with an SDE, which results from discretization in space only; related stability bounds for the solution to SDEs may then be used in Section \ref{spdeh-time1} to verify rates of convergence for the complete discretization in space and time.

\subsubsection{Spatial discretization and error estimates for linear SPDEs}\label{num-spde1a} 

We use the notations from part {\bf c.} of Section \ref{not1}.
A finite element method of \eqref{spde1}  
reads: Find
$X_h(\cd) \in L^2_{{\mathbb F}}\big(\Omega; C([0,T]; {\mathbb V}_h)\big)$ that  satisfies 
\begin{equation}\label{spde1-h}
\setlength\abovedisplayskip{3pt}
\lt\{
\begin{array}{ll}
\ds {\rm d}X_h(t)=\big[\Delta_hX_h(t)+  f_h(t)\big] \rd t+ \big[\b X_h(t)+ g_h (t)\big] \rd W(t)\qq  t\in (0,T]\, , \\
\ns\ds X_h(0)= \Pi_hx\, ,
\end{array}
\rt.
\setlength\belowdisplayskip{3pt}
\end{equation}
where $f_h(\cd)\,, g_h(\cd)\in L^2_\dbF(0,T;\dbV_h)$ are approximations of $f(\cd)\,,g(\cd)$, respectively.
Note that for any $t\in[0,T]$, $X_h(t)$ takes values in the finite dimensional space $\dbV_h$. Hence 
problem
\rf{spde1-h} is a system of linear SDEs, and its solvability is obvious.

\bl{w1002l4}
Suppose that $x\in \dbH_0^1\cap\dbH^2\,,f_h(\cd)\,,g_h(\cd)\in L^2_\dbF(0,T;\dbV_h)$, and $X_h(\cd)$ solves \rf{spde1-h}. 
There exists a constant $\cC$ such that for each $\g=-1,0,1,2$,
\begin{equation}\label{w1002e6}
\setlength\abovedisplayskip{3pt}
\setlength\belowdisplayskip{3pt}
\bal
&\me\Big[\sup_{t\in[0,T]}\|X_h(t)\|_{\dot \dbH^\g_h}^2+\int_0^T\|X_h(t)\|_{\dot \dbH^{\g+1}_h}^2\rd t\Big] 
\leq \cC \bigg\{ \|x\|_{\dot \dbH^\g_h}^2+\me\Big[\int_0^T\|f_h(t)\|_{\dot \dbH^{\g-1}_h}^2+\|g_h(t)\|_{\dot \dbH^\g_h}^2\rd t\Big]\bigg\} \,.
\eal
\end{equation}
Furthermore, for a uniform partition $I_\t=\{t_n\}_{n=0}^N\subset[0,T]$, there exists a constant $\cC$ independent of $\t$ such that
\bel{w1002e5}
\setlength\abovedisplayskip{3pt}
\setlength\belowdisplayskip{3pt}
\bal
&\sum_{n=0}^{N-1}\me\Big[\int_{t_n}^{t_{n+1}}\|X_h(t)-X_h(t_n)\|_{\dot\dbH^{\g}_h}^2+\|X_h(t)-X_h(t_{n+1})\|_{\dot\dbH^{\g}_h}^2\rd t \Big]\\
&\qq\leq \cC\t \bigg\{ \|x\|_{\dot\dbH^{\g+1}}^2+\me\Big[\int_0^T\|f_h(t)\|_{\dot\dbH^{\g}_h}^2+\|g_h(t)\|_{\dot\dbH^{\g+1}_h}^2\rd t\Big]\bigg\} \qq \g=0,1\,.
\eal
\ee
\el

\begin{proof}

The estimates \rf{w1002e6} can be proved in the same vein as  in the proof of Lemma \ref{reg-spde1}.
Here, we only bound $X_h(t)-X_h(t_n)$ on the left-hand side of \rf{w1002e5}; 
the remaining part can be derived by a similar argument.
Based on  the It\^o isometry and \rf{w1002e6}, we can obtain
\begin{equation*}
\setlength\abovedisplayskip{3pt}
\setlength\belowdisplayskip{3pt}
\bal
&\sum_{n=0}^{N-1}\me\Big[\int_{t_n}^{t_{n+1}}\|X_h(t)-X_h(t_n)\|_{\dot\dbH^{\g}_h}^2\rd t \Big]\\
&\q\leq \cC \sum_{n=0}^{N-1}\int_{t_n}^{t_{n+1}}\me\Big[\t\int_{t_n}^{t_{n+1}}\|\D_h X_h(s)\|_{\dot\dbH^{\g}_h}^2+\|f_h(s)\|_{\dot\dbH^{\g}_h}^2\rd s
+\int_{t_n}^{t_{n+1}}\|X_h(s)\|_{\dot\dbH^{\g}_h}^2+\|g_h(s)\|_{\dot\dbH^{\g}_h}^2\rd s\Big]\rd t \\
&\q\leq \cC\t \bigg\{ \|\Pi_h x\|_{\dot\dbH^{\g+1}_h}^2+\me\Big[\int_0^T\|f_h(t)\|_{\dot\dbH^{\g}_h}^2+\|g_h(t)\|_{\dot\dbH^{\g+1}_h}^2\rd t\Big]\bigg\}\,,
\eal
\end{equation*}
which, together with 
$\|\Pi_hx\|_{\dot\dbH^1_h}\leq \cC\|x\|_{\dot\dbH^1}$ and
$\|\Pi_hx\|_{\dot\dbH^2_h}\leq \cC\|x\|_{\dot\dbH^2}$
(see \rf{w1022e1a}, \rf{w1022e1d} in Lemma \ref{w1017l1}),
 settles the desired assertion.
\end{proof}

The following result bounds the difference between $X(\cd)$ and $X_h(\cd)$.

\bt{w1002t1}
Let $X(\cd)$ resp.~$X_h(\cd)$ be solutions to \rf{spde1} resp.~\rf{spde1-h}. Then there exists a constant $\cC$ 
such that
{\small
\begin{subequations}\label{w1002e2}
    \begin{empheq}[left={\empheqlbrace\,}]{align}
      &\me\Big[\sup_{t\in[0,T]}\|X(t)-X_h(t)\|_{\dot \dbH^\g}^2\Big]+\me\Big[\int_0^T\|X(t)-X_h(t)\|_{\dot \dbH^{\g+1}}^2\rd t\Big]\nonumber\\
&\q\leq \cC h^{2(1-\g)} \bigg\{ \|x\|_{\dbH_0^1}^2+\me\Big[\int_0^T\|f(t)\|^2+\|g(t)\|_{\dbH_0^1}^2\rd t\Big]\bigg\}\nonumber\\
&\qq+\cC\, \me\Big[\int_0^T\|\Pi_h f(t)-f_h(t)\|_{\dot \dbH^{\g-1}_h}^2+\|\Pi_h g(t)-g_h(t)\|_{\dot \dbH^{\g}_h}^2\rd t\Big] \q \g=-1,0\,,\label{w1002e2a}\\
& \me\Big[ \sup_{t\in[0,T]} \|X(t)-X_h(t)\|^2\Big]\nonumber\\
&\q\leq \cC h^4\bigg\{ \|x\|_{\dbH_0^1\cap\dbH^2}^2 +\me\Big[\int_0^T \|f(s)\|_{\dbH_0^1\cap\dbH^2}^2+\|g(s)\|_{\dbH_0^1\cap\dbH^2}^2\rd s \Big]\bigg\} \nonumber \\
&\qq+\cC\, \me\Big[\int_0^T  \big\|\Pi_h\big[f(s)-f_h(s)\big]\big\|^2+\big\|\Pi_h\big[g(s)-g_h(s)\big]\big\|^2\rd s\Big]\,. \label{w1002e2b}
\end{empheq}
\end{subequations}
}
\et

\begin{proof}
{\bf (1)  Verification of \rf{w1002e2a}.}
Define the $\dbV_h$-valued process $e_X(\cd)=\Pi_h X(\cd)-X_h(\cd)$.
We have
\begin{equation*}
\setlength\abovedisplayskip{3pt}
\lt\{
\begin{array}{ll}
\ds {\rm d} e_X(t)=\big[\D_h e_X(t)+(\Pi_h\D-\D_h\Pi_h)X(t)+\big(\Pi_h f(t)-f_h(t)\big)\big]\rd t\\
\ns\ds \qq\qq\qq+\big[\b e_X(t)+\big(\Pi_h g(t)-g_h(t)\big)\big]\rd W(t) \qq t\in (0,T]\,,\\
\ns\ds e_X(0)=0\,.
\end{array}
\rt.
\setlength\belowdisplayskip{3pt}
\end{equation*}
Then Lemmata \ref{w1002l4} and \ref{w1017l1} yield
\bel{w1002e1}
\setlength\abovedisplayskip{3pt}
\setlength\belowdisplayskip{3pt}
\bal
&\me\Big[\sup_{t\in[0,T]}\|e_X(t)\|_{\dbH^{-1}}^2\Big]+\me\Big[\int_0^T\|e_X(t)\|^2\rd t\Big]   \\
&\leq \cC\, \me\Big[\int_0^T\|(\Pi_h\D-\D_h\Pi_h)X(t)\|_{\dot\dbH^{-2}_h}^2 \Big]   
+\cC\, \me\Big[\int_0^T\|\Pi_h f(t)-f_h(t)\|_{\dot\dbH^{-2}_h}^2+\|\Pi_h g(t)-g_h(t)\|_{\dot\dbH^{-1}_h}^2\rd t\Big]   \\
&\leq \cC h^4 \bigg\{ \|x\|_{\dbH_0^1}^2+\me\Big[\int_0^T\|f(t)\|^2+\|g(t)\|_{\dbH_0^1}^2\rd t\Big]\bigg\}  \\
&\qq+\cC\, \me\Big[\int_0^T\|\Pi_h f(t)-f_h(t)\|_{\dot\dbH^{-2}_h}^2+\|\Pi_h g(t)-g_h(t)\|_{\dot\dbH^{-1}_h}^2\rd t\Big]\,. 
\eal
\ee
Now, by \rf{w1002e1} and Lemmata \ref{reg-spde1}, \ref{w1017l1}, we  conclude that
\begin{eqnarray*}
&&\me\Big[\sup_{t\in[0,T]}\|X(t)-X_h(t)\|_{\dbH^{-1}}^2\Big]+\me\Big[\int_0^T\|X(t)-X_h(t)\|^2\rd t\Big]\\
&&\qq\leq \cC \bigg\{\me\Big[\sup_{t\in[0,T]}\|X(t)-\Pi_h X(t)\|_{\dbH^{-1}}^2\Big]
+\me\Big[\sup_{t\in[0,T]}\|e_X(t)\|_{\dbH^{-1}}^2\Big]\\
&&\qq\q+\me\Big[\int_0^T\|X(t)-\Pi_h X(t)\|^2\rd t\Big]
+\me\Big[\int_0^T\|e_X(t)\|^2\rd t\Big]\bigg\}\\
&&\qq\leq  \cC h^4 \bigg\{ \|x\|_{\dbH_0^1}^2+\me\Big[\int_0^T\|f(t)\|^2+\|g(t)\|_{\dbH_0^1}^2\rd t\Big]\bigg\}\\
&&\qq\q+\cC\, \me\Big[\int_0^T\|\Pi_h f(t)-f_h(t)\|_{\dot\dbH^{-2}_h}^2+\|\Pi_h g(t)-g_h(t)\|_{\dot\dbH^{-1}_h}^2\rd t\Big]\,.
\end{eqnarray*}
That completes the assertion \rf{w1002e2a} for $\g=-1$.

\ss
A corresponding argument
yields
 the assertion \rf{w1002e2a} for $\g=0$.

\ss
{\bf (2) Verification of \rf{w1002e2b}.} We use tools from semigroup theory. By \rf{spde1} and \rf{spde1-h}, it follows that
\begin{equation*}
\setlength\abovedisplayskip{4pt}
\setlength\belowdisplayskip{4pt}
\bal
&X(t)=E(t)x+\int_0^t E(t-s)f(s)\rd s +\int_0^t E(t-s)\big[\b X(s)+g(s)\big]\rd W(s)\,,\\
&X_h(t)\!=\!E_h(t)\Pi_hx\!+\!\int_0^t\! E_h(t\!-\!s)f_h(s)\rd s 
\!+\!\int_0^t E_h(t\!-\!s)\big[\b X_h(s)\!+\!g_h(s)\big]\!\rd W(s)\,.
\eal
\end{equation*}
\no On setting $\d X(\cd)=X(\cd)-X_h(\cd)$ and using the notation $G_h(\cd)$ from \rf{w1022e4},
\begin{equation*}
\setlength\abovedisplayskip{3pt}
\bal
\d X(t)&=G_h(t)x+ \int_0^t G_h(t-s)f(s) +E_h(t-s)\Pi_h\big[f(s)-f_h(s)\big]\rd s \\
&\q+\int_0^t \b G_h(t-s)X(s) +\b E_h(t-s)\Pi_h\d X(s)\\
&\qq+G_h(t-s)g(s)+E_h(t-s)\Pi_h\big[g(s)-g_h(s)\big]\rd W(s)\,.
\eal
\setlength\belowdisplayskip{4pt}
\end{equation*}
After squaring both sides, taking expectations, applying Lemma \ref{error-G_h} 
with $\rho=\g=2$ and Lemma \ref{reg-spde1}, we can deduce that
\begin{equation*}
\setlength\abovedisplayskip{3pt}
\bal
\me\big[\|\d X(t)\|^2\big]
&\leq \cC\bigg\{ \|G_h(t)x\|^2 +\me\Big[\int_0^t \|G_h(t-s)f(s)\|^2+\|G_h(t-s)X(s)\|^2\\
&\qq +\|G_h(t-s)g(s)\|^2+\|\d X(s)\|^2+\big\|\Pi_h\big[f(s)-f_h(s)\big]\big\|^2\\
&\qq+\big\|\Pi_h\big[g(s)-g_h(s)\big]\big\|^2\rd s\Big]\bigg\}\\
&\leq \cC h^4\bigg\{ \|x\|_{\dbH_0^1\cap\dbH^2}^2 +\me\Big[\int_0^t \|f(s)\|_{\dbH_0^1\cap\dbH^2}^2+\|g(s)\|_{\dbH_0^1\cap\dbH^2}^2\rd s \Big] \bigg\} \\
&\q+\cC\, \me\Big[\int_0^t  \|\d X(s)\|^2+\big \|\Pi_h\big[f(s)-f_h(s)\big]\big\|^2+\big\|\Pi_h\big[g(s)-g_h(s)\big]\big\|^2\rd s\Big]\,,
\eal
\setlength\belowdisplayskip{3pt}
\end{equation*}
which, together with Gronwall's inequality, yields the assertion with the supremum outside $\me[\cd]$; using the BDG inequality then implies \rf{w1002e2b}.
\end{proof}


\subsubsection{Temporal discretization and error estimates for linear SDEs}\label{spdeh-time1}
In this part, we adopt the Euler method to temporally discretize \rf{spde1-h}, {\em i.e.},
\begin{equation}\label{dis-sde}
\setlength\abovedisplayskip{3pt}
\lt\{
\begin{array}{ll}
\ds X_{n+1}\!-\!X_n\!=\!\t \big[\D_h X_{n+1} \!+\!  \wt f_h(t_n)\big] \!+\! \big[\b X_n\!+\!\wt g_h(t_n)\big]\D_{n+1}W
 \q n=0,1,\cds,N,\\
\ns\ds x=\Pi_h x\,,
\end{array}
\rt.
\setlength\belowdisplayskip{3pt}
\end{equation}
where $\wt f_h(\cd)\,,\wt g_h(\cd)\in D_\dbF\big([0,T];L^2(\O;\dbV_h)\big)$ are approximations of $f_h(\cd)\,,g_h(\cd)$ respectively.
While the next lemma settles relevant stability properties for \rf{dis-sde}, Theorem \ref{w1002t3} provides rates for it.

\bl{w1002l2}
Let $X_\cd \equiv \{X_n\}_{n=0}^N$ be the solution of \rf{dis-sde}. There exists a constant $\cC$ such that
\begin{equation*}
\setlength\abovedisplayskip{3pt}
\setlength\belowdisplayskip{3pt}
\bal
&\max_{0\leq n\leq N}\me\big[\|X_n\|_{\dot \dbH^\g_h}^2\big] 
+\sum_{n=0}^{N-1} \me\big[\| X_{n+1}-X_n\|_{\dot\dbH^\g_h}^2\big]
+\t \sum_{n=1}^{N}\me\big[\| X_n\|_{\dot \dbH^{\g+1}_h}^2\big]\\
&\qq\q\leq  \cC \Big[\|x\|_{\dot \dbH^\g_h}^2+\max_{0\leq n\leq N} \me\big[\|\wt f_h(t_n)\|_{\dot \dbH^{\g-1}_h}^2+\|\wt g_h(t_n)\|_{\dot \dbH^\g_h}^2\big] \Big] \qq \g=0,1\,.
\eal
\end{equation*}
\el
 
 \begin{proof}
 Here we only prove the case $\g=0$; the case $\g=1$ can be derived accordingly.
By testing \rf{dis-sde} with $X_{n+1}$, using the relation $(x-y,x)=\frac 1 2 (\|x\|^2-\|y\|^2+\|x-y\|^2)$, 
and then taking expectations, we arrive at
\begin{eqnarray*}
&&\frac 1 2 \Big\{ \me\big[\|X_{n+1}\|^2\big]-\me\big[\|X_{n}\|^2\big]+\me\big[\|X_{n+1}-X_n\|^2\big] \Big\}
+\t \me\big[\| X_{n+1}\|_{\dot\dbH^1_h}^2\big]\\
&&\qq=\t \me\big[\big(\wt f_h(t_n), X_{n+1}\big)_{\dbL^2}\big] 
+\me\big[\big(\D_{n+1}W\big[\b X_n+\wt g_h(t_n)\big], X_{n+1}-X_n\big)_{\dbL^2}\big]\,,
\end{eqnarray*}
since $-\me\big[\big(\D_{n+1}W\big[\b X_n+\wt g_h(t_n)\big], X_n\big)_{\dbL^2}\big]=0$. 
By Young's inequality and the It\^o isometry, we resume that
\begin{equation*}
\setlength\abovedisplayskip{3pt}
\setlength\belowdisplayskip{3pt}
\bal
&\leq \frac \t 2 \me\big[\|X_{n+1}\|_{\dot\dbH^{1}_h}^2\big] +\frac \t 2 \me\big[\|\wt f_h(t_n)\|_{\dot\dbH^{-1}_h}^2\big]
+ \frac 1 4 \me\big[\|X_{n+1}-X_n\|^2\big] 
+2\b^2\t \me\big[\|X_{n}\|^2\big] +2\t \me\big[\|\wt g_h(t_n)\|^2\big]\,.
\eal
\end{equation*}
And then
\bel{w1002e3}
\setlength\abovedisplayskip{3pt}
\setlength\belowdisplayskip{3pt}
\bal
& \me\big[\|X_{n+1}\|^2\big]+\frac 1 2 \me\big[\|X_{n+1}-X_n\|^2\big] +\t \me\big[\| X_{n+1}\|_{\dot\dbH^1_h}^2\big]\\
&\qq\leq \big[1+4\b^2\t\big]  \me\big[\|X_n\|^2\big] +\t  \me\big[\|\wt f_h(t_n)\|_{\dot\dbH^{-1}_h}^2\big] +4\t  \me\big[\|\wt g_h(t_n)\|^2\big]\,.
\eal
\ee
By applying discrete Gronwall's inequality and then taking summation in \rf{w1002e3} over $n$ from $0$ to $N-1$, we get
\begin{equation*}
\setlength\abovedisplayskip{3pt}
\setlength\belowdisplayskip{3pt}
\bal
&\me\big[\|X_N\|^2\big] +\frac 1 2\sum_{n=0}^{N-1} \me\big[\| X_{n+1}-X_n\|^2\big]+\t\sum_{n=0}^{N-1} \me\big[\| X_{n+1}\|_{\dot\dbH^1_h}^2\big]\\
\setlength\abovedisplayskip{3pt}
\setlength\belowdisplayskip{3pt}
&\qq\leq \cC \Big[\|x\|^2+\max_{0\leq n\leq N} \me\big[\|\wt f_h(t_n)\|_{\dot\dbH^{-1}_h}^2\big]+\max_{0\leq n\leq N} \me\big[\|\wt g_h(t_n)\|^2\big] \Big]\,.
\eal
\end{equation*}
That completes the proof. 
 \end{proof}
 
\bt{w1002t3}
Let $X_h(\cd)$ be the solution to SDE \rf{spde1-h}, and $X_\cd$ solves \rf{dis-sde}. There exists a constant $\cC$ such 
that
\begin{subequations}\label{w1028e1}
    \begin{empheq}[left={\empheqlbrace\,}]{align}
      &\max_{0\leq n\leq N}\me\big[\|X_h(t_n)-X_n\|_{\dot\dbH^{-1}_h}^2\big]
      +\sum_{n=0}^{N-1}\me\Big[\int_{t_n}^{t_{n+1}}\|X_h(t)-X_n\|^2\rd t\Big] \nonumber \\
&\q\leq \cC\t \bigg\{ \|x\|_{\dbH_0^1}^2+\me\Big[\int_0^T\|f_h(t)\|^2+\|g_h(t)\|_{\dbH_0^1}^2\rd t\Big]\bigg\}\nonumber \\
&\qq+\cC\sum_{n=0}^{N-1} \me\Big[\int_{t_n}^{t_{n+1}} \|f_h(t)-\wt f_h(t_n)\|_{\dot \dbH^{-1}_h}^2+\|g_h(t)-\wt g_h(t_n)\|_{\dot \dbH^{-1}_h}^2 \rd t \Big]\,,  \label{w1028e1a}\\
& \max_{0\leq n\leq N-1}\!\sup_{t\in[t_n,t_{n+1})}\!\me\big[\|X_h(t)\!-\!X_n\|^2\big]
\!+\!\sum_{n=0}^{N-1}\me\Big[\!\int_{t_n}^{t_{n+1}}\!\|X_h(t)\!-\!X_n\|_{\dbH_0^1}^2\rd t\Big]\nonumber \\
 &\q\leq \cC\t \bigg\{\|x\|_{\dbH_0^1\cap\dbH^2}^2+\me\Big[\int_0^T  \|f_h(t)\|_{\dbH_0^1}^2+\|g_h(t)\|_{\dbH_0^1\cap\dbH^2}^2 \rd t \Big]\nonumber \\
&\qq+\cC \sum_{n=0}^{N-1} \me\Big[\int_{t_n}^{t_{n+1}} \|f_h(t)-\wt f_h(t_n)\|^2+\|g_h(t)-\wt g_h(t_n)\|^2 \rd t \Big]\,.  \label{w1028e1b}
\end{empheq}
\end{subequations}
\et

\begin{proof}
We only prove  \rf{w1028e1a}. The argumentation for \rf{w1028e1b} is very close to \rf{w1028e1a} and we again leave it
to the interested reader.

For any $n=0,1,\cds,N$, consider the $\dbV_h$-valued random variable $e_n=X_n-X_h(t_n)$.
Then $e_\cd$ satisfies
\bel{w1017e2}
\setlength\abovedisplayskip{3pt}
\setlength\belowdisplayskip{3pt}
\lt\{
\begin{array}{ll}
\ds e_{n+1}\!-\!e_n
 =\t\D_h e_{n+1}\!+\!\int_{t_n}^{t_{n+1}}\big[\D_h(X_h(t_{n+1})\!-\!X_h(t))\big]\!+\!\big[\wt f_h(t_n)\!-\!f_h(t)\big]\!\rd t\\
\ns\ds\qq\q+ \b e_n\D_{n+1}W\! +\!\int_{t_n}^{t_{n+1}}\!\b \big[X_h(t_n)\!-\!X_h(t)\big]\!+\!\big[\wt g_h(t_n)\!-\!g_h(t)\big]\!\rd W(t)\\
\ns\ds\qq\qq\qq\qq\qq \qq n=0,1,\cds,N-1\,,\\
\ns\ds e_0=0\,.
\end{array}
\rt.
\ee
By testing this equation with $(-\D_h)^{-1}e_{n+1}$ and 
using the same idea as in the proof of Lemma \ref{w1002l2}, 
we have
\begin{eqnarray}
&&\frac  1 2\Big\{ \me \big[\|(-\D_h)^{-1/2}e_{n+1}\|^2\big] - \me \big[\|(-\D_h)^{-1/2}e_n\|^2\big] \notag\\
&&\q + \me \big[\|(-\D_h)^{-1/2}(e_{n+1}-e_n)\|^2\big] \Big\}+\t \me \big[\|e_{n+1}\|^2\big] \notag \\
&&=\me\Big[\int_{t_n}^{t_{n+1}} \big( X_h(t)-X_h(t_{n+1}), e_{n+1}\big)_{\dbL^2}\rd t \Big] \notag \\
&&\q+\me\Big[\int_{t_n}^{t_{n+1}} \big( (-\D_h)^{-1/2}\big(\wt f_h(t_n)-f_h(t)\big),  (-\D_h)^{-1/2}(e_{n+1}-e_n)\big)_{\dbL^2}\rd t \Big] \notag \\
&&\q+\me\Big[\int_{t_n}^{t_{n+1}} \big( (-\D_h)^{-1/2}\big(\wt f_h(t_n)-f_h(t)\big),  (-\D_h)^{-1/2}e_n\big)_{\dbL^2}\rd t \Big]\notag \\
&&\q+\me\Big[\big(\D_{n+1}W\b (-\D_h)^{-1/2}e_n, (-\D_h)^{-1/2}(e_{n+1}-e_n)\big)_{\dbL^2}\Big] \notag \\
&&\q+\me\Big[\Big(\int_{t_n}^{t_{n+1}}  (-\D_h)^{-1/2} \b\big(X_h(t_n)-X_h(t)\big)\rd W(t),  (-\D_h)^{-1/2}(e_{n+1}-e_n)\Big)_{\dbL^2} \Big]\notag \\
&&\q+\me\Big[\Big(\!\int_{t_n}^{t_{n+1}} \! (-\D_h)^{-1/2}\big(\wt g_h(t_n)\!-\!g_h(t)\big)\!\rd W(t),  (-\D_h)^{-1/2}\!(e_{n+\!1}\!-\!e_n)\Big)_{\dbL^2} \Big]\,.\label{w1209e1}
\end{eqnarray}
A standard procedure, which uses the It\^o isometry, Young's inequality, discrete Gronwall's inequality 
and the fact that $e_0=0$, then leads to
\begin{equation*}
\setlength\abovedisplayskip{4pt}
\setlength\belowdisplayskip{3pt}
\bal
\max_{0\leq k\leq N}\me \big[\|e_k\|_{\dot \dbH^{-1}_h}^2\big] 
&\leq \cC\sum_{n=0}^{N-1}\bigg\{\me\Big[\int_{t_n}^{t_{n+1}}  \|X_h(t)-X_h(t_{n+1})\|^2+\|X_h(t)-X_h(t_n)\|_{\dot \dbH^{-1}_h}^2 \rd t \Big]\\
&\qq+ \me\Big[\int_{t_n}^{t_{n+1}} \|f_h(t)-\wt f_h(t_n)\|_{\dot \dbH^{-1}_h}^2+\|g_h(t)-\wt g_h(t_n)\|_{\dot \dbH^{-1}_h}^2 \rd t \Big]\bigg\}\,.
\eal
\end{equation*}
We may now use this result to verify assertion \rf{w1028e1a}. Therefore, we come back to 
 \rf{w1209e1} and do summation over $n$ from $0$ to $N-1$ to conclude
\begin{eqnarray*}
&& \sum_{n=0}^{N-1}\Big[\me \big[\|e_{n+1}-e_n\|_{\dot \dbH^{-1}_h}^2\big] +\t \me \big[\|e_{n+1}\|^2\big]\Big]\\
&&\qq\leq (1+4\b^2)\t  \sum_{n=0}^{N-1}\me \big[\|e_n\|_{\dot \dbH^{-1}_h}^2\big]\\
&&\qq\q+\sum_{n=0}^{N-1}\bigg\{\me\Big[\int_{t_n}^{t_{n+1}}  \|X_h(t)-X_h(t_{n+1})\|^2+8\b^2\|X_h(t)-X_h(t_n)\|_{\dot \dbH^{-1}_h}^2 \rd t \Big]\\
&&\qq\q+ \me\Big[\int_{t_n}^{t_{n+1}} (4\t+1)\|f_h(t)-\wt f_h(t_n)\|_{\dot \dbH^{-1}_h}^2+8\|g_h(t)-\wt g_h(t_n)\|_{\dot \dbH^{-1}_h}^2 \rd t \Big]\bigg\}\\
&&\qq\leq \cC\sum_{n=0}^{N-1}\bigg\{\me\Big[\int_{t_n}^{t_{n+1}}  \|X_h(t)-X_h(t_{n+1})\|^2+\|X_h(t)-X_h(t_n)\|_{\dot \dbH^{-1}_h}^2 \rd t \Big]\\
&&\qq\q+ \me\Big[\int_{t_n}^{t_{n+1}} \|f_h(t)-\wt f_h(t_n)\|_{\dot \dbH^{-1}_h}^2+\|g_h(t)-\wt g_h(t_n)\|_{\dot \dbH^{-1}_h}^2 \rd t \Big]\bigg\}\,.
\end{eqnarray*}
This bound, together with \rf{w1002e5} for $\g=0$ in Lemma \ref{w1002l4}, then leads to the estimate
\begin{equation*}
\setlength\abovedisplayskip{3pt}
\setlength\belowdisplayskip{3pt}
\bal
&\max_{0\leq n\leq N}\me\big[\|e_n\|_{\dot\dbH^{-1}_h}^2\big]
+\sum_{n=0}^{N-1}\me\Big[\int_{t_n}^{t_{n+1}}\|X_h(t)-X_n\|^2\rd t\Big]\\
&\q\leq \max_{0\leq n\leq N}\me\big[\|e_n\|_{\dot\dbH^{-1}_h}^2\big]
+\cC\sum_{n=0}^{N-1}\me\Big[\int_{t_n}^{t_{n+1}}\|X_h(t)-X_h(t_n)\|^2\rd t
+\t \|e_{n+1}\|^2\Big]\\
&\q\leq \cC\t \bigg\{ \|x\|_{\dbH_0^1}^2+\me\Big[\int_0^T\|f_h(t)\|^2+\|g_h(t)\|_{\dbH_0^1}^2\rd t\Big]\bigg\}\\
&\qq+\cC \sum_{n=0}^{N-1} \me\Big[\int_{t_n}^{t_{n+1}} \|f_h(t)-\wt f_h(t_n)\|_{\dot \dbH^{-1}_h}^2+\|g_h(t)-\wt g_h(t_n)\|_{\dot \dbH^{-1}_h}^2 \rd t \Big]\,.
\eal
\end{equation*}
That settles the assertion  \rf{w1028e1a}.
\end{proof}

\subsection{Discretization of linear BSPDEs and rates}\label{nume-bspde1}
We split the derivation of rates of convergence for the scheme \rf{bspde-h-t} to discretize BSPDE (\ref{bshe}) into two steps: in Section \ref{num-bspde1a} we begin with BSDE \rf{bsde}, which is a discretization scheme of \rf{bshe} in space only; related stability bounds for the solutions to BSDEs may then serve in Section \ref{dis-bspde-t} to verify rates of convergence for the complete discretization in space and time.

\subsubsection{Spatial discretization and error estimates for linear BSPDEs }\label{num-bspde1a}
 
We now consider a finite element method of the BSPDE \eqref{bshe}.
Let $Y_{T,h} \in L^2_{{\mathcal F}_T}(\Omega; {\mathbb V}_h)$, 
 $ f_h(\cd) \in L^2_{{\mathbb F}}(0,T; \dbV_h)$ be approximations of $Y_T\,,f(\cd)$, respectively.
Then the finite element discretization of \rf{bshe} is the following BSDE:
\bel{bsde}
\setlength\abovedisplayskip{3pt}
\setlength\belowdisplayskip{3pt}
\lt\{
\bal
&{\rm d}Y_h(t)=\big[-\Delta_hY_h(t)-\b Z_h(t)+ f_h(t)\big] \rd t+ Z_h(t)\rd W(t) \q  t\in [0,T)\, ,\\
&Y_h(T)=Y_{T,h}\, .
\eal
\rt.
\ee
The well-posedness  of a solution tuple $\big(Y_h(\cd), Z_h(\cd)\big)\in L^2_\dbF\big(\O;C([0,T];\dbV_h)\big)\times L^2_\dbF(0,T;\dbV_h)$ follows from \cite[Theorem 2.1]{ElKaroui-Peng-Quenez97}.
The following result is on the stability of the solution to \rf{bsde}.

\bl{w1008l1}
Suppose that $Y_{T,h}\in L^2_{\mf_T}(\O;\dbV_h)$ and $f_h(\cd)\in L^2_\dbF(0,T;\dbV_h)$.
There exists a constant $\cC$ independent of $h$ such that
\bel{w1008e1}
\setlength\abovedisplayskip{3pt}
\setlength\belowdisplayskip{3pt}
\bal
& \me\Big[\sup_{t \in [0,T]} \| Y_h(t)\|_{\dot \dbH^\g_h}^2\Big] +
\me \Big[\int_0^T  \| Y_h(t)\|_{\dot \dbH^{\g+1}_h}^2+
\|Z_h(t)\|_{\dot \dbH^\g_h}^2\rd t \Big]\\
&\qq\leq
\cC\, \me \Big[ \lt\|  Y_{T,h}\rt\|_{\dot \dbH^\g_h}^2+ \int_0^T \| f_h(t)\|_{\dot \dbH^{\g-1}_h}^2\rd t\Big] \qq \g=-1,0,1,2\,.\\
\eal
\ee
For a uniform partition $I_\t=\{t_n\}_{n=0}^N$ of size $\t>0$ covering $[0,T]$, it holds that
\bel{w1008e2}
\setlength\abovedisplayskip{3pt}
\setlength\belowdisplayskip{3pt}
\bal
\sum_{n=0}^{N-1}\me\Big[\int_{t_n}^{t_{n+1}} \|Y_h(t)-Y_h(t_n)\|_{\dot \dbH^{\g}_h}^2\rd t\Big]
\leq \cC\t \me\Big[\|Y_{T,h}\|_{\dot \dbH^{\g+1}_h}^2+\int_0^T \|f_h(t)\|_{\dot \dbH^{\g}_h}^2\rd t\Big]  \qq \g=0,1\,.
\eal
\ee
\el

\begin{proof}
The assertion \rf{w1008e1} can be derived by It\^o's formula, Gronwall's inequality and the BDG inequality;
 see {\em e.g.}~\cite[Lemma 3.1]{Dunst-Prohl16}. 
To prove \rf{w1008e2}, we follow a similar technique as in the proof of \rf{w1002e5} in  Lemma \ref{w1002l4}.
By the It\^o isometry and \rf{w1008e1}, BSDE \rf{bsde} easily gives
\begin{equation*}
\setlength\abovedisplayskip{3pt}
\setlength\belowdisplayskip{3pt}
\bal
\sum_{n=0}^{N-1}\me\Big[\int_{t_n}^{t_{n+1}}\|Y_h(t)-Y_h(t_n)\|_{\dot \dbH^{\g}_h}^2\rd t \Big]
&\leq \cC\t \me\Big[\int_0^T\|\D_h Y_h(t)\|_{\dot \dbH^{\g}_h}^2+\|Z_h(t)\|_{\dot \dbH^{\g}_h}^2+\|f_h(t)\|_{\dot \dbH^{\g}_h}^2\rd t\Big] \\
&\leq \cC\t  \me\Big[\|Y_{T,h}\|_{\dot \dbH^{\g+1}_h}^2+\int_0^T \|f_h(t)\|_{\dot \dbH^{\g}_h}^2\rd t\Big] \,.
\eal
\end{equation*}
That completes the proof.
\end{proof}
 
 The  following result settles rates of convergence for discretization \rf{bsde}, which also
improves the result in \cite[Theorem 3.2]{Dunst-Prohl16}.

\begin{theorem}\label{w0911t1}
Suppose that $Y_T \in L^2_{{\mathcal F}_T}(\Omega; {\mathbb H}^1_0\cap\dbH^2)$ and $f(\cd)\in L^2_\dbF(0,T;\dbH_0^1)$.
Let $\big(Y(\cd),Z(\cd)\big)$ resp.~$\big(Y_h(\cd), Z_h(\cd)\big)$  be solutions of \eqref{bshe} resp.~\eqref{bsde}. 
There exists
a constant $\cC$ independent of $h$
such that
\begin{equation}\label{w1028e2a}
\setlength\abovedisplayskip{3pt}
\setlength\belowdisplayskip{3pt}
\bal
&\me\Big[\sup_{t\in[0,T]}\|Y(t)-Y_h(t)\|_{\dot \dbH^{\g}}^2\Big]
+\me\Big[\int_0^T\|Y(t)-Y_h(t)\|_{\dot\dbH^{\g+1}}^2+\|Z(t)-Z_h(t)\|_{\dot \dbH^{\g}}^2\rd t\Big]  \\
&\qq\leq \cC h^{2(1-\g)} \me\Big[ \|Y_T\|_{\dbH_0^1}^2+\int_0^T\|f(t)\|^2\rd t\Big] \\
&\qq\q+\cC\,\me\Big[\|\Pi_h Y_T-Y_{T,h}\|^2_{\dot \dbH^{\g}_h}+\int_0^T\|\Pi_h f(t)-f_h(t)\|_{\dot \dbH^{\g-1}_h}^2\rd t\Big]
\qq\g=-1,0\,,  
\eal
\end{equation}
and
\begin{equation}\label{w1028e2b}
\setlength\abovedisplayskip{3pt}
\setlength\belowdisplayskip{3pt}
\bal
\sup_{t\in[0,T]}\me\big[\|Y(t)-Y_h(t)\|^2\big]  
&\leq \cC h^4  \me\Big[\|Y(T)\|_{\dbH_0^1\cap\dbH^2}^2 +\int_0^T \|f(s)\|_{\dbH_0^1\cap \dbH^2}^2 \rd s\Big] \\
&\qq+\cC  \me\Big[\|Y(T)-Y_h(T)\|^2+\int_0^T \|f(s)-f_h(s)\|^2 \rd s\Big]\,.  
\eal
\end{equation}

\end{theorem}

\begin{proof}

{\bf (1) Verification of \rf{w1028e2a}.} By defining two $\dbV_h$-valued stochastic processes
$e_Y(\cd)=\Pi_h Y(\cd)-Y_h(\cd)$ and $ e_Z(\cd)=\Pi_h Z(\cd)-Z_h(\cd)$,
we find that
\begin{equation*}
\setlength\abovedisplayskip{3pt}
\setlength\belowdisplayskip{3pt}
\lt\{
\begin{array}{ll}
\ds {\rm d} e_Y(t)=\big[- \D_h e_Y(t)-\b e_Z(t)-(\Pi_h\D-\D_h\Pi_h)Y(t)+\big(\Pi_h f(t)-f_h(t)\big)\big]\rd t\\
\ns \ds \qq\qq\qq\qq+e_Z(t)\rd W(t) \qq t\in[0,T)\,,\\
\ns\ds e_Y(T)=\Pi_h Y_T-Y_{T,h}\,.
\end{array}
\rt.
\end{equation*}
Based on \rf{w1008e1} in Lemma \ref{w1008l1}, \rf{w1022e1e} in  Lemma \ref{w1017l1} and Lemma \ref{reg-bshe}, we have
\begin{equation*}
\setlength\abovedisplayskip{3pt}
\setlength\belowdisplayskip{3pt}
\bal
&\me\Big[\sup_{t\in[0,T]}\|e_Y(t)\|_{\dot \dbH_h^{-1}}^2\Big]+\me\Big[\int_0^T\|e_Y(t)\|^2+\|e_Z(t)\|_{\dot \dbH_h^{-1}}^2\rd t\Big]\\
&\q\leq \cC\, \me\Big[\|\Pi_h Y_T-Y_{T,h}\|^2_{\dot \dbH^{-1}_h}+\int_0^T\|(-\D_h)^{-1}(\D_h\cR_h-\D_h\Pi_h)Y(t)\|^2\\
&\qq+\|\Pi_h f(t)-f_h(t)\|^2_{\dot \dbH^{-2}_h}\rd t\Big]\\
&\q\leq \cC h^4 \me \Big[ \|Y_T\|_{\dbH_0^1}^2+\int_0^T\|f(t)\|^2\rd t\Big]
+\cC\, \me\Big[\|\Pi_h Y_T-Y_{T,h}\|^2_{\dot \dbH^{-1}_h}+\int_0^T\|\Pi_h f(t)-f_h(t)\|^2_{\dot \dbH^{-2}_h}\rd t\Big]\,.
\eal
\end{equation*}
On the other hand, Lemmata \ref{w1017l1} and \ref{reg-bshe} imply that
\begin{equation*}
\setlength\abovedisplayskip{3pt}
\setlength\belowdisplayskip{3pt}
\bal
&\me\Big[\sup_{t\in[0,T]}\|Y(t)-\Pi_hY(t)\|_{\dot \dbH_h^{-1}}^2\Big]
+\me\Big[\int_0^T\|Y(t)-\Pi_hY(t)\|^2+\|Z(t)-\Pi_hZ(t)\|_{\dot \dbH_h^{-1}}^2\rd t\Big]\\
&\qq\leq\cC h^4 \bigg\{\sup_{t\in[0,T]}\me\big[\|Y(t)\|_{\dbH_0^1}^2\big]+\me\Big[\int_0^T\|Y(t)\|_{\dbH_0^1\cap\dbH^2}^2+\|Z(t)\|_{\dbH_0^1}^2\rd t\Big]\bigg\}\\
&\qq\leq \cC h^4 \bigg\{ \me\big[\|Y_T\|_{\dbH_0^1}^2\big]+\me\Big[\int_0^T\|f(t)\|^2\rd t\Big]\bigg\}\,.
\eal
\end{equation*}
The assertion for $\g=-1$ then follows from the estimates above.

\ss

To derive assertion \rf{w1028e2a} for $\g=0$, we may proceed correspondingly.  Firstly by applying  \rf{w1008e1} in Lemma \ref{w1008l1} and \rf{w1022e1e} in Lemma \ref{w1017l1}, we conclude that
\begin{eqnarray*}
&&\me\Big[\sup_{t\in[0,T]}\|e_Y(t)\|^2\Big]+\me\Big[\int_0^T\|\nb e_Y(t)\|^2+\|e_Z(t)\|^2\rd t\Big]\\
&&\q\leq \cC\, \me\Big[\|\Pi_h Y_T-Y_{T,h}\|^2
+\int_0^T\|(\Pi_h\D-\D_h\Pi_h)Y(t)\|_{\dot\dbH^{-1}_h}^2+\|\Pi_h f(t)-f_h(t)\|_{\dot \dbH^{-1}_h}^2\rd t\Big]\\
&&\q\leq \cC h^2\, \me \Big[ \|Y_T\|_{\dbH_0^1}^2+\int_0^T\|f(t)\|^2\rd t\Big]
+\cC\, \me\Big[\|\Pi_h Y_T-Y_{T,h}\|^2+ \int_0^T  \|\Pi_h f(t)-f_h(t)\|_{\dot \dbH^{-1}_h}^2\rd t\Big]\,.
\end{eqnarray*}
It is now immediate to complete the argument, which settles the assertion for $\g=0$.

\ss
{\bf (2)  Verification of \rf{w1028e2b}.}
To derive the assertion,
we consider a time partition $I_{\t_0}=\{T_n\}_{n=0}^{N_0}$ with $\t_0\leq \frac{1}{3\b^2}$ to control the term led by $\b$;
if $\b=0$, we do not need any partition.
By setting $\d Y(\cd)=Y(\cd)-Y_h(\cd)\,,\d f(\cd)=f(\cd)-f_h(\cd)$ and applying \eqref{bshe} and \eqref{bsde}, for any $t\in [T_n,T_{n+1})$ we have
\begin{equation*}
\setlength\abovedisplayskip{3pt}
\setlength\belowdisplayskip{3pt}
\bal
&\d Y(t) +\int_t^{T_{n+1}} E(s-t)Z(s)-E_h(s-t)Z_h(s) \rd W(s)\\
&\qq=E(T_{n+1}-t)Y(T_{n+1})-E_h(T_{n+1}-t)Y_h(T_{n+1})
+\b \int_t^{T_{n+1}} E(s-t)Z(s)-E_h(s-t)Z_h(s) \rd s\\
&\qq\q- \int_t^{T_{n+1}} E(s-t)f(s)-E_h(s-t)f_h(s) \rd s\,.
\eal
\end{equation*}
Using mutual independence of $\d Y(t)$ and $\int_t^{T_{n+1}}E(s-t)Z(s)-E_h(s-t)Z_h(s)\rd W(s)$, squaring on both sides,
taking expectations, and applying \rf{w1022e4},  we  arrive at
\begin{equation*}
\setlength\abovedisplayskip{3pt}
\setlength\belowdisplayskip{3pt}
\bal
&\me\big[\|\d Y(t)\|^2\big] + \me\Big[\int_t^{T_{n+1}} \|E(s-t)Z(s)-E_h(s-t)Z_h(s)\|^2 \rd s\Big]\\
&\qq\leq 6 \me\big[\|G_h(T_{n+1}-t) Y(T_{n+1})\|^2\big] + 6 \me\big[\|E_h(T_{n+1}-t)\Pi_h \d Y(T_{n+1})\|^2\big]\\
&\qq\q+3\b^2 \t_0  \me\Big[\int_t^{T_{n+1}} \|E(s-t)Z(s)-E_h(s-t)Z_h(s)\|^2 \rd s\Big]\\
&\qq\q+6  \t_0  \me\Big[\int_t^{T_{n+1}} \|G_h(s-t)f(s)\|^2+\| E_h(s-t)\Pi_h\d f(s)\|^2 \rd s\Big]\,,
\eal
\end{equation*}
which, together with the fact that $3\b^2\t_0\leq 1$, yields
\begin{eqnarray}
&\me\big[\|\d Y(t)\|^2\big]&\leq 6 \me\big[\|G_h(T_{n+1}-t) Y(T_{n+1})\|^2\big] +6\me\big[ \|E_h(T_{n+1}-t)\Pi_h \d Y(T_{n+1})\|^2\big]  \notag \\
&&\q+6  \t_0  \me\Big[\int_t^{T_{n+1}} \|G_h(s-t)f(s)\|^2+\| E_h(s-t)\Pi_h\d f(s)\|^2 \rd s\Big] \notag \\
&&\leq \cC h^4  \me\Big[\|Y(T_{n+1})\|_{\dbH_0^1\cap \dbH^2}^2
+\t_0 \int_{T_n}^{T_{n+1}} \|f(t)\|_{\dbH_0^1\cap \dbH^2}^2 \rd t\Big] \notag \\
&&\q+\cC\,  \me\Big[\|\Pi_h\d Y(T_{n+1})\|^2+\t_0  \int_{T_n}^{T_{n+1}} \|\Pi_h\d f(t)\|^2 \rd t\Big]\,. \label{w1015e1}
\end{eqnarray}
By taking $t=T_n$,  $n=N_0-1,N_0-2,\cds,0$, repeating the above procedure $N_0$ times and applying \rf{vari-2}, we conclude that
\begin{equation*}
\setlength\abovedisplayskip{3pt}
\setlength\belowdisplayskip{3pt}
\bal
\max_{0\leq n\leq N_0}\me\big[\|\d Y(T_n)\|^2\big] 
&\leq \cC h^4   \me\Big[\|Y(T)\|_{\dbH_0^1\cap \dbH^2}^2
+\int_0^T \|f(t)\|_{\dbH_0^1\cap \dbH^2}^2 \rd t\Big]\\
&\q+\cC\,  \me\Big[\|\Pi_h \d Y(T)\|^2 + \int_0^T \|\Pi_h \d f(t)\|^2 \rd t\Big] \,,
\eal
\end{equation*}
which, together with \rf{w1015e1}, implies the assertion \rf{w1028e2b}.
\end{proof}


\subsubsection{Temporal discretization and error estimates for linear BSDEs}\label{dis-bspde-t}

In this part, we prove rates of convergence for the temporal discretization of BSDE \rf{bsde} by means of 
the Euler method:
\begin{equation}\label{bspde-h-t}
\setlength\abovedisplayskip{3pt}
\setlength\belowdisplayskip{3pt}
\lt\{\!\!\!
\begin{array}{ll}
\ds Y_{h\t}(t_n) = A_0 \me^{t_n}\Big[ \big(Y_{h\t}(t_{n+1})-\t f_h(t_{n+1})\big)\big(1+\b\D_{n+1}W\big)\Big]  \qq & n=0,1,\cds,N-1\,, \\
\ns\ds Z_{h\t}(t_n)= \frac 1 \t \me^{t_n}\Big[ \big(Y_{h\t}(t_{n+1})-\t f_h(t_{n+1})\big)\D_{n+1}W\Big] \qq&  n=0,1,\cds,N-1,\\
\ns\ds Y_{h\t}(T)=Y_{T,h}\,,
\end{array}
\rt.
\end{equation}
where $A_0=(\mathds{1}_h-\t\D_h)^{-1}$, and $\me^{t_n}[\cd]$ is a conditional expectation with respect to $\mf_{t_n}$;
see part {\bf a.} in Section \ref{not1}.
To initiate  the error analysis, we introduce the following  auxiliary equation which serves as a bridge between 
BSDE \rf{bsde} and difference equation \rf{bspde-h-t}, 
\begin{equation}\label{dis-bsde}
\setlength\abovedisplayskip{3pt}
\setlength\belowdisplayskip{3pt}
\lt\{\!\!\!
\begin{array}{ll}
\ds Y_n-Y_{n+1}=\t \D_h Y_n + \int_{t_n}^{t_{n+1}} \b  \bar Z_0(t)- \wt f_h(t)\rd t
-\int_{t_n}^{t_{n+1}} Z_0(t)\rd W(t)\q n=0,1,\cds,N,\\
\ns\ds Y_N=\wt Y_{T,h}\,,
\end{array}
\rt.
\end{equation}
where 
\begin{equation}
\setlength\abovedisplayskip{3pt}
\setlength\belowdisplayskip{3pt}
\label{w1026e4}
\bar Z_0(t)=\frac 1\t\me^{t_n}\Big[\int_{t_n}^{t_{n+1}}Z_0(t)\rd t\Big]\,,
\end{equation}
and $\wt Y_{T,h}\,,\wt f_h(\cd)$ are approximations of $ Y_{T,h}\,, f_h(\cd)$ respectively. 
Note that when $\b\neq 0$, equation \rf{dis-bsde} is {\em not implementable}, while \rf{bspde-h-t} is, 
since conditional expectations can be simulated; see Section \ref{stati-1} below.

\br{w1019r1}
{\bf (i)} For sufficiently small $\t$, we can follow {\em e.g.}~\cite{Pardoux-Peng90} to derive 
the well-posedness of \rf{dis-bsde}.

{\bf (ii)} The following observation connects the two equations \rf{bspde-h-t} and \rf{dis-bsde}.
If $\wt Y_{T,h}\,,\wt f_h(\cd)$ are chosen as follows,
\bel{w1117e1}
\setlength\abovedisplayskip{3pt}
\setlength\belowdisplayskip{3pt}
\wt Y_{T,h}= Y_{T,h}\qq
\wt f_h(t)=f_h(t_{n+1}) \qq t\in [t_n,t_{n+1})\,, n=0,1,\cds,N-1\,,
\ee
then it holds that
\begin{equation*}
\setlength\abovedisplayskip{3pt}
\setlength\belowdisplayskip{3pt}
Y_{h\t}(t_n)=Y_n\qq Z_{h\t}(t_n)=\bar Z_0(t_n) \qq n=0,1,\cds,N-1\,.
\end{equation*}
Hence, if an error estimate is obtained for $\big(Y_h(\cd)-Y_\cd$, $Z_h(\cd)-Z_0(\cd)\big)$ in a suitable norm,
where $\big(Y_h(\cd), Z_h(\cd)\big)$ is from \rf{bsde}, it transfers to
the Euler method \rf{bspde-h-t} as well.
\er

We are now ready to verify rates for scheme \rf{dis-bsde}. 

\bt{w1019t1} 
Let $\big(Y_h(\cd),Z_h(\cd)\big)$ resp.~$\big(Y_\cd, Z_0(\cd)\big)$ be solutions to \rf{bsde} resp.~\rf{dis-bsde}. Then there exists a
constant $\cC$ such that
\begin{equation*}
\setlength\abovedisplayskip{3pt}
\setlength\belowdisplayskip{3pt}
\bal
&\max_{0\leq n\leq N}\me\big[\|Y_h(t_n)\!-\!Y_n\|_{\dot \dbH^{-1}_h}^2\big]
\!+\!\sum_{n=0}^{N-1}\me\Big[\int_{t_n}^{t_{n+1}}\|Y_h(t)\!-\!Y_n\|^2\!+\! \|Z_h(t)\!-\!Z_0(t)\|_{\dot \dbH^{-1}_h}^2\rd t\Big]\\
&\qq\leq \cC\t \me\Big[\|Y_{T,h}\|_{\dot \dbH^2_h}^2+\int_0^T\|f_h(t)\|_{\dot \dbH^1_h}^2\rd t\Big]
+\cC\, \me\Big[\|Y_{T,h}-\wt Y_{T,h}\|_{\dot \dbH^{-1}_h}^2\\
&\qq\q+\sum_{n=0}^{N-1}\int_{t_n}^{t_{n+1}}\|Z_h(t)-\bar Z_h(t)\|_{\dot \dbH^{-1}_h}^2+\|f_h(t)-\wt f_h(t)\|_{\dot \dbH^{-1}_h}^2\rd t\Big]\,.
\eal
 \end{equation*}
\et

\begin{proof}
For any $n=0,1,\cds,N$,
we define the $\dbV_h$-valued random variable $e_n^Y=Y_n- Y_h(t_n)$.
It then follows that
\begin{equation}\label{w1}
\setlength\abovedisplayskip{3pt}
\setlength\belowdisplayskip{3pt}
\bal
&e_n^Y-e_{n+1}^Y+\int_{t_n}^{t_{n+1}}Z_0(t)-Z_h(t)\rd W(t)\\
&=\t\D_h e_n^Y+\int_{t_n}^{t_{n+1}}\D_h \big[  Y_h(t_n)- Y_h(t)\big] \\
&\qq+ \b \big[\bar Z_0(t)-\bar Z_h(t)\big]+ \b \big[\bar Z_h(t)- Z_h(t)\big]- \big[\wt f_h(t)-f_h(t)\big]\rd t\,.
\eal
\end{equation}
Testing this equation with $(-\D_h)^{-1}e_n^Y$, then applying the mutual independence of $(-\D_h)^{-1}e^Y_n$ and $\int_{t_n}^{t_{n+1}}Z_0(t)-Z_h(t)\rd W(t)$, as well as discrete Gronwall's inequality, and finally taking summation over $n$ we arrive at
\bel{w1026e2}
\setlength\abovedisplayskip{3pt}
\setlength\belowdisplayskip{3pt}
\bal
&\max_{0\leq k\leq N}\|e_k^Y\|_{\dot \dbH^{-1}_h}^2\!+\!\sum_{n=0}^{N-1}\t \|e_n^Y\|^2\\
&\q\leq \cC\Big[\|e_N^T\|_{\dot \dbH^{-1}_h}^2+\sum_{n=0}^{N-1} \int_{t_n}^{t_{n+1}}\| Y_h(t)- Y_h(t_n)\|^2
+\|Z_0(t)-Z_h(t)\|_{\dot \dbH^{-1}_h}^2\\
&\qq\q+\|Z_h(t)-\bar Z_h(t)\|_{\dot \dbH^{-1}_h}^2+\|\wt f_h(t)-f_h(t)\|_{\dot \dbH^{-1}_h}^2\rd t\Big]\,.
\eal
\ee
Together with estimate \rf{w1008e2} in Lemma \ref{w1008l1}, the triangle inequality then leads to
\bel{w1019e2}
\setlength\abovedisplayskip{3pt}
\setlength\belowdisplayskip{3pt}
\bal
&\max_{0\leq n\leq N}\me\big[\|Y_n- Y_h (t_n)\|_{\dot \dbH^{-1}_h}^2\big]
+\sum_{n=0}^{N-1}\me\Big[\int_{t_n}^{t_{n+1}}\|Y_h(t)-Y_n\|^2\rd t\Big]  \\
&\q\leq \max_{0\leq n\leq N}\me\big[\|e_n^Y\|_{\dot \dbH^{-1}_h}^2\big] 
+\cC\sum_{n=0}^{N-1}\me\Big[\int_{t_n}^{t_{n+1}}\|Y_h(t)-Y_h(t_n)\|^2+\|e_n^Y\|^2\rd t\Big]  \\
&\q\leq \cC\t \me \Big[ \|Y_{T,h}\|_{\dot \dbH^{1}_h}^2+\int_0^T\|f_h(t)\|^2 \rd t \Big]
+\cC\, \me\big[ \|Y_{T,h}-\wt Y_{T,h}\|_{\dot \dbH^{-1}_h}^2\big]  \\
&\qq+\cC\sum_{n=0}^{N-1}\me\Big[\int_{t_n}^{t_{n+1}}\|Z_0(t)-Z_h(t)\|_{\dot \dbH^{-1}_h}^2+\|Z_h(t)-\bar Z_h(t)\|_{\dot \dbH^{-1}_h}^2
+\|f_h(t)-\wt f_h(t)\|_{\dot \dbH^{-1}_h}^2\rd t\Big]\,.
\eal
\ee

Estimate \rf{w1019e2} contains a term for $Z_0(\cd)-Z_h(\cd)$ that we need to bound;
for this purpose,  we rewrite \rf{w1} and get
\begin{equation*}
\setlength\abovedisplayskip{3pt}
\setlength\belowdisplayskip{3pt}
\bal
&(\mathds{1}_h-\t\D_h) e^Y_n+\int_{t_n}^{t_{n+1}}Z_0(t)-Z_h(t)\rd W(t)\\
&=e^Y_{n+1}+\int_{t_n}^{t_{n+1}}\D_h\big[Y_h(t)-Y_h(t_n)\big]+\b\big[\bar Z_0(t)-\bar Z_h(t)\big]
+\b\big[\bar Z_h(t)-Z_h(t)\big]+\big[f_h(t)-\wt f_h(t)\big] \rd t\,.
\eal
 \end{equation*}
By squaring both sides, taking expectations and then applying Young's inequality as well as H\"older's inequality, we have  ($\e>0$)
\begin{equation*}
\setlength\abovedisplayskip{3pt}
\setlength\belowdisplayskip{3pt}
\bal
&\me\big[\|(\mathds{1}_h-\t\D_h) e^Y_n\|_{\dot \dbH^{-1}_h}^2\big]
+\me\Big[\int_{t_n}^{t_{n+1}}\|Z_0(t)-Z_h(t)\|_{\dot \dbH^{-1}_h}^2 \rd t\Big]\\
&\leq (1+4\e)\me\big[\|e^Y_{n+1}\|_{\dot \dbH^{-1}_h}^2\big]
+\Big[4+\frac 1\e\Big]\t \me\Big[\int_{t_n}^{t_{n+1}}\|\D_h\big[Y_h(t)-Y_h(t_n)\big]\|_{\dot \dbH^{-1}_h}^2
+\b^2\|Z_0(t)-Z_h(t)\|_{\dot \dbH^{-1}_h}^2\\
&\qq+\b^2\|Z_h(t)-\bar Z_h(t)\|_{\dot \dbH^{-1}_h}^2 
+\|f_h(t)-\wt f_h(t)\|_{\dot \dbH^{-1}_h}^2 \rd t\Big] \,.
\eal
 \end{equation*}
By choosing $\e$ such that $\big[4+\frac 1\e\big]\b^2\t\leq \frac 1 2$ (for example, we may choose $\t$ small enough satisfying $\t\leq \frac 1{16\b^2}$, and take $\e=\frac{2\b^2\t}{1-8\b^2\t}$), we may absorb the third term on the right-hand
side, such that
\begin{eqnarray*}
&&\me\big[\|e^Y_n\|_{\dot \dbH^{-1}_h}^2\big]
+\frac 1 2 \me\Big[\int_{t_n}^{t_{n+1}}\|Z_0(t)-Z_h(t)\|_{\dot \dbH^{-1}_h}^2 \rd t\Big]\\
&&\leq \Big[1+\frac{8\b^2}{1-8\b^2\t}\t\Big] \me\big[\|e^Y_{n+1}\|_{\dot \dbH^{-1}_h}^2\big]\\
&&\q+\frac 1{8\b^2} \me\Big[\int_{t_n}^{t_{n+1}}\|\D_h\big[Y_h(t)-Y_h(t_n)\big]\|_{\dot \dbH^{-1}_h}^2
+\b^2\|Z_h(t)-\bar Z_h(t)\|_{\dot \dbH^{-1}_h}^2 +\|f_h(t)-\wt f_h(t)\|_{\dot \dbH^{-1}_h}^2 \rd t\Big]\,.
\end{eqnarray*}
We now apply discrete Gronwall's inequality and take summation  over $n$
together with \rf{w1008e2} for $\g=1$ in Lemma \ref{w1008l1} to conclude
\bel{w1019e3}
\setlength\abovedisplayskip{3pt}
\setlength\belowdisplayskip{3pt}
\bal
&\max_{0\leq n\leq N}\me\big[\|e^Y_k\|_{\dot \dbH^{-1}_h}^2\big]
+\me\Big[\int_0^T\|Z_0(t)-Z_h(t)\|_{\dot \dbH^{-1}_h}^2 \rd t\Big]\\
&\qq\leq  \cC\t \me\Big[\|Y_{T,h}\|_{\dot \dbH^{2}_h}^2+\int_0^T \|f_h(t)\|_{\dot \dbH^{1}_h}^2\rd t\Big]
+\cC\bigg\{ \me\big[\|Y_h(T)-Y_N\|_{\dot \dbH^{-1}_h}^2\big]\\
&\qq\q
+ \sum_{n=0}^{N-1}\me\Big[\int_{t_n}^{t_{n+1}} \b^2\|Z_h(t)-\bar Z_h(t)\|_{\dot \dbH^{-1}_h}^2 
+\|f_h(t)-\wt f_h(t)\|_{\dot \dbH^{-1}_h}^2 \rd t\Big]\bigg\}\,.
\eal
\ee 

Finally, the desired assertion can be derived by combining with \rf{w1019e2} and \rf{w1019e3}.
\end{proof}

\subsection{Discretization of Problem {\bf(SLQ)} and rates}\label{nume-fbspde1}
We split the construction and convergence analysis of a numerical method to approximate Problem {\bf(SLQ)} from
Section \ref{ch-1} into two steps: in Section \ref{dis-slq-h} we begin with Problem {\bf (SLQ)}$_h$, which is 
a discretization scheme in space only; related stability bounds for its solution will be provided in Section \ref{fbspde-num1z} to verify rates of convergence for Problem 
{\bf (SLQ)$_{h\t}$} --- the spatio-temporal discretization of Problem {\bf (SLQ)}. 
Conceptually, we follow the approach to `first discretize, then optimize'; see also \cite{Hinze-Pinnau-Ulbrich-Ulbrich09}.

The numerical analysis of
Problem {\bf (SLQ)} was also considered in \cite{Prohl-Wang21,Li-Zhou21, Lv-Wang-Wang-Zhang22,Prohl-Wang22}.
Here we proceed as follows to obtain optimal convergence rates:
\begin{enumerate}[{\rm (a)}]

\item 
We use weaker assumptions on data $x$ 
and $\si(\cd)$ as were imposed in \cite{Prohl-Wang21, Prohl-Wang22}, and derive improved convergence rates in space. 

\item  To  construct  an implementable algorithm for Problem {\bf (SLQ)},
we adopt the Euler
method for temporal discretization.
Then by Malliavin calculus, we verify additional regularity of the optimal control $U^*_h(\cd)$ of Problem {\bf (SLQ)$_{h}$} to  prove an optimal  convergence rate in time for it; see Sections \ref{hidden-reg} and \ref{fbspde-num1z}. 

\end{enumerate}

\subsubsection{Spatial discretization of Problem {\bf(SLQ)}}\label{dis-slq-h}

In this part, we present a semi-discretized  method in space of Problem {\bf (SLQ)}, and prove a rate of convergence. 
Relying on the finite element method, the spatial discretization is stated as follows.\\
\no{\bf Problem (SLQ)$_h$.} Find the optimal control
$U^*_h(\cd) \in L^2_\dbF(0,T; \dbV_h) $ that minimizes the cost functional
\begin{equation} \label{w1003e2h}
\cJ_h \big(U_h(\cd)\big)=\frac 1 2 \me \Big[\int_0^T  \| X_h(t)\|^2+ \| U_h(t) \|^2 \rd t\Big]
+  \frac {\a}{2}\me \big[\|X_h(T)\|^2\big] \,,
\end{equation}
subject to the following SDE
\begin{equation}\label{spde-h}   
\lt\{
\begin{array}{ll}
{\rm d}X_h(t)=\big[\D_h X_h(t)\!+\! U_h(t)\big]\, {\rm d}t\!+\!  \big[\b X_h(t)\!+\!\Pi_h\si(t)\big] {\rm d}W(t) \q  t \in (0,T]\,,\\
\ns\ds X_h(0)= \Pi_h x\,.
\end{array}
\rt.
\end{equation}

The existence of a unique optimal control $ U^*_h(\cd)$ follows from \cite[Chapter 6]{Yong-Zhou99}. Moreover, 
thanks to Pontryagin's maximum principle, $U^*_h(\cd)$ enjoys an  open-loop representation:
\begin{equation}\label{pontr1a}  
 U^*_h(t) - Y_h(t) = 0 \qquad  t \in [0,T]\, ,
\end{equation}
where $\big(Y_h(\cd), Z_h(\cd)\big) \in L^2_{{\mathbb F}}\big( \Omega; C([0,T]; {\mathbb V}_h)\big) \times L^2_{{\mathbb F}}\big( 0,T; {\mathbb V}_h\big)$ solves the BSDE
\begin{equation}\label{bshe1a}
\lt\{
\begin{array}{ll}
\ds {\rm d}Y_h(t) \!=\! \big[\!-\!\D_h Y_h(t)\!-\! \b Z_h(t)\!+\! X^*_h(t)  \big]\rd t
 \!+\!Z_h(t) \rd W(t)  \q   t \in [0,T)\, ,\\
\ns\ds Y_h(T)= -\a X^*_h(T) \,.
\end{array}
\rt.
\end{equation}

The main result in this section is the following theorem which shows optimal rates of convergence for the optimal pair
of Problem {\bf (SLQ)$_h$}.

\bt{rate1}
Suppose that  {\rm \bf (A)} in part {\bf a.} of Section \ref{not1}
holds. Let $\big(X^*(\cd), U^*(\cd)\big)$ and $\big(X^*_h(\cd), U^*_h(\cd)\big)$ be the
optimal pairs to Problems {\bf (SLQ)} and {\bf (SLQ)$_h$}, respectively. 
There exists $\cC $ independent of $h$ such that
\begin{subequations}\label{w1028e4}
    \begin{empheq}[left={\empheqlbrace\,}]{align}
 & \me \Big[\int_0^T \|X^*(t)-X_h^*(t)\|^2+\|U^*(t)-U_h^*(t)\|^2 \rd t\Big]
 \leq \cC h^4 \big[\|x\|_{\dbH_0^1}^2+\|\si(\cd)\|^2_{L^2_\dbF(0,T;\dbH_0^1)} \big]  \, ,\label{w1028e4a}\\
&  \me \Big[\sup_{0 \leq t \leq T}\|X^*(t)-X_h^*(t)\|^2\Big]  
\leq  \cC h^4 \big[\|x\|_{\dbH_0^1\cap\dbH^2}^2+\|\si(\cd)\|^2_{L^2_\dbF(0,T;\dbH_0^1\cap\dbH^2)}\big] \, . \label{w1028e4b}
\end{empheq}
\end{subequations}
\et

\ss
To prove Theorem \ref{rate1}, we introduce several operators $\cS\,,\cT^1\,,\cT^2$ which are related to solution $\big(X(\cd), Y(\cd), Z(\cd)\big)$ of FBSPDE \rf{fbspde}. 
Specifically, the  `control-to-state' map ${\mathcal S}: L^2_\dbF (0,T; \dbL^2) \rightarrow L^2_{{\mathbb F}}\big(\Omega; C([0,T]; {\mathbb H}^1_0)\big) \cap
L^2_\dbF(0,T; \dbH_0^1\cap\dbH^2) \index{${\mathcal S}$}$
is defined by $\cS\big(U(\cd)\big)(\cd)=X\big(\cd;x,U(\cd)\big)$.
If $x = 0$ and $\si(\cd) = 0$, we denote this solution map by $\cS^0$\index{${\mathcal S}^0$}. 
We know that
the solution to equation \eqref{fbspdeb} depends on $X^*(\cd)$, which may be written as 
$\big(Y\big(\cd;X^*(\cd)\big),\,Z\big(\cd;X^*(\cd)\big)\big) = \big(\cT^1\big(X^*(\cd)\big)(\cd),\,\cT^2\big(X^*(\cd)\big)(\cd)\big)$, where 
$\cT^1: L^2_{{\mathbb F}}\big(\Omega; C([0,T]; \dbH_0^1)\big) \rightarrow L^2_\dbF\big(\Omega; C([0,T]; {\mathbb H}^1_0)\big) \cap L^2_\dbF(0,T; \dbH^1_0 \cap \dbH^2)\, , \index{${\mathcal T}^1$}$
$\cT^2: L^2_\dbF\big(\Omega; C([0,T]; \dbH_0^1)\big) \rightarrow  L^2_\dbF(0,T; {\mathbb H}^1_0)\, . \index{${\mathcal T}^2$}$

To shorten notations, we will write $\cS(\cd;U)\,,\cS^0(\cd;U)\,,\cT^i\big(\cd;\cS(U)\big)$ instead of $\cS\big(U(\cd)\big)(\cd)$, $\cS^0\big(U(\cd)\big)(\cd)$, 
$\cT^i\big(\cS\big(U(\cd)\big)\big)(\cd)$ for $i=1,2$.

%
\bl{w212l1}
For every $U(\cd) \in L^2_\dbF (0,T; {\dbL^2}) $, the Fr\'echet derivative
$D \cJ\big(U(\cd)\big)$  is a bounded operator on $L^2_{{\mathbb F}} (0,T; {\dbL^2})$
which takes the form
\begin{equation}\label{derivative-cont}
\setlength\abovedisplayskip{3pt}
\setlength\belowdisplayskip{3pt}
D \cJ\big(U(\cd)\big) = U(\cd) -\cT^1\big(\cd;\cS(U)\big) \, .
\end{equation}
\el

\begin{proof}
By Lemma \ref{reg-spde1} on the stability of SPDE \rf{fbspdea}, we have
\begin{equation*}
\setlength\abovedisplayskip{3pt}
\setlength\belowdisplayskip{3pt}
\me \big[\| \cS^0(t;V)\|^2\big]\leq \cC \|V(\cd)\|^2_{L^2_\dbF (0,T;\dbL^2)} \qq \forall\,t\in [0,T]\,,\, V(\cd)\in L^2_\dbF(0,T;\dbL^2)\,.
\end{equation*}
Hence, for any $U(\cd),\,V(\cd)\in L^2_\dbF(0,T;\dbL^2)$, since the linearity of SPDE \rf{fbspdea} yields $\cS(\cd; U+V)=\cS(\cd; U)+\cS^0(\cd; V)$, we can get for the quadratic  cost functional $\cJ(\cd)$ defined in \rf{intro-1a}
\begin{equation*}
\setlength\abovedisplayskip{3pt}
\setlength\belowdisplayskip{3pt}
\bal
&\cJ\big(U(\cd)+V(\cd)\big)-\cJ\big(U(\cd)\big)-\Big[\big(\cS(\cd; U),\cS^0(\cd; V)\big)_{L^2_\dbF(0,T;\dbL^2)}\\
 &\qq+\big(\a \cS(T;U),\cS^0(T;V)\big)_{L^2_{\mf_T}(\O;\dbL^2)}+\big(U(\cd),V(\cd)\big)_{L^2_\dbF(0,T;\dbL^2)}\Big]\\
&\q=\frac 1 2 \Big[ \big\| \cS^0(\cd; V)\big\|^2_{L^2_\dbF(0,T;\dbL^2)}+\|V(\cd)\|^2_{L^2_\dbF(0,T;\dbL^2)}
+ \a\big\| \cS^0(T; V)\big\|^2_{L^2_{\mf_T}(\O;\dbL^2)}\Big]\,.
\eal
\end{equation*}
On the other hand, It\^o's formula to $\big(\cS^0(t; V),\cT^1\big(t; \cS(U)\big)\big)_{\dbL^2}$ yields 
\begin{equation*}
\setlength\abovedisplayskip{3pt}
\setlength\belowdisplayskip{3pt}
\bal
\big( \cS(\cd; U),\cS^0(\cd; V) \big)_{L^2_\dbF(0,T;\dbL^2)}+ \big( \a\cS(T;U),\cS^0(T;V) \big)_{L^2_{\mf_T}(\O;\dbL^2)}
=-\big(\cT^1\big(\cd; \cS(U)\big), V(\cd)\big)_{L^2_\dbF(0,T;\dbL^2)}\,.
\eal
\end{equation*}
By the definition of Fr\'echet derivative, the desired result now follows.
\end{proof}

Similar to the definition of $\cS$, by the unique solvability property of \eqref{spde-h}, we associate to this equation the  bounded solution operator 
${\mathcal S}_h: L^2_{{\mathbb F}}(0,T; \dbV_h) \rightarrow L^2_{{\mathbb F}}\big(\Omega; C([0,T]; {\mathbb V}_h)\big)\,.
\index{${\mathcal S}_h$}$
When $x = 0$ and $\si(\cd) = 0$, we denote this solution operator by $\cS_h^0$.\index{${\mathcal S}^0_h$} 
Similarly, the solution pair to BSDE \eqref{bshe1a} may be written in the following form 
$\big(Y_h\big(\cd; X^*_h(\cd)\big)\,, Z_h\big(\cd; X^*_h(\cd)\big)\big) = \big(\cT^1_h\big(X^*_h(\cd)\big)(\cd)\,,\cT^2_h\big(X^*_h(\cd)\big)(\cd)\big)$,
where 
$\cT^1_h: L^2_{{\mathbb F}}\big(\Omega; C([0,T]; {\mathbb V}_h) \big) \rightarrow L^2_{{\mathbb F}}\big(\Omega; C([0,T]; {\mathbb V}_h) \big)\, , \index{${\mathcal T}^1_h$}$
$\cT^2_h: L^2_{{\mathbb F}}\big(\Omega; C([0,T]; {\mathbb V}_h) \big) \rightarrow L^2_{{\mathbb F}}(0,T; {\mathbb V}_h )\, . \index{${\mathcal T}^2_h$}$

Like the simplified notations $\cS(\cd;U)\,,\cS^0(\cd;U)\,,\cT^i\big(\cd;\cS(U)\big)$ for $i=1,2$, we here also introduce
$\cS_h(\cd;U_h)\,,\cS^0_h(\cd;U_h)\,,\cT^i_h\big(\cd;\cS_h(U_h)\big)$.

\ss
We are now in a position to prove Theorem \ref{rate1}.

\begin{proof}[\bf {Proof of Theorem \ref{rate1}}] 

{\bf (1) Verification of \rf{w1028e4a}. }
For any control variable  $U_h(\cd) \in L^2_\dbF(0,T; \dbV_h ) $, the Fr\'echet derivative 
$D \cJ_h\big(U_h(\cd)\big)$ is a bounded operator (uniformly in $h$) on $L^2_\dbF(0,T; \dbV_h) $, and has the form
\begin{equation}\label{derivative-semidisc}
\setlength\abovedisplayskip{3pt}
\setlength\belowdisplayskip{3pt}
D \cJ_h \big(U_h(\cd)\big)=U_h(\cd)-\cT_h^1\big(\cd; \cS_h (U_h)\big)\, ,
\end{equation}
which can be deduced by a similar procedure as  in the proof of Lemma \ref{w212l1}.
Let $U_h(\cd)\,,R_h(\cd) \in L^2_{{\mathbb F}}(0,T; {{\mathbb V}_h})$ be arbitrary; it is due to the quadratic structure of the 
cost functional $\cJ_h(\cd)$ given in \eqref{w1003e2h} that
\begin{equation*}
\setlength\abovedisplayskip{3pt}
\setlength\belowdisplayskip{3pt}
D^2  \cJ_h\big(U_h(\cd)\big)R_h(\cd)=R_h(\cd)-\cT^1_h\big(\cd; \cS_h^0(R_h)\big)\,.
\end{equation*}
Note that the right-hand side of the above equation is independent of $U_h(\cd)$.
Then by It\^o's formula to $\big( \cS^0_h(t; R_h), \cT^1_h\big(t; \cS^0_h(R_h)\big)\big)_{\dbL^2}$ we get
\begin{equation*}
\setlength\abovedisplayskip{3pt}
\setlength\belowdisplayskip{3pt}
\bal
&\Big(\cT^1_h\big(\cd; \cS_h^0(R_h)\big), R_h(\cd)\Big)_{L^2_\dbF(0,T;\dbL^2)}\\
&=-\a\big(\cS_h^0(T; R_h),\cS_h^0(T; R_h)\big)_{L^2_{\mf_T}(\O;\dbL^2)}
-\big(\cS_h^0(\cd; R_h),\cS_h^0(\cd; R_h)\big)_{L^2_\dbF(0,T;\dbL^2)}\,,
\eal
\end{equation*}
and therefore 
\begin{equation*}
\setlength\abovedisplayskip{3pt}
\setlength\belowdisplayskip{3pt}
\bal
&\Big( D^2 \cJ_h\big(U_h(\cd)\big) R_h(\cd),R_h(\cd) \Big)_{L^2_\dbF(0,T;\dbL^2)} \\
&\qq= \| R_h(\cd) \|^2_{L^2_\dbF(0,T;\dbL^2)}+  \|\cS^0_h(\cd; R_h) \|^2_{L^2_\dbF(0,T;\dbL^2)} 
   +\a \|\cS^0_h(T; R_h) \|^2_{L^2_{\mf_T}(\O;\dbL^2)}     \\
&\qq\geq \| R_h(\cd)\|^2_{L^2_\dbF(0,T; {\mathbb L}^2)}
\qq \forall\, R_h(\cd) \in 
L^2_\dbF(0,T; \dbV_h )\, .
\eal
\end{equation*}
As a consequence, on putting $R_h(\cd) = U^*_h(\cd) - \Pi_h U^*(\cd)$, we have
\begin{equation*}
\setlength\abovedisplayskip{3pt}
\setlength\belowdisplayskip{3pt}
\bal
 \| U^*_h(\cd) - \Pi_h U^*(\cd)\|_{L^2_\dbF(0,T; {\mathbb L}^2)}^2
&\leq  
\big( D\cJ_h\big(U^*_h(\cd)\big) ,U^*_h(\cd) - \Pi_h U^*(\cd)\big)_{L^2_\dbF(0,T;\dbL^2)} \\
&\q- \big( D\cJ_h\big(\Pi_h U^*(\cd)\big) ,U^*_h(\cd) - \Pi_h U^*(\cd)\big)_{L^2_\dbF(0,T;\dbL^2)}\, .
\eal
\end{equation*}
and then
\begin{equation*}
\setlength\abovedisplayskip{3pt}
\setlength\belowdisplayskip{3pt}
 \| U^*_h(\cd) - \Pi_h U^*(\cd)\|_{L^2_\dbF(0,T; {\mathbb L}^2)}
\leq   \big\| D\cJ_h\big(U^*_h(\cd)\big) - D\cJ_h\big(\Pi_h U^*(\cd)\big) \big\|_{L^2_\dbF(0,T;\dbL^2)}\, .
\end{equation*}
Noticing  that $D\cJ_h\big(U^*_h(\cd)\big) = 0$ by \eqref{pontr1a}, as well as $D \cJ\big(U^*(\cd)\big) = 0$ by \eqref{op-condition1}, 
we find that
\bel{w229e1}
\setlength\abovedisplayskip{3pt}
\setlength\belowdisplayskip{3pt}
\bal
 \| U^*_h (\cd)- \Pi_h U^*(\cd)\|_{L^2_\dbF(0,T; {\mathbb L}^2)}^2
&\leq2 \Big[  \big\| D\cJ\big(U^*(\cd)\big)-D\cJ\big(\Pi_hU^*(\cd)\big) \big\|^2_{L^2_\dbF(0,T;\dbL^2)} \\
&\q       + \big\|  D\cJ\big(\Pi_hU^*(\cd)\big)-D\cJ_h\big(\Pi_hU^*(\cd)\big) \big\|^2_{L^2_\dbF(0,T;\dbL^2)}  \Big]\\
& =:2(I_1+I_2)\,.       
\eal
\ee
In the following, we bound these two terms independently.

 We use \eqref{derivative-cont} to bound $I_1$ as follows,
\begin{equation*}
\setlength\abovedisplayskip{3pt}
\setlength\belowdisplayskip{3pt}
I_1\! \leq \! 2 \Big[\|U^*(\cd) \!-\! \Pi_h U^*(\cd)\|^2_{L^2_\dbF(0,T;\dbL^2)} 
      \! + \big\|\cT^1\big(\cd; \cS(\Pi_h U^*) \big) \!-\!\cT^1\big(\cd; \cS(U^*) \big)\big\|^2_{L^2_\dbF(0,T;\dbL^2)}\Big] \,.
\end{equation*}
By  stability property \eqref{vari-2} in Lemma \ref{reg-bshe}  for BSPDE \eqref{fbspdeb} with data 
$Y(T)=-\a \big[\cS(T; \Pi_hU^*)- \cS(T; U^*)\big]\,,f(\cd)=\cS(\cd; \Pi_hU^*)-\cS(\cd; U^*)$, as well as \rf{reg-e1} in Lemma \ref{reg-spde1} for SPDE \eqref{fbspdea} with $x=0\,, f(\cd)=U^*(\cd)-\Pi_hU^*(\cd)$, $g(\cd)=0$, 
the second term in the above inequality may be bounded by
\bel{estima}
\setlength\abovedisplayskip{3pt}
\setlength\belowdisplayskip{3pt}
\bal
&\leq \cC \Big[\| \big(\cS(T; U^*)-\cS(T; \Pi_h U^*)\big) \|^2_{L^2_{\mf_T}(\O;\dot\dbH^{-1})} 
           +  \| \cS(\cd; U^*)-\cS(\cd; \Pi_h U^*)\|^2_{L^2_\dbF(0,T; \dot\dbH^{-2})}\Big]\\ 
&\leq \cC \| U^*(\cd) - \Pi_h U^*(\cd)\|^2_{L^2_\dbF(0,T; \dot\dbH^{-1})}\, .
\eal
\ee
By optimality condition \eqref{op-condition1},  and the regularity properties of the solution
of FBSPDE \eqref{fbspde}, we know that 
$U^*(\cd)=Y(\cd) \in L^2_\dbF(0,T; \dbH_0^1\cap\dbH^2)$; as a consequence,  combining  Lemma \ref{w1017l1}
with \rf{w1008e5b} in Lemma \ref{w1008l2} leads to
\begin{equation*}
\setlength\abovedisplayskip{3pt}
\setlength\belowdisplayskip{3pt}
\bal
I_1&\leq \cC h^4\|Y(\cd)\|^2_{L^2_\dbF(0,T;\dbH_0^1\cap\dbH^2)}\\
&\leq \cC h^4\big[\|X^*(T)\|_{L^2_{\mf_T}(\O;\dbH_0^1)}^2+\|X^*(\cd)\|^2_{L^2_\dbF(0,T;\dbL^2)}\big]\\
&\leq \cC h^4 \big[\|x\|_{\dbH_0^1}^2+\|\si(\cd)\|^2_{L^2(0,T;\dbH_0^1)}\big]\,.
\eal
\end{equation*}

In the below, 
we apply  representations \rf{derivative-cont}, \rf{derivative-semidisc} to bound $I_2$ via
\begin{equation*}
\setlength\abovedisplayskip{3pt}
\setlength\belowdisplayskip{3pt}
\bal
I_2 
&\leq  2 \Big[\big\|\cT^1\big(\cd; \cS(\Pi_h U^*) \big) - \cT^1 \big( \cd; \cS_h(\Pi_h U^*)\big)\big\|^2_{L^2_\dbF(0,T;\dbL^2)}\\
 &\qq+\big\|\cT^1\big(\cd; \cS_h(\Pi_h U^*) \big) - \cT^1_h\big(\cd; \cS_h(\Pi_h U^*)\big)\big\|^2_{L^2_\dbF(0,T;\dbL^2)}\Big]\\
&=: 2 (I_{21}+I_{22})\,.
\eal
\end{equation*}
In order to bound $I_{21}$, we use the stability properties in Lemma \ref{reg-bshe} for  BSPDE \eqref{fbspdeb}, in combination with 
Theorem \ref{w1002t1} on the error estimate for SPDE \eqref{fbspdea} 
(note that $f(\cd)=f_h(\cd)=\Pi_hU^*(\cd)$, $g(\cd)=\si(\cd)\,,g_h(\cd)=\Pi_h\si(\cd)$), as well as
 \rf{w1008e5b} in Lemma \ref{w1008l2} to conclude that
 \begin{equation*}
\setlength\abovedisplayskip{3pt}
\setlength\belowdisplayskip{3pt}
\bal
\ds I_{21} &\leq \cC \Big[ \big\| \cS(\cd; \Pi_h U^*) - \cS_h(\cd; \Pi_h U^*) \big\|^2_{L^2_\dbF(0,T; \dot\dbH^{-1})}
+\big\| \cS(T; \Pi_h U^*)- \cS_h(T; \Pi_h U^*) \big\|^2_{L^2_{\mf_T}(\O;\dot\dbH^{-2})}\Big]\\
&\leq \cC h^4 \big[\|x\|_{\dbH_0^1}^2+\|\si(\cd)\|^2_{L^2(0,T;\dbH_0^1)}\big]\,.
\eal
\end{equation*}
For $I_{22}$, we use Theorem \ref{w0911t1}  with $Y_T=Y_{T,h}=-\a \cS_h(T; \Pi_h U^*)\,, f(\cd)=f_h(\cd)=  \cS_h(\cd; \Pi_h U^*)$ on the error estimate for BSPDE \eqref{fbspdeb}, and then Lemma \ref{w1002l4}  with $f_h(\cd)=\Pi_h U^*(\cd)$, $g_h(\cd)=\Pi_h\si(\cd)$ on the stability properties  for 
SDE \rf{spde-h} to find
 \begin{equation*}
\setlength\abovedisplayskip{3pt}
\setlength\belowdisplayskip{3pt}
\bal
I_{22} 
&\leq \cC h^4 \big[ \a^2\|\cS_h(T; \Pi_h U^*)\|_{L^2_{\mf_T}(\O;\dbH_0^1)}^2+\|\cS_h(\cd; \Pi_h U^*)\|^2_{L^2_\dbF(0,T;\dbL^2)}\big]\,.\\
&\leq \cC h^4 \big[\|x\|_{\dbH_0^1}^2+\|\si(\cd)\|^2_{L^2(0,T;\dbH_0^1)}\big]\,.
\eal
\end{equation*}
We now insert these estimates into \eqref{w229e1}, and utilize the optimality  condition \eqref{op-condition1} to obtain the bound
\begin{equation*}
\setlength\abovedisplayskip{3pt}
\setlength\belowdisplayskip{3pt}
\bal
\|U^*(\cd)-U^*_h(\cd)\|^2_{L^2_\dbF(0,T;\dbL^2)} 
&\leq 2\big[ \|U^*(\cd)-\Pi_hU^*(\cd)\|^2_{L^2_\dbF(0,T;\dbL^2)} +\| U^*_h(\cd) - \Pi_h U^*(\cd) \|_{L^2_\dbF(0,T; \dbL^2)}^2\big] \\
& \leq \cC h^4 \big[ \|x\|_{\dbH_0^1}^2+\|\si(\cd)\|^2_{L^2(0,T;\dbH_0^1)}\big] \, . 
\eal
\end{equation*}
This is the estimate for $U^*(\cd)-U^*_h(\cd)$ in assertion \rf{w1028e4a} in Theorem \ref{rate1}. 
Then applying Theorem \ref{w1002t1} by setting $f(\cd)=U^*(\cd)\,,f_h(\cd)=U^*_h(\cd)\,, g(\cd)=\si(\cd)\,,g_h(\cd)=\Pi_h\si(\cd)$, 
we have
\begin{equation*}
\setlength\abovedisplayskip{3pt}
\setlength\belowdisplayskip{3pt}
\|X^*(\cd)-X^*_h(\cd)\|^2_{L^2_\dbF(0,T;\dbL^2)} 
 \leq \cC h^4 \big[ \|x\|_{\dbH_0^1}^2+\|\si(\cd)\|^2_{L^2(0,T;\dbH_0^1)}\big] \, ,
\end{equation*}
which is  the estimate for $X^*(\cd)-X^*_h(\cd)$ in assertion \rf{w1028e4a}.

\ss
{\bf (2) Verification of \rf{w1028e4b}. }
By Theorem \ref{w1002t1} with $f(\cd)=U^*(\cd)\,, f_h(\cd)=U^*_h(\cd)\,, g(\cd)=\si(\cd)\,,g_h(\cd)=\Pi_h\si(\cd)$, and assertion \rf{w1028e4a}, 
it follows that
\begin{equation*}
\setlength\abovedisplayskip{3pt}
\setlength\belowdisplayskip{3pt}
\bal
&\me \Big[\sup_{t\in[0,T]}  \|X^*(t)-X_h^*(t)\|^2\Big]\\
&\leq \cC
h^4\bigg\{ \|x\|_{\dbH_0^1\cap\dbH^2}^2 +\me\Big[\int_0^T \|U^*(t)\|_{\dbH_0^1\cap\dbH^2}^2+\|\si(t)\|_{\dbH_0^1\cap\dbH^2}^2\rd t \Big]\bigg\}
+\cC\, \me\Big[\int_0^T  \|U^*(t)-U^*_h(t)\|^2\rd t\Big]\\
& \leq \cC h^4 \big[\|x\|_{\dbH_0^1\cap\dbH^2}^2+\|\si(\cd)\|^2_{L^2(0,T;\dbH_0^1\cap\dbH^2)}\big] \, .
\eal
\end{equation*}
That completes the proof of assertion \rf{w1028e4b}.
\end{proof}

\br{w1121r1}
We may use
Theorem \ref{rate1} to verify further error bounds,
\vspace{-2ex}
\begin{subequations}\label{w1028e45}
    \begin{empheq}[left={\empheqlbrace\,}]{align}
 &   {\mathbb E} \Big[ \sup_{t\in[0,T]} \|X^*(t)-X_h^*(t)\|^2\Big]
       + \me\Big[\int_0^T \|X^*(t)-X_h^*(t)\|_{\dbH_0^1}^2 \rd t\Big]  \nonumber  \\
&\qq\qq\qq\qq\leq \cC h^2 \big[\|x\|_{\dbH_0^1}^2+\|\si(\cd)\|^2_{L^2(0,T;\dbH_0^1)}\big]  \, ,  \label{w1028e4c}\\
 &   \me \Big[\!\sup_{t\in[0,T]} \|Y(t)\!-\!Y_h(t)\|^2\! \Big]\!+\! \me\Big[ \! \int_0^T\! \|Y(t)\!-\!Y_h(t)\|_{\dbH_0^1}^2\!+\!\|Z(t)\!-\!Z_h(t)\|^2\rd t \Big]   \nonumber  \\
  &\qq\qq\qq\qq\leq \cC h^2 \big[\|x\|_{\dbH_0^1}^2+\|\si(\cd)\|^2_{L^2(0,T;\dbH_0^1)}\big]  \, , \label{w1028e4d}\\
 &  \sup_{ t\in[0,T] }\me \big[ \|Y(t)-Y_h(t)\|^2 \big] 
  \leq \cC h^4 \big[\|x\|_{\dbH_0^1\cap\dbH^2}^2+\|\si(\cd)\|^2_{L^2(0,T;\dbH_0^1\cap\dbH^2)}\big]  \, .\label{w1028e4e}
\end{empheq}
\end{subequations}

Indeed, Theorem \ref{w1002t1} with $f(\cd)=U^*(\cd)\,, f_h(\cd)=U^*_h(\cd)\,, g(\cd)=\si(\cd)\,,g_h(\cd)=\Pi_h\si(\cd)$, and assertion \rf{w1028e4a} easily settle \rf{w1028e4c}.

\rf{w1028e2a} in Theorem \ref{w0911t1} with $Y(T)=-\a X^*(T)\,, Y_h(T)=-\a X^*_h(T)$,  $f(\cd)=X^*(\cd)\,, f_h(\cd)=X^*_h(\cd)$, \rf{w1022e1c} in Lemma \ref{w1017l1} and assertion \rf{w1028e4a} lead to  \rf{w1028e4d}.

Similarly, by 
 \rf{w1028e2b} in Theorem \ref{w0911t1},  Lemma \ref{reg-spde1} with $f(\cd)=U^*(\cd)\,, g(\cd)=\si(\cd)\,, \g=2$,  \rf{w1008e5b} in Lemma \ref{w1008l2}, and assertion \rf{w1028e4b}, we arrive at   \rf{w1028e4e}.
\er

\subsubsection{Bounds in strong norms for the optimal pair to Problem {\bf (SLQ)$_h$}}\label{hidden-reg}

In this part, we verify uniform (in $h$) bounds in stronger norms for
the optimal pair $\big(X^*_h(\cd),U^*_h(\cd)\big)$ to Problem {\bf (SLQ)$_h$}, which
will be applied in the error analysis of the temporal discretization of Problem {\bf (SLQ)$_h$}.

To begin with, we introduce a family of SLQ problems, which are parametrized by $t \in [0,T]$; for this purpose, we consider the controlled SDE
\bel{see}
\setlength\abovedisplayskip{5pt}
\setlength\belowdisplayskip{5pt}
\lt\{\!\!\!
\begin{array}{ll}
\ds{\rm d}X_h(s)\!=\!\big[\D_h X_h(s)\!+\! U_h(s)\big]\rd s\!+\!  \big[\b X_h(s)\!+\!\Pi_h\si(s)\big] \rd W(s) 
     \q  s \in (t,T]\,,\\
\ns \ds X_h(t)= \Pi_h x
\end{array}
\rt.
\ee
with  $x\in \dbH_0^1\cap\dbH^2$, and the (parametrized) cost functional
\bel{cost}
\setlength\abovedisplayskip{5pt}
\setlength\belowdisplayskip{5pt}
\cJ_h\big(t,x;U_h(\cd)\big)=\frac 1 2 \me \Big[\int_t^T \|X_h(s)\|^2+\|U_h(s)\|^2 \rd s \Big]
    +\frac \a 2 \me \big[\|X_h(T)\|^2\big]\,.
\ee
Now we define auxiliary SLQ problems as follows.\\
\no{\bf Problem (SLQ)$^{t;h}_{\tt aux}$.} For any  given $t\in[0,T)$ and $x\in \dbL^2$, search for 
$U^*_h(\cd)\in L^2_\dbF(t,T;\dbV_h)$ such that
\bel{SLQ-cost-h}
\setlength\abovedisplayskip{5pt}
\setlength\belowdisplayskip{5pt}
\cJ_h\big(t,x; U^*_h(\cd)\big)=\inf_{U_h(\cd) \in L^2_\dbF(t,T;\dbV_h)}\cJ_h\big(t,x; U_h(\cd)\big) \,.
\ee

Obviously   $\cJ_h\big(U_h(\cd)\big)=\cJ_h\big(0,x;U_h(\cd)\big)$.
The unique solvability of Problem {\bf (SLQ)$^{t;h}_{\tt aux}$} is shown in \cite{Yong-Zhou99}.
Moreover, the stochastic Riccati equation related to  Problem {\bf (SLQ)$^{t;h}_{\tt aux}$} reads:
\bel{riccati}
\setlength\abovedisplayskip{5pt}
\setlength\belowdisplayskip{5pt}
\lt\{\!\!\!
\begin{array}{ll}
\ds \cP_h'(t)+\cP_h(t) \D_h + \D_h \cP_h(t)+ \b^2 \cP_h(t)+\mathds{1}_h
  - \cP_h^2(t)=0\q  t\in [0,T]\,,\\
\ns \ds \cP_h(T)=\a\mathds{1}_h \,,
\end{array}
\rt.
\ee
and we consider a backward ODE 
\bel{ode}
\setlength\abovedisplayskip{5pt}
\setlength\belowdisplayskip{5pt}
\lt\{\!\!\!
\begin{array}{ll}
\ds \eta_h'(t)+\big[\D_h-\cP_h(t)\big]\eta_h(t)+  \cP_h(t)\Pi_h\si(t)=0 \qq t\in [0,T]\,,\\
\ns\ds \eta_h(T)=0 \,.
\end{array}
\rt.
\ee
By \cite[Chapter 6, Theorems 6.1, 7.2]{Yong-Zhou99}, 
the stochastic Riccati equation \eqref{riccati} 
admits a unique solution
$\cP_h(\cd)\in C\big([0,T];\dbS_+(\dbV_h)\big)$, subsequently \rf{ode} has a unique solution $\eta_h(\cd)\in C([0,T];\dbV_h)$, and that
\bel{w229e12}
\setlength\abovedisplayskip{5pt}
\setlength\belowdisplayskip{5pt}
\bal
\cJ_h\big(t,x; U^*_h(\cd)\big)=&\frac 1 2 \big(\cP_h(t)\Pi_hx, \Pi_hx \big)_{\dbL^2}+  \big(\eta_h(t),\Pi_hx \big)_{\dbL^2}\\
&+\frac 1 2 \int_t^T\big[ \big(\cP_h(s)\Pi_h\si(s),\Pi_h \si(s) \big)_{\dbL^2}+\|\eta_h(s)\|^2\big] \rd s\, ,
\eal
\ee
where  $U^*_h(\cd)$ is the optimal control owning the following state feedback form
\bel{w1011e2}
\setlength\abovedisplayskip{5pt}
\setlength\belowdisplayskip{5pt}
U^*_h(t)=-\cP_h(t)X^*_h(t)-\eta_h(t) \qq t\in[0,T]\,.
\ee

\bl{w229l4}
Let $\cP_h(\cd)$ be the solution to Riccati equation \eqref{riccati}, and $\eta_h(\cd)$ solves \rf{ode}. There exists a constant $\cC$ 
independent of $h$ such that
\begin{subequations}\label{w1029e1}
    \begin{empheq}[left={\empheqlbrace\,}]{align}
      &\sup_{t\in[0,T]}\|\cP_h(t)\|_{\cL(\dbL^2|_{\dbV_h})}\leq \cC\,,  \label{w1029e1a}\\
&  \sup_{t\in [0,T]}\|\eta_h(t)\|^2+ \int_0^T \big[ \|\nb \eta_h(t)\|^2+\big(\cP_h(t)\eta_h(t),\eta_h(t)\big)_{\dbL^2} \big] \rd t 
 \leq \cC\|\si(\cd)\|_{L^2(0,T;\dbL^2)}^2 \,. \label{w1029e1b}
\end{empheq}
\end{subequations}

%
\el

\begin{proof}
{\bf (1) Verification of \rf{w1029e1a}.} We consider Problem {\bf (SLQ)$^{t;h}_{\tt aux}$}
for \rf{see} with $\si(\cd)= 0$.
In this case, it follows that $\eta_h(\cd)=0$.
For the control variable 
$U_h(\cd)= 0$,  Lemma \ref{w1002l4} yields
\begin{equation*}
\setlength\abovedisplayskip{5pt}
\setlength\belowdisplayskip{5pt}
\sup_{s\in[t,T]}\me \big[\|X_h(s;\Pi_hx, 0)\|^2\big] \leq \cC\, \me\big[\|\Pi_hx\|^2\big].
\end{equation*}
Hence, \eqref{SLQ-cost-h} and \eqref{w229e12} lead to
\begin{equation*}
\setlength\abovedisplayskip{5pt}
\setlength\belowdisplayskip{5pt}
\big(\cP_h(t)\Pi_hx, \Pi_hx \big)_{\dbL^2}\leq 2\cJ_h(t,x; 0)\leq  \cC\, \me\big[\|\Pi_hx\|^2\big] \, ,
\end{equation*}
which, together with the facts that $\cP_h(\cd)$ is nonnegative and $\Pi_h$ is surjective, implies assertion  \rf{w1029e1a}.

\ss
{\bf (2) Verification of \rf{w1029e1b}.}  Since $\cP_h(\cd)$ is nonnegative, we infer from \rf{ode} that
\begin{equation*}
\setlength\abovedisplayskip{5pt}
\setlength\belowdisplayskip{5pt}
\bal
&\|\eta_h(t)\|^2+2\int_t^T\lt[ \|\nb \eta_h(s)\|^2+\big(\cP_h(s)\eta_h(s),\eta_h(s)\big)_{\dbL^2} \rt]\rd s\\
&\qquad \leq \int_t^T \|\eta_h(s)\|^2\rd s+\int_t^T \|\cP_h(s)\Pi_h\si(s)\|^2\rd s\,.
\eal
\end{equation*}
Then Gronwall's inequality and \rf{w1029e1a} settle the assertion. 
\end{proof}

The following lemma collects bounds for the solution $X^*_h(\cd)$ and its Malliavin derivatives 
$D_\th X^*_h(\cd)$, $D_\mu D_\th X^*_h(\cd)$
to SDE \eqref{see} with $U_h(\cd) = U^*_h(\cd)$, by exploiting the state
feedback representation \rf{w1011e2} of the optimal control $U^*_h(\cd)$, and the bounds for the stochastic 
Riccati operator in Lemma \ref{w229l4}, in particular.
These bounds will be applied to 
derive a convergence rate for an open-loop based scheme.

\bl{w229l2}
Let $\big(X^*_h(\cd),U^*_h(\cd)\big)$ be the optimal pair of Problem {\bf (SLQ)$_h$}.  Then it holds that $X^*_h(t) \in \dbD^{2,2}(\dbL^2)$ for any $t \in [0,T]$,
 and there exists a constant $\cC$ independent of $h$ such that
 {\small
 \begin{subequations}\label{w1029e2}
    \begin{empheq}[left={\empheqlbrace\,}]{align}
 &\me \Big[\sup_{t\in [0,T]}\|X^*_h(t)\|_{\dot\dbH^\g_h}^2\Big] +\me\Big[\int_0^T\|X^*_h(t)\|_{\dot \dbH^{\g+1}_h}^2\rd t\Big] \nonumber \\
&\qq\qq\leq \cC \big[\|x\|_{\dot \dbH^{\g}}^2+\|\si(\cd)\|_{L^2(0,T;\dot \dbH^{\g})}^2\big]\qq \g=0,1\,,  \label{w1029e2a}\\
& \me \Big[\sup_{t\in [0,T]}\|U^*_h(t)\|^2\Big] +\me\Big[\int_0^T\|U^*_h(t)\|_{\dot \dbH^{1}_h}^2\rd t\Big]  \notag\\
&\qq\qq \leq \cC \big[\|x\|^2+\|\si(\cd)\|_{L^2(0,T;\dbL^2)}^2\big] \,, \label{w1029e2b}\\
&\Big(\sup_{t\in[0,T]}\me\big[\|U^*_h(t)\|^4\big]\Big)^{1/2}
   \leq C \big[\|x\|^2+\|\si(\cd) \|_{C([0,T];\dbL^2)}^2\big]\, , \label{w1029e2c}
\end{empheq}
\end{subequations}
}
{\small
 \begin{subequations}\label{w1029e3}
    \begin{empheq}[left={\empheqlbrace\,}]{align}
 &\sup_{\th\in[0,T]}\me \Big[\sup_{t\in[\th,T]}\|D_\th X^*_h(t)\|^2\Big]
+\sup_{\th,\mu \in[0,T]}\me \Big[\sup_{t\in[\mu\vee\th,T]}\|D_\mu D_\th X^*_h(t)\|^2\Big] \nonumber \\
&\qq\qq\leq \cC \big[\|x\|^2+\|\si(\cd)\|_{C([0,T];\dbL^2)}^2\big]\,,  \label{w1029e3a}\\
&\sup_{\th\in[0,T]}\bigg\{ \me\Big[\sup_{t\in[\th,T]}\|D_\th X^*_h(t)\|_{\dot \dbH^1_h}^2\Big]+\me\Big[\int_\th^T\|D_\th X^*_h(t)\|_{\dot \dbH^{2}_h}^2\rd t\Big]\bigg\} \nonumber \\
&\qq\qq\leq \cC\big[\|x\|_{\dbH_0^1}^2+\|\si(\cd) \|_{C([0,T];\dbH_0^1)}^2\big] \,, \label{w1029e3b}
\end{empheq}
\end{subequations}
 }
 and
 {\small
 \begin{subequations}\label{w1029e33}
    \begin{empheq}[left={\empheqlbrace\,}]{align}
&\me \big[\|X^*_h(t)-X^*_h(s)\|^2\big] \notag\\
&\qq\qq \leq \cC |t-s| \big[ \|x\|_{\dbH_0^1}^2+\|\si(\cd)\|^2_{C([0,T];\dbL^2) \cap L^2(0,T;\dbH_0^1)} \big]\qq t,s\in [0,T]\, , \label{w1029e2d}\\
&\me\Big[\sup_{t\in[\th_1,T]}\|(D_{\th_1}-D_{\th_2}) X^*_h(t)\|^2\Big] \nonumber \\
&\qq\qq \leq \cC(\th_1-\th_2)\big[\|x\|_{\dbH_0^1}^2+\|\si(\cd) \|_{C([0,T];\dbH_0^1)}^2+L_{\si,1/2}^2\big]\qq \th_2\leq \th_1\, . \label{w1029e3c}
\end{empheq}
\end{subequations}
 }
\el

\begin{proof}
{\bf (1) Verification of \rf{w1029e2a}--\rf{w1029e2c}. } 
We may follow the same steps that lead to \rf{w1028e6}, \rf{w1008e5c}, \rf{w1008e5b} and \rf{w1008e5a}.

%
 

\ss
{\bf (2) Verification of \rf{w1029e2d}. }  Based on \rf{w1029e2a}--\rf{w1029e2b}, the It\^o isometry together with the triangle
inequality implies that for $s\leq t$, 
\begin{equation*}
\setlength\abovedisplayskip{3pt}
\bal
\me \big[\|X^*_h(t)-X^*_h(s)\|^2\big]&\leq \cC\bigg\{ |t-s| \me\Big[\int_0^T \|\D_h X^*_h(\th)\|^2 +  \|U^*_h(\th)\|^2 \rd \th \Big] \\
&\qq+\me\Big[\int_s^t \|\b X^*_h(\th)\|^2 +  \|\Pi_h\si (\th)\|^2 \rd \th \Big] \bigg\}\\
&\leq \cC\|t-s\| \big[ \|x\|_{\dbH_0^1}^2+\|\si(\cd)\|^2_{L^2(0,T;\dbH_0^1)\cap C([0,T];\dbL^2)}\big]\,.
\eal
\end{equation*}

{\bf (3) Verification of \rf{w1029e3a}. } 
To estimate the first term on the left-hand side, we use the facts that $ \cP_h(\cd)$ and $\eta_h(\cd)$ are deterministic,
 and \cite[Theorem 2.2.1]{Nualart06} to conclude that the first Malliavin derivative of $X_h^*(\cd)$ exists.  We take the Malliavin
derivative on both sides of \eqref{see}, together with the sate feedback control \rf{w1011e2} to get
\bel{w301e8}
\lt\{\!\!\!
\begin{array}{ll}
\ds \rd D_\th X^*_h(t)\!=\!\big[\D_h\!-\! \cP_h(t) \big] D_\th X^*_h(t)\rd t
   \!+\! \b D_\th X^*_h(t)\rd W(t) \q  t\in (\th,T]\,,\\
\ns\ds  D_\th X^*_h(\th)=\b  X^*_h(\th)+\Pi_h\si(\th)\, ,\\
\ns\ds D_\th X^*_h(t)=0 \qq  t\in [0,\th)\, .
\end{array}
\rt.
\ee
Lemma \ref{w1002l4} and the fact that $\cP_h(\cd)\in C\big([0,T];\dbS_+(\dbV_h)\big)$ lead to
\begin{equation*}
\bal
\me \big[\sup_{t\in[\th,T]}\|D_\th X^*_h(t)\|^2\big]
\leq \cC\, \me \big[\|X_h^*(\th)+\Pi_h\si(\th)\|^2 \big]
 \cC\big[\|x\|^2+\|\si(\cd)\|_{C([0,T];\dbL^2)}^2\big]\, ,
\eal
\end{equation*}
where $\cC$ is independent of $\th$.
Hence, $X^*_h(t)\in \dbD^{1,2}(\dbL^2)$. In a similar vein, by \cite[Theorem 2.2.2]{Nualart06}, 
it follows that
\bel{d2-spde}
\lt\{\!\!\!
\begin{array}{ll}
\ds {\rm d}D_\mu D_\th X_h^*(t)=\big[\D-\cP_h(t)\big] D_\mu D_\th X^*_h(t)\rd t
+\b D_\mu D_\th X^*_h(t) \rd W(t) \qq  t \in [\mu\vee\th,T]\,,\\
\ns \ds D_\mu D_\th X_h^*(\mu\vee\th)= \b D_\mu X^*_h(\th)\chi_{\{\mu<\th\}}+\b D_\th X^*_h(\mu)\chi_{\{\th\leq \mu\}}   \,.
\end{array}
\rt.
\ee
We may then deduce that
 $X^*_h(t)\in \dbD^{2,2}(\dbL^2)$, and
the remaining part of \rf{w1029e3a}.

\ss
{\bf (4) Verification of \rf{w1029e3b}. } Based on  the fact that $\cP_h(\cd)\in C\big([0,T];\dbS_+(\dbV_h)\big)$ and Lemma \ref{w1002l4}, it follows that
\begin{equation*}
\setlength\abovedisplayskip{3pt}
\setlength\belowdisplayskip{3pt}
\bal
&\me\Big[\sup_{t\in[\th,T]}\|D_\th X^*_h(t)\|_{\dot \dbH^1_h}^2\Big]+\me\Big[\int_\th^T\|D_\th X^*_h(t)\|_{\dot \dbH^{2}_h}^2\rd t\Big]\\
&\qq\leq \cC\, \me\big[ \|D_\th X^*_h(\th)\|_{\dot \dbH^1_h}^2\big]
\leq \cC\big[\|x\|_{\dbH_0^1}^2+\|\si(\cd) \|_{C([0,T];\dbH_0^1)}^2\big]\,.
\eal
\end{equation*}
Note that the constant $\cC$ is independent of $\th$, which settles the assertion \rf{w1029e3b}.

\ss
{\bf (5) Verification of \rf{w1029e3c}. } Using the linearity of SDE \eqref{w301e8}, the fact that $\cP_h(\cd)\in C\big([0,T];\dbS_+(\dbV_h)\big)$, Lemma \ref{w1002l4}, and
the triangle inequality,  we find that
\begin{equation*}
\setlength\abovedisplayskip{3pt}
\setlength\belowdisplayskip{3pt}
\bal
&\me\Big[\sup_{t\in[\th_1,T]}\|(D_{\th_1}-D_{\th_2}) X^*_h(t)\|^2\Big]
\leq \cC\, \me\big[\|(D_{\th_1}-D_{\th_2}) X^*_h(\th_1)\|^2\big]\\
&\q \leq \cC\Big[\me\big[\|D_{\th_1}X^*_h(\th_1)-D_{\th_2}X^*_h(\th_2)\|^2\big]
+ \me\big[\|D_{\th_2}X^*_h(\th_1)-D_{\th_2}X^*_h(\th_2)\|^2\big] \Big]\\
&\q=:\cC(I_1+I_2)\,.
\eal
\end{equation*}
For $I_1$, the result \rf{w1029e2d} implies that
\begin{equation*}
\setlength\abovedisplayskip{3pt}
\setlength\belowdisplayskip{3pt}
\bal
I_1&\leq \cC\Big[ \me \big[\|X^*_h(\th_1)-X^*_h(\th_2)\|^2\big]+\me\big[\|\Pi_h\si(\th_1)-\Pi_h\si(\th_2)\|^2\big] \Big]\\
&\leq \cC |\th_1-\th_2| \big[ \|x\|_{\dbH_0^1}^2+\|\si(\cd)\|^2_{L^2(0,T;\dbH_0^1)\cap C([0,T];\dbL^2)}+L_{\si,1/2}^2\big]\,.
\eal
\end{equation*}
For $I_2$, \rf{w1029e3a} and \rf{w1029e3b} lead to
\begin{equation*}
\setlength\abovedisplayskip{3pt}
\setlength\belowdisplayskip{3pt}
\bal
I_2&\leq (\th_1-\th_2)\me\Big[\int_0^T\|\D_h D_{\th_2}X^*_h(t)\|^2\rd t\Big]
  +\me\Big[\int_{\th_2}^{\th_1}\|D_{\th_2}X^*_h(t)\|^2\rd t\Big]\\
   &\leq \cC\big[\|x\|_{\dbH_0^1}^2+\|\si(\cd) \|_{C([0,T];\dbH_0^1)}^2\big]\,.
\eal
\end{equation*}
A combination of the above three estimates now shows the assertion \rf{w1029e3c}.
\end{proof}

%

The following result is on the regularity of $\big(Y_h(\cd),Z_h(\cd)\big)$ of BSDE \rf{bshe1a}.

\bl{w229l3}
Suppose that $\big(Y_h(\cd),Z_h(\cd)\big)$ solves BSDE \eqref{bshe1a}, and $I_\t=\{t_n\}_{n=0}^N$ is a uniform time mesh of $[0,T]$. Then
there exists a constant $\cC$ independent of $h$ such that 
 \begin{subequations}\label{w1029e4}
    \begin{empheq}[left={\empheqlbrace\,}]{align}
 &\me \Big[\sup_{t\in [0,T]}\|Y_h(t)\|_{\dot \dbH^\g_h}^2\Big]
+\me \Big[\int_0^T \|Y_h(t)\|_{\dot \dbH^{\g+1}_h}^2+\| Z_h(t)\|_{\dot \dbH^\g_h}^2 \rd t \Big] \nonumber \\
&\qq\qq\leq \cC \big[\|x\|_{\dot \dbH^\g}^2+\|\si(\cd)\|_{L^2(0,T;\dot \dbH^{\g})}^2\big] \qq\g=0,1\,,  \label{w1029e4a}\\
&\big\|Y_h(\cd)-Y_h\big(\nu(\cd)\big)\big\|_{L^2_\dbF(0,T;\dbL^2)}^2 
\leq \cC \t \big[\|x\|_{\dbH_0^1}^2+\|\si\|_{L^2(0,T;\dbH_0^1)}^2\big]\,, \label{w1029e4b}
\end{empheq}
\end{subequations}
and
 \begin{subequations}\label{w1029e5}
    \begin{empheq}[left={\empheqlbrace\,}]{align}
 &\sup_{\th\in[0,T]} \bigg\{\me\Big[ \sup_{t\in [\th, T]} \|D_\th Y_h(t)\|^2\Big]+\me\Big[\int_\th^T \|D_\th Z_h(t)\|^2\rd t\Big]\bigg\} \nonumber \\
&\qq     +\sup_{\th,\mu\in[0,T]}\me\Big[ \sup_{t\in [\mu\vee\th, T]} \|D_\mu D_\th Y_h(t)\|^2\Big]      
      \leq \cC \big[\|x\|^2+\|\si(\cd) \|_{C([0,T];\dbL^2)}^2\big]\,,  \label{w1029e5a}\\
 &\sup_{\th\in[0,T]}\bigg\{\me\Big[ \sup_{t\in [\th,T]} \| D_\th Y_h(t)\|_{\dot\dbH^1_h}^2\Big]
+\me\Big[\int_\th^T  \| D_\th Y_h(t)\|_{\dot\dbH^2_h}^2+\| D_\th Z_h(t)\|_{\dot\dbH^1_h}^2  \rd t\Big]\bigg\} \nonumber \\
&\qq \leq  \cC \big[\|x\|_{\dbH_0^1}^2+\|\si(\cd)\|_{C([0,T];\dbH_0^1)}^2\big]\, , \label{w1029e5b}\\
&\me\Big[ \sup_{t\in [\th_1,T]} \|(D_{\th_1}-D_{\th_2})Y_h(t)\|^2\Big] \nonumber \\
&\qq\leq \cC(\th_1-\th_2) \big[\|x\|_{\dbH_0^1}^2+\|\si(\cd)\|_{C([0,T];\dbH_0^1)}^2+L_{\si,1/2}^2\big]  \qq  \th_2\leq \th_1\,,  \label{w1029e5c}\\
&\me \big[\|Z_h(t)-Z_h(s)\|^2\big]  
 \leq \cC|t-s|  \big[\|x\|_{\dbH_0^1}^2+\|\si(\cd)\|_{C([0,T];\dbH_0^1)}^2+L_{\si,1/2}^2\big] \qq t,s \in [0,T]\,, \label{w1029e5d}
\end{empheq}
\end{subequations}
where 
$\nu(\cd)$ is defined in \rf{w827e1}.
\el

\begin{proof}
{\bf (1) Verification of \rf{w1029e4a}.}  This assertion follows from \rf{w1008e1} in Lemma \ref{w1008l1} and \rf{w1029e2a} in Lemma \ref{w229l2}.
%

\ss
{\bf (2) Verification of \rf{w1029e4b}. } This assertion follows from \rf{w1008e2} in Lemma \ref{w1008l1} and \rf{w1029e2a} 
in Lemma \ref{w229l2}.


\ss
{\bf (3) Verification of \rf{w1029e5a}, \rf{w1029e5b}. } On noting that $X^*_h(t)\in \dbD^{2,2}(\dbL^2)$ for any $t\in [0,T]$,  
 by \cite[Proposition 5.3]{ElKaroui-Peng-Quenez97} we can get
\bel{w301e5}
\setlength\abovedisplayskip{3pt}
\setlength\belowdisplayskip{3pt}
\lt\{\!\!\!
\begin{array}{ll}
\ds {\rm d} D_\th Y_h(t)=\big[-\D_h D_\th Y_h(t)- D_\th Z_h(t)+ D_\th X^*_h(t)\big]\rd t   
   +D_\th Z_h(t)\rd W(t)\q t\in [\th, T)\, ,\\
\ns\ds D_\th Y_h(T)=-\a D_\th X^*_h(T)\,,\\
\ns\ds D_\th Y_h(t) =0,\,\,D_\th Z_h(t)=0 \qq t\in [0,\th)\,,       
\end{array}
\rt.
\ee
\bel{w1029e6}
Z_h(t)=D_t Y_h(t)\qq \ae\, t\in[0,T]\,,
\ee
and
\bel{w301e6}
\setlength\abovedisplayskip{3pt}
\setlength\belowdisplayskip{3pt}
\lt\{\!\!\!
\begin{array}{ll}
\ds {\rm d} D_\mu D_\th Y_h(t)=\big[-\D_h D_\mu D_\th Y_h(t)- D_\mu D_\th Z_h(t)
 + D_\mu D_\th X^*_h(t)\big]\rd t\\
\ns\ds \qq\qq\qq\qq\qq        +D_\mu D_\th Z_h(t)\rd W(t)\qq t\in [\mu\vee\th, T),\\
\ns\ds D_\mu D_\th Y_h(T)=-\a  D_\mu D_\th X^*_h(T)\,.      
\end{array}
\rt.
\ee
Then, \rf{w1008e1} in Lemma \ref{w1008l1} and \eqref{w1029e3a}, \eqref{w1029e3b} in Lemma \ref{w229l2} lead to \rf{w1029e5a} and \rf{w1029e5b}.

\ss
{\bf (4) Verification of \rf{w1029e5c}.} Lemma \ref{w1008l1}, the linearity of BSDE \rf{bshe1a}, and \rf{w1029e3c} in 
Lemma \ref{w229l2} imply
\begin{equation*}
\setlength\abovedisplayskip{3pt}
\setlength\belowdisplayskip{3pt}
\bal
\me\Big[\! \sup_{t\in [\th_1,T]} \|(D_{\th_1}-D_{\th_2})Y_h(t)\|^2\Big]
&\leq \cC \bigg\{\me\big[\|(D_{\th_1}-D_{\th_2})X^*_h(T)\|^2\big]\!+\! \me\Big[\!\int_{\th_1}^T\!\|(D_{\th_1}-D_{\th_2})X^*_h(\th)\|^2\rd \th\Big]\bigg\}\\
&\leq \cC(\th_1-\th_2) \big[\|x\|_{\dbH_0^1}^2+\|\si(\cd)\|_{C([0,T];\dbH_0^1)}^2+L_{\si,1/2}^2\big]\,,
\eal
\end{equation*}
which settles the assertion \rf{w1029e5c}.

\ss
{\bf (5) Verification of \rf{w1029e5d}. } Suppose that $s\leq t$. By applying the fact that $Z_h(\cd)=D_\cd Y_h(\cd), \ae$ and the assertion \rf{w1029e5c}, we arrive at 
\bel{w301e10}
\setlength\abovedisplayskip{3pt}
\setlength\belowdisplayskip{3pt}
\bal
&\me \big[\|Z_h(t)-Z_h(s)\|^2\big] \\
&\q\leq 2\me\big[\|(D_t-D_s)Y_h(t)\|^2\big]+2\me\big[\big\|D_s\big(Y_h(t)-Y_h(s)\big)\big\|^2\big]\\
& \q\leq \cC|t-s|  \big[\|x\|_{\dbH_0^1}^2+\|\si(\cd)\|_{C([0,T];\dbH_0^1)}^2+L_{\si,1/2}^2\big]
+ 2\me\big[\big\|D_s\big(Y_h(t)-Y_h(s)\big)\big\|^2\big]\,.
\eal
\ee
By \rf{w301e5} and \rf{w1029e6}, we have
\begin{equation*}
\setlength\abovedisplayskip{3pt}
\setlength\belowdisplayskip{3pt}
\bal
&\me\big[\big\|D_s\big(Y_h(t)-Y_h(s)\big)\big\|^2\big] \\
&\leq 
\cC\bigg\{ \me\Big[\int_s^t\|D_sD_\th Y_h(\th)\|^2\rd \th\Big]
\!+\!(t-s)\me\Big[ \int_s^T \|\D_h D_sY_h(\th)\|^2+\| D_sZ_h(\th)\|^2
+\|D_sX^*_h(\th)\|^2 \rd \th \Big]\bigg\}\,,
\eal
\end{equation*}
which, together with \rf{w1029e5a}, \rf{w1029e5b} and \rf{w1029e3a}, leads to
\bel{w1023e2}
\setlength\abovedisplayskip{3pt}
\setlength\belowdisplayskip{3pt}
\bal
\me\big[\big\|D_s\big(Y_h(t)-Y_h(s)\big)\big\|^2\big]
 \leq  \cC|t-s| \big[\|x\|_{\dbH_0^1}^2+\|\si(\cd)\|_{C([0,T];\dbH_0^1)}^2\big]\,.
\eal
\ee
Now \rf{w1029e5d} can be derived by \rf{w301e10}--\rf{w1023e2}.
\end{proof}

Relying on Lemmata \ref{w1002l4} and \ref{w229l3}, we can derive the following strong bound for the optimal control $U^*_h(\cd)$ of Problem {\bf (SLQ)$_h$}.

\bl{w1019l3} 
Let $U^*_h(\cd)$ be the optimal control of Problem {\bf (SLQ)$_h$}. There exists a constant $\cC$
such that
\begin{equation*}
\me\Big[\sup_{t\in[0,T]}\|U_h^*(t)\|_{\dot \dbH^2_h}^2\Big]
 \leq \cC\big[ \|   x\|_{\dbH_0^1\cap \dbH^2}^2
+ \|\si(\cd)\|^2_{L^2(0,T;\dbH_0^1\cap\dbH^2)}\big]\,.
\end{equation*}
\el

\begin{proof}

By Lemma \ref{w1002l4} with $\g=2$, the optimality condition \rf{pontr1a} and \rf{w1029e4a} with $\g=0$ in Lemma \ref{w229l3}, we arrive at
\begin{equation*}
\setlength\abovedisplayskip{3pt}
\setlength\belowdisplayskip{3pt}
\bal
\me\Big[\sup_{t\in[0,T]}\|X_h^*(t)\|_{\dot \dbH^{2}_h}^2\Big]
&\leq \cC\, \me \Big[ \lt\|   X^*_{h}(0)\rt\|_{\dot \dbH^2_h}^2
+ \int_0^T \| U^*_h(t)\|_{\dot \dbH^1_h}^2+\| \Pi_h \si (t)\|_{\dot \dbH^2_h}^2\rd t\Big]\\
&\leq \cC\big[ \|   x\|_{\dbH_0^1\cap \dbH^2}^2
+ \|\si(\cd)\|^2_{L^2(0,T;\dbH_0^1\cap\dbH^2)}\big]\,.
\eal
\end{equation*}
Subsequently, by using the optimality condition \rf{pontr1a} and Lemma \ref{w1008l1} with $\g=2$, we derive 
\begin{equation*}
\setlength\abovedisplayskip{3pt}
\setlength\belowdisplayskip{3pt}
\bal
\me\Big[\sup_{t\in[0,T]}\|U_h^*(t)\|_{\dot \dbH^2_h}^2\Big]
&\leq \cC\, \me \Big[ \|  \a X^*_{h}(T)\|_{\dot \dbH^2_h}^2+ \int_0^T \| X^*_h(t)\|^2_{\dot\dbH^1_h}\rd t\Big]\\
&\leq \cC\big[ \|   x\|_{\dbH_0^1\cap \dbH^2}^2
+ \|\si(\cd)\|^2_{L^2(0,T;\dbH_0^1\cap\dbH^2)}\big]\,.
\eal
\end{equation*}
That completes the proof.
\end{proof}

\subsubsection{Temporal discretization of Problem {\bf (SLQ)$_h$}}\label{fbspde-num1z}

In this part, we propose a temporal discretization scheme  of Problem {\bf (SLQ)$_h$}, consider its well-posedness  and derive 
error estimates. 
For this purpose, we use
a mesh $I_\tau$ which covers $[0,T]$, and consider step processes $\big(X_{h\tau}(\cd), U_{h\tau}(\cd)\big) \in \dbX_{\dbF} \times \dbU_{\dbF}\subset L^2_\dbF\big(0,T;\dbV_h\big)\times L^2_\dbF\big(0,T;\dbV_h\big)$, where
\begin{equation*}
\setlength\abovedisplayskip{3pt}
\setlength\belowdisplayskip{3pt}
\bal \dbX_{\dbF} &\deq\! \big\{X(\cd) \in L^2_\dbF (0,T;\dbV_h)\,\big| \, X(t)\! =\! X(t_n), \,\,   \forall\,t\in [t_n, t_{n+1}),  \,\, n=0,1,\cds, \, N-1\big\}\, ,\\
\dbU_{\dbF} &\deq\!   \big\{U(\cd) \in L^2_\dbF (0,T;\dbV_h)\,\big|\, U(t)\! =\! U(t_n),  \,\,    \forall\,t\in [t_n, t_{n+1}),  \,\, n=0,1,\cds, \, N-1\big\}\, ;
\eal
\end{equation*}
for any $X(\cd)\in \dbX_{\dbF}$  and $U(\cd)\in \dbU_{\dbF}$, we define
\begin{equation*}
\setlength\abovedisplayskip{3pt}
\setlength\belowdisplayskip{3pt}
\|X(\cd)\|_{\dbX_{\dbF}}\deq \Big(\t \sum_{n=0}^{N-1}\me\big[\|X(t_n)\|^2\big]\Big)^{1/2}\,,\,
\mbox{ and } \,\, \|U(\cd)\|_{\dbU_{\dbF}} \deq \Big(\t \sum_{n=0}^{N-1}\me \big[\|U(t_n)\|^2\big]\Big)^{1/2}.
\end{equation*}

Now a temporal discretization scheme of Problem {\bf (SLQ)$_h$} reads as follows.\\
\no{\bf Problem  (SLQ)$_{h\t}$.} Find an optimal control
$U_{h\t}^*(\cd)\in \dbU_{\dbF}$ which minimizes the quadratic  cost functional 
\bel{w1003e8}
\setlength\abovedisplayskip{3pt}
\setlength\belowdisplayskip{3pt}
\bal
\cJ_{h,\t}\big(U_{h\t}(\cd)\big)
=\frac {1} 2 \big[\| X_{h\t}(\cd)\|^2_{\dbX_{\dbF}}
+ \| U_{h\t} (\cd)\|_{\dbU_{\dbF}}^2\big]
+\frac \a 2 \me \big[\|X_{h\t}(T)\|^2\big]\, ,
\eal
\ee
subject to the forward semi-implicit  difference equation
\begin{equation}\label{w1003e7}
\setlength\abovedisplayskip{3pt}
\setlength\belowdisplayskip{3pt}
\left\{\!\!\!
\begin{array}{ll}
\ds X_{h\t}(t_{n+1})-X_{h\t}(t_n)= \tau \big[\D_h X_{h\t}(t_{n+1})
 +U_{h\t}(t_n) \big]\\
\ns \ds \qq\qq+ \big[\b X_{h\t}(t_n)+\Pi_h \si(t_n)\big] \D_{n+1}W
 \qq n=0,1,\cds,N-1\, ,\\
\ns \ds   X_{h\t}(0)= \Pi_h x\,.
 \end{array}
 \right.  
\end{equation}
%
The following result states a {(discrete)} Pontryagin's maximum principle for the uniquely  solvable  Problem {\bf (SLQ)$_{h\tau}$}. 

\bt{MP}
Let $\ds A_0 =\lt(\mathds{1}_h-\tau \D_h\rt)^{-1}$.
The unique optimal pair $\big(X^*_{h\t}(\cd), U^*_{h\t}(\cd)\big)\in \dbX_{\dbF} \times \dbU_{\dbF}$ of
Problem {\bf (SLQ)$_{h\t}$}  solves the following coupled  equalities for $n=0,1,\cds,N-1$:
\begin{subequations}\label{w1212e3}
    \begin{empheq}[left={\empheqlbrace\,}]{align}
&X_{h\t}^*(t_{n+1})=A_0^{n+1}\prod_{j=1}^{n+1}\lt(1+\b \D_j W\rt)X_{h\t}({0}) \nonumber \\
&\qq+\sum_{j=0}^{n}A_0^{n+1-j}\prod_{k=j+2}^{n+1} (1+\b \D_kW ) \big[ \t U_{h\t}^*(t_j)+\Pi_h\si(t_j)\D_{j+1}W\big]\,,
 \label{w1212e3a}\\
&Y_{h\t}(t_n)=-\t \me^{t_n}\Big[ \sum_{j=n+1}^{N-1} A_0^{j-n} \prod_{k={n+2}}^j(1+\b \D_kW) X^*_{h\t}(t_j) \Big]  \nonumber \\
&\qq-\a \me^{t_n} \Big[ A_0^{N-n} \prod_{k=n+2}^N(1+\b \D_kW)  X^*_{h\t}(T) \Big]\,, \label{w1212e3b}\\
&X^*_{h\t}(0)=\Pi_h x\, , \label{w1212e3c}
\end{empheq}
\end{subequations}
 together with the discrete optimality condition
\bel{w1003e12}     
U^*_{h\t}(t_n)- Y_{h\t}(t_n)=0 \, .
\ee
\et
%
Below, we write $Y_{h\t}(\cd;X^*_{h\t})$ 
for $Y_{h\t}(\cd)$ to indicate the dependence on $X^*_{h\t}(\cd)$.

We mention the appearing
mapping $Y_{h\tau}(\cd)$ in \rf{w1212e3b}, which is used to give the discrete optimality condition \rf{w1003e12}; see also Remark \ref{w829r1}. 
We will use this characterization in Theorem \ref{MP} of $\big(X^*_{h\t}(\cd), U^*_{h\t}(\cd)\big)\in \dbX_{\dbF} \times \dbU_{\dbF}$ to deduce rates of convergence for the optimal pair of Problem {\bf (SLQ)$_{h\t}$} in the next Theorem 
\ref{rate2}; but to actually implement this numerical method is still too ambitious: this is the reason why
another method that is suitable for implementation will be proposed in the subsequent Sections \ref{numopt}
and \ref{impl-1}.

\bt{rate2}
Suppose that  {\rm \bf (A)} holds. Let $\big(X_h^*(\cd),U^*_h(\cd)\big)$ be the optimal pair to Problem {\bf (SLQ)$_h$}, and $\big(X^*_{h\t}(\cd), U^*_{h\t}(\cd)\big)$ solves Problem {\bf (SLQ)$_{h\t}$}. Then, there exists a constant
$\cC$ independent of $h, \tau$ such that
\begin{subequations}\label{w1030e1}
    \begin{empheq}[left={\empheqlbrace\,}]{align}
& \sum_{n=0}^{N-1}\me \Big[\int_{t_n}^{t_{n+1}}\|U^*_h(t)-U^*_{h\t}(t_n)\|^2
+\|X^*_h(t)-X^*_{h\t}(t_n)\|^2 \rd t \Big] \nonumber \\
&\qq     \leq \cC\t \big[ \|   x\|_{\dbH_0^1\cap \dbH^2}^2
+ \|\si(\cd)\|^2_{C([0,T];\dbH_0^1)\cap L^2(0,T;\dbH_0^1\cap\dbH^2)}+L_{\si,1/2}^2\big]\, ,   \label{w1030e1a}\\
& \max_{0 \leq n \leq N}  {\mathbb E} \big[\|X_h^*(t_n)-X_{h\t}^*(t_n)\|^2 \big]
+\tau \sum_{n=1}^N {\mathbb E}  \big[\|X_h^*(t_n)-X^*_{h\t}(t_n)\|_{\dbH_0^1}^2\big] \nonumber \\
&\qq \leq  \cC\t \big[ \|   x\|_{\dbH_0^1\cap \dbH^2}^2
+ \|\si(\cd)\|^2_{C([0,T];\dbH_0^1)\cap L^2(0,T;\dbH_0^1\cap\dbH^2)}+L_{\si,1/2}^2\big] \, . \label{w1030e1b}
\end{empheq}
\end{subequations}
\et

\br{w829r1} 
{\bf (i)}~Equation
\eqref{w1212e3a} is the time-implicit approximation of SDE \eqref{spde-h} with $U_{h\t}(\cd)=U^*_{h\t}(\cd)$, while 
$Y_{h\t} (\cd)$ is 
an approximation of $Y_h(\cd)$ to BSDE \eqref{bshe1a}; in fact, \eqref{w1212e3b} is different from the 
temporal discretization of \eqref{bshe1a} via the implicit Euler method 
--- which is why
our approach differs from the one known as
`first optimize, then discretize'.
See
also Lemma \ref{w229l1} for an (auxiliary, non-implementable) approximation of $Y_h(\cd)$ that is based on the implicit Euler method. 
In Lemma \ref{w301l1}, we estimate the difference between these two approximations.

{\bf (ii)}~If we endow the discrete state space with the norm
$\|X(\cd)\|_{\dbX_{\dbF}}\deq \big(\t \sum_{n=1}^{N}\me\big[\|X(t_n)\|^2\big]\big)^{1/2}$,
then \rf{w1212e3b} turns to
\bel{w1201e1}
\setlength\abovedisplayskip{3pt}
\setlength\belowdisplayskip{3pt}
\bal
Y_{h\t}(t_n)&=-\t \me^{t_n}\Big[ \sum_{j=n+1}^{N} A_0^{j-n} \prod_{k={n+2}}^j(1+\b \D_kW) X^*_{h\t}(t_j) \Big] \\
&\q-\a \me^{t_n} \Big[ A_0^{N-n} \prod_{k=n+2}^N(1+\b \D_kW)  X^*_{h\t}(T) \Big]\,;
\eal
\ee
see \cite[Theorem 3.2]{Prohl-Wang22}.
Hence, when SPDE \rf{spde} is driven by {\em additive} noise only ({\em i.e.}, $\b=0$), $Y_{h\t}(\cd)$ in 
\rf{w1201e1} satisfies
\begin{equation*}
\setlength\abovedisplayskip{3pt}
\setlength\belowdisplayskip{3pt}
\lt\{\!\!\!
\begin{array}{ll}
\ds [\mathds{1}_h - \tau \Delta_h]Y_{h\t}(t_n) = \me^{t_n}\big[ Y_{h\t}(t_{n+1})- {\tau} X_{h\t}^*(t_{n+1}) \big] \qq n=0,1,\cds,N-1\, , \\
\ns \ds Y_{h\t}(T)=-\a X_{h\t}^*(T)\, ,
\end{array}
\rt.
\end{equation*}
which is the same as the semi-implicit Euler method;
see Lemma \ref{w229l1} and \cite[Theorem 4.2]{Prohl-Wang21}.
In this case, to `first discretize, then optimize', and to `first optimize, then discretize' lead
to the same discrete optimality system.


\er

\begin{proof}[\bf {Proof of Theorem \ref{MP}}]
We divide the proof into three steps.

\ss

{\bf (1)}  Recall the notation $ A_0=(\mathds{1}_h-\tau \D_h)^{-1}$.
We define the bounded operators $\G : \dbV_h\rightarrow \dbX_{\dbF}$ and $L: \dbU_{\dbF}\rightarrow \dbX_{\dbF}$
as follows, for any $ n=0,1,\cds,N$,
\bel{rep1}
\setlength\abovedisplayskip{3pt}
\setlength\belowdisplayskip{3pt}
\bal
& \big(\G X_{h\t}({0})\big)(t_n)=A_0^n\prod_{j=1}^n\lt(1+\b \D_j W\rt)X_{h\t}({0})\,, \\
\mbox{and }\q &\big(LU_{h\t}(\cd) \big)(t_n)= \t \sum_{j=0}^{n-1}A_0^{n-j}\prod_{k=j+2}^n\lt(1+\b \D_kW\rt) U_{h\t}(t_j) \,,
\eal
\ee
respectively, where $\big(X_{h\tau}(\cd), U_{h\tau}(\cd)\big)$ is an admissible pair of Problem {\bf (SLQ)}$_{h\tau}$; see \rf{w1003e8}--\rf{w1003e7}. We also need $f(\cd)$, which we define as
\begin{equation*}
\setlength\abovedisplayskip{3pt}
\setlength\belowdisplayskip{3pt}
f(t_n)=\sum_{j=0}^{n-1} A_0^{n-j}\prod_{k=j+2}^{n}\lt(1+\b \D_kW \rt)\Pi_h\si(t_j)\D_{j+1}W \qquad \forall\, 
n=0, 1 ,\cds,N\,,
\end{equation*}
and use below the abbreviations
\begin{equation}\label{w1003e14}
\setlength\abovedisplayskip{5pt}
\setlength\belowdisplayskip{5pt}
\h \G {\Pi_h x} \deq \lt(\G\Pi_h x\rt)(T)\, , 
\qquad \h L{U_{h\tau}}(\cd) \deq \big(LU_{h\tau}(\cd)\big)(T)\,,
\qq \h f\deq f(T) \,.
\end{equation}

By \rf{w1003e7}, we find that
\bel{w1024e2}
\setlength\abovedisplayskip{5pt}
\setlength\belowdisplayskip{5pt}
X_{h\t}(t_n)=\big(\G X_{h\t}(0)\big)(t_n)+\big(LU_{h\t}(\cd)\big)(t_n)+f(t_n) \qq n=0,1,\cds,N-1\,.
\ee

It is easy to check that, for any $\xi(\cd)\in \dbX_{\dbF}$, and any $\eta\in L^2_{\mf_T}(\O;\dbV_h)$,
\begin{equation*}
\setlength\abovedisplayskip{3pt}
\setlength\belowdisplayskip{3pt}
\bal
\big(&L^*\xi(\cd)\big)(t_n)= \me^{t_n}\Big[ \t \sum_{j=n+1}^{N-1} A_0^{j-n} \prod_{k={n+2}}^j(1+\b \D_kW) \xi(t_j)  \Big]\,,\\
\mbox{and}\q \big(& \h L^*\eta\big)(t_n) = \me^{t_n} \Big[  A_0^{N-n} \prod_{k=n+2}^N (1+\b \D_kW) \eta \Big]\q n=0,1,\cds, N-1\, . 
\eal
\end{equation*}
{\bf (2)} By \rf{rep1}--\rf{w1024e2}, 
we can rewrite $\cJ_{h,\t}\big(U_{h\t}(\cd) \big)$ in \rf{w1003e8} as follows:
\begin{equation*}
\setlength\abovedisplayskip{3pt}
\setlength\belowdisplayskip{3pt}
\bal
\cJ_{h,\t}\big(U_{h\t}(\cd) \big)
&= \frac 1 2 \Big\{  \Big(\big(\big[\mathds{1}_h +L^* L+\a \h L^* \h L \big]U_{h\t}(\cd)\big)(\cd) ,U_{h\t}(\cd)  \Big)_{L^2_{{\mathbb F}}(0,T; \dbL^2)}  \\
 &\qq   +2 \Big(\big(\big[L^* \G+\a \h L^*\h\G \big] {\Pi_h x}\big)(\cd) +\big(L^*f(\cd)\big)(\cd)+\a \big(\h L^*\h f\,\big)(\cd) ,U_{h\t}(\cd) \Big)_{L^2_{{\mathbb F}}(0,T; \dbL^2)}\\
  &\qq  +\Big[ \big( (\G {\Pi_h x})(\cd)+f(\cd), (\G {\Pi_h x})(\cd)+f(\cd)\big)_{L^2_{{\mathbb F}}(0,T; \dbL^2)}\\
&\qq\q   + \a \big( \h\G {\Pi_h x}+\h f, \h\G \Pi_h x+\h f\, \big)_{L^2_{\mf_T}(\O;\dbL^2)} \Big]     \Big\}\\
&=: \frac 1 2\Big[ \Big(\big(\fN U_{h\t}(\cd)\big)(\cd) ,U_{h\t}(\cd) \Big)_{L^2_{{\mathbb F}}(0,T; \dbL^2)} +2 \big( \fH({\Pi_h x,f})(\cd),U_{h\t}(\cd)\big)_{L^2_{{\mathbb F}}(0,T; \dbL^2)}\\
&\qq    + \fM(\Pi_hx,f)  \Big]\, .
\eal
\end{equation*}
We use this re-writing of $\cJ_{h,\t}\big(U_{h\t}(\cd) \big)$  which involves
mappings $\fN$ and $\fH$, and note that the last term does not depend on $U_{h\t}(\cd)$,
to now involve the optimality condition for $U^*_{h\t}(\cd)$: since
$\fN  =\mathds{1}_h +L^* L+\a \h L^* \h L $  is positive definite, there exists a unique $U^*_{h\t}(\cd) \in \dbU_{\dbF}$ such that
\bel{w1117e2}
\setlength\abovedisplayskip{3pt}
\setlength\belowdisplayskip{3pt}
\big(\fN U^*_{h\t}(\cd)\big)(\cd) +\fH({\Pi_h x},f )(\cd) =0\, .
\ee
Therefore, for any $U_{h\t}(\cd) \in \dbU_{\dbF}$ such that $U_{h\t}(\cd) \neq U^*_{h\t}(\cd)$, 
\begin{equation*}
\setlength\abovedisplayskip{3pt}
\setlength\belowdisplayskip{3pt}
\bal
\cJ_{h,\t}\big(U_{h\t}(\cd)\big)-\cJ_{h,\t}\big(U^*_{h\t}(\cd)\big)
=\frac 1 2 \Big( \fN \big(U_{h\t}(\cd)-U^*_{h\t}(\cd)\big)(\cd), U_{h\t}(\cd)-U^*_{h\t}(\cd)\Big)_{L^2_{{\mathbb F}}(0,T; \dbL^2)}
>0\, ,
\eal
\end{equation*}
which means that $U^*_{h\t}(\cd)$ is the unique optimal control, and $\big(X^*_{h\t}(\cd),U^*_{h\t}(\cd)\big)$ is the unique optimal pair.

\ss
{\bf (3)}
By the definition of $\fN, \fH$,  and \rf{w1024e2},
 we deduce for \rf{w1117e2} that
\begin{equation*}
\setlength\abovedisplayskip{3pt}
\setlength\belowdisplayskip{3pt}
0=\big(\fN U^*_{h\t}(\cd)\big)(\cd) +\fH({\Pi_h x,f })(\cd)
=U^*_{h\t} (\cd)+\big(L^* X^*_{h\t}(\cd) +\a \h L^*  X^*_{h\t}(T) \big)(\cd)\, . \\
\end{equation*}
For $Y_{h\t}(\cd)$ in \rf{w1212e3b}, we compute
\bel{w1008e4}
\setlength\abovedisplayskip{3pt}
\setlength\belowdisplayskip{3pt}
\big(L^* X^*_{h\t}(\cd) +\a \h L^*  X^*_{h\t}(T) \big)(\cd)=-Y_{h\t}(\cd)\,.
\ee
Then \eqref{w1003e12} can be deduced by these two equalities. That completes the proof.
\end{proof}

Our next goal is to verify Theorem \ref{rate2} which needs preparatory results that are stated in Lemmata 
\ref{w229l1} through \ref{w1019l2}.

In Remark \ref{w829r1} we have already mentioned that $Y_{h\t}(\cd)$
in 
\eqref{w1212e3b} neither solves the temporal discretization of BSDE \eqref{bshe1a} by the explicit Euler, nor by the implicit Euler method in general. 
The following two lemmata study the difference between $Y_{h\t}(\cd)$ and the temporal discretization of BSDE \eqref{bshe1a} by
the implicit Euler method. 
\bl{w229l1}
Suppose that $\big(Y_\cd,\, Z_0(\cd)\big)$ solves the following backward equation:
\begin{equation}\label{w229e7}
\setlength\abovedisplayskip{3pt}
\setlength\belowdisplayskip{3pt}
\lt\{\!\!\!
\begin{array}{ll}
\ds Y_n\!-\!Y_{n+1}=\t \D_h Y_n \!+\! \int_{t_n}^{t_{n+1}} \b \bar Z_0(t)\!-\! X^*_{h\t}\big(\mu(t)\big) \rd t
\!-\!\int_{t_n}^{t_{n+1}} Z_0(t)\rd W(t)\,,\\
\ns\ds Y_N=-\a X^*_{h\t}(T)\,,
\end{array}
\rt.
\end{equation}
%
where $X^*_{h\t}$ is the optimal state of Problem {\bf (SLQ)$_{h\t}$}, $\mu(\cd)$ is defined in \rf{w827e1}.
and  $\bar Z_0(\cd)$ is a piecewise constant process
 which is defined by
\begin{equation*}
\setlength\abovedisplayskip{3pt}
\setlength\belowdisplayskip{3pt}
\bar Z_0(t)=\frac 1 \t \me^{t_n}\Big[\int_{t_n}^{t_{n+1}}  Z_0(s)\rd s\Big]\qquad \forall\, t\in [t_n,t_{n+1})\,, \qquad  n=0,1,\cds,N-1\, .
\end{equation*}
 Then, 
 \begin{equation*}
\setlength\abovedisplayskip{3pt}
\setlength\belowdisplayskip{3pt}
\bal
Y_n
=& - \t \me^{t_n} \Big[ \sum_{j=n+1}^NA_0^{j-n}\prod_{k=n+1}^{j}\lt(1+\b \D_kW\rt)  X^*_{h\t}(t_j) \Big] \\
 &  -\a \me^{t_n} \Big[ A_0^{N-n}\prod_{k={n+1}}^N \lt(1+\b \D_kW\rt)  X^*_{h\t}(T)  \Big]  \qquad n=0,1,\cds,N-1\, .
 \eal
\end{equation*}
%
\el
\begin{proof}
The well-posedness of \eqref{w229e7} can be derived in the same vein as that in \cite{Pardoux-Peng90}.
For any $n=0,1,\cds,N-1$,
it is easy to see that
 \begin{equation*}
\setlength\abovedisplayskip{3pt}
\setlength\belowdisplayskip{3pt}
\int_{t_n}^{t_{n+1}}Z_0(t)\rd t
=\me^{t_n}\Big[ \big( Y_{n+1}-\t X^*_{h\t}(t_{n+1}) \big) \D_{n+1}W\Big]\,,
\end{equation*}
which yields the desired result.
\end{proof}

The solution $\big(Y_\cd,\,Z_0(\cd)\big)$ in \rf{w229e7} depends on $X^*_{h\t}$. Hence, in what follows, we may write it
in the form $\big(Y_\cd(X^*_{h\t}),  \,Z_0(\cd;X^*_{h\t})\big)$.

Similar to $\cS$ and $\cS_h$, we can define the  solution operator for the difference equation \rf{w1003e7},
$\cS_{h\t}: \dbU_{\dbF}\to \dbX_{\dbF}\,.\index{${\mathcal S}_{h\tau}$}$
In the next lemma, we estimate the 
difference between $Y_{h\t}\big(\cd;\cS_{h\t}(\Pi_\t U^*_h)\big)$,
which was introduced in \rf{w1212e3b}, and $Y_\cd\big(\cS_{h\t}(\Pi_\t U^*_h)\big)$ from Lemma \ref{w229l1}, 
which is crucial in proving rates of
convergence for the temporal discretization of Problem {\bf (SLQ)$_h$}.

\bl{w301l1} 
Suppose that $\big(Y_\cd,Z_0(\cd)\big)$ solves \eqref{w229e7}. Then
it holds that
\begin{equation*}
\bal
\max_{0\leq n\leq N-1}\me\big[ \big\| Y_n \big(\cS_{h\t}(\Pi_\t U^*_h)\big )-Y_{h\t} \big(t_n;\cS_{h\t}(\Pi_\t U^*_h) \big)
\big\|^2\big]
\leq \cC\t \big[\|x\|^2+\|\si(\cd) \|_{C([0,T];\dbL^2)}^2\big] \,,
\eal
\end{equation*}
where $\Pi_\t$ is defined in \eqref{w828e1}, $Y_{h\t}(\cd)$ is given by \eqref{w1212e3b}, and the constant $\cC$ is independent of $h$ and $\t$.
\el

\begin{proof}
By Lemma \ref{w229l1} and Theorem \ref{MP}, we find that for any $n=0,1,\cds,N-1$
\begin{equation*}
\setlength\abovedisplayskip{3pt}
\setlength\belowdisplayskip{3pt}
\bal
&Y_n\big( \cS_{h\t}(\Pi_\t U^*_h)\big)-Y_{h\t}\big(t_n;\cS_{h\t}(\Pi_\t U^*_h)\big)\\
&\qq=\bigg\{-\a\b  \me^{t_n}\Big[ A_0^{N-n} \D_{n+1}W \prod_{k={n+2}}^N (1+\b \D_kW)\cS_{h\t}(T;\Pi_\t U^*_h)\Big]\\
&\qq\q-\t \me^{t_n}\Big[ A_0^{N-n}\prod_{i=n+2}^{N}\lt(1+\b \D_iW\rt)  \cS_{h\t}(T;\Pi_\t U^*_h)  \Big]\bigg\}\\
&\qq\q-\t \b  \sum_{k=n+1}^N \me^{t_n}\Big[ A_0^{k-n} \D_{n+1}W \prod_{i=n+2}^{k} \lt(1+\b \D_iW\rt)  \cS_{h\t}(t_k;\Pi_\t U^*_h)  \Big]\\
&\qq=: I_n^0-\t\b \sum_{k=n+1}^N I_n^k\,.
\eal
\end{equation*}
For $I_n^k$, by  \eqref{w1212e3a}, we have
\bel{w1117e3}
\setlength\abovedisplayskip{3pt}
\setlength\belowdisplayskip{3pt}
\bal
I_n^k&= \me^{t_n}\Big[ A_0^{k-n} \D_{n+1}W \prod_{i=n+2}^{k} \lt(1+\b \D_iW\rt) A_0^k \prod_{l=1}^k\lt(1+\b \D_l W\rt)X_{h\t}({0}) \Big]   \\
&\q+\! \me^{t_n}\! \Big[\!  \t A_0^{k-n} \! \D_{n+1}\! W \prod_{i=n+2}^{k} \! (1\! +\! \b \D_iW)\sum_{l=0}^{k-1} A_0^{k-l} \! \prod_{m=l+2}^k\! (1\! +\! \b \D_mW) U^*_{h}(t_l)  \Big]  \\
&\q+ \me^{t_n}\Big[  A_0^{k-n} \D_{n+1}W \prod_{i=n+2}^{k} \lt(1+\b \D_iW\rt)\sum_{l=0}^{k-1} A_0^{k-l}  \\
&\qq\times \prod_{m=l+2}^k\lt(1+\b \D_mW\rt) \D_{l+1}W \Pi_h\si(t_l)  \Big]  \\
&=:\sum_{i=1}^3 I_{ni}^k\,.
\eal
\ee
By the mutual independence of $\{\D_nW\}_{n=1}^N$, and the fact that $X_{h\t}(0)$  
is deterministic, we obtain
\bel{w304e1}
\setlength\abovedisplayskip{3pt}
\setlength\belowdisplayskip{3pt}
\bal
\me\big[\|I_{n1}^k\|^2\big]&\leq \me\Big[\Big\| \prod_{i=n+2}^{k}(1+\b\D_iW )\D_{n+1}W
\prod_{l=1}^k (1+\b\D_l W )X_{h\t}({0}) \Big\|^2\Big]  \\
&=\prod_{i=n+2}^k\me\big[(1+\b \D_iW)^4\big] \me\big\{\big[\D_{n+1}W+\b\big(\D_{n+1}W\big)^2\big]^2\big\} \\
&\qq\times  \prod_{l=1}^n\me\big[(1+\b\D_l W)^2\big]\lt\| X_{h\t}(0)\rt\|^2  \\
& \leq \cC \t \| X_{h\t}(0)\|\,.
\eal
\ee
For $I_{n2}^k$, firstly we have
\begin{equation*}
\setlength\abovedisplayskip{3pt}
\setlength\belowdisplayskip{3pt}
\me\big[\|I_{n2}^k\|^2\big]
\leq \cC\t\sum_{l=0}^{k-1} \me\Big[\Big\| \prod_{i=n+2}^{k}\lt(1+ \b \D_iW\rt)\D_{n+1}W
\prod_{m=l+2}^k\lt(1+\b \D_mW\rt) U^*_{h}(t_l) \Big\|^2\Big]\,.
\end{equation*}
Now we divide this estimate into two cases.
\ss

\no{\bf Case (i)}  $l\leq n$. In this case, by using the same trick as that to derive \rf{w304e1}, 
we can see that
\bel{w304e2}
\setlength\abovedisplayskip{3pt}
\setlength\belowdisplayskip{3pt}
\bal
&\me\Big[\Big\| \prod_{i=n+2}^{k}\lt(1+\b  \D_iW\rt)\D_{n+1}W
\prod_{m=l+2}^k\lt(1+\b  \D_mW\rt) U^*_{h}(t_l) \Big\|^2\Big] \\
& \leq \prod_{i=n+2}^k\me\big[(1+\b  \D_iW)^4\big]
\me\big[\big(\D_{n+1}W+\b  (\D_{n+1}W)^2\big)^2\big]  
  \prod_{m=l+2}^n\me\big[(1+\b  \D_mW)^2\big]\me\big[\|U^*_h(t_l)\|^2\big] \\
&\leq \cC\t \sup_{t\in[0,T]}\me\big[\|U^*_h(t)\|^2\big]\,. 
\eal
\ee

\no{\bf Case (ii)}  $l> n$. Still applying the mutual independence of $\{\D_nW\}_{n=1}^N$, we have
\bel{w304e3}
\setlength\abovedisplayskip{3pt}
\setlength\belowdisplayskip{3pt}
\bal
&\me\Big[\Big\| \prod_{i=n+2}^{k}\lt(1+\b  \D_iW\rt)\D_{n+1}W
\prod_{m=l+2}^k\lt(1+\b  \D_mW\rt) U^*_{h}(t_l) \Big\|^2\Big]  \\
&\q\leq \prod_{i=l+2}^k\me\big[(1+\b  \D_i W)^4\big]\me\big[(1+\b  \D_{l+1}W )^2\big]  \\
&\qq\qq\times \Big\{ \prod_{m=n+2}^l\me\big[(1+\b  \D_m W)^4\big]\me\big[(\D_{n+1}W)^4\big]\me\big[\|U^*_h(t_l)\|^4\big]\Big\}^{1/2} \\
&\q \leq \cC\t  \Big(\sup_{t\in[0,T]} \me\big[\|U^*_h(t)\|^4\big] \Big)^{1/2}\, .
\eal
\ee
We may now combine both estimates and use Lemma \ref{w229l2} to conclude that
\begin{equation*}
\setlength\abovedisplayskip{3pt}
\setlength\belowdisplayskip{3pt}
\bal
\me\big[\|I_{n2}^k\|^2\big]
&\leq \cC\t  \Big[\sup_{t\in[0,T]}\me\big[\|U^*_h(t)\|^2\big]+\Big(\sup_{t\in[0,T]} \me\big[\|U^*_h(t)\|^4\big] \Big)^{1/2}\Big]\\
&\leq \cC \t  \big[\|x\|^2+\|\si(\cd) \|_{C([0,T];\dbL^2)}^2\big]\,.
\eal
\end{equation*}

To estimate $\me\big[\|I_{n3}^k\|^2\big]$, we need a longer computation. By introducing a notation
\begin{equation*}
\setlength\abovedisplayskip{3pt}
\setlength\belowdisplayskip{3pt}
\bal
  a_l 
\deq  A_0^{2k-n-l} \D_{n+1}W \prod_{i=n+2}^{k} \lt(1+\b \D_iW\rt)
  \prod_{m=l+2}^k\lt(1+\b \D_mW\rt) \D_{l+1}W \Pi_h\si(t_l)  \,,
\eal
\end{equation*}
we find that $I_{n3}^k=\sum_{l=0}^{k-1}\me^{t_n}[a_l]$, and divide the summation into three cases.

\ss
\no{\bf Case (i)} $l\leq n-1$. 
We use the facts that $\{\D_mW\}_{m\leq n}$ is $\mf_{t_n}$-measurable and  the mutual independence of $\{\D_i W\}_{i=n+1}^k$ to conclude that 
\begin{equation*}
\setlength\abovedisplayskip{3pt}
\setlength\belowdisplayskip{3pt}
\bal
\me^{t_n}\big[a_l\big]&=A_0^{2k-n-l}\Pi_h\si(t_l)\D_{l+1}W\prod_{m=l+2}^n(1+\b \D_mW)\\
&\q\times \me^{t_n} \big[(1+\b \D_{n+1}W)\D_{n+1}W\big]\prod_{i=n+2}^k\me^{t_n}\big[(1+\b \D_iW)^2\big]\\
&=\b \t A_0^{2k-n-l}\Pi_h\si(t_l)\D_{l+1}W\prod_{m=l+2}^n(1+\b \D_mW)\prod_{i=n+2}^k(1+\b ^2\t)\,,
\eal
\end{equation*}
and by the fact that $\|A_0\|_{\cL(\dbL^2|_{\dbV_h})}\leq 1$ we arrive at
\begin{equation*}
\setlength\abovedisplayskip{3pt}
\setlength\belowdisplayskip{3pt}
\bal
\me\big[\big\|\me^{t_n}\big[a_l\big]\big\|^2\big]
\leq \cC\t^2 \|\si(t_l)\|^2 \me\big[\D_{l+1}^2W\big] \prod_{m=l+2}^n\me\big[(1+\b \D_mW)^2\big]\prod_{i=n+2}^k(1+\b ^2\t)^2
\leq \cC\t^3 \|\si(t_l)\|^2\,.
\eal
\end{equation*}

\no{\bf Case (ii)} $l=n$.
Similar to the above case,
it follows that
\begin{equation*}
\setlength\abovedisplayskip{3pt}
\setlength\belowdisplayskip{3pt}
\bal
\me^{t_n}\big[a_n\big]
&=A_0^{2k-2n}\Pi_h\si(t_n)\t \prod_{i=n+2}^k(1+\b ^2\t)\,,
\eal
\end{equation*}
and 
\begin{equation*}
\setlength\abovedisplayskip{3pt}
\setlength\belowdisplayskip{3pt}
\bal
\me\big[\big\|\me^{t_n}\big[a_n\big]\big\|^2\big]
\leq \cC\t^2 \|\si(t_n)\|^2 \prod_{i=n+2}^k(1+\b ^2\t)^2
\leq \cC\t^2 \|\si(t_l)\|^2\,.
\eal
\end{equation*}

\no{\bf Case (iii)} $l\geq n+1$.
Based on the mutual independence of $\{\D_i W\}_{i=n+1}^k$ and the fact that $\me\big[\D_{n+1}W\big]=0$, we have
\begin{equation*}
\setlength\abovedisplayskip{3pt}
\setlength\belowdisplayskip{3pt}
\bal
\me^{t_n}\big[a_l\big] 
&=A_0^{2k-n-l}\Pi_h\si(t_l) \me\big[\D_{n+1}W\big] \prod_{i=n+2}^l\me\big[(1+\b \D_iW)  \big] \\
&\qq\times \me\big[(1+\b \D_{l+1}W)\D_{l+1}W\big] \prod_{m=l+2}^k\me\big[(1+\b \D_mW)^2\big]=0\,.
\eal
\end{equation*}

Combining with the above estimates, we conclude that
\begin{equation*}
\setlength\abovedisplayskip{3pt}
\setlength\belowdisplayskip{3pt}
\bal
&\me\big[\|I_{n3}^k\|^2\big]
=\me\Big[\Big\|\sum_{l=0}^{n-1} \me^{t_n}\big[a_l\big]+ \me^{t_n}\big[a_n\big]+\sum_{l=n+1}^{k-1} \me^{t_n}\big[a_l\big]\Big\|^2\Big]\\
&\qq\leq \cC \bigg\{n \sum_{l=0}^{n-1} \me\Big[\big\| \me^{t_n}\big[a_l\big]\big\|^2\Big]
+\me\big[\big\| \me^{t_n}\big[a_n\big]\big\|^2\big] \bigg\}\\
&\qq\leq \cC\t \|\si(\cd) \|_{C([0,T];\dbL^2)}^2 \,.
\eal
\end{equation*}

Note that the constant $\cC$ in the above estimates  is independent of $k$ and $n$. Finally, by applying these estimates
in \rf{w1117e3},
we deduce that
\bel{w1005e1}
\setlength\abovedisplayskip{3pt}
\setlength\belowdisplayskip{3pt}
\bal
\me\Big[\Big\| \t\b \sum_{k=n+1}^N I_n^k\Big\|^2\Big]
&\leq \cC\t^2 (N-n)\sum_{k=n+1}^N \sum_{i=1}^3\me\big[\big\|I_{ni}^k\big\|^2\big]\\
&\leq \cC \max_{0\leq n\leq N-1}\max_{n+1\leq k\leq N}\sum_{i=1}^3\me\big[\big\|I_{ni}^k\big\|^2\big]\\
&\leq \cC\t \big[\|x\|^2+\|\si(\cd) \|_{C([0,T];\dbL^2)}^2\big]  \,.
\eal
\ee
For term $I_n^0$, in the same vein, we also can derive
\begin{equation*}
\setlength\abovedisplayskip{3pt}
\setlength\belowdisplayskip{3pt}
\bal
\me\big[\|I_n^0\|^2\big]\leq \cC\t   \big[\|x\|^2+\|\si(\cd) \|_{C([0,T];\dbL^2)}^2\big]  \,,
\eal
\end{equation*}
which, together with \rf{w1005e1}, yields the assertion.
\end{proof}

Below, our aim is to
estimate  the difference between $Y_h\big(\cd;\cS_{h\t}(\Pi_\t U^*_h)\big)$ and $Y_\cd\big(\cS_{h\t}(\Pi_\t U^*_h)\big)$.
To accomplish it, we need to estimate  $Z_h\big(\cd;\cS_{h\t}(\Pi_\t U^*_h)\big)$.

\bl{w1024l1}
Let
$\big(Y_h\big(\cd;\cS_{h\t}(\Pi_\t U^*_h)\big),Z_h\big(\cd;\cS_{h\t}(\Pi_\t U^*_h)\big)\big)$ solve
\eqref{bshe1a} with $X^*_h(\cd)=\cS_{h\t}(\cd;\Pi_\t U^*_h)$. There exists a constant $\cC$ independent of $h\,,\t$ such that  
\bel{w1024e9}
\bal
&\me\Big[\|Z_h\big(t; \cS_{h\t}(\Pi_\t U^*_h)\big)-Z_h\big(\nu(t); \cS_{h\t}(\Pi_\t U^*_h)\big)\|^2 \Big]\\
&\qq\leq   \cC|t-\nu(t)|  \big[\|x\|_{\dbH_0^1}^2+\|\si(\cd)\|_{C([0,T];\dbH_0^1)}^2+L_{\si,1/2}^2\big] \qq\forall\, t\in [0,T] \, ,
\eal
\ee
where $ \nu(\cd)$ is defined in \rf{w827e1}.
\el

\begin{proof} The proof consists of three steps.
\ss

{\bf (1)} 
We claim that $\cS_{h\t}(t_n;\Pi_\t U^*_h)\in \dbD^{2,2}(\dbL^2)$, for any $n=0,1,\cds,N$.

Indeed, by  \rf{w1024e2} we know that 
\begin{equation*}
\setlength\abovedisplayskip{3pt}
\setlength\belowdisplayskip{3pt}
\bal
\cS_{h\t}(t_n;\Pi_\t U^*_h)
=&A_0^n\prod_{j=1}^n\lt(1\!+\!\b\D_j W\rt)X_{h\t}({0})
\!+\!\t \sum_{j=0}^{n-1}A_0^{n-j}\prod_{k=j+2}^n\lt(1\!+\!\b\D_kW\rt) U_{h}^*(t_j)\\
&+\sum_{j=0}^{n-1} A_0^{n-j}\prod_{k=j+2}^{n}\lt(1+\b\D_kW \rt)\Pi_h\si(t_j)\D_{j+1}W\,.
\eal
\end{equation*}
In the following, we only prove that the second term on the right-hand side of the above representation is in $\dbD^{2,2}(\dbL^2)$.
The other two terms can be proved in a similar vein.

By the optimality condition \rf{pontr1a} and Lemma \ref{w229l3}, 
we know that $U^*_h(t_j)\in \dbD^{2,2}(\dbL^2)$, for any $j=0,1,\cds,N-1$. 
For any $\th,\mu \in [0,T]$,
without loss of generality, suppose that $\th\in [t_l,t_{l+1}),\, \mu\in [t_m,t_{m+1})$. Then, by the chain rule and the fact that 
$D_\th (1+\b\D_kW)=\b\d_{lk},\,D_\mu (1+\b\D_kW)=\b\d_{mk}$, where $\d_{lk},\,\d_{mk}$ are Kronecker delta functions, we have
\begin{equation*}
\setlength\abovedisplayskip{3pt}
\setlength\belowdisplayskip{3pt}
\bal
&D_\th \Big(\t \sum_{j=0}^{n-1}A_0^{n-j}\prod_{k=j+2}^n\lt(1+\b\D_kW\rt) U_{h}^*(t_j)\Big)\\
&=\t \sum_{j=0}^{n-1}\b \chi_{\{j+2\leq l\leq n\}} A_0^{n-j}
\prod_{\scriptstyle k=j+2 \atop \scriptstyle k\neq l}^n
 (1+\b\D_kW) U_{h}^*(t_j)
  +\t \sum_{j=0}^{n-1}A_0^{n-j}\prod_{k=j+2}^n(1+\b\D_kW) D_\th U_{h}^*(t_j)\,,
\eal
\end{equation*}
and
\begin{equation*}
\setlength\abovedisplayskip{3pt}
\setlength\belowdisplayskip{3pt}
\bal
&D_\mu D_\th \Big(\t \sum_{j=0}^{n-1}A_0^{n-j}\prod_{k=j+2}^n\lt(1+\b\D_kW\rt) U_{h}^*(t_j)\Big)\\
&\qq=\t \sum_{j=0}^{n-1}\b^2 \chi_{\{j+2\leq l\leq n\}}\chi_{\{j+2\leq m\leq n\}}(1-\d_{ml}) A_0^{n-j} \prod_{\scriptstyle k=j+2 \atop \scriptstyle k\neq l,k\neq m}^n \lt(1+\b\D_kW\rt) U_{h}^*(t_j)\\
&\qq\q +\t \sum_{j=0}^{n-1}\b \chi_{\{j+2\leq l\leq n\}}A_0^{n-j}\prod_{\scriptstyle k=j+2 \atop \scriptstyle k\neq l}^n \lt(1+\b\D_kW\rt) D_\mu U_{h}^*(t_j)\\
&\qq\q  +\t \sum_{j=0}^{n-1}\chi_{\{j+2\leq m\leq n\}}A_0^{n-j}\prod_{\scriptstyle k=j+2 \atop \scriptstyle k\neq m}^n\lt(1+\b\D_kW\rt) D_\th U_{h}^*(t_j)\\
&\qq\q +\t \sum_{j=0}^{n-1}A_0^{n-j}\prod_{k=j+2}^n\lt(1+\b\D_kW\rt) D_\mu D_\th U_{h}^*(t_j) \,,
\eal
\end{equation*}
and subsequently, by \rf{w1029e4a} and \rf{w1029e5a} in Lemma  \ref{w229l3}, 
\begin{equation*}
\setlength\abovedisplayskip{3pt}
\setlength\belowdisplayskip{3pt}
\bal
&\me\Big[\Big\|D_\mu D_\th \Big(\t \sum_{j=0}^{n-1}A_0^{n-j}\prod_{k=j+2}^n\lt(1+\D_kW\rt) U_{h}^*(t_j)\Big)\Big\|^2\Big]\\
&\qq\leq \cC\max_{0\leq j\leq N-1} \me\Big[\|U^*_h(t_j)\|^2+\|D_\th U^*_h(t_j)\|^2
  +\|D_\mu D_\th U^*_h(t_j)\|^2\Big]\\
&\qq\leq \cC \big[\|x\|^2+\|\si(\cd)\|_{C([0,T];\dbL^2)}^2\big]\,,
\eal
\end{equation*}
which leads to $\t \sum_{j=0}^{n-1}A_0^{n-j}\prod_{k=j+2}^n\lt(1+\D_kW\rt) U_{h}^*(t_j)\in \dbD^{2,2}(\dbL^2)$.
Hence, for any $t\in[0,T]$, $\cS_{h\t}\big(\nu(t);\Pi_\t U^*_h\big)\in \dbD^{2,2}(\dbL^2)$, and
\begin{equation}\label{w1024e8}
\setlength\abovedisplayskip{3pt}
\setlength\belowdisplayskip{3pt}
\bal
&\sup_{\th\in[0,T]}\sup_{\nu(t)\in[\th,T]} \me\big[\| D_\th \cS_{h\t}(\nu(t);\Pi_\t U^*_h)\|^2\big]  
+ \sup_{\mu,\th\in[0,T]}\sup_{\nu(t)\in[\mu\vee \th,T]} \me\big[\|D_\mu D_\th \cS_{h\t}(\nu(t);\Pi_\t U^*_h)\|^2\big]   \\
&\qq\qq\leq \cC \big[\|x\|^2+\|\si(\cd)\|_{C([0,T];\dbL^2)}^2\big]\,.
\eal
\end{equation}

Applying the same trick, the optimality condition \rf{pontr1a}, \rf{w1029e4a} and \rf{w1029e5b} with $\g=1$ in Lemma \ref{w229l3}
we can deduce that there exists a constant $\cC$ such that for any $\th\in [0,T]$,
\bel{w1024e6}
\setlength\abovedisplayskip{3pt}
\setlength\belowdisplayskip{3pt}
\bal
\me\big[ \|\nb D_\th \cS_{h\t}(T;\Pi_\t U^*_h)\|^2\big]
\leq \cC \big[\|x\|_{\dbH_0^1}^2+\|\si(\cd)\|_{C([0,T];\dbH_0^1)}^2\big]\,.
\eal
\ee

\ss
{\bf (2)} 
By the same procedure as  in the proof of \rf{w1029e5a}, \rf{w1029e5b} in Lemma \ref{w229l3}, 
thanks to \rf{w1024e8}, we can obtain
\bel{w1024e11}
\setlength\abovedisplayskip{3pt}
\setlength\belowdisplayskip{3pt}
\bal
&\sup_{\th\in[0,T]}\me\Big[\int_{\th}^T\big\| D_\th Z_h\big(t; \cS_{h\t}(\Pi_\t U^*_h)\big)\big\|^2\rd t\Big]\\
&\qq\leq \cC \sup_{\th\in[0,T]}\sup_{\nu(t)\in[\th,T]}\me\big[\big\|D_\th \cS_{h\t}\big(\nu(t);\Pi_\t U^*_h\big)\big\|^2\big]\\
&\qq \leq \cC \big[\|x\|^2+\|\si(\cd)\|_{C([0,T];\dbL^2)}^2\big]\,,
\eal
\ee
\bel{w1024e12}
\setlength\abovedisplayskip{3pt}
\setlength\belowdisplayskip{3pt}
\bal
&\sup_{\mu,\th\in[0,T]}\sup_{t\in[\mu\vee \th,T]}\me\big[\big\|D_\mu D_\th Y_h\big(t; \cS_{h\t}(\Pi_\t U^*_h)\big)\big\|^2\big]\\
&\qq\leq \cC \sup_{\mu,\th\in[0,T]}\sup_{\nu(t)\in[\mu\vee \th,T]}\me\big[\big\|D_\mu D_\th \cS_{h\t}\big(\nu(t);\Pi_\t U^*_h\big)\big\|^2\big]\\
&\qq \leq \cC \big[\|x\|^2+\|\si(\cd)\|_{C([0,T];\dbL^2)}^2\big]\,,
\eal
\ee
and
\bel{w1024e13}
\setlength\abovedisplayskip{3pt}
\setlength\belowdisplayskip{3pt}
\bal
&\sup_{\th\in[0,T]}\me\Big[\int_{\th}^T\| \D_h D_\th Y_h(t; \cS_{h\t}(\Pi_\t U^*_h))\|^2\rd t\Big]\\
&\qq\leq C \sup_{\th\in[0,T]}\sup_{\nu(t)\in[\th,T]}\me\big[\big\|\nb D_\th \cS_{h\t}\big(\nu(t);\Pi_\t U^*_h\big)\big\|^2\big]\\
&\qq\leq C \big[\|x\|_{\dbH_0^1}^2+\|\si(\cd)\|_{C([0,T];\dbH_0^1)}^2\big]\,.
\eal
\ee

{\bf (3)} 
Applying the fact that $Z_h\big(\cd;\cS_{h\t}(\Pi_\t U^*_h)\big)=D_\cd Y_h\big(\cd;\cS_{h\t}(\Pi_\t U^*_h)\big) \, \ae$, for any $t\in [t_n,t_{n+1})$, $n=0,1,\cds,N-1$, we arrive at
\bel{w1024e3}
\setlength\abovedisplayskip{3pt}
\setlength\belowdisplayskip{3pt}
\bal
&\me \big[ \big \|Z_h\big(t;\cS_{h\t}(\Pi_\t U^*_h)\big)-Z_h\big(\nu(t);\cS_{h\t}(\Pi_\t U^*_h)\big)\big\|^2\big] \\
&\q \leq 2\me\Big[\big\|(D_t-D_{t_n})Y_h\big(t;\cS_{h\t}(\Pi_\t U^*_h)\big)\big\|^2\\
&\qq\q +\Big\|D_{t_n}\Big(Y_h\big(t;\cS_{h\t}(\Pi_\t U^*_h)\big)-Y_h\big(t_n;\cS_{h\t}(\Pi_\t U^*_h)\big)\Big)\Big\|^2\Big]\\
&\q=: 2(I_1+I_2)\,.\\
\eal
\ee
Similar to \rf{w1029e5c} in Lemma \ref{w229l3}, 
\begin{equation*}
\setlength\abovedisplayskip{3pt}
\setlength\belowdisplayskip{3pt}
I_1\leq \cC \max_{n\leq i\leq N} \me\big[\|(D_t-D_{t_n})\cS_{h\t}(t_i; \Pi_\t U^*_h)\|^2\big]\,,
\end{equation*}
which, together with
\begin{equation*}
\setlength\abovedisplayskip{3pt}
\setlength\belowdisplayskip{3pt}
\bal
(D_t-D_{t_n})\cS_{h\t}(t_i; \Pi_\t U^*_h)
=\t \sum_{j=0}^{i-1}A_0^{i-j}\prod_{k=j+2}^i \lt(1+\b\D_kW\rt) (D_t-D_{t_n})U_{h}^*(t_j) \q \ae\,,
\eal
\end{equation*}
 the optimality condition \rf{pontr1a}, and  \rf{w1029e5c} in Lemma \ref{w229l3}, yields
\bel{w1024e4}
\setlength\abovedisplayskip{3pt}
\setlength\belowdisplayskip{3pt}
\bal
I_1&\leq  \cC\max_{n\leq j\leq N} \me\big[\|(D_t-D_{t_n})U^*_h(t_j)\|^2\big]\\
&\leq C |t-t_n| \big[\|x\|_{\dbH_0^1}^2+\|\si(\cd)\|_{C([0,T];\dbH_0^1)}^2+L_{\si,1/2}^2\big] \,.
\eal
\ee
By virtue of \eqref{w1024e8}, \eqref{w1024e11}--\eqref{w1024e13}, we can get
\begin{eqnarray}
&I_2&\leq \cC\bigg\{ \me\Big[\int_{t_n}^t\|D_{t_n}D_\th Y_h\big(\th;\cS_{h\t}(\Pi_\t U^*_h)\big)\|^2\rd \th\Big] \notag \\
&&\q+(t-{t_n})\me\Big[ \int_{t_n}^T \|\D_h D_{t_n}Y_h\big(\th;\cS_{h\t}(\Pi_\t U^*_h)\big)\|^2+\| D_{t_n}Z_h\big(\th;\cS_{h\t}(\Pi_\t U^*_h)\big)\|^2 \notag \\
&&\qq\qq\qq+\|D_{t_n}\cS_{h\t}\big(\nu(\th);\Pi_\t U^*_h\big)\|^2 \rd \th \Big]\bigg\} \notag \\
&& \leq  C|t-{t_n}|  \big[\|x\|_{\dbH_0^1}^2+\|\si(\cd)\|_{C([0,T];\dbH_0^1)}^2\big]\,.\label{w1024e5}
\end{eqnarray}
Now the desired result \rf{w1024e9} can be derived by \rf{w1024e3}--\rf{w1024e5}.
\end{proof}

\bl{w1019l2}
Suppose that 
$\big(Y_h\big(\cd;\cS_{h\t}(\Pi_\t U^*_h)\big),Z_h\big(\cd;\cS_{h\t}(\Pi_\t U^*_h)\big)\big)$ solves BSDE
\eqref{bshe1a} with $X^*_h(\cd)=\cS_{h\t}(\cd;\Pi_\t U^*_h)$ and $\big(Y_\cd\big(\cS_{h\t}(\cd;\Pi_\t U^*_h)\big),Z_0\big(\cd;\cS_{h\t}(\cd;\Pi_\t U^*_h)\big)\big)$ solves \eqref{w229e7} with
$X^*_{h\t}(\cd)=\cS_{h\t}(\cd;\Pi_\t U^*_h)$.
Then, there exist a constant $\cC$ independent of $h\,,\t$ such that
\begin{eqnarray*}
&&\sum_{n=0}^{N-1}\me\Big[\int_{t_n}^{t_{n+1}}\|Y_h\big(t;\cS_{h\t}(\Pi_\t U^*_h)\big)-Y_n\big(\cS_{h\t}(\Pi_\t U^*_h)\big)\|^2\rd t\Big]\\
&&\qq\leq  \cC\t \big[ \|   x\|_{\dbH_0^1\cap \dbH^2}^2
+ \|\si(\cd)\|^2_{C([0,T];\dbH_0^1)\cap L^2(0,T;\dbH_0^1\cap\dbH^2)}+L_{\si,1/2}^2\big]\,.
\end{eqnarray*}
\el

\begin{proof}
For simplicity, in the proof we write solution pairs $\big( Y_h\big(\cd;\cS_{h\t}(\Pi_\t U^*_h)\big), Z_h\big(\cd;\cS_{h\t}(\Pi_\t U^*_h)\big)\big)$
resp.~$\big(Y_\cd\big(\cS_{h\t}(\Pi_\t U^*_h)\big), Z_0\big(\cd;\cS_{h\t}(\Pi_\t U^*_h)\big)\big)$
as $\big(\wt Y_h(\cd), \wt Z_h(\cd)\big)$ resp.~$\big(\wt Y_n, \wt Z_0(\cd)\big)$.

Firstly we apply Theorem \ref{w1019t1} with $Y_{T,h}=\wt Y_{Y,h}=-\a \cS_{h\t}(T; \Pi_\t U^*_h)$ and 
$f_h(\cd)=\wt f_h(\cd)=\cS_{h\t}(\cd; \Pi_\t U^*_h)$, and arrive at
\begin{equation*}
\setlength\abovedisplayskip{3pt}
\setlength\belowdisplayskip{3pt}
\bal
&\sum_{n=0}^{N-1}\me\Big[\int_{t_n}^{t_{n+1}}\|\wt Y_h(t)-\wt Y_n\|^2\rd t\Big]\\
&\leq \cC\t \me\Big[\|\cS_{h\t}(T; \Pi_\t U^*_h)\|_{\dot \dbH^2_h}^2+\int_0^T\|\cS_{h\t}(t; \Pi_\t U^*_h)\|_{\dot \dbH^1_h}^2\rd t\Big]
+\cC\sum_{n=0}^{N-1}\me\Big[\int_{t_n}^{t_{n+1}}\|\wt Z_h(t)-\overline{\wt { Z}}_h(t)\|_{\dot \dbH^{-1}_h}^2 \rd t\Big]\\
&=:\cC \t \sum_{i=1}^2 I_i+\cC I_3\,.
\eal
\end{equation*}
For $I_1$, by Theorem \ref{MP}, Lemmata \ref{w1017l1} and \ref{w1019l3}, we find that
\begin{equation*}
\setlength\abovedisplayskip{5pt}
\setlength\belowdisplayskip{5pt}
\bal
I_1
  &\leq \cC\big[ \|   x\|_{\dbH_0^1\cap \dbH^2}^2
+ \|\si(\cd)\|^2_{L^2(0,T;\dbH_0^1\cap\dbH^2)}\big]\,.
\eal
\end{equation*}
Term $I_2$ can be estimated in the same vein and
\begin{equation*}
\setlength\abovedisplayskip{5pt}
\setlength\belowdisplayskip{5pt}
I_2\leq \cC\big[ \|   x\|_{\dbH_0^1}^2
+ \|\si(\cd)\|^2_{L^2(0,T;\dbH_0^1)}\big]\,.
\end{equation*}
For $I_3$, by 
Lemma \ref{w1024l1} we have
\begin{equation*}
\setlength\abovedisplayskip{5pt}
\setlength\belowdisplayskip{5pt}
\bal
I_3&\leq \sum_{n=0}^{N-1}\me\Big[\int_{t_n}^{t_{n+1}}\|\wt Z_h(t)-\wt  Z_h(t_n)\|_{\dot \dbH^{-1}_h}^2 \rd t\Big]\\
&\leq \cC \t \big[\|x\|_{\dbH_0^1}^2+\|\si(\cd)\|_{C([0,T];\dbH_0^1)}^2+L_{\si,1/2}^2\big]\,.
\eal
\end{equation*}

A combination of  estimates for $I_1$ through $I_3$ settles the desired assertion. 
\end{proof}

We are now ready to verify rates of convergence for the optimal pair of Problem {\bf (SLQ)$_{h\t}$}; it is as in 
Section \ref{dis-slq-h} that 
 $\cS_{h\t}: \dbU_{\dbF} \rightarrow \dbX_{\dbF}$\index{${\mathcal S}_{h\tau}$} is used, which is the solution operator to the forward equation \eqref{w1212e3a}.

\begin{proof} [\bf {Proof of Theorem \ref{rate2}}]

We divide the proof into two steps.
\ss

 {\bf (1)}
We  follow the argument
in the proof of Theorem \ref{rate1}.
Firstly, we have
\bel{w1206e4}
\setlength\abovedisplayskip{3pt}
\setlength\belowdisplayskip{3pt}
\bal
&\| U^*_{h\t}(\cd) - \Pi_\t U^*_h(\cd)\|_{L^2_\dbF(0,T; \dbL^2)}^2\\
&\q\leq \Big[ \Big(  D \cJ_h\big(U^*_h(\cd)\big)-  D \cJ_{h}\big(\Pi_\t U^*_h(\cd)\big),U^*_{h\t}(\cd) - \Pi_\t U^*_h(\cd)\Big)_{L^2_\dbF(0,T; \dbL^2)} \\ 
& \qq+ \Big(  D\! \cJ_{h}\big(\Pi_\t U^*_h(\cd)\big)\!-\! D\! \cJ_{h,\t}\big(\Pi_\t U^*_h(\cd)\big) ,U^*_{h\t}(\cd)\! - \!\Pi_\t U^*_h(\cd)
\Big)_{L^2_\dbF(0,T; \dbL^2)}\Big]\, .
\eal
\ee
Therefore,
\begin{eqnarray}
&\| U^*_{h\t} - {\Pi_\t U^*_h}\|_{L^2_\dbF(0,T; \dbL^2)}^2 
 &\leq  3 \Big[ \big \|  D\cJ_h\big(U^*_h(\cd)\big)-  D\cJ_{h}\big(\Pi_\t U^*_h(\cd)\big) \big \|_{L^2_\dbF(0,T;\dbL^2)}^2\nonumber \\
&&\q+ \big \| \cT_h^1(\cd; \cS_h(\Pi_\t U^*_h))-Y_\cd(\cS_{h\t}(\Pi_\t U^*_h))
\big \|_{L^2_\dbF(0,T;\dbL^2)}^2 \nonumber\\ 
&&\q+\big  \| Y_\cd(\cS_{h\t}(\Pi_\t U^*_h))-Y_{h\t}(\cd; \cS_{h\t}(\Pi_\t U^*_h))
\big \|_{L^2_\dbF(0,T;\dbL^2)}^2\Big]  \nonumber \\ 
 &&= 3\sum_{i=1}^3 I_i \, . \label{w1212e1}
\end{eqnarray}
We use \eqref{derivative-semidisc} and the optimality condition  \eqref{pontr1a}
 to bound $I_1$ as follows,
\bel{w1206e5}
\setlength\abovedisplayskip{3pt}
\setlength\belowdisplayskip{3pt}
\bal
I_1 
&\leq 2 \big[\| U^*_h(\cd) - \Pi_\t U^*_h(\cd) \|_{L^2_\dbF(0,T; \dbL^2)}^2
  +\big \|\cT_h^1 \big(\cd; {\mathcal S}_h(\Pi_\t U^*_h) \big) -\cT_h^1\big(\cd; {\mathcal S}_h(U^*_h) \big) \big \|_{L^2_\dbF(0,T; \dbL^2)}^2\big]\, .
\eal
\ee
 By Lemma \ref{w1008l1} on  stability properties of solutions to  BSDE \eqref{bshe1a} 
 with $X^*_h(\cd)=\cS_h(\cd; \Pi_\t U^*_h)$ and  Lemma \ref{w1002l4} for
 SDE \eqref{spde1-h} with $x=0$, $f_h(\cd)=U^*_h(\cd)-\Pi_\t U^*_h(\cd)$, $g_h(\cd)=0$, we obtain
\bel{w1206e6}
\setlength\abovedisplayskip{3pt}
\setlength\belowdisplayskip{3pt}
\bal
 & \lt\| \cT_h^1 \big(\cd; {\mathcal S}_h(\Pi_\t U^*_h) \big) -\cT_h^1\big(\cd; {\mathcal S}_h(U^*_h) \big) \rt\|_{L^2_\dbF(0,T; \dbL^2)}^2\\ 
 &\qq \leq  \cC  \Big[\| \cS_h(T; U^*_h)-\cS_h(T; \Pi_\t U^*_h) \|^2_{L^2_{\mf_T}(\O;\dot\dbH^{-1}_h)}
     + \| \cS_h(\cd; U^*_h)-\cS_h(\cd; \Pi_\t U^*_h) \|^2_{L^2_\dbF(0,T;\dot\dbH^{-2}_h)} \Big] \\
 &\qq  \leq   \cC \| U^*_h(\cd)-\Pi_\t U^*_h(\cd) \|^2_{L^2_\dbF(0,T;\dbL^2)} \, .
\eal
\ee 
By the optimality condition \eqref{pontr1a}, estimate \rf{w1008e2} of Lemma \ref{w1008l1} with 
$Y_{T,h}=-\a \cS_h(T; U^*_h)\,,f_h(\cd)=\cS_h(\cd; U^*_h)$, then Lemma \ref{w1002l4} with $f_h(\cd)=U_h^*(\cd)\,,g_h(\cd)=\Pi_h\si(\cd)$, and \rf{w1029e2b} in Lemma \ref{w229l2},
we have
\bel{w1206e7}
\setlength\abovedisplayskip{3pt}
\setlength\belowdisplayskip{3pt}
\bal
 &\| U^*_h(\cd)-\Pi_\t U^*_h(\cd) \|^2_{L^2_\dbF(0,T; \dbL^2)} 
 = \| Y_h(\cd)-\Pi_\t Y_h(\cd) \|^2_{L^2_\dbF((0,T; \dbL^2)} \\
&\qq  \leq \cC\t  \me\Big[\|\cS_h(T; U^*_h)\|_{\dbH_0^1}^2+\int_0^T \|\cS_h(t; U^*_h)\|^2\rd t\Big] \\
&\qq\leq \cC \t \big[\|x\|_{\dbH_0^1}^2+\|\si(\cd)\|_{L^2(0,T;\dbH_0^1)}^2 \big]\,.
\eal
\ee

Next, we turn to $I_2$. The triangle inequality leads to
\begin{equation*}
\setlength\abovedisplayskip{3pt}
\setlength\belowdisplayskip{3pt}
\bal
 I_2
 &\leq2  \Big[\big\| \cT_h^1 \big(\cd; \cS_{h}(\Pi_\t U^*_h) \big) -  \cT_h^1 \big(\cd; \cS_{h\t}(\Pi_\t U^*_h) \big) \big\|^2_{L^2_\dbF(0,T; \dbL^2)} \\
  &\q+\big\| \cT_h^1 \big(\cd; \cS_{h\t}(\Pi_\t U^*_h) \big) -Y_\cd \big(\cS_{h\t}(\Pi_\t U^*_h)\big)  \big\|^2_{L^2_\dbF(0,T; \dbL^2)}\Big]\\
 &=: 2 \big(I_{21}+I_{22}\big) \,. 
\eal
\end{equation*}
In order to bound $I_{21}$, we use Lemma \ref{w1008l1}, then \rf{w1028e1b} in Theorem \ref{w1002t3} with $f_h(\cd)=\wt f_h(\cd)=\Pi_\t U^*_h(\cd)$, 
$g_h(\cd)=\wt g_h(\cd)=\Pi_h\si(\cd)$, and finally \rf{w1029e4a} in Lemma \ref{w229l3}  to conclude that 
\begin{equation*}
\setlength\abovedisplayskip{3pt}
\setlength\belowdisplayskip{3pt}
\bal
I_{21}&  \leq \cC \, \me\Big[\| \cS_{h}(T; \Pi_\t U^*_h) -  \cS_{h\t}(T; \Pi_\t U^*_h) \|^2
+\int_0^T \|\cS_{h}(t; \Pi_\t U^*_h) -  \cS_{h\t}(t; \Pi_\t U^*_h) \|_{\dot\dbH^{-1}_h}^2\rd t\Big]\\
&\leq \cC \t  \big[\|x\|_{\dbH_0^1\cap\dbH^2}^2+\|\si(\cd) \|_{L^2(0,T;\dbH_0^1\cap\dbH^2)}^2\big]\,.
\eal
\end{equation*}
For $I_{22}$,  Lemma \ref{w1019l2} implies that
\begin{equation*}
\setlength\abovedisplayskip{3pt}
\setlength\belowdisplayskip{3pt}
\bal
I_{22}\leq   \cC\t \big[ \|   x\|_{\dbH_0^1\cap \dbH^2}^2
+ \|\si(\cd)\|^2_{C([0,T];\dbH_0^1)\cap L^2(0,T;\dbH_0^1\cap\dbH^2)}+L_{\si,1/2}^2\big]\,.
\eal
\end{equation*}
%
%

Finally,
Lemma \ref{w301l1} leads to 
\begin{equation*}
\setlength\abovedisplayskip{5pt}
\setlength\belowdisplayskip{5pt}
\bal
I_3 \leq \cC\t \big[\|x\|^2+\|\si(\cd) \|_{C([0,T];\dbL^2)}^2\big] \,.
\eal
\end{equation*}

Now we insert above estimates into \eqref{w1212e1} to obtain the $U^*_h(\cd)-U^*_{h\t}(\cd)$ part of the assertion \rf{w1030e1a}.

\ss

{\bf (2)}
By setting $f_h(\cd)=U^*_h(\cd)$, $\wt f_h(\cd)=U^*_{h\t}(\cd)$ and $g_h(\cd)=\wt g_h(\cd)=\Pi_h \si(\cd)$, and 
applying the
assertion \rf{w1030e1a} as well as
Theorem \ref{w1002t3},  \rf{w1029e2b} in Lemma \ref{w229l2}, we deduce that
\begin{equation*}
\setlength\abovedisplayskip{3pt}
\setlength\belowdisplayskip{3pt}
\bal
&\sum_{n=0}^{N-1}\me\Big[\int_{t_n}^{t_{n+1}}\|X^*_h(t)-X^*_{h\t}(t_n)\|^2\rd t\Big]\\
&\qq\leq \cC\t \bigg\{ \|x\|_{\dbH_0^1}^2+\me\Big[\int_0^T\|U^*_h(t)\|^2+\|\si(t)\|_{\dbH_0^1}^2\rd t\Big]\bigg\}\\
&\qq\q+\cC\sum_{n=0}^{N-1} \me\Big[\int_{t_n}^{t_{n+1}} \|U^*_h(t)-U^*_{h\t}(t_n)\|^2+\|\si(t)-\si(t_n)\|^2 \rd t \Big]\\
&\qq\leq  \cC\t \big[ \|   x\|_{\dbH_0^1\cap \dbH^2}^2
+ \|\si(\cd)\|^2_{C([0,T];\dbH_0^1)\cap L^2(0,T;\dbH_0^1\cap\dbH^2)}+L_{\si,1/2}^2\big]\,,
\eal
\end{equation*}
which is the $X^*_h(\cd)-X^*_{h\t}(\cd)$ part of the  assertion \rf{w1030e1a},
and
\begin{equation*}
\setlength\abovedisplayskip{3pt}
\setlength\belowdisplayskip{3pt}
\bal
&\max_{0\leq n\leq N-1}\sup_{t\in[t_n,t_{n+1})}\me\big[\|X_h(t)-X_n\|^2\big]
+\sum_{n=0}^{N-1}\me\Big[\int_{t_n}^{t_{n+1}}\|X_h(t)-X_n\|_{\dbH_0^1}^2\rd t\Big]\\
&\qq\leq \cC\t \bigg\{\| x\|_{\dbH_0^1\cap\dbH^2}^2+\me\Big[\int_0^T  \|U^*_h(t)\|_{\dbH_0^1}^2+\|\Pi_h \si(t)\|_{\dbH_0^1\cap\dbH^2}^2 \rd t \Big]\bigg\}\\
&\qq\q+\cC\sum_{n=0}^{N-1} \me\Big[\int_{t_n}^{t_{n+1}} \|U^*_h(t)-U^*_{h\t}(t_n)\|^2+\|\si(t)-\si(t_n)\|^2 \rd t \Big]\\
&\qq\leq \cC\t \big[ \|   x\|_{\dbH_0^1\cap \dbH^2}^2
+ \|\si(\cd)\|^2_{C([0,T];\dbH_0^1)\cap L^2(0,T;\dbH_0^1\cap\dbH^2)}+L_{\si,1/2}^2\big]\,.
\eal
\end{equation*}
That settles the  assertion \rf{w1030e1b} and completes the proof.
\end{proof}

\subsubsection{The gradient descent method for Problem {\bf (SLQ)}$_{h\tau}$}\label{numopt}

By Theorem \ref{MP}, solving Problem {\bf (SLQ)$_{h\t}$} is 
equivalent to solving the system of the coupled forward-backward difference equations 
\eqref{w1212e3} and \eqref{w1003e12}. We may exploit the variational
character of Problem {\bf (SLQ)$_{h\t}$} to construct a gradient descent method
where approximate iterates of the optimal control $U^*_{h\t}(\cd)$ in the 
Hilbert space $\dbU_{\dbF}$ are obtained; see also \cite{Nesterov04,Kabanikhin12} for more details.
%
%

\ss

\begin{algorithm}\label{alg1}
Let $U_{h\t}^{(0)}(\cd)\in \dbU_{\dbF}$, and fix $\kappa > 0$. For any $\ell \in {\mathbb N}_0$, update $U_{h\t}^{(\ell)}(\cd) \in \dbU_{\dbF}$ as follows:
\begin{enumerate}[{\rm 1.}]
\item Compute $X_{h\t}^{(\ell)}(\cd)\in \dbX_{\dbF}$ by 
\bel{w0115e4}
\setlength\abovedisplayskip{3pt}
\setlength\belowdisplayskip{3pt}
\lt\{\!\!\!
\begin{array}{ll} 
 \ds [\mathds{1} - \tau \Delta_h]X^{(\ell)}_{h\t}(t_{n+1})= X^{(\ell)}_{h\t}(t_n)+ \tau U^{(\ell)}_{h\t}(t_n) \\
\ns\ds \qq\qq+ \big[\b X^{(\ell)}_{h\t}(t_n)+\Pi_h\si(t_n)\big] \D_{n+1}W \q n=0,1,\cds,N-1\, ,\\
\ns\ds X_{h\t}^{(\ell)}(0)=\Pi_h x\, .
\end{array}
\rt.
\ee
\item Use $X_{h\t}^{(\ell)}(\cd)\in \dbX_{\dbF}$ to compute $Y_{h\t}^{(\ell)}(\cd)\in \dbX_{\dbF}$ via
\begin{equation*}
\setlength\abovedisplayskip{3pt}
\setlength\belowdisplayskip{3pt}
\bal
Y_{h\t}^{(\ell)} (t_n)=&-\t \me^{t_n}\Big[ \sum_{j=n+1}^{N-1} A_0^{j-n} \prod_{k={n+2}}^j(1+\b \D_kW) X^{(\ell)}_{h\t}(t_j) \Big]  \\
&-\a \me^{t_n} \Big[ A_0^{N-n} \prod_{k=n+2}^N(1+\b \D_kW) X^{(\ell)}_{h\t}(T) \Big] \q n=0,1,\cds,N-1\,.\\
\eal
\end{equation*}

\item Update $U_{h\t}^{(\ell+1)} (\cd)\in \dbU_{\dbF}$ via
\begin{equation*}
\setlength\abovedisplayskip{3pt}
\setlength\belowdisplayskip{3pt}
U^{(\ell+1)}_{h\t}(\cd)=U^{(\ell)}_{h\t}(\cd)-\frac 1 {\kappa} \big[U^{(\ell)}_{h\t}(\cd) -Y_{h\t}^{(\ell)} (\cd)\big] \, .
\end{equation*}
\end{enumerate}
\end{algorithm}

If compared with \eqref{w1212e3}--\eqref{w1003e12}, steps 1 and 2 are now decoupled: the first step requires to solve a  spatio-temporal discretized  equation of  SPDE \rf{fbspdea}, while the second requires to solve a  spatio-temporal discretized equation of the BSPDE \rf{fbspdeb} which is not the numerical solution by the implicit Euler method (see Lemmata 
 \ref{w229l1} and \ref{w301l1} for the difference).
A similar method to solve Problem {\bf (SLQ)$_{h\t}$} 
has been proposed in \cite{Dunst-Prohl16, Prohl-Wang21}.

Below,
we  present a lower bound for $\kappa$ and discuss the speed of convergence of Algorithm \ref{alg1}.
 For this purpose, we first show Lipschitz continuity of 
$D\cJ_{h,\tau}(\cd)$.
By borrowing notations  defined in \rf{rep1}--\rf{w1003e14} in the proof of Theorem \ref{MP}, and applying \rf{w1024e2}
and \rf{w1008e4},
we know that
\begin{equation*}
\setlength\abovedisplayskip{5pt}
\setlength\belowdisplayskip{5pt}
\bal
 D\cJ_{h,\t}\big(U_{h\t}(\cd)\big)
&= U_{h\t}(\cd) +L^*\big[\big(\G X_{h\t}(0)\big)(\cd)+\big(LU_{h\t}(\cd)\big)(\cd)+f(\cd)\big](\cd) \\
 &\q+\a \h L^* [\h\G X_{h\t}(0)+\h LU_{h\t}+\h f\,](\cd)\,.
\eal
\end{equation*}
Thus by the linearity of $L\,,\h L\,,\G\,,\h\G$,  for any $U^1_{h\t}(\cd)\,,U^2_{h\t}(\cd)\in\dbU_\dbF$, it follows that
$\big\|D\h\cJ_{h,\t}\big(U^1_{h\t}(\cd)\big)-D\h\cJ_{h,\t}\big(U^2_{h\t}(\cd)\big)\big\|_{\dbU_{\dbF}}
\leq \|\mathds{1}_h +L^* L+\a\h L^*\h L\|_{\cL(\dbU_{\dbF})}  \|U^1_{h\t}(\cd)-U^2_{h\t}(\cd)\|_{\dbU_{\dbF}}$.
Furthermore, we can show a upper bound of $ \|\mathds{1}_h +L^* L+\a\h L^*\h L\|_{\cL(\dbU_{\dbF})}$.
Indeed,
on noting that $\|A_0\|_{\cL(\dbV_h|_{\dbL^2})}=\|(\mathds{1}_h-\t\D_h)^{-1}\|_{\cL(\dbV_h|_{\dbL^2})}\leq 1$,
we find that
\begin{equation*}
\setlength\abovedisplayskip{5pt}
\setlength\belowdisplayskip{5pt}
\bal
 \big \|\big(LU_{h\t}(\cd)\big)(\cd)\big\|_{\dbX_{\dbF}}^2
=\sum_{n=1}^N \tau \me \Big[\Big\| \t \sum_{j=0}^{n-1} A_0^{n-j}\prod_{k=j+2}^n\lt(1+\D_kW \rt) U_{h\t}(t_j) \Big\|^2  \Big]
\leq  T^2e^{\b^2T}\|U_{h\t}(\cd)\|_{\dbU_{\dbF}}^2\, .
\eal
\end{equation*}
In a similar vein, we can prove that 
$\|\h LU_{h\t}(\cd)\|_{L^2_{\mf_T}(\O;\dbL^2)}^2
\leq   Te^{\b^2T}\|U_{h\t}\|_{\dbU_{\dbF}}^2$.
Hence
\begin{equation*}
\setlength\abovedisplayskip{5pt}
\setlength\belowdisplayskip{5pt}
\|\mathds{1}_h +L^* L+\a\h L^*\h L\|_{\cL(\dbU_{\dbF})}\leq 1+\a Te^{\b^2T}+T^2e^{\b^2T}\, .
\end{equation*}
Since Algorithm \ref{alg1} is the gradient descent method for Problem {\bf (SLQ)$_{h\t}$}, we have the following estimates.

\bt{gradient-rate} 
Suppose that $\kappa \geq \|\mathds{1}_h +L^* L+\a\h L^*\h L\|_{\cL(\dbU_{\dbF})}$.
Let $\ds \{U^{(\ell)}_{h\t}(\cd)\}_{\ell \in {\mathbb N}_0} \subset \dbU_{\dbF}$ be  generated by Algorithm \ref{alg1}, and $U^*_{h\t}(\cd)$ solves Problem {\bf (SLQ)$_{h\t}$}. Then,  for $\ell=1,2,\cds$,
\begin{subequations}\label{w1030e2}
    \begin{empheq}[left={\empheqlbrace\,}]{align}
&  \big \| U^{(\ell)}_{h\t}(\cd)-U^*_{h\t}(\cd)\big \|_{\dbU_{\dbF}}^2
\leq \Big(1-\frac 1 {\kappa}\Big)^{\ell} \big\| U^{(0)}_{h\t}(\cd)-U^*_{h\t}(\cd) \big\|_{\dbU_{\dbF}}^2  \, ,   \label{w1030e2a}\\
&\cJ_{h,\t}\big(U^{(\ell)}_{h\t}(\cd)\big)-\cJ_{h,\t}\big(U^*_{h\t}(\cd)\big)
\leq \frac{2 \kappa \big\| U^{(0)}_{h\t}(\cd)-U^*_{h\t}(\cd) \big\|_{\dbU_{\dbF}}^2}{\ell}\,  , \label{w1030e2b}\\
& \max_{0 \leq n \leq N}  {\mathbb E}\big[\big \|X_{h\t}^*(t_n)-X^{(\ell)}_{h\t}(t_n)\big \|^2\big]
+\tau \sum_{n=1}^N {\mathbb E}   \big[\big \|X^*_{h\t}(t_n)- X^{(\ell)}_{h\t}(t_n) \big \|_{\dbH_0^1}^2\big] \nonumber \\
& \qq\qq
\leq \cC\Big(1-\frac 1 {\kappa}\Big)^{\ell} \big\| U^{(0)}_{h\t}(\cd)-U^*_{h\t}(\cd) \big\|_{\dbU_{\dbF}}^2 \, , \label{w1030e2c}
\end{empheq}
\end{subequations}
where $\cC$ is independent of $h,\,\t,\,\ell$.
\et

\begin{proof}
The estimates \rf{w1030e2a} and \rf{w1030e2b} are standard for the gradient descent method (see {\em e.g.}~\cite[Theorem 1.2.4]{Nesterov04}); more details can also be found in \cite[Section 5]{Prohl-Wang21}. In the following, we restrict to assertion \rf{w1030e2c}.

For all $n=0,1,\cds,N$, define $\bar e_{X}^{n,\ell}=X^*_{h\t}(t_n)-X^{(\ell)}_{h\t}(t_n)$. Subtracting \rf{w0115e4} 
from \eqref{w1003e7} (where $U_{h\t}(\cd)=U^*_{h\t}(\cd)$) leads to
\begin{equation*}
\setlength\abovedisplayskip{5pt}
\setlength\belowdisplayskip{5pt}
\bar e_{X}^{n+1,\ell} - \bar e_X^{n,\ell} 
=  \t  \D_h \bar e_X^{n+1,\ell} + \b \bar e_X^{n,\ell }\D_{n+1}W+
\t \big[U_{h\tau}^*(t_n)-U^{(\ell)}_{h\t}(t_n)\big] \, .
\end{equation*}
Following the procedure as that in the proof of Lemma \ref{w1002l2}, we can derive \rf{w1030e2c}.
%
\end{proof}

\subsection{Implementable modifications of the gradient descent method}\label{impl-1}
Algorithm \ref{alg1} is an iterative method to approximate the 
optimal control $U^*_{h\t}(\cd)$  of Problem {\bf (SLQ)$_{h\t}$} --- and its benefit lies in the decoupled computation of iterates which approximate the optimal state, control, and (adjoint) states. However, it is still not practical, since step 2 of
Algorithm \ref{alg1} involves {\em conditional expectations} $\me^{t_n}[\cdot]$.

In this section, we present two modifications of  Algorithm \ref{alg1} which are {implementable}: the first uses an {\em exact} 
representation of the conditional expectation in its step 2; as a consequence, the rates obtained in 
Section \ref{nume-fbspde1} remain valid here.
As will be detailed in 
Theorem \ref{imple-gdm} of Section \ref{impl-1a}, this strategy is possible for $\beta = 0$. The second strategy is 
proposed in Section \ref{stati-1}: it  uses an approximation of each $\me^{t_n}[\cdot]$ by statistical regression, 
and therefore introduces a further (statistical) error, whose estimation will be left open. In comparative simulations, it turns out that the second method 
requires far more computational resources, but is applicable for general $\beta \in {\mathbb R}$.

\subsubsection{Exact representation of conditional expectations}\label{impl-1a}

Let $\beta = 0$ in Problem {\bf (SLQ)$_{h\t}$}. The following theorem asserts that 
the updated control iterate  $U_{h\tau}^{(\ell+1)}(\cd)$ in step 3 of Algorithm \ref{alg1} may be found {\em without} the
computation of conditional expectations $\me^{t_n}[\cdot]$ --- as is stated in step 2, where the new adjoint $Y_{h \tau}^{(\ell)}(\cd)$ is determined. 
This idea has first been exploited in the context of stochastic boundary control \cite{C-M-P-W24}.
In fact,  $Y_{h \tau}^{(\ell)}(\cd)$ now satisfies the identity \rf{y1024e3} below as well, which does not involve 
conditional expectations.


\bt{imple-gdm}
Suppose that $\b=0$.
In the Algorithm \ref{alg1},
if $U^{(0)}_{h\t}(\cd)$ is of the form
\begin{equation*}
\setlength\abovedisplayskip{3pt}
\setlength\belowdisplayskip{3pt}
U^{(0)}_{h\t}(t_n)=\sum_{m=1}^n\prod_{i=1}^m(1+\D_iW)f^{(0)}_{n,m}+\sum_{m=1}^n\D_mW\wt f^{(0)}_{n,m}+g^{(0)}_n\q n=0,1,\cds,N-1\,,
\end{equation*}
where $f^{(0)}_{n,m},\wt f^{(0)}_{n,m}, g^{(0)}_n$ are deterministic for $m=1,2,\cds,n$.
Then for any $\ell\in\dbN_0$, it holds that
\begin{equation}\label{y1024e3}
\setlength\abovedisplayskip{3pt}
\setlength\belowdisplayskip{3pt}
Y^{(\ell)}_{h\t}(t_n)=\sum_{m=1}^n\prod_{i=1}^m(1+\D_iW)F^{(\ell)}_{n,m}+\sum_{m=1}^n\D_mW\wt F^{(\ell)}_{n,m}+G^{(\ell)}_n\,, 
\end{equation}
where $F^{(\ell)}_{n,m}\,,\wt F^{(\ell)}_{n,m}\,\,,G^{(\ell)}_n$ are deterministic and of the forms
\bel{y1024e4}
\setlength\abovedisplayskip{3pt}
\setlength\belowdisplayskip{3pt}
\lt\{
\bal
 F^{(\ell)}_{n,m}&=-\t^2 \sum_{j=n+1}^{N-1}   \sum_{l=m}^{j-1} A_0^{2j-l-n}f^{(\ell)}_{l,m} 
-\a\t \sum_{l=m}^{N-1}A_0^{2N-l-n} f^{(\ell)}_{l,m} \\
&\qq\qq\qq\qq\qq\qq\qq\qq\qq m=1,2,\cds,n-1\,,  \\
 F^{(\ell)}_{n,n}&=
-\t^2 \sum_{j=n+1}^{N-1}  \Big[ \sum_{l=n}^{j-1} A_0^{2j-l-n}f^{(\ell)}_{l,n}+ \sum_{l=n+1}^{j-1} A_0^{2j-l-n}\sum_{m=n+1}^l f^{(\ell)}_{l,m}\Big]  \\
&\q-\a\t \Big[\sum_{l=n}^{N-1}A_0^{2N-l-n} f^{(\ell)}_{l,n}+\sum_{l=n+1}^{N-1}A_0^{2N-l-n} \sum_{m=n+1}^l f^{(\ell)}_{l,m}\Big]\,,  \\
 \wt F^{(\ell)}_{n,m}&=-\t^2\sum_{j=n+1}^{N-1}\sum_{l=m}^{j-1} A_0^{2j-l-n} \wt f^{(\ell)}_{l,m}
- \a\t \sum_{l=m}^{N-1}A_0^{2N-l-n} \wt f^{(\ell)}_{l,m}  \\
&\q- \t \sum_{j=n+1}^{N-1}   A_0^{2j-m-n+1} \Pi_h \si(t_{m-1})
- \a A_0^{2N-m-n+1} \Pi_h \si(t_{m-1}) \\
&\qq\qq\qq\qq\qq\qq\qq\qq\qq  m=1,2,\cds,n \,,  \\
 G^{(\ell)}_n&= -\t \sum_{j=n+1}^{N-1} A_0^{2j-n}\Pi_hx -\a  A_0^{2N-n}\Pi_hx
- \t^2 \sum_{j=n+1}^{N-1}   \sum_{l=0}^{j-1}A_0^{2j-l-n} g^{(\ell)}_l   \\
&\q-\a\t \sum_{l=0}^{N-1}A_0^{2N-l-n} g^{(\ell)}_l   \,.
\eal
\rt.
\ee
Moreover, $f^{(\ell+1)}_{n,m},\wt f^{(\ell+1)}_{n,m}, g^{(\ell+1)}_n$ in $U^{(\ell+1)}_{h\t}(t_n)$ for $\ell\in \dbN_0$ can be derived recursively by
\begin{equation*}
\setlength\abovedisplayskip{3pt}
\setlength\belowdisplayskip{3pt}
\lt\{
\bal
f^{(\ell+1)}_{n,m}&=\Big(1-\frac 1 \kappa\Big)f^{(\ell)}_{n,m}+\frac 1 \kappa F^{(\ell)}_{n,m}\,,\\
\wt f^{(\ell+1)}_{n,m}&=\Big(1-\frac 1 \kappa\Big)\wt f^{(\ell)}_{n,m}+\frac 1 \kappa \wt F^{(\ell)}_{n,m}\qq m=1,2,\cds,n\,,\\
g^{(\ell+1)}_{n}&=\Big(1-\frac 1 \kappa\Big)g^{(\ell)}_{n}+\frac 1 \kappa G^{(\ell)}_{n}\,.
\eal
\rt.
\end{equation*}
\et

\begin{proof}
For simplicity, in the proof we rewrite $X^{(\ell)}_{h\t}(t_n), Y^{(\ell)}_{h\t}(t_n), U^{(\ell)}_{h\t}(t_n), \Pi_h\si(t_n)$ as $X^{(\ell)}_n$, $Y^{(\ell)}_n$, $U^{(\ell)}_n$, $\si_n$. 
We argue by induction to derive the assertion. Suppose that 
\bel{w1119e5}
\setlength\abovedisplayskip{3pt}
\setlength\belowdisplayskip{3pt}
U^{(\ell)}_n=\sum_{m=1}^n\prod_{i=1}^m(1+\D_iW)f^{(\ell)}_{n,m}+\sum_{m=1}^n\D_mW\wt f^{(\ell)}_{n,m}+g^{(\ell)}_n \q n=0,1,\cds,N-1\,,
\ee
with deterministic $f^{(\ell+1)}_{n,m},\wt f^{(\ell+1)}_{n,m}, g^{(\ell+1)}_n$.

Based on \rf{w0115e4}, we can deduce that
\bel{y1024e1}
\setlength\abovedisplayskip{3pt}
\setlength\belowdisplayskip{3pt}
X_{n+1}^{(\ell)}=A_0^{n+1}\Pi_hx
+\t \sum_{j=0}^{n}A_0^{n+1-j}  U_j^{(\ell)}
+\sum_{j=0}^{n}A_0^{n+1-j}\si_j\D_{j+1}W\,.
\ee
Then inserting \rf{y1024e1} into $Y^{(\ell)}_n$ of step 2 in Algorithm \ref{alg1}, we have
\begin{eqnarray}
&Y_n^{(\ell)}
&=-\t \me^{t_n}\Big[ \sum_{j=n+1}^{N-1} A_0^{j-n}  X^{(\ell)}_j  \Big] 
-\a \me^{t_n}\big[A_0^{N-n}X^{(\ell)}_N\big] \nonumber \\
&&=-\t \sum_{j=n+1}^{N-1} A_0^{2j-n}\Pi_hx -\a  A_0^{2N-n}\Pi_hx \nonumber \\
&&\q-\t \sum_{j=n+1}^{N-1}  \t  \sum_{l=0}^{j-1} A_0^{2j-l-n}\me^{t_n}\big[ U_l^{(\ell)}  \big] 
-\a \t  \sum_{l=0}^{N-1} A_0^{2N-l-n}\me^{t_n}\big[ U_l^{(\ell)}  \big] \nonumber  \\
&&\q- \sum_{j=n+1}^{N-1}  \t  \sum_{l=0}^{j-1} A_0^{2j-l-n}\me^{t_n}\big[ \si_l \D_{l+1}W  \big] 
-\a  \sum_{l=0}^{N-1} A_0^{2N-l-n}\me^{t_n}\big[ \si_l \D_{l+1}W \big] \nonumber \\
&&=: -\t \sum_{j=n+1}^{N-1} A_0^{2j-n}\Pi_hx -\a  A_0^{2N-n}\Pi_hx -\sum_{i=1}^4I_i  \,. \label{w1119e4}
\end{eqnarray}
Now we compute $I_i$ for $1\leq i\leq 4$. 
By the assumption for $U_l^{(\ell)} $, if $l\leq n$ we have
\bel{w1119e1}
\setlength\abovedisplayskip{5pt}
\setlength\belowdisplayskip{5pt}
\me^{t_n}\big[ U_l^{(\ell)}  \big]
=U_l^{(\ell)} \,,
\ee
and if $l> n$, then by \rf{w1119e5}
\bel{w1119e2}
\setlength\abovedisplayskip{3pt}
\setlength\belowdisplayskip{3pt}
\bal
\me^{t_n}\big[ U_l^{(\ell)}  \big]
&= \sum_{m=1}^{n-1}\prod_{i=1}^m(1+\D_iW)f^{(\ell)}_{l,m}
+\prod_{i=1}^n(1+\D_iW)\Big[f^{(\ell)}_{l,n}+\sum_{m=n+1}^l f^{(\ell)}_{l,m}\Big]
+ \sum_{m=1}^{n}\D_mW\wt f^{(\ell)}_{l,m}+g^{(\ell)}_{l}\,.
\eal
\ee
Since $\si(\cd)$ is deterministic, we have
\begin{equation}\label{w1119e3}
\setlength\abovedisplayskip{3pt}
\setlength\belowdisplayskip{3pt}
\me^{t_n}\big[ \si_l \D_{l+1}W \big]
=
\lt\{\!\!\!
\begin{array}{ll}
\ds  \si_l \D_{l+1}W \qq\q& l\leq n-1\,,\\
\ns\ds 0 \qq& l\geq n\,.
\end{array}
\rt.
\end{equation}

For $I_1$, by \rf{w1119e1}--\rf{w1119e2}, we then arrive at
\begin{equation*}
\setlength\abovedisplayskip{3pt}
\setlength\belowdisplayskip{3pt}
\bal
I_1
&=\t^2 \sum_{j=n+1}^{N-1}   \sum_{l=0}^{n}A_0^{2j-l-n}\Big[\sum_{m=1}^l \prod_{i=1}^m(1+\D_iW)f^{(\ell)}_{l,m}+\sum_{m=1}^l \D_mW\wt f^{(\ell)}_{l,m}+g^{(\ell)}_l\Big]\\
&\q+\t^2 \sum_{j=n+1}^{N-1}   \sum_{l=n+1}^{j-1}A_0^{2j-l-n}\bigg\{\sum_{m=1}^{n-1}\prod_{i=1}^m(1+\D_iW)f^{(\ell)}_{l,m}
\\
&\qq\q +\prod_{i=1}^n(1+\D_iW)\Big[f^{(\ell)}_{l,n}+\sum_{m=n+1}^l f^{(\ell)}_{l,m}\Big]
+ \sum_{m=1}^{n}\D_mW\wt f^{(\ell)}_{l,m}+g^{(\ell)}_{l} \bigg\}\,.
\eal
\end{equation*}
By changing the order of the summation, we continue to get
\begin{equation*}
\setlength\abovedisplayskip{3pt}
\setlength\belowdisplayskip{3pt}
\bal
&=\sum_{m=1}^{n-1} \prod_{i=1}^m(1+\D_iW) \t^2 \sum_{j=n+1}^{N-1}   \sum_{l=m}^{j-1}A_0^{2j-l-n} f^{(\ell)}_{l,m}
\\
&\q+\prod_{i=1}^n(1+\D_iW)\t^2 \sum_{j=n+1}^{N-1}  \Big[ \sum_{l=n}^{j-1}A_0^{2j-l-n} f^{(\ell)}_{l,n}+ \sum_{l=n+1}^{j-1}A_0^{2j-l-n}\sum_{m=n+1}^l f^{(\ell)}_{l,m}\Big]\\
&\q+ \sum_{m=1}^n \D_mW \t^2 \sum_{j=n+1}^{N-1}   \sum_{l=m}^{j-1}A_0^{2j-l-n} \wt f^{(\ell)}_{l,m}
+ \t^2 \sum_{j=n+1}^{N-1}   \sum_{l=0}^{j-1}A_0^{2j-l-n} g^{(\ell)}_l \,.
\eal
\end{equation*}
Similarly, for $I_2$, we have
\begin{equation*}
\setlength\abovedisplayskip{3pt}
\setlength\belowdisplayskip{3pt}
\bal
I_2
&=\sum_{m=1}^{n-1} \prod_{i=1}^m(1+\D_iW) \a\t \sum_{l=m}^{N-1}A_0^{2N-l-n} f^{(\ell)}_{l,m}\\
&\q+\prod_{i=1}^n(1+\D_iW)\a\t \Big[\sum_{l=n}^{N-1}A_0^{2N-l-n} f^{(\ell)}_{l,n}+\sum_{l=n+1}^{N-1}A_0^{2N-l-n} \sum_{m=n+1}^l f^{(\ell)}_{l,m}\Big]\\
&\q+ \sum_{m=1}^n \D_mW \a\t \sum_{l=m}^{N-1}A_0^{2N-l-n} \wt f^{(\ell)}_{l,m}
+ \a\t \sum_{l=0}^{N-1}A_0^{2N-l-n} g^{(\ell)}_l \,.
\eal
\end{equation*}

For $I_3$ and $I_4$, based on \rf{w1119e3}, we can conclude that
\begin{equation*}
\setlength\abovedisplayskip{3pt}
\setlength\belowdisplayskip{3pt}
\bal
I_3=\sum_{j=n+1}^{N-1}  \t  \sum_{l=0}^{n-1} A_0^{2j-l-n} \si_l \D_{l+1}W
= \sum_{m=1}^n \D_mW \t \sum_{j=n+1}^{N-1}   A_0^{2j-m-n+1} \si_{m-1}\,,
\eal
\end{equation*}
and
\begin{equation*}
\setlength\abovedisplayskip{3pt}
\setlength\belowdisplayskip{3pt}
\bal
I_4= \sum_{l=0}^{n-1} \a A_0^{2N-l-n} \si_l \D_{l+1}W
= \sum_{m=1}^n \D_mW \a A_0^{2N-m-n+1} \si_{m-1}\,.
\eal
\end{equation*}

Finally,  a combination of \rf{w1119e4}, the above computation of $I_1$ through $I_4$ and \rf{y1024e4} 
then yields the assertion
\rf{y1024e3}. That completes the proof.
\end{proof}

\subsubsection{A regression estimator for  conditional expectations}\label{stati-1}

The strategy in this section is to simulate
Problem {\bf (SLQ)$_{h\t}$} for general $\beta \in {\mathbb R} $ via Algorithm \ref{alg1}. 
For fixed $\ell \in \dbN_0$ in Algorithm \ref{alg1}, step 2 has the form
$Y_{h\tau}^{(\ell)}(t_n) = {\mathbb E}^{t_n}\big[\Theta^{(\ell)}_{h\tau}(t_n)\big]$, which 
involves  a conditional 
expectation. 
We approximate it by a {\em `data dependent'} regression estimator (see {\em e.g.}~\cite[Chapter 13]{Gyorfi-Kohler-Krzyzak-Walk02}) which is able to cope with the  high dimensionality
 $1 \ll   {\mathfrak d} \deq \dim(\dbV_h)$ of the involved state space.

 As will be outlined below, related data dependent partitionings $\sP^{n, (\ell)}_{\tt M}$ of ${\mathbb R}^{{\mathfrak d}}$ will automatically be generated, with finer cells at places where the regression function (see below) changes fastly, and coarser ones elsewhere. This nonlinear approximation approach significantly reduces
complexities of simulations, in particular when ${\mathfrak d}$ is large --- and  where well-known {\em uniform} partitioning strategies 
for low-dimensional BSDEs {loose their competitivity}.

The regression estimator from \cite{Dunst-Prohl16} in this section may
 be considered as a modification of the least squares Monte Carlo method (see {\em e.g.}~\cite{Gobet-Turkedjiev16,Gyorfi-Kohler-Krzyzak-Walk02}) to make accessible
 the simulation of high-dimensional BSDEs as required in step 2 of Algorithm \ref{alg1}. Choose ${\tt M}, {\tt R} \in {\mathbb N}$, with ${\tt R} \ll {\tt M}$.
 We use the following tools to simulate ${\mathbb E}^{t_n}\big[ \Theta^{(\ell)}_{h\tau}(t_n)\big]$: 
\begin{enumerate} [(a)]
\item  {\bf (algebraic reformulation)} Make use of the finite element basis functions $\{ \phi_{h,j}\}_{j=1}^{\mathfrak{d}}$ of ${\mathbb V}_h$ to represent appearing ${\mathbb V}_h$-valued ${Y}^{\cdot,(\ell)}_{h\tau}$ resp.~${X}^{\cdot,(\ell)}_{h\tau}$
 by ${\mathbb R}^{\mathfrak d}$-valued $\mathfrak{Y}^{\cdot,(\ell)}$ resp.~$\mathfrak{X}^{\cdot,(\ell)}$, with related entries $\mathfrak{y}^{\cdot,(\ell)}_j$ resp.~$\mathfrak{x}^{\cdot,(\ell)}_j$ via 
\begin{equation*}
\setlength\abovedisplayskip{3pt}
\setlength\belowdisplayskip{3pt}
\bal
 & \mathfrak{X}^{\cdot,(\ell)}  = \sum_{j=1}^{{\mathfrak d}} \mathfrak{x}^{\cdot,(\ell)}_j
\phi_{h,j}  \qquad 
 \mbox{and}  \qquad
&\mathfrak{Y}^{\cdot,(\ell)}  = \sum_{j=1}^{{\mathfrak d}} \mathfrak{y}^{\cdot,(\ell)}_j
\phi_{h,j} \,.
\eal
\end{equation*}

\item {\bf ({\tt M}-sample ${\tt D}^{n,(\ell)}_{{\tt M},j}$)} Let  $1 \leq j\leq {\mathfrak d}$; the regression function  ${\mathcal Y}^{n, (\ell)}_{j}: {\mathbb R}^{\mathfrak{d}} \rightarrow {\mathbb R}$ satisfies
\begin{equation*}
\setlength\abovedisplayskip{3pt}
\setlength\belowdisplayskip{3pt}
\bal
\mathfrak{y}^{n,(\ell)}_j = {\mathcal Y}^{n, (\ell)}_{j}  \big(\mathfrak{X}^{n,(\ell)}\big)\, ,
\eal
\end{equation*}
which we approximate by the estimator
$\widehat{{\mathcal Y}}^{n, (\ell)}_{j}$ in (iv) below. Therefore, let
$${\tt D}^{n,(\ell)}_{{\tt M},j} \deq \bigl\{\big(\mathfrak{X}^{n,(\ell)}_{{\tt m}}, \mathfrak{\theta}^{n,(\ell)}_{{\tt m},j}\big)\bigr\}_{{\tt m}=1}^{ {\tt M}} \qquad (1 \leq n \leq N) $$
be an ${\tt M}$-sample, with $\bigl(\mathfrak{X}^{n,(\ell)}_{{\tt m}}, \mathfrak{\theta}^{n,(\ell)}_{{\tt m},j} \bigr) \sim \bigl(\mathfrak{X}^{n,(\ell)}, \mathfrak{\theta}^{n,(\ell)}_j \bigr)$, and $\mathfrak{\theta}^{n,(\ell)}_j$ the $j$-th entry of the ${\mathbb R}^{\mathfrak{d}}$-valued representation of $\Theta^{(\ell)}_{h\tau}(t_n)$ as in (i). 
\item  {\bf (data dependent partition $\sP^{n, (\ell)}_{\tt M}$)} 
Use  (the first entries of) 
${\tt D}^{n,(\ell)}_{{\tt M},j}$ to generate a partition (see also \cite[Chapter 13]{Gyorfi-Kohler-Krzyzak-Walk02})
$$\sP^{n, (\ell)}_{\tt M} = \bigl\{{\mathscr A}^{n, (\ell)}_{\tt r} \bigr\}_{{\tt r}=1}^{{\tt R}} 
$$
of  ${\mathbb R}^{\mathfrak d}$ into statistically equivalent cells ${\mathscr A}^{n, (\ell)}_{\tt r}$ via
the BTC strategy in \cite[p.~A2741]{Dunst-Prohl16}.
\item {\bf (estimator $\widehat{{\mathcal Y}}^{n, (\ell)}_{j}$)} 
%
Use ${\tt D}^{n,(\ell)}_{{\tt M},j}$ to compute $\widehat{\mathcal Y}^{n, (\ell)}_{j}: {\mathbb R}^{\mathfrak{d}} \rightarrow {\mathbb R}$ via
$$\widehat{{\mathcal Y}}^{n, (\ell)}_{j}(\cdot) =  \sum_{{\tt r} =1}^{{\tt R}}\frac{\sum_{{\tt m}=1}^{\tt M}\mathfrak{\theta}^{n,(\ell)}_{{\tt m},j} \cdot \mathds{1}_{\{ \mathfrak{X}^{n,(\ell)}_{{\tt m}} \in {\mathscr A}^{n, (\ell)}_{\tt r}\}}}{\sum_{{\tt m}=1}^{\tt M} \mathds{1}_{\{ \mathfrak{X}^{n,(\ell)}_{{\tt m}} \in {\mathscr A}^{n, (\ell)}_{\tt r}\}}} \mathds{1}_{{\mathscr A}_{\tt r}}(\cdot)\, .
$$
\end{enumerate}
We refer to \cite{Dunst-Prohl16} for more algorithmic details.  

\section{Discretization based on the closed-loop approach}\label{ch-closed}

As has been reviewed in Section \ref{se-openclosed}, the optimal control $U^*(\cd)$ of Problem {\bf (SLQ)} admits 
the feedback law \rf{feedback-0} with the help of $\cP(\cd)$ and $\eta(\cd)$ (see \rf{Riccati} and \rf{varphi}).
In comparison to 
Section \ref{ch-open},
both terminal problems on \rf{Riccati} and \rf{varphi} are deterministic ---  a relevant property that
we may now benefit from in a numerical discretization of Problem  {\bf (SLQ)}:
a canonical strategy therefore is now to approximate 
$U^*(\cd)$ by first discretizing \rf{Riccati} and \rf{varphi}
with the help of tools from deterministic numerics; and to finally insert it into SPDE \rf{spde} thanks to \rf{feedback-0}, and to
then solve it numerically to get an approximation of $X^*(\cd)$.

As a consequence, to properly discretize the stochastic Riccati equation
(\ref{Riccati}) is the key step within this program, and to derive related optimal error estimates is the main aim in this section. In fact, it is an open problem even in the context of deterministic control ({\em i.e.}, $\beta = 0$ and $\sigma(\cd) = 0$ in Problem {\bf (SLQ)}) whether the implicit Euler for the {\em deterministic Riccati equation} performs with optimal rates: the reason for it is that an applied
discretization can spoil the specific role that the continuous Riccati equation plays within the (limiting) control problem. 
To preserve this role for the constructed discretization is
of primary relevancy in this section.
By adopting an idea from \cite{Wang23}, we develop a proper discretization scheme 
of \rf{Riccati} --- which, in effect, differs from a standard time numerical integrator for (\ref{Riccati}); see \rf{dis-Riccati1-h-t} and \rf{dif-Riccati}.
In the second step, we apply a perturbation argument, where we leave the ground of optimization
to arrive at a modified scheme whose complexity is comparable to standard numerical integrators for \rf{Riccati}. 
The benefit of this construction strategy  for the related scheme \rf{feedback-5}--\rf{dis-state}
is that we
can show optimal rates of convergence for iterates to approximate $\big(X^*(\cd), U^*(\cd)\big)$.

Based on LQ theory, the error analysis for different spatio-temporal discretization schemes  of the Riccati equation is given in Section \ref{Riccati-dis};
see Theorems \ref{Riccati-rate-h} (for spatial discretization), \ref{Riccati-rate}
(for temporal discretization in $\|\cd\|_{\cL(\dbL^2)}$) under some assumption on $\b$ (see assumption {\bf (H)} in Section \ref{dis-lq}), and Theorems \ref{Riccati-rate2}
(for temporal discretization in $\|\cd\|_{\cL(\dbH_0^1;\dbL^2)}$)
{\em without} any smallness assumption  on $\b$.

In Section \ref{SLQ-rate}, a mixed strategy to numerically solve Problem {\bf (SLQ)} is proposed, which uses 
the `open-loop approach' for spatial discretization, and then the  `closed-loop approach' for temporal discretization. 
This mixed approach improves the results in \cite{Prohl-Wang21,Prohl-Wang22, Lv-Wang-Wang-Zhang22} which adopts the `open-loop approach', and 
in \cite{Prohl-Wang24} where methods were only taken from the `closed-loop approach' at the stage of the numerical analysis. 

This part starts with Section \ref{tools-1}, where we introduce a family of LQ problems (Problem {\bf (LQ)$^t_{\tt aux}$}) related to Problem {\bf (SLQ)}, which will play a prominent role in the numerical analysis. 
In fact, the {\em stochastic} Riccati equation for Problem 
{\bf (SLQ)} coincides with the {\em deterministic} Riccati equation for Problem {\bf (LQ)$^t_{\tt aux}$}, where
the equation is a (deterministic) PDE instead, containing a shifted operator on a scale $\beta$ if compared with the original diffusion operator in (\ref{spde}).
\subsection{Discretization of the  Riccati equation}\label{tools-1}

The {\em stochastic} Riccati equation \rf{Riccati}  is related to Problem {\bf (SLQ)}; but it can also be considered as a {\em deterministic} Riccati equation associated with Problem {\bf (LQ)$^\cd_{\tt aux}$} below. We will detail this observation in this section, which will play a relevant role in the construction of a proper discretization scheme for \rf{Riccati} --- if interpreted in the latter sense. 
In this section, we propose two different temporal discretization  schemes  for the Riccati equation \rf{Riccati} 
which are both consistent --- {\em i.e.,} its solutions take values in $\dbS_+(\dbL^2|_{\dbV_h})$ ---
and two further schemes which are not; their construction again follows the strategy to `first discretize, then optimize'. 
To state the main ideas, in this section we only detail the first consistent discretization; the others will be stated in part {\bf b.} of 
Section \ref{Riccati-dis-2}, and in Section \ref{Riccati-dis-3}.

\subsubsection{An auxiliary LQ problem for the Riccati equation}\label{deter-lq}

We now define a family of {\em deterministic} LQ problems --- where each of them is later referred to as Problem {\bf (LQ)$_{\tt aux}^t$} for any $t \in [0,T)$ ---, {whose corresponding Riccati equation is \rf{Riccati}} on $[t,T]$.
To do that, we define a family of cost functionals first, which for any $t\in [0,T)$ and  $z\in \dbL^2$ reads as 
\begin{equation} \label{cost-t}
\setlength\abovedisplayskip{3pt}
\setlength\belowdisplayskip{3pt}
\cG\big(t,z; u(\cd) \big)=\frac 1 2 \int_t^T\big[ \| x(s) \|^2+ \|u(s) \|^2 \big] \rd s +\frac \a 2 \|x(T)\|^2\,,
\end{equation}
subject to the (controlled forward) {\em deterministic} PDE
\bel{pde}
\setlength\abovedisplayskip{3pt}
\setlength\belowdisplayskip{3pt}
\lt\{\!\!\!
\begin{array}{ll} 
\ds x'(s)=\D x(s)+\frac {\b^2} 2 x(s)+u(s)  \qq  s \in (t,T]\,,\\
\ns\ds x(t)=z \,.
\end{array}
\rt.
\ee
Note that the involved $\big(x(\cd),u(\cd)\big)$ in (\ref{cost-t})--(\ref{pde}) are deterministic functions --- as opposed to the related 
stochastic processes in (\ref{intro-1a})--(\ref{spde}). 
\ss

\no {\bf Problem (LQ)$^t_{\tt aux}$.} Fix an initial pair $(t,z)\in [0,T)\times \dbL^2$. Find $u^*(\cd)\in L^2(t,T;\dbL^2)$ such that
\bel{LQ-cost}
\setlength\abovedisplayskip{3pt}
\setlength\belowdisplayskip{3pt}
\cG\big(t,z; u^*(\cd) \big)=\inf_{u(\cd)\in L^2(t,T;\dbL^2)}\cG\big(t,z; u(\cd)\big)=:V(t,z)\,.
\ee

Any $u^*(\cd) \in L^2(t,T;\dbL^2)$ that satisfies \rf{LQ-cost} is called an {\em optimal control} for the given initial pair $(t,z)$, 
with $x^*(\cd)$ the corresponding {\em optimal state}; the tuple $\big(x^*(\cd),u^*(\cd)\big)$ will be referred to as an {\em optimal pair} for $(t,z)$, and $V(\cd,\cd)$ is 
called the {\em value function} of Problem {\bf (LQ)}$_{\tt aux}^t$.

\bs

It is well-known {that Problem {\bf (LQ)}$_{\tt aux}^t$} admits a unique optimal pair $\big(x^*(\cd),u^*(\cd)\big)$, and the optimal control $u^*(\cd)$ admits the following state feedback representation:
\bel{feedback-1}
\setlength\abovedisplayskip{3pt}
\setlength\belowdisplayskip{3pt}
u^*(s) =-\cP(s) x^*(s) \qquad  s \in [t,T]\,,
\ee
where ${\mathcal P}(\cd)$ comes from the terminal problem \rf{Riccati}. Furthermore, the value function $V(\cd,\cd)$ is explicitly presented by
\bel{val-1}
\setlength\abovedisplayskip{3pt}
\setlength\belowdisplayskip{3pt}
V(t,z)=\frac{1}{2}\big(\cP(t)z,z\big)_{\dbL^2}\,.
\ee
We refer the reader to \cite[Chapter 17]{Zabczyk20} for more details.
Based on \rf{LQ-cost} and \rf{val-1}, for any admissible control $u(\cd)\in L^2(t,T;\dbL^2)$ we may  conclude that
\begin{equation*}
\setlength\abovedisplayskip{3pt}
\setlength\belowdisplayskip{3pt}
\big(\cP(t)z,z \big)_{\dbL^2}\leq  2\cG\big(t,z; u(\cd)\big) \qquad \forall \, t \in [0,T]\,,
\end{equation*}
which will be used frequently 
in Section \ref{Riccati-dis}.

\subsubsection{Discretization of Problem {\bf (LQ)}$_{\tt aux}^t$
and a related difference Riccati equation}\label{dis-lq}
%

We derive the difference Riccati equation (\ref{dis-Riccati1-h-t}) with the help of two discretization schemes  
of Problem  {\bf (LQ)}$_{\tt aux}^t$. 
For abbreviation, 
 we
denote $\cA=\D+\frac {\b^2}{2}\mathds{1}$\index{${\mathcal A}$} and shall make use of the following assumption: \\
\no{\bf Assumption (H):} There exists a positive constant $\cC_0$ such that
\bel{assumption2}
\setlength\abovedisplayskip{3pt}
\setlength\belowdisplayskip{3pt}
\|\nb\f\|^2-\frac{\b^2}{2}\|\f\|^2\geq \cC_0\|\f\|^2 \qq\forall\,\f\in \dbH_0^1\,.
\ee
This assumption constrains the strength of admissible noise, but is suitable for the first numerical analysis in this section
and Section \ref{Riccati-dis-2}; it will, however, be weakened in Section \ref{Riccati-dis-3}.
Under the coercivity assumption {\bf (H)}, $\cA $   generates an analytic semigroup $\h E(\cd)$ with $\h E(t)=e^{t\cA }$\index{$\hat E(t)$} 
for $t\geq 0$. Also, we can define  $\cA_h=\Delta_h+\frac{\b^2}{2}\mathds{1}_h:\dbV_h \to \dbV_h$\index{${\mathcal A}_h$}, and consider its corresponding semigroup
$\h E_h(\cd)$\index{$\hat E_h(t)$},
as well as $\cA_0=(\mathds{1}_h-\t \cA_h)^{-1}$\index{${\mathcal A}_0$} and the related $\h E_{h,\t}(\cd)$. \index{$\hat E_{h,\tau}(t)$} 
Similarly, we can define 
$
\h G_h(\cd)\,, \h G_\t(\cd)
\,;\index{$\hat G_h(t)$}\index{$\hat G_\tau(t)$}
$
for notations see Section \ref{not1}. Besides, Lemmata \ref{w228l1}--\ref{w207l1} still hold for $\cA$. 

\ss
{\bf a.~Spatial discretization of {\bf Problem (LQ)$^{t}_{\tt aux}$}.} By virtue of $\cA$, \rf{Riccati} can be rewritten as 
\bel{Riccati3}
\setlength\abovedisplayskip{3pt}
\setlength\belowdisplayskip{3pt}
\lt\{\!\!\!
\begin{array}{ll} 
\ds\cP'(t)+\cA \cP(t)+\cP(t)\cA+\mathds{1}-\cP^2(t)=0\qq t\in [0,T]\,,\\
\ns\ds\cP(T)=\a\mathds{1}\,,
\end{array}
\rt.
\ee
and its spatial discretization by the finite element method is
\bel{dis-Riccati1-h}
\setlength\abovedisplayskip{3pt}
\setlength\belowdisplayskip{3pt}
\lt\{\!\!\!
\begin{array}{ll} 
\ds\cP_h'(t)+\cA_h\cP_h(t)+\cP_h(t)\cA_h+\mathds{1}_h-\cP_h^2(t)=0\qq t\in [0,T]\,,\\
\ns\ds\cP_h(T)=\a\mathds{1}_h\,.
\end{array}
\rt.
\ee
By \cite[Theorem 7.2]{Yong-Zhou99}, \rf{dis-Riccati1-h} admits a unique solution $\cP_h(\cd)\in C\big([0,T];\dbS(\dbL^2|_{\dbV_h})\big)$.
To measure the difference $\cP(\cd)-\cP_h(\cd)\Pi_h$, we introduce a family of auxiliary LQ problems.\\
\no{\bf Problem (LQ)$^{t;h}_{\tt aux}$.} For any  given $t\in[0,T)$ and $z\in \dbL^2$, search for 
$u^*_h(\cd)\in L^2(t,T;\dbL^2)$ such that
\bel{LQ-cost-h}
\setlength\abovedisplayskip{3pt}
\setlength\belowdisplayskip{3pt}
\cG_{h}\big(t,z; u^*_h(\cd)\big)=\inf_{u_h(\cd) \in L^2(t,T;\dbL^2)}\cG_{h}\big(t,z; u_h(\cd)\big)=:V_{h}(t,z)\,,
\ee
where the cost functional is
\bel{cost-t-h}
\setlength\abovedisplayskip{3pt}
\setlength\belowdisplayskip{3pt}
\cG_{h}\big(t,z; u_h(\cd)\big)=\frac 1 2 \int_t^T\big[ \| x_h(s) \|^2+ \|u_h(s) \|^2 \big] \rd s +\frac \a 2 \|x_h(T)\|^2
\ee
and the state variable $x_h(\cd)\in C([t,T];\dbV_h)$ satisfies
\bel{pde-h}
\setlength\abovedisplayskip{3pt}
\setlength\belowdisplayskip{3pt}
\lt\{\!\!\!
\begin{array}{ll} 
\ds x_h'(s)=\cA_h x_h(s)+\Pi_h u_h(s)  \qq  s \in (t,T]\,,\\
\ns\ds x_h(t)=\Pi_hz \,.
\end{array}
\rt.
\ee

Similar to Problem {\bf (LQ)}$_{\tt aux}^t$, Problem {\bf (LQ)}$_{\tt aux}^{t;h}$ has a unique optimal control $u^*_h(\cd)\in L^2(0,T;\dbL^2)$,
satisfying  the following state feedback representation:
\bel{feedback-h-1}
\setlength\abovedisplayskip{3pt}
\setlength\belowdisplayskip{3pt}
u^*_h(s) =-\cP_h(s) x^*_h(s) \qquad  s \in [t,T]\,;
\ee
 and the value function $V_h(\cd,\cd)$ is explicitly presented by
\bel{val-h-1}
\setlength\abovedisplayskip{3pt}
\setlength\belowdisplayskip{3pt}
V_h(t,z)=\frac{1}{2}\big(\cP_h(t)\Pi_h z,z\big)_{\dbL^2}\,.
\ee
Based on \rf{LQ-cost-h} and \rf{val-h-1}, for any admissible control $u_h(\cd)\in L^2(t,T;\dbL^2)$ we find  that
\begin{equation*}
\setlength\abovedisplayskip{3pt}
\setlength\belowdisplayskip{3pt}
\big(\cP_h(t)\Pi_h z,z \big)_{\dbL^2}\leq  2\cG_h\big(t,z; u_h(\cd)\big) \qquad \forall \, t \in [0,T]\,,
\end{equation*}
which will be used 
in Section \ref{Riccati-dis}.

Moreover, we can deduce that the optimal control $u^*_h(\cd)\in L^2(t, T;\dbV_h)$. 
In fact, since 
$\Pi_h u^*_h(\cd)$ is an admissible control, we conclude by \rf{w1024e1} that
\bel{w1026e5}
\setlength\abovedisplayskip{3pt}
\setlength\belowdisplayskip{3pt}
0\!\leq\! \cG_h\big(t,z; \Pi_h u_h^*(\cd)\big)\!-\!\cG_h\big(t,z; u^*_h(\cd)\big)\!=\!\frac1 2 \int_t^T\big[ \|\Pi_hu^*_h(s)\|^2\!-\!\|u^*_h(s)\|^2\big]\rd s\!\leq\! 0\,,
\ee
which yields $u^*_h(\cd)=\Pi_hu^*_h(\cd)\in L^2(t,T;\dbV_h)$.

\ss

{\bf b.~Spatio-temporal discretization of Problem  {\bf(LQ)$^{t}_{\tt aux}$}.} 
For a given $t\in[0,T)$, we now discretize Problem {\bf (LQ)}$_{\tt aux}^{t;h}$ in time. 
Firstly, we need two spaces for state and control variables: for any $l=0,1,\cds,N-1$,
\bel{w1015e4}
\setlength\abovedisplayskip{3pt}
\setlength\belowdisplayskip{3pt}
\bal
 \dbX(t_l,T)&\deq  \Big\{x_\cd \equiv \big\{x_n\big\}_{n=l}^{N-1} \, \Big| \, x_{n}\in\dbV_h \,\,\,   \forall\, n=l,\cds, \, N-1\,, 
 \mbox{and } \t \sum_{n=l}^{N-1}\|x_n\|^2 <\infty\Big\}\, ,\index{$\dbX(t_l,T)$}\\
\dbU(t_l,T) &\deq \Big\{u_\cd \equiv \big\{u_n\big\}_{n=l}^{N-1} \, \Big| \, u_{n}\in\dbL^2 \,\,\,   \forall\, n=l\cds, \, N-1\,,   \mbox{and } \t \sum_{n=l}^{N-1}\|u_n\|^2 <\infty\Big\}\, .\index{$\dbU(t_l,T)$}
\eal
\ee
%
We endow the discrete state and control spaces with the following norms, 
\begin{equation*}
\setlength\abovedisplayskip{3pt}
\setlength\belowdisplayskip{3pt}
\|x_\cd\|_{\dbX(t_l,T)}\deq \Big(\t \sum_{n=l}^{N-1}\|x_n\|^2 \Big)^{1/2} \qquad
\mbox{and} \qquad 
\|u_\cd\|_{\dbU(t_l,T)}\deq \Big(\t \sum_{n=l}^{N-1}\|u_n\|^2\Big)^{1/2}.
\end{equation*}

Throughout this section, to deal with the difference equation, we use a subscript index for the time variable, such as $x_n, u_n$, and $x_\cd\,,u_\cd$.

By adopting  the (semi-)implicit Euler method in time, we approximate 
system \rf{pde-h} and the cost functional \rf{cost-t-h} as follows:
For any given $l= 0,1,\cds,N-1$, and $z\in \dbL^2$, the discrete state $x_\cd$ satisfies 
\bel{dde}
\setlength\abovedisplayskip{3pt}
\setlength\belowdisplayskip{3pt}
\lt\{\!\!\!
\begin{array}{ll} 
\ds x_{n+1}=\cA_0x_{n}+\t \cA_0{\Pi_h u_n}  \qq n=l, l+1,\cds,N-1\,,\\
\ns\ds x_l=  \Pi_hz\,,
\end{array}
\rt.
\ee
where $\cA_0=\big(\mathds{1}_h-\t\cA_h\big)^{-1}$, and
\bel{cost-ht1}
\setlength\abovedisplayskip{3pt}
\setlength\belowdisplayskip{3pt}
\cG_{h,\t}(t_l,z; u_\cd)
= \frac 1 2 \big[ \|x_\cd\|^2_{\dbX(t_l,T)}+\|u_\cd\|^2_{\dbU(t_l,T)}\big]+\frac \a 2\|x_N\|^2\,.
\ee
A discretized version in time of the family of Problems {\bf (LQ)}$_{\tt aux}^{t;h}$ is\\
\no{\bf Problem (LQ)$^{t_{l};h,\t}_{\tt aux}$.} For any  given $l=0,1,\cds,N-1$ and $z\in \dbL^2$, search for 
$u^*_\cd\in \dbU(t_l,T)$ such that
\begin{equation*}
\setlength\abovedisplayskip{3pt}
\setlength\belowdisplayskip{3pt}
\cG_{h,\t}(t_l,z; u^*_\cd)=\inf_{u_\cd \in \dbU(t_l,T)}\cG_{h,\t}(t_l,z; u_\cd)=:V_{h,\t}(t_l,z)\,.
\end{equation*}

Note that for Problems {\bf(LQ)$^{t;h}_{\tt aux}$} resp.~{\bf(LQ)$^{t_l;h,\t}_{\tt aux}$} the elements 
$x_h(\cd)\in L^2(0,T;\dbV_h)$ resp.~$x_\cd\in \dbX(t_l,T)$ take values in $\dbV_h$, while $u_h(\cd)\in L^2(0,T;\dbL^2)$ 
resp.~$u_\cd\in \dbU(t_l,T)$ take values in $\dbL^2$. This choice is used to prove convergence rates for the proposed 
discretization schemes  of the Riccati equation;
see Theorems \ref{Riccati-rate-h} and \ref{Riccati-rate}.
Actually, with the same argument as in \rf{w1026e5}, we can deduce that the optimal control $u^*_\cd$ to Problem {\bf(LQ)$^{t_l;h,\t}_{\tt aux}$} is $\dbV_h$-valued.

We call the solution of Problem {\bf (LQ)}$^{t_{l};h,\t}_{\tt aux}$ its {\em optimal control} 
$ u^*_\cd  \in \dbU(t_l,T) $, and the corresponding state $x^*_\cd \in \dbX(t_l,T)$ an {\em optimal state}. 
The following result show the solvability of Problem {\bf (LQ)}$^{t_{l};h,\t}_{\tt aux}$ and a {\em discrete
feedback law} for the optimal control.
We refer the reader to \cite[Chapter 5]{Heij-Ran-Schagen21} for deterministic systems, and to \cite{Zhou02} for stochastic systems.

\bl{w1021l1}
For any $l=0,1,\cds, N-1$ and $z\in\dbL^2$, Problem {\bf (LQ)}$^{t_{l};h,\t}_{\tt aux}$ admits a unique optimal control 
which enjoys the discrete state feedback law
\bel{feedback-2a}
\setlength\abovedisplayskip{3pt}
\setlength\belowdisplayskip{3pt}
u^*_n=- \cK_n^{-1}\cH_n x^*_n \qquad n=l,l+1,\cds,N-1\,.
\ee
Moreover, the value function satisfies
\bel{val-2a}
\setlength\abovedisplayskip{3pt}
\setlength\belowdisplayskip{3pt}
V_{h,\t}(t_l,z)=\frac 1 2 ( \cP_l\Pi_h z,z )_{\dbL^2} \qquad  \forall\, z \in {\mathbb L}^2\,,
\ee
where $\cP_\cd\equiv \{\cP_n\}_{n=0}^N\subset \dbS_+(\dbL^2|_{\dbV_h})$ solves the following difference Riccati equation 
\bel{dis-Riccati1-h-t}
\setlength\abovedisplayskip{3pt}
\setlength\belowdisplayskip{3pt}
\lt\{\!\!\!
\begin{array}{ll} 
\ds \cP_n=\cA_0\cP_{n+1}\cA_0+\t\mathds{1}_h -\t \cH_n \cK_n^{-1}\cH_n \qq n=0,1,\cds,N-1\,,\\
\ns\ds \cP_N=\a \mathds{1}_h\,,\\
\ns\ds \cH_n=\cA_0\cP_{n+1}\cA_0\,,\\
\ns\ds \cK_n=\mathds{1}_h+\t \cA_0\cP_{n+1}\cA_0\, .
\end{array}
\rt.
\ee 

\el

Here we emphasize  that
$\cP_l\in \dbS_+(\dbL^2|_{\dbV_h})$ for any $l=0,1,\cds,N$; for the used notations, we again refer to Section \ref{not1}. 
Also, \rf{val-2a}  leads to
%
\begin{equation*}
\setlength\abovedisplayskip{3pt}
\setlength\belowdisplayskip{3pt}
( \cP_l\Pi_h z,z )_{\dbL^2}\leq 2 \, {\mathcal G}_{h,\t}(t_l,z, u_\cd)\qquad  \forall\, z \in {\mathbb L}^2 \quad \forall\, u_\cd \in \dbU(t_l,T)\,.
\end{equation*}
These properties make difference Riccati equation (\ref{dis-Riccati1-h-t}) stand out if compared to standard discretization (see {\em e.g.}~\cite{Stillfjord18,Stillfjord18-2,Benner-Stillfjord-Trautwein22}),
and allow to verify optimal convergence behavior for (\ref{dis-Riccati1-h-t}) in Section \ref{Riccati-dis}. 


\br{w322r1}
An essential step in the above program --- which is detailed in Section \ref{Riccati-dis} --- is to study Problem ${\bf (LQ)}_{\tt aux}^{\cd;h,\t}$, {\em i.e.}, the temporal discretization of Problem ${\bf (LQ)}_{\tt aux}^{\cd;h}$.
%
%
%
\er

\subsection{Approximation of the Riccati equation}\label{Riccati-dis}

To obtain the convergence rates in different norms --- $\cL(\dbL^2)$ and $\cL(\dbH_0^1;\dbL^2)$ --- for the discretization of Riccati equation \rf{Riccati},
we divide this section into two parts: Sections \ref{Riccati-dis-h} and \ref{Riccati-dis-2} focus on the first consistent discretization
$\cP_\cd$ that is given by \rf{dis-Riccati1-h-t}.
The related error analysis requires  assumption {\bf (H)}, and is based on the close connection of $\cP(\cd)$ with Problem {\bf (LQ)}$_{\tt aux}^\cd$, and of $\cP_\cd$ with Problem {\bf (LQ)}$_{\tt aux}^{\cd;h,\tau}$. 
Section  \ref{Riccati-dis-3} presents  the second consistent discretization $\bar \cP_\cd$ provided in \rf{dif-Riccati}; 
 there results are derived without assumption {\bf (H)}.


In the following, we use the following notations for solutions $x\big(\cd;t, z,u(\cd)\big)$, $x_h\big(\cd; t,z,u_h(\cd)\big)$, and $x_\cd(t_l,z,u_\cd)$  of \rf{pde}, \rf{pde-h}, and \rf{dde} to emphasize their dependence on initial conditions and control variables;
when there is no danger of confusion, we also write $x(\cd;t,z)$, $x_h(\cd;t,z)$, $x_\cd(t_l,z)$ for the sake of simplicity.
Similarly, we denote by $u^*(\cd;t,z)$, $u^*_h(\cd; t,z)$, $u^*_\cd(t_l,z)$
the optimal control of Problems {\bf (LQ)}$^t_{\tt aux}$, {\bf (LQ)}$^{t;h}_{\tt aux}$, and {\bf (LQ)$^{t_l;h,\t}_{\tt aux}$}, respectively.
Let us recall that all state and control variables in this section are deterministic, which is different from 
related quantities used in Section
\ref{SLQ-rate};  
they are introduced here {\em only} to verify the main results ({\em i.e.}, Theorems \ref{Riccati-rate}
and \ref{Riccati-rate2}) 
which is a part of our scheme to approximate
Problem {\bf (SLQ)}.

\subsubsection{Estimates for the spatially semi-discrete Riccati equation}\label{Riccati-dis-h}

In the following result,  we deduce the properties that $\cP_h(\cd)$
is uniformly bounded, and that the optimal pair
$\big(x^*_h(\cd),u^*_h(\cd)\big)$ of Problem {\bf (LQ)$^{t;h}_{\tt aux}$} is bounded by the initial state.

\bl{w1012l1}
Let $\cP_h(\cd)$ be the solution of Riccati equation \rf{dis-Riccati1-h} and
for any $t\in [0,T)$, $z\in\dbL^2$, let $\big(x^*_h(\cd;t,z), u^*_h(\cd;t,z)\big)$ be the optimal pair of Problem {\bf (LQ)$^{t; h}_{\tt aux}$}.
Then, there exists a constant $\cC$ independent of $h$ such that
\vspace{-1ex}
\begin{subequations}\label{w1012e1}
    \begin{empheq}[left={\empheqlbrace\,}]{align}
      & \sup_{t\in[0,T]} \|\cP_h(t)\|_{\cL(\dbL^2|_{\dbV_h})}\leq \cC \,, \label{w1012e1a}\\
      &\sup_{t\in[0,T]} \Big[ \sup_{s\in[t,T]} \big[\| x^*_h(s; t,z)\|^2  +\|u^*_h(s; t,z)\|^2\big] 
       \notag \\
       &\qq\q+\int_t^T \| x^*_h(s; t,z)\|_{\dbH_0^1}^2 +\|u^*_h(s; t,z)\|_{\dbH_0^1}^2\rd s\Big] \leq \cC\|z\|^2\,. \label{w1012e1b}
    \end{empheq}
\end{subequations}
%
\el

\begin{proof}
{\bf (1) Verification of \rf{w1012e1a}.} 
This assertion has been derived in Lemma \ref{w229l4} by virtue of Problem {\bf (SLQ)$_{\tt aux}^{\cd;h}$}; and
with the same idea
it  can also be obtained by the explicit representation \rf{val-h-1} for the value function $V_{h}(\cd,\cd)$.
Indeed, for any $z \in \dbL^2$, by the fact that $\cP_h(t)\in \dbS_+(\dbL^2|_{\dbV_h})$ and \rf{val-h-1}, it follows that
\begin{equation*}
\setlength\abovedisplayskip{3pt}
\setlength\belowdisplayskip{3pt}
\bal
0&\leq \big( \cP_h(t)\Pi_h z,\Pi_h z \big)_{\dbL^2}= \big( \cP_h(t)\Pi_h z,z\big)_{\dbL^2}
\leq 2\cG_{h}(t,z; 0)\\
&=\frac 1 2 \int_t^T\|x_h(s;t,z,0)\|^2\rd s +\frac \a 2 \|x_h(T;t,z,0)\|^2\\
&\leq \cC \|\Pi_h z\|^2\,,
\eal
\end{equation*}
where $\cC$ is independent of $t$, which settles the assertion \rf{w1012e1a}.

\ss

{\bf (2) Verification of \rf{w1012e1b}.} The feedback law \rf{feedback-h-1} and PDE \rf{pde-h} imply that
\begin{equation*}
\setlength\abovedisplayskip{3pt}
\setlength\belowdisplayskip{3pt}
\bal
& \|x^*_h(s;t,z)\|^2+2\int_t^s\|\nb x^*_h(\th;t,z)\|^2\rd \th\\
&\qquad \leq  \|\Pi_h z\|^2+2\sup_{t\in[0,T]}\|\cP_h(t)\|_{\cL(\dbL^2|_{\dbV_h})}\int_t^s\|x^*_h(\th;t,z)\|^2\rd \th\,,
\eal
\end{equation*}
which, together with assertion \rf{w1012e1a} and Gronwall's inequality, leads to
\bel{w1012e2}
\setlength\abovedisplayskip{3pt}
\setlength\belowdisplayskip{3pt}
\sup_{t\in [0,T]}\Big[\sup_{s\in[t,T]} \|x^*_h(s;t,z)\|^2+\int_t^T\|\nb x^*_h(\th;t,z)\|^2\rd \th\Big]
\leq \cC\|\Pi_h z\|^2\,.
\ee
Then we adopt Pontryagin's maximum principle to verify the assertion on $u_h^*(\cd;t,z)$. By the deterministic version of \rf{intro-1c}--\rf{intro-1d} or 
\cite[Chapter 12]{Zabczyk20}, we know that
\bel{w1013e14}
\setlength\abovedisplayskip{3pt}
\setlength\belowdisplayskip{3pt}
u^*_h(s;t,z)=y_h(s)\qq t\in[0,T)\,, s\in[t,T]\,,
\ee
where $y_h(\cd)$ solves
\bel{bpde-h}
\setlength\abovedisplayskip{3pt}
\setlength\belowdisplayskip{3pt}
\lt\{\!\!\!
\begin{array}{ll} 
\ds y_h'(s)=-\cA_h y_h(s)+x^*_h(s;t,z)  \qq  s \in [t,T)\,,\\
\ns\ds y_h(T)=-\a x^*_h(T;t,z)\,.
\end{array}
\rt.
\ee
Relying on the stability property of $y_h(\cd)$, \rf{w1012e2} and \rf{w1013e14}, we have
\begin{equation*}
\setlength\abovedisplayskip{3pt}
\setlength\belowdisplayskip{3pt}
\bal
&\sup_{s\in[t,T]}\|u^*_h(s;t,z)\|^2+\int_t^T\|\nb u^*_h(s;t,z)\|^2\rd s \\
&\qquad \leq \cC\Big[ \|x^*_h(T;t,z)\|^2+\int_t^T\|x^*_h(s;t,z)\|^2\rd s\Big]\\
&\qquad \leq \cC\|z\|^2\,,
\eal
\end{equation*}
where $\cC$ is independent of $t$.
That completes the proof.
\end{proof}

In the same vein, we can prove the following result for the solution $\cP(\cd)$ to Riccati equation \rf{Riccati}.
\bl{w1012l2} 
Fix $t\in[0,T]$ and $z\in \dbL^2$. Let $\big(x^*(\cd;t,z), u^*(\cd;t,z)\big)$ be the optimal pair of Problem {\bf (LQ)}$^t_{\tt aux}$.
Let $\cP(\cd)$ be the solution to Riccati equation \rf{Riccati}.
There exists $\cC$ such that 
\begin{subequations}\label{w1012e3}
    \begin{empheq}[left={\empheqlbrace\,}]{align}
      & \sup_{t\in[0,T]}\|\cP(t)\|_{\cL(\dbL^2)}\leq \cC\,, \label{w1012e3a}\\
      &\sup_{t\in[0,T]} \Big[ \sup_{s\in[t,T]} \big[\| x^*(s; t,z)\|^2  +\|u^*(s; t,z)\|^2\big] 
       \notag \\
       &\qq\q+\int_t^T \| x^*(s; t,z)\|_{\dbH_0^1}^2 +\|u^*(s; t,z)\|_{\dbH_0^1}^2\rd s\Big] \leq \cC\|z\|^2\,.  \label{w1012e3b}
    \end{empheq}
\end{subequations}
\el

\ss

For any given $t\in[0,T)$, to estimate $\|\cP(t)-\cP_h(t)\Pi_h\|_{\cL(\dbL^2)}$ we follow a related
strategy as in \cite[Section 3]{Kroller-Kunisch91}:
For $z\in\dbL^2$, we shall compare $\big(\cP(t)z,z\big)_{\dbL^2}$ 
and $\big(\cP_h(t)\Pi_hz,z\big)_{\dbL^2}$. There are two possibilities.

\ss

\no{\bf Case (i)}  $\big(\cP_h(t)\Pi_h z,z\big)_{\dbL^2}\leq \big(\cP(t)z,z\big)_{\dbL^2}.$

In this case, based on the explicit representations of the value functions $V(\cd,\cd)$ resp.~$V_{h}(\cd,\cd)$ in \rf{val-1} resp.~\rf{val-h-1}, 
we arrive at
\bel{w1012e4}
\setlength\abovedisplayskip{3pt}
\setlength\belowdisplayskip{3pt}
\bal
0&\leq \big( \cP(t)z,z \big)_{\dbL^2}- \big( \cP_h(t)\Pi_hz,z \big)_{\dbL^2}
=2V(t,z)-2V_{h}(t,z)\\
&=2 \cG\big(t,z; u^*(\cd;t,z) \big) -2 \cG_{h}\big(t,z; u^*_h(\cd;t,z)\big)\,.
\eal
\ee
To estimate the difference of the right-hand side of \rf{w1012e4}, we replace $u^*(\cd;t,z)$ by the (possibly) non-optimal control 
$u^*_h(\cd;t,z)\in L^2(t,T;\dbL^2)$. 
By the optimality of 
$u^*(\cd;t,z)$ in Problem {\bf (LQ)$^{t}_{\tt aux}$}, we then conclude that
\bel{w1012e5}
\setlength\abovedisplayskip{3pt}
\setlength\belowdisplayskip{3pt}
0 \leq \big( \cP(t)z,z \big)_{\dbL^2}- \big( \cP_h(t)\Pi_hz,z \big)_{\dbL^2}
\leq 2\cG\big(t,z; u^*_h(\cd;t,z)\big)-2\cG_{h}\big(t,z; u^*_h(\cd; t,z)\big)\,.
\ee
Hence, to bound the difference of the quadratic  cost functionals, we have to estimate the term
$\big\|x\big(\cd;t,z,u^*_h(\cd;t,z)\big)- x^*_h\big(\cd; t,z,u^*_h(\cd; t,z)\big)\big\|$,
which will be done in Lemma \ref{w1012l3} below.

\ss

\no{\bf Case (ii)}  $\big(\cP_h(t)\Pi_hz,z\big)_{\dbL^2}> \big(\cP(t)z,z\big)_{\dbL^2}.$

In this case, similar to Case (i), we can deduce that 
\bel{w1012e6}
\setlength\abovedisplayskip{3pt}
\setlength\belowdisplayskip{3pt}
\bal
0 \leq \big( \cP_h(t)\Pi_hz,z \big)_{\dbL^2}- \big( \cP(t)z,z \big)_{\dbL^2}
\leq 2\cG_{h}\big(t,z;  u^*(\cd; t,z)\big)-2\cG\big(t,z; u^*(\cd;t,z)\big)\,.
\eal
\ee
Note that $u^*(\cd; t,z)\in L^2(t,T;\dbL^2)$ may be a non-optimal control in Problem {\bf (LQ)$^{t;h}_{\tt aux}$}.
Then, the difference between $x_h\big(\cd; t,z, u^*(\cd; t,z)\big)$ and $x^*\big(\cd;t,z,u^*(\cd;t,z)\big)$ need be estimated, 
which will be 
dealt with in Lemma \ref{w1012l4}.

\bl{w1012l3}
For any $t\in[0,T)$ and $z\in\dbL^2$,
suppose that $\big(x^*_h(\cd; t,z), u^*_h(\cd; t,z)\big)$ is the optimal pair of Problem {\bf (LQ)$^{t;h}_{\tt aux}$}.
Let
$\bar x(\cd;t,z) \deq x\big(\cd;t,z, u^*_h(\cd;t,z)\big)$. Then there exists a constant $\cC$ independent of $h$ 
such that
\begin{subequations}\label{w1012e7}
    \begin{empheq}[left={\empheqlbrace\,}]{align}
      & \sup_{t\in[0,T]} \sup_{s\in[t,T]}\|\bar x(s;t,z)\|\leq \cC\|z\|\,,\label{w1012e7a}\\
&\|\bar x(s;t,z)- x^*_h(s;t,z)\| \leq \cC h^2\Big[\frac{1}{s-t}+\ln{\frac 1 h}\Big]\|z\| \qq \forall\, s\in(t,T]\,, \label{w1012e7b}\\
 &\int_t^T \|\bar x(s;t,z)-x^*_h(s;t,z)\| \rd s \leq \cC h^2 \ln{\frac 1 h}\|z\| \,.\label{w1012e7c}
\end{empheq}
\end{subequations}
%
\el

\begin{proof}
Throughout the proof, we write $\bar x(\cd)$ for $\bar x(\cd;t,z)$. Recall that $\h E(\cd)$ resp.~$\h E_h(\cd)$ denote 
semigroups generated by $\cA$ resp.~$\cA_h$, and that $\h G(\cd)=\h E(\cd)-\h E_h(\cd)\Pi_h$; see Section \ref{dis-lq}.

\ss

{\bf (1) Verification of \rf{w1012e7a}.}
By \rf{pde}, Lemmata \ref{w228l1} and \ref{w1012l1}, it follows that
\begin{equation*}
\setlength\abovedisplayskip{3pt}
\setlength\belowdisplayskip{3pt}
\bal
\|\bar x(s)\|&=\Big\|\h E(s-t)z+\int_{t}^{s}\h E(s-\th)  u^*_h(\th;t,z)\rd \th \Big\|\\
&\leq \cC \Big[\|z\|+\sup_{s\in[t,T]}\|u^*_h(s; t,z)\|\Big]\\
&\leq \cC \|z\|\,,
\eal
\end{equation*}
where $\cC$ is independent of $h$ and $t$. That settles \rf{w1012e7a}.

\ss

{\bf (2) Verification of  \rf{w1012e7b}.} By \rf{pde}, \rf{pde-h} and Lemma \ref{w1012l1},   
Lemma \ref{error-G_h} for $\h G_h(\cd)$ with $(\rho,\g)=(0,0)$ and $(\rho,\g)=(0,2)$, we find that
\begin{equation*}
\setlength\abovedisplayskip{3pt}
\setlength\belowdisplayskip{3pt}
\bal
\|\bar x(s)- x^*_h(s;t,z)\|&\leq  \big\| \h G_h(s-t)z \big\|+\Big\| \int_{(s-h^2)\vee t}^s \h G_h(s-\th)u^*_h(\th; t,z)\rd \th \Big\|\\
&\q+\Big\| \int_t^{(s-h^2)\vee t} \h G_h(s-\th)u^*_h(\th; t,z)\rd \th \Big\|\\
&\leq \cC \frac {h^2}{s-t}\|z\|+\cC h^2\sup_{s\in[t,T]}\|u^*_h(s;t,z)\|
+\cC \int_t^{(s-h^2)\vee t}  \frac{h^2}{s-\th}\rd \th \sup_{s\in[t,T]}\|u^*_h(s;t,z)\|\\
&\leq \cC h^2\Big[\frac{1}{s-t}+1+\ln{(s-t)}-\ln{h}\Big] \|z\|\,,
\eal
\end{equation*}
which yields the assertion \rf{w1012e7b} since
\begin{equation*}
\setlength\abovedisplayskip{3pt}
\setlength\belowdisplayskip{3pt}
\bal
\ln{(s-t)}-\ln{h}\leq \ln{T}+\ln{\frac 1 h}=\ln{\frac 1 h}\Big[\frac{\ln{T}}{\ln{\frac 1 {h_0}}}+1\Big]=\big[\cC_{h_0}+1\big]\ln{\frac 1 h}\,.
\eal
\end{equation*}



{\bf (3) Verification of \rf{w1012e7c}.}  
Based on systems \rf{pde}, \rf{pde-h} and assertions \rf{w1012e1b}, \rf{w1012e7a}, \rf{w1012e7b}, we arrive at
\begin{equation*}
\setlength\abovedisplayskip{3pt}
\setlength\belowdisplayskip{3pt}
\bal
\int_t^T& \|\bar x(s;t,z)-x^*_h(s;t,z)\| \rd s\\
& \q \leq  \int_t^{(t+h^2)\wedge T} \sup_{s\in[t,T]}  \big[\|\bar x(s;t,z)\|+\|x^*_h(s;t,z)\|\big]\rd s
+\cC h^2 \int_{(t+h^2)\wedge T}^T\Big[ \frac{1}{s-t}+\ln{\frac 1 h}\Big]\|z\|\rd s\\
&\quad  \leq \cC h^2 \ln{\frac 1 h}\|z\| \,.
\eal
\end{equation*}
That completes the proof.
\end{proof}

In the same vein as  in Lemma \ref{w1012l3}, we derive the following
\bl{w1012l4}
For any $t\in [0,T)$ and $z\in\dbL^2$,
let $\big(x^*(\cd;t_l,z), u^*(\cd; t_l,z)\big)$ be the optimal pair of Problem {\bf (LQ)$^{t_l}_{\tt aux}$}.
Denote by $\h x_h\big(\cd;t,z)\deq x_h(\cd; t,z, u^*(\cd; t,z)\big)$. Then
\begin{subequations}\label{w1012e9}
    \begin{empheq}[left={\empheqlbrace\,}]{align}
      & \sup_{t\in[0,T]} \sup_{s\in[t,T]}\|\h x_h(s;t,z)\| \leq \cC\|z\| \,, \label{w1012e9a}\\
      &\|\h x_h(s;t,z)- x^*(s;t,z)\| \leq \cC h^2\Big[\frac{1}{s-t}+\ln{\frac 1 h}\Big]\|z\|  \qq \forall\, s\in(t,T]\,,  \label{w1012e9b} \\
&\int_t^T \|\h x_h(s;t,z)-x^*(s;t,z)\| \rd s \leq \cC h^2 \ln{\frac 1 h}\|z\| \,.  \label{w1012e9c}
\end{empheq}
\end{subequations}
\el

So far, we have obtained the following results in this section:
\begin{enumerate}[(a)]

\item
By \rf{w1012e3a} in Lemma \ref{w1012l2}, the $\cL(\dbL^2)$-valued solution $\cP(\cd)$ to the Riccati equation 
\rf{Riccati} is bounded; by its close connection to Problem {\bf (LQ)$^t_{\tt aux}$} --- see formulae \rf{feedback-1} and 
\rf{val-1} --- the stability bound \rf{w1012e3a} for $\cP(\cd)$ leads to the bound \rf{w1012e3b} for the optimal pair
$\big(x^*(\cd; t,z), u^*(\cd; t,z) \big)$ of Problem {\bf (LQ)$^t_{\tt aux}$}.

\item
According to Lemma \ref{w1012l1}, these stability bounds are inherited by $\cP_h(\cd)$, which solves the spatial 
discretization \rf{dis-Riccati1-h}. This family of $\cL(\dbL^2|_{\dbV_h})$-valued operators is also linked to 
Problem {\bf (LQ)$^{t;h}_{\tt aux}$}, whose optimal pairs $\big(x^*_h(\cd; t,z), u^*_h(\cd; t,z) \big)$ are bounded (uniformly in $h$) in terms of $z$; see \rf{w1012e1b}.

\item
In Lemma \ref{w1012l3}, we compare the solution $\bar x(\cd; t,z)$ to PDE \rf{pde} 
(with control $u(\cd)=u^*_h(\cd;t,z)$) 
with the optimal state $x^*_h(\cd; t,z) $ of Problem {\bf (LQ)$^{t;h}_{\tt aux}$}.

\item
The `reverse case' is addressed in Lemma \ref{w1012l4}, where we compare the solution $\h x(\cd; t,z)$ to 
semi-discretization \rf{pde-h} (with the given control $u_h(\cd)=u^*(\cd;t,z)$)
with the optimal state $x^*(\cd;t,z)$ of Problem  {\bf (LQ)$^{t}_{\tt aux}$}.

\end{enumerate}

The following theorem is  the main result in this section, which bounds the error between $\cP(\cd)$ 
and $\cP_h(\cd)$.

\bt{Riccati-rate-h}
Suppose that $\cP(\cd)$ and $\cP_h(\cd)$ are solutions to Riccati equations \rf{Riccati3} and \rf{dis-Riccati1-h} respectively.
Then there exists a constant $\cC$ independent of $h$ such that
\bel{w1012e10}
\setlength\abovedisplayskip{3pt}
\setlength\belowdisplayskip{3pt}
\|\cP(t)-\cP_h(t)\Pi_h\|_{\cL(\dbL^2)}\leq \cC h^2 \Big[\frac {\a}{T-t}+\ln{\frac 1 h}\Big] \qq \forall\, t\in[0,T)\,.
\ee
\et

\begin{proof}
For given $t\in[0,T)$, and $z\in\dbL^2$, based on \rf{w1012e5} and \rf{w1012e6} we split the proof into the following two cases.

\ss

\no {\bf Case (i)}. $(\cP_h(t)\Pi_hz,z)_{\dbL^2}\leq (\cP(t)z,z)_{\dbL^2}\,.$

In this case,  by \rf{w1012e5} and the definition of the cost functionals $\cG(\cd,\cd; \cd)$ and $\cG_{h}(\cd,\cd; \cd)$,
as well as $\bar x(\cd;t,z)$ that is defined in Lemma \ref{w1012l3}, we may use
binomial formula and estimate 
\begin{eqnarray*}
&0  &\leq  \big( \cP(t)z,z \big)_{\dbL^2}-\big( \cP_h(t)\Pi_hz,z \big)_{\dbL^2}
\leq 2\cG\big(t,z; u^*_h(\cd;t,z)\big)-2\cG_{h}\big(t,z; u^*_h(\cd;t,z)\big)\\
&&=\int_t^T\big[\|\bar x(s;t,z)\|^2-\|x^*_h(s;t,z)\|^2 \big] \rd s 
 +\a \big[ \|\bar x(T;t,z)\|^2-\|x^*_h(T;t,z)\|^2 \big]\\
&&\leq \cC \sup_{s\in[t,T]}\big[ \|\bar x(s;t,z)\|+\|x^*_h(s;t,z)\| \big]\\
&&\qq\times \Big[\int_t^T \|\bar x(s;t,z)-x^*_h(s;t,z)\| \rd s +\a \|\bar x(T;t,z)-x^*_h(T;t,z)\|\Big]\\
&&\leq \cC h^2\Big[\frac \a{T-t}+\ln{\frac 1 h}\Big]\|z\|^2\,.
\end{eqnarray*}

\ss

\no {\bf Case (ii)}. $\big(\cP_h(t)\Pi_hz,z\big)_{\dbL^2}> \big(\cP(t)z,z\big)_{\dbL^2}.$

As in Case (i), 
by \rf{w1012e6} and Lemmata \ref{w1012l2} and \ref{w1012l4}, we can deduce that
\begin{equation*}
\setlength\abovedisplayskip{3pt}
\setlength\belowdisplayskip{3pt}
\bal
0 &\leq \big( \cP_h(t)\Pi_hz,z \big)_{\dbL^2}-\big( \cP(t)z,z \big)_{\dbL^2}
\leq 2\cG_{h}(t,z;  u^*(\cd; t,z))-2\cG(t,z; u^*(\cd;t,z))\\
&\leq \cC \sup_{s\in[t,T]}\big[ \|\h x_h(s;t,z)\|+\|x^*(s;t,z)\| \big]\\
&\qq\times \Big[\int_t^T \|\h x_h(s;t,z)-x^*(s;t,z)\| \rd s +\a \|\h x_h(T;t,z)-x^*(T;t,z)\|\Big]\\
&\leq \cC h^2\Big[\frac \a{T-t}+\ln{\frac 1 h}\Big]\|z\|^2\,.
\eal
\end{equation*}

A combination of these two estimates leads to
\begin{equation*}
\setlength\abovedisplayskip{3pt}
\setlength\belowdisplayskip{3pt}
\bal
\big| \big( \cP_h(t)\Pi_hz,z \big)_{\dbL^2}-\big( \cP(t)z,z \big)_{\dbL^2}\big |
\leq \cC h^2\Big[\frac \a{T-t}+\ln{\frac 1 h}\Big]\|z\|^2\,,
\eal
\end{equation*}
which implies the assertion \rf{w1012e10}.
\end{proof}

%

\subsubsection{The first consistent  difference Riccati equation}\label{Riccati-dis-2}

The main result in this section is Theorem \ref{Riccati-rate}, which provides an error estimate for 
\rf{dis-Riccati1-h-t} ---  the temporal discretization of \rf{dis-Riccati1-h}. For its verification,
we parallel the steps in Section \ref{Riccati-dis-h}, and correspondingly  introduce related 
quantities $\bar x_h(\cd; t_l,z)$ in Lemma \ref{w207l5}, and $\h x_\cd(t_l,z)$ in Lemma \ref{w207l6}
to quantify the temporal error inherent to \rf{dis-Riccati1-h-t}
step by step.

 The iterates of the  {\em difference Riccati equation} \rf{dis-Riccati1-h-t} are known to
take values in ${\mathbb S}_+(\dbL^2|_{\dbV_h})$; see Lemma \ref{w1021l1}. 
In this respect, this discretization is {\em consistent}.

\ss

The following result characterizes the solution of \rf{dde}; since the proof is obvious, we skip it here.
\bl{w207l2}
For any $l=0,1,\cds,N-1$, $z\in\dbL^2$, $u_\cd\in\dbU(t_l,T)$, let $x_\cd(t_l,z,u_\cd)$ solve \rf{dde}. Then for any $n=l,l+1,\cds,N$, it holds that
\bel{w305e2}
\setlength\abovedisplayskip{3pt}
\setlength\belowdisplayskip{3pt}
x_n(t_l,z,u_\cd)=\cA_0^{n-l}\Pi_hz+\t\sum_{k=l}^{n-1} \cA_0^{n-k}\Pi_h u_k\,.
\ee
\el

%

The following lemma establishes that $\cP_\cd$ from \rf{dis-Riccati1-h-t}
is uniformly bounded, and also provides a uniform bound for  {$x^*_\cd(t_l,z)$} from Problem  {\bf (LQ)$^{t_l; h,\t}_{\tt aux}$}, whose derivation
is based on \rf{w305e2}. It will turn out later that \rf{w305e2} is useful at other
places as well, including the estimation of the difference between $x\big(\cd;t_l,z,u(\cd)\big)$ and $x_\cd(t_l,z,u_\cd)$; see 
Lemmata \ref{w207l5} and \ref{w207l6}.
%

\bl{w207l4}
Let $\cP_\cd$ be the solution to Riccati equation \rf{dis-Riccati1-h-t}, and
for any $l=0,1,\cds,N-1$ and $z\in\dbL^2$, let $\big(x^*_\cd(t_l,z), u^*_\cd(t_l,z)\big)$ be the optimal pair of Problem {\bf (LQ)$^{t_l; h,\t}_{\tt aux}$}.
There exists a constant $\cC$ independent of $h$ and $\t$ such that
\begin{subequations}\label{w225e1}
    \begin{empheq}[left={\empheqlbrace\,}]{align}
      & \max_{0\leq l\leq N-1}\|\cP_l\|_{\cL(\dbL^2|_{\dbV_h})}\leq \cC \,, \label{w225e1a}\\
      &  \max_{0\leq l\leq N-1}\max_{l\leq n\leq N-1} \big[\| x^*_n(t_l,z)\|  +\|u^*_n(t_l,z)\|\big] \leq \cC\|z\|\,. \label{w225e1b}
    \end{empheq}
\end{subequations}
%
\el

\begin{proof}
{\bf (1) Verification of \rf{w225e1a}.}
This assertion can be obtained by the same vein as in the proof of Lemma \ref{w1012l1}, along with the explicit representation \rf{val-2a}
for the value function $V_{h,\t}(\cd,\cd)$,
and the fact that $\|\cA_0\|_{\cL(\dbL^2|_{\dbV_h})}\leq 1$.

\ss

{\bf (2) Verification of \rf{w225e1b}.} By the feedback law \rf{feedback-2a}
and the difference equation \rf{dde},
for any $n\geq l$, we have
\begin{equation*}
\setlength\abovedisplayskip{3pt}
\setlength\belowdisplayskip{3pt}
x^*_{n+1}(t_l,z)=\cA_0x^*_{n}(t_l,z)-\t \cA_0\cK_n^{-1}\cH_nx^*_n(t_l,z)\,.
\end{equation*}
By the definition of $\cK_\cd$ and $\cH_\cd$ in (\ref{dis-Riccati1-h-t}), 
assertion \rf{w225e1a} and the fact that $\{\cP_n\}_{n=0}^N \subset \dbS_+(\dbL^2|_{\dbV_h})$, we  arrive at
\begin{equation*}
\setlength\abovedisplayskip{3pt}
\setlength\belowdisplayskip{3pt}
\bal
\|\cA_0\cK_n^{-1}\cH_n\|_{\cL(\dbL^2|_{\dbV_h})}
\leq \|\cP_{n+1}\|_{\cL(\dbL^2|_{\dbV_h})}
\leq \cC\,,
\eal
\end{equation*}
which leads to
\begin{equation*}
\setlength\abovedisplayskip{5pt}
\setlength\belowdisplayskip{5pt}
\|x^*_{n}(t_l,z)\|
\leq (1+\cC\t)\|x^*_{n-1}(t_l,z)\|
\leq \cds 
\leq (1+\cC\t)^{n-l}\|\Pi_h z\|
\leq \cC\|\Pi_hz\|\,.
\end{equation*}
Finally, relying on the discrete feedback form \rf{feedback-2a} again, we have
\begin{equation*}
\setlength\abovedisplayskip{5pt}
\setlength\belowdisplayskip{5pt}
\|u^*_{n}(t_l,z)\|=\|-\cK_n^{-1}\cH_nx^*_n(t_l,z)\|\leq  \cC\|x^*_n(t_l,z)\|
\leq  \cC \|\Pi_hz\|\,.
\end{equation*}
That completes the proof.
\end{proof}

In the next step, we follow similar arguments as in Section \ref{Riccati-dis-h} to estimate $\|\cP_h(t_l)\Pi_h-\cP_l\Pi_h\|_{\cL^2(\dbL^2)}$. 
The following two results (Lemmata \ref{w207l5} and \ref{w207l6}) show some estimates in time, where similar estimates in space are presented in Lemmata \ref{w1012l3} and \ref{w1012l4}.

\bl{w207l5}
For any $l=0,1,\cds,N-1$ and $z\in\dbL^2$,
let $\big(x^*_\cd(t_l,z), u^*_\cd(t_l,z)\big)$ be the optimal pair of Problem {\bf (LQ)$^{t_l;h,\t}_{\tt aux}$}.
Suppose that 
\bel{w207e5}
\setlength\abovedisplayskip{5pt}
\setlength\belowdisplayskip{5pt}
\bar u_h(t;t_l,z)\deq u^*_n(t_l,z) \qq \forall\, t\in [t_n,t_{n+1})\,, \quad n=l,l+1,\cds,N-1\,,
\ee
and
$\bar x_h(\cd;t_l,z) \deq x_h\big(\cd;t_l,z, \bar u(\cd;t_l,z)\big)$. There exists a constant $\cC$ independent of $h\,,\t$ 
such that
\begin{subequations}\label{w225e11}
    \begin{empheq}[left={\empheqlbrace\,}]{align}
      & \max_{0\leq l\leq N-1}\sup_{t\in[t_l,T]}\|\bar x_h(t;t_l,z)\|\leq \cC\|z\|\,,\label{w225e11a}\\
      &\sum_{n=l}^{N-1}\!\int_{t_n}^{t_{n+1}}\!\big[\|\bar x_h(t;t_l,z)\!-\!\bar x_h(t_n;t_l,z)\|+\|\bar x_h(t;t_l,z)\!-\!\bar x_h(t_{n+1};t_l,z)\|\big]\rd t 
\leq \cC \t\ln{\frac 1 \t}\|z\| \,,\label{w225e11b}\\
&\|\bar x_h (t_{n};t_l,z)\!-\!x^*_{n}(t_l,z)\|\leq \cC\t\Big[\frac 1 {t_n-t_l}+\ln{\frac 1 \t}\Big]\|z\| \q n=l+1,\cds,N\,. \label{w225e11c}
\end{empheq}
\end{subequations}
%
\el

\begin{proof}
For simplicity, we write $\bar x_h(\cd)$ instead of $\bar x_h(\cd;t_l,z)$
in the proof.
\ss

{\bf (1) Verification \rf{w225e11a}.}
By \rf{pde-h} and  \rf{w225e1b} in Lemma \ref{w207l4}, it follows that
\begin{equation*}
\setlength\abovedisplayskip{3pt}
\setlength\belowdisplayskip{3pt}
\bal
\|\bar x_h(t)\|&=\Big\|\h E_h(t-t_l)\Pi_hz+\int_{t_l}^{t}\h E_h(t-\th) \bar u_h(\th;t_l,z) \rd \th \Big\|\\
&\leq \cC\big[\|\Pi_h z\|+\max_{l\leq k\leq N-1}\|u^*_k(t_l,z)\|\big]
\leq \cC\|\Pi_h z\|\,.
\eal
\end{equation*}

\ss
{\bf (2) Verification of \rf{w225e11b}.}  We only prove
\begin{equation*}
\setlength\abovedisplayskip{3pt}
\setlength\belowdisplayskip{3pt}
I:=\sum_{n=l}^{N-1}\int_{t_n}^{t_{n+1}}\|\bar x_h(t)-\bar x_h(t_n)\|\rd t\leq 
 \cC \t\ln{\frac 1 \t} \|z\|\,,
\end{equation*}
since the remaining part can be shown similarly.
By \rf{pde-h}, we find that
\begin{eqnarray}
&I &\leq \sum_{n=l}^{N-1}\int_{t_n}^{t_{n+1}}\big\|\big[\h E_h(t-t_l)-\h E_h(t_n-t_l) \big]\Pi_hz\big\|\rd t \nonumber\\
&&\q +\sum_{n=l}^{N-1}\int_{t_n}^{t_{n+1}}\Big\|\int_{t_l}^{t_n} \big[\h E_h(t-s)-\h E_h(t_n-s) \big] \bar u_h(s; t_l,z)\rd s \Big\|\rd t\nonumber\\
&&\q +\sum_{n=l}^{N-1}\int_{t_n}^{t_{n+1}}\Big\|\int_{t_n}^{t} \h E_h(t-s) u^*_n(t_l,z)\rd s\Big\|\rd t \nonumber\\
&&=: I_1+I_2+I_3\,. \label{w225e7}
\end{eqnarray}

Now we estimate the summands on the right-hand side of \rf{w225e7} independently. For $I_1$, by virtue of $\nu(\cd)$ which is defined in \rf{w827e1}, 
as well as \rf{w228e2} for $\h E_h(\cd)$ with $\g=0$, and \rf{w320e01} for $\h E_h(\cd)$  with $\g=1$ in Lemma \ref{w228l1}, it follows that
\bel{w225e8}
\setlength\abovedisplayskip{3pt}
\setlength\belowdisplayskip{3pt}
\bal
I_1 &\leq \int_{t_{l}}^{(t_{l}+2\t)\wedge T}\cC\|z\|\rd t
 +\int_{(t_{l}+2\t)\wedge T}^T\cC \frac{t-\nu(t)}{\nu(t)-t_l}\|z\|\rd t \\
&\leq  \cC \|z\|\Big[\t+\t \int_{(t_{l}+2\t )\wedge T}^T \frac{1}{t-\t-t_l}\rd t \Big] 
\leq \cC\|z\|  \t\ln{\frac 1 \t}\,.
\eal
\ee
Here, the constant $\cC$ comes from Lemma \ref{w228l1} and also depends on $\cC_{\t_0}=\frac{\ln{T}}{\ln{\frac{1}{\t_0}}}$; in the second inequality, to bound the integral (after computation) we apply the fact that $\frac{1}{\nu(t)-t_l}\leq \frac{1}{t-\t-t_l}$.

For $I_{2}$, we proceed accordingly to \rf{w225e8},  and use \rf{w320e01} in Lemma \ref{w228l1} to get
\bel{w225e08}
\setlength\abovedisplayskip{3pt}
\setlength\belowdisplayskip{3pt}
\bal
I_{2}&\leq  \cC\sum_{n=l+1}^{N-1}\int_{t_n}^{t_{n+1}}\int_{t_{n}-\t }^{t_n} \|\bar u_h(s;t_l,z)\|\rd s \rd t  \\
&\q+\cC\sum_{n=l+1}^{N-1}\int_{t_n}^{t_{n+1}}\int_{t_l}^{t_{n}-\t} \frac{t-t_n}{t_n-s} \|\bar u_h(s;t_l,z)\|\rd s \rd t  \\
&\leq \cC \|z\|\Big[ \t
+\t \sum_{n=l}^{N-1}\int_{t_n}^{t_{n+1}}\int_{t_l}^{t_{n}-\t} \frac{1}{t_n-s}\rd s \rd t\Big] 
\leq \cC\|z\|  \t\ln{\frac 1 \t}\,. 
\eal
\ee

For $I_{3}$, by \rf{w225e1b}, it is easy to see that
\bel{w225e10}
\setlength\abovedisplayskip{5pt}
\setlength\belowdisplayskip{5pt}
I_{3}\leq \cC\t \max_{l\leq n\leq N-1}\| u^*_n(t_l,z)\| \leq \cC\t \|z\|\,.
\ee
A combination of \rf{w225e7}--\rf{w225e10} now yields the assertion.

\ss

{\bf (3) Verification of \rf{w225e11c}.} By \rf{pde-h} and \rf{w305e2}, we find that
\begin{eqnarray*}
&\|\bar x_h(t_n)- x^*_{n}(t_l,z)\|  
&\leq \big\| \h E_h(t_n-t_l)\Pi_h z-\cA_0^{n-l}\Pi_hz \big\| \nonumber\\
&&\q+\Big\| \sum_{j=l}^{n-1}\int_{(t_j, t_{j+1}]}\big[ \h E_h(t_{n}-t)\Pi_h-\cA_0^{n-j}\Pi_h\big]u^*_j(t_l,z)\rd t \Big\|\nonumber\\
&&=:J_1+J_2\,.
\end{eqnarray*}
We estimate these two terms independently.
By \rf{w1115e2c} in Lemma \ref{w207l1} with $\g=2$ and \rf{w320e01} in Lemma \ref{w228l1} with $\g=1$, for any $\e\in(0,\t)$  we have
\begin{equation*}
\setlength\abovedisplayskip{3pt}
\setlength\belowdisplayskip{3pt}
\bal
J_1
&\leq \big\| \big[\h E_h(\e)-\mathds{1}_h \big] \h E_h(t_n-t_l-\e)\Pi_h z \big\|+\big\|\h G_{\t}(t_n-t_l-\e)z  \big\|\\
&\leq \cC\frac{\e}{t_n-t_l-\e}\|z\|+\cC \t \frac{1}{t_n-t_l-\e}\|z\|\,,
\eal
\end{equation*}
where $\cC> 0$ is independent of $\e$. For the definition and properties of $\h G_\t(\cd)$, recall Section \ref{dis-lq} and Lemma \ref{w207l1}. 
Note that $\h E_h(t_n-t_l)z-\cA_0^{n-l}\Pi_hz\neq \h G_{\t}(t_n-t_l)z$.
Hence we divide it into two terms
\begin{equation*}
\setlength\abovedisplayskip{3pt}
\setlength\belowdisplayskip{3pt}
\bal
\big[\h E_h(t_n-t_l)-\h E_h(t_n-t_l-\e)\big]\Pi_h z\,, \qq \h E_h(t_n-t_l-\e)\Pi_hz-\cA_0^{n-l}\Pi_hz\,,
\eal
\end{equation*}
and the latter one is just $\h G_{\t}(t_n-t_l-\e)z$.
Since $\e>0$ is arbitrary, we conclude that
\begin{equation*}
\setlength\abovedisplayskip{3pt}
\setlength\belowdisplayskip{3pt}
\bal
J_1 \leq \cC \t \frac{1}{t_n-t_l}\|z\|\,.
\eal
\end{equation*}
Also,  by \rf{w225e1b} in Lemma \ref{w207l4} and \rf{w1115e2c} in Lemma \ref{w207l1},
\begin{equation*}
\setlength\abovedisplayskip{3pt}
\setlength\belowdisplayskip{3pt}
\bal
J_2 
&\leq  \sum_{j=l}^{n-1}\int_{(t_j, t_{j+1}]}\|\h G_\t (t_n-t)\|_{\cL(\dbL^2)}\max_{l\leq j\leq N-1} \|u^*_j(t_l,z)\|\rd t \\
&\leq \cC\|z\| \Big[\int_{t_l}^{t_n-\t}\frac \t {t_n-t}\rd t
+\int_{t_n-\t }^{t_n}\cC\rd t\Big]
\leq  \cC\|z\| \t \ln{\frac 1 \t}\,.
\eal
\end{equation*}

A combination of these two estimates now settles the assertion.
\end{proof}

\bl{w207l6}
For any $l=0,1,\cds,N-1$ and $z\in\dbL^2$,
let $\big(x^*_h(\cd;t_l,z), u^*_h(\cd; t_l,z)\big)$ be the optimal pair of Problem {\bf (LQ)$^{t_l;h}_{\tt aux}$}.
Let
\bel{w225e6}
\setlength\abovedisplayskip{3pt}
\setlength\belowdisplayskip{3pt}
\h u_n(t_l,z)\deq \frac 1 \t \int_{t_n}^{t_{n+1}}u^*_h(t;t_l,z)\rd t \qq \forall\,n=l,l+1,\cds,N-1\,,
\ee
and $\h x_\cd(t_l,z)\deq x_\cd\big(t_l,z,\h u_\cd(t_l,z)\big)$. Then there exists $\cC$ independent of $h$ and $\t$ such that
\begin{subequations}\label{w225e13}
    \begin{empheq}[left={\empheqlbrace\,}]{align}
      &\max_{0\leq l\leq N-1} \max_{l\leq n\leq N}\|\h x_n(t_l,z)\|\leq \cC\|z\| \,, \label{w225e13a}\\
      &\sum_{n=l}^{N-1}\!\int_{t_n}^{t_{n+1}}\!\big[\| x^*_h(t;t_l,z)\!-\! x^*_h(t_n;t_l,z)\|+\| x^*_h(t;t_l,z)\!-\! x^*_h(t_{n+1};t_l,z)\|\big]\rd t \notag\\
&\qq\qq\qq\qq\qq\leq \cC \t\ln{\frac 1 \t}\|z\| \,, \label{w225e13b}\\
& \|\h x_{n}(t_l,z)- x^*_h(t_{n};t_l,z)\|  
\leq \cC \t\Big[\frac 1 {t_n-t_l}+\ln{\frac 1 \t}\Big]\|z\| \q n=l+1,\cds,N \,. \label{w225e13c}
\end{empheq}
\end{subequations}
\el

\begin{proof}
{\bf (1) Verification of \rf{w225e13a}.}  
By Lemmata \ref{w207l2} and \ref{w1012l1}, for any $n=l,l+1,\cds,N$, we have 
\begin{equation*}
\setlength\abovedisplayskip{3pt}
\setlength\belowdisplayskip{3pt}
\bal
\|\h x_n(t_l,z)\|&=\Big\|\cA_0^{n-l}\Pi_hz+\sum_{k=l}^{n-1}\cA_0^{n-k}\Pi_h\int_{t_k}^{t_{k+1}}u^*_h(t;t_l,z)\rd t\Big\|\\
&\leq \cC\|\Pi_h z\|+\cC\sup_{t\in[t_l,T]}\|u^*_h(t;t_l,z)\|
\leq \cC\|z\|\,.
\eal
\end{equation*}

{\bf (2) Verification of \rf{w225e13b} and \rf{w225e13c}.}  As done for \rf{w225e11b} and \rf{w225e11c}.
\end{proof}

The following lemma gives a regularity result for the optimal control to Problem {\bf (LQ)$^{t_l}_{\tt aux}$}, which will be 
needed below. Its derivation uses a tool from the `open-loop approach' --- but now for {\em deterministic} control problems.

\bl{w226l1}
For a given uniform partition $I_\t$ of size $\t\in(0,\t_0]\subset(0,1)$,
for any $l=0,1,\cds,N-1$ and $z\in \dbL^2$,
suppose that $u_h^*(\cd;t_l,z)$ is the optimal control of Problem {\bf (LQ)$^{t_l;h}_{\tt aux}$}. Then, for $s,s_0\in [t_n,t_{n+1}), n=l,l+1,\cds,N-1$, 
there exists  a constant $\cC$ independent of $\t$ such that
\bel{w226e4}
\setlength\abovedisplayskip{3pt}
\setlength\belowdisplayskip{3pt}
\|u^*_h(s;t_l,z)-u^*_h(s_0;t_l,z)\|\leq \cC \t\Big[\frac{\a}{T-s\vee s_0}+\ln{\frac 1 \t}\Big] \|z\| \,,
\ee
where $s\vee s_0=\max\{s,s_0\}$ defined in \rf{w1025e1}.
\el

\begin{proof}
We adopt Pontryagin's maximum principle to verify the assertion. By the adjoint equation \rf{bpde-h},
it is easy to write its mild solution as
\begin{equation*}
\setlength\abovedisplayskip{3pt}
\setlength\belowdisplayskip{3pt}
y_h(s;t_l,z)
=-\a \h E_h(T-s) x^*_h(T; t_l,z)+
\int_s^T \h E_h(\th-s)   x_h^*(\th;t_l,z)\rd \th \q s\in[t_l,T]\,.
\end{equation*}
Then, with the help of \rf{w1013e14},
we arrive at (without loss of generality, we take $s_0\leq s$)
\begin{equation*}
\setlength\abovedisplayskip{3pt}
\setlength\belowdisplayskip{3pt}
\bal
&\|u^*_h(s;t_l,z)-u^*_h(s_0;t_l,z)\|
=\|y_h(s;t_l,z)-y_h(s_0;t_l,z)\|\\
&\quad \leq \a \big\|\big[\h E_h(T-s)-\h E_h(T-s_0)\big] x^*_h(T;t_l,z) \big\|\\
&\qq+\Big\|\int_s^T \big[\h E_h(\th-s_0)-\h E_h(\th-s) \big] x^*_h(\th;t_l,z)  \rd \th \Big\|\\
&\qquad +\Big\|\int_{s_0}^s  \h E_h(\th-s_0)  x^*_h(\th;t_l,z) \rd \th  \Big\|
=:\sum_{i=1}^3 I_i\,.
\eal
\end{equation*}
Lemmata \ref{w1012l1} and \ref{w207l1} lead to
\begin{equation*}
\setlength\abovedisplayskip{5pt}
\setlength\belowdisplayskip{5pt}
I_1\leq \cC \t\frac{\a}{T-s\vee s_0} \|z\|\,.
\end{equation*}
For $I_2\,,I_3$, we can proceed similarly to deduce that
\begin{equation*}
\setlength\abovedisplayskip{3pt}
\setlength\belowdisplayskip{3pt}
\bal
I_2
&\leq \cC \int_s^{(s+\t)\wedge T} \sup_{t\in[t_l,T]}\|x^*_h(t;t_l,z)\| \rd \th
+\cC \int_{(s+\t)\wedge T}^{T}\frac{s-s_0}{\th-s} \sup_{t\in[t_l,T]}\|x^*_h(t;t_l,z)\|  \rd \th  \\
&\leq \cC \t\ln{\frac 1 \t}\|z\| \,,
\eal
\end{equation*}
and
\begin{equation*}
\setlength\abovedisplayskip{3pt}
\setlength\belowdisplayskip{3pt}
\bal
I_3
\leq \cC (s-s_0)\sup_{t\in[t_l,T]}\|x^*_h(t;t_l,z)\| 
\leq \cC\t \|z\|\,.
\eal
\end{equation*}
Combining these estimates settles the assertion \rf{w226e4}.
\end{proof}

The following result is comparable to Theorem \ref{Riccati-rate-h}, which was for $\cP_h(\cd)$ from the semi-discrete Riccati equation \rf{dis-Riccati1-h}:
now we address $\cP_\cd$ from the {\em difference Riccati equation}  \rf{dis-Riccati1-h-t}. Its proof is based on the results that we obtain so far in this section.
\bt{Riccati-rate}
Let $\cP_h(\cd)$ resp.~$\cP_\cd$ be solutions to Riccati equations \rf{dis-Riccati1-h} resp.~\rf{dis-Riccati1-h-t}.
Then there exists a constant $\cC$ independent of $h$ and $\t$ such that
\bel{w207e9}
\setlength\abovedisplayskip{3pt}
\setlength\belowdisplayskip{3pt}
\|\cP_h(t_l)\Pi_h-\cP_l\Pi_h\|_{\cL(\dbL^2)}\leq \cC \t \Big[\frac{\a}{T-t_l}+\ln{\frac 1 \t}\Big] \q \forall\, l=0,1,\cds,N-1\,.
\ee
\et

\begin{proof}
For given $l=0,1,\cds,N-1$, and $z\in\dbL^2$, we divide the proof into the following two cases. 

\ss

\no {\bf Case (i)}. $(\cP_l\Pi_hz,z)_{\dbL^2}\leq (\cP_h(t_l)\Pi_hz,z)_{\dbL^2}\,.$

In this case,  by \rf{val-h-1}, \rf{val-2a}, and the definition of the cost functionals $\cG_h(\cd,\cd; \cd)$ and $\cG_{h,\t}(\cd,\cd; \cd)$,
as well as of $\bar u_h(\cd;t_l,z)$ and $\bar x_h(\cd;t_l,z)$ that are defined in Lemma \ref{w207l5}, we can deduce that
\begin{equation*}
\setlength\abovedisplayskip{3pt}
\setlength\belowdisplayskip{3pt}
\bal
0  &\leq  \big( \cP_h(t_l)\Pi_hz,z \big)_{\dbL^2}-( \cP_l\Pi_hz,z )_{\dbL^2}\\
&\leq 2\cG_h\big(t_l,z; \bar u_h(\cd;t_l,z)\big)-2\cG_{h,\t}\big(t_l,z; u^*_\cd(t_l,z)\big)\\
&=\sum_{n=l}^{N-1}\int_{t_n}^{t_{n+1}}\big[\|\bar x_h(t;t_l,z)\|^2-\|\bar x_h(t_{n};t_l,z)\|^2 \big] \rd t \\
&\q +\sum_{n=l}^{N-1}\int_{t_n}^{t_{n+1}}\big[\|\bar x_h(t_{n};t_l,z)\|^2- \| x^*_{n}(t_l,z)\|^2 \big]\rd t  \\
&\q+\a\big[\|\bar x_h(T;t_l,z)\|^2-\|x^*_N(t_l,z)\|^2\big]
=: \sum_{i=1}^3I_i\,.
\eal
\end{equation*}
Here we apply the fact that $\|\bar u_h(\cd;t_l,z)\|_{L^2(t_l,T;\dbL^2)}=\|u^*_\cd(t_l,z)\|_{\dbU(t_l,T)}$ (see \rf{w207e5}).
In what follows, we estimate the above three terms one by one.

For $I_1$, by \rf{w225e11a} and \rf{w225e11b} in Lemma \ref{w207l5}, we have
\begin{equation*}
\setlength\abovedisplayskip{3pt}
\setlength\belowdisplayskip{3pt}
\bal
I_1 &\leq \sum_{n=l}^{N-1}\int_{t_n}^{t_{n+1}}\big[\|\bar x_h(t;t_l,z)\|+\|\bar x_h(t_{n};t_l,z)\| \big] \big[\|\bar x_h(t;t_l,z)-\bar x_h(t_{n};t_l,z)\|\big]\rd t \\
&\leq \cC\|z\| \sum_{n=l}^{N-1}\int_{t_n}^{t_{n+1}}\big[\|\bar x_h(t;t_l,z)-\bar x_h(t_{n};t_l,z)\|\big]\rd t 
\leq  \cC\|z\|^2  \t \ln{\frac 1 \t}\,.
\eal
\end{equation*}
For $I_2$, by virtue of $\nu(\cd),\pi(\cd)$ which are defined in \rf{w827e1}, as well as Lemmata \ref{w207l4} and \ref{w207l5} it follows that
\begin{eqnarray*}
&I_2&\leq \int_{t_l}^T\big[\big\|\bar x_h\big(\nu(t);t_l,z\big)\big\|+ \big\| x^*_{\pi(t)-1}(t_l,z)\big\| \big]
 \big[\big\|\bar x_h\big(\nu(t);t_l,z\big)- x^*_{\pi(t)-1}(t_l,z)\big\|\big]\rd t \\
&&\leq \cC\|z\| \int_{(t_{l+2})\wedge T}^T  \big[\big\|\bar x_h\big(\nu(t);t_l,z\big)- x^*_{\pi(t)-1}(t_l,z)\big\|\big]\rd t
     +\cC\|z\|^2\int_{t_l}^{(t_{l+2})\wedge T}1\rd t\\
&&\leq \cC \t \|z\|^2 +\cC \|z\|^2\int_{(t_{l+2})\wedge T}^T\Big[\frac{\a}{\nu(t)-t_l}+\ln{\frac 1 \t}\Big]\rd t
\leq \cC \t\ln{\frac 1 \t} \|z\|^2 \,.
\end{eqnarray*}
For $I_3$, by Lemma \ref{w207l5}, it is obvious that
\begin{equation*}
\setlength\abovedisplayskip{3pt}
\setlength\belowdisplayskip{3pt}
\bal
I_3\leq \cC\t\Big[\frac{\a}{T-t_l}+\ln{\frac 1 \t}\Big]\|z\|^2\,.
\eal
\end{equation*}

A combination of the above estimates now leads to
\bel{w207e6}
\setlength\abovedisplayskip{3pt}
\setlength\belowdisplayskip{3pt}
0\leq  \big( \cP(t_l)z,z \big)_{\dbL^2}-\big( \cP_l\Pi_hz,z \big)_{\dbL^2}
\leq \cC\t\Big[\frac{\a}{T-t_l}+\ln{\frac 1 \t}\Big]\|z\|^2\,.
\ee


\no {\bf Case (ii)}. $(\cP_l\Pi_hz,z)_{\dbL^2}> (\cP_h(t_l)\Pi_hz,z)_{\dbL^2}.$

In a similar way as in Case (i), 
by Lemmata \ref{w1012l1} and \ref{w207l6}  we infer that
\begin{eqnarray*}
&0 &\leq \big( \cP_l\Pi_hz,z \big)_{\dbL^2}-\big( \cP_h(t_l)\Pi_h z,z \big)_{\dbL^2}\\
&&\leq \cC\|z \|\sum_{n=l}^{N-1}\int_{t_n}^{t_{n+1}} \|\h x_{n}(t_l,z)- x^*_h(t;t_l,z)\|  \rd t \\
&&\q + \cC\|z\|  \sum_{n=l}^{N-1}\int_{t_n}^{t_{n+1}}\Big\| \frac 1 \t\int_{t_n}^{t_{n+1}}u^*_h(\th;t_l,z)\rd \th -u^*_h(t;t_l,z)   \Big\|\rd t  \\
&&\q +\cC\|z\| \|\h x_N(t_l,z)-x^*_h(T;t_l,z)\|
=:\cC\|z\| \sum_{i=1}^3J_i\,.
\end{eqnarray*}

For $J_1$, the triangle inequality yields
\begin{equation*}
\setlength\abovedisplayskip{3pt}
\setlength\belowdisplayskip{3pt}
\bal
J_1&\leq \sum_{n=l}^{N-1}\int_{t_n}^{t_{n+1}} \|\h x_{n}(t_l,z)- x^*_h(t_{n};t_l,z)\|+\| x^*_h(t_{n};t_l,z)- x^*_h(t;t_l,z)\|  \rd t \,.\\
\eal
\end{equation*}
In the same vein as that to estimate $I_2$, and thanks to  \rf{w225e13c} and \rf{w225e13b},  we get
\begin{equation*}
\setlength\abovedisplayskip{3pt}
\setlength\belowdisplayskip{3pt}
\bal
J_{1}\leq \cC\t \ln{\frac 1\t}\|z\|\,.
\eal
\end{equation*}

For $J_2$, by 
utilizing  Lemmata \ref{w1012l1} and \ref{w226l1},
we have
\begin{equation}\label{w208e1}
\setlength\abovedisplayskip{3pt}
\setlength\belowdisplayskip{3pt}
\bal
J_2
&\leq
\int_{(T-2\t)\vee t_l}^{T}\cC\|z\|\rd t
+\int_{t_l}^{(T-2\t)\vee t_l }\| \h u_{\pi(t)-1}(t_l,z) - u^*_h(t;t_l,z) \|\rd t   \\
&\leq  \cC \t \|z\| 
+ \cC\t \|z\| \int_{t_l}^{(T-2\t)\vee t_l} \Big[\frac {\a}{T-\mu(t)}+\ln{\frac 1 \t } \Big]\rd t 
 \leq \cC\t \ln{\frac 1\t}\|z\|\,.
\eal
\end{equation}
Thanks to \rf{w225e13c} in Lemma \ref{w207l6}, we get
\begin{equation*}
\setlength\abovedisplayskip{3pt}
\setlength\belowdisplayskip{3pt}
I_3\leq \cC \t\Big[\frac \a {T-t_l}+\ln{\frac 1 \t}\Big]\|z\| \,.
\end{equation*}

A combination of these estimates then leads to
\bel{w207e7}
\setlength\abovedisplayskip{3pt}
\setlength\belowdisplayskip{3pt}
0\leq ( \cP_l\Pi_hz,z )_{\dbL^2}- \big( \cP_h(t_l)\Pi_hz,z \big)_{\dbL^2}
\leq \cC\|z\|^2 \t\Big[\frac{\a}{T-t_l}+\ln{\frac{1}{\t}}\Big]\,,
\ee
and estimates \rf{w207e6} and \rf{w207e7} together imply the the assertion \rf{w207e9}.
\end{proof}

\subsubsection{The second consistent difference Riccati equation}\label{Riccati-dis-3}

The error analysis in Sections \ref{Riccati-dis-h} and \ref{Riccati-dis-2} for Problem  {\bf (LQ)$^{t}_{\tt aux}$} assume
{\bf (H)} --- which is a non-checkable assumption in applications. Another limitation of Theorem \ref{Riccati-rate} is
that {\em no} convergence way be expected for the `initial' time step $t_l=t_{N-1}$ unless $\a=0$, since the estimate
 \rf{w207e9} at this time is of the form 
 \begin{equation*}
\setlength\abovedisplayskip{3pt}
\setlength\belowdisplayskip{3pt}
\bal
\|\cP_h(t_{N-1})\Pi_h-\cP_{N-1}\Pi_h\|_{\cL(\dbL^2)}\leq \cC\Big[\a+\t \ln{\frac 1 \t}\Big]\,,
\eal
\end{equation*}
such that the right-hand side does {\em not} tend to zero for $\t\to 0$. However, the estimate  \rf{w207e9}
is sufficient to derive the convergence rate for the spatio-temporal discretization of Problem {\bf (SLQ)}; see also 
\cite[Theorem 4.5]{Prohl-Wang24}.

In this section, we present different  spatio-temporal discretization schemes that are exempted  from both deficiencies of the 
{\em difference Riccati equation} \rf{dis-Riccati1-h-t} above:
\begin{enumerate}[ (a) ]

\item General values for $\b\in\dbR$ are allowed to scale noise. Hence assumption {\bf (H)} is {\em not} needed anymore.

\item Theorem \ref{Riccati-rate2} 
will be shown, which even ensures convergence rate for iterates of the new {\em difference Riccati equations} \rf{dif-Riccati}. 
\end{enumerate}

In the following, we present another discretization for differential Riccati equation \rf{dis-Riccati1-h}. Since we do {\em not}
assume  {\bf (H)} any more, we rewrite Riccati equation \rf{dis-Riccati1-h} and system \rf{pde-h} by using $\D_h$ instead of $\cA_h$ as follows
\bel{dis-Riccati1-h2}
\setlength\abovedisplayskip{3pt}
\setlength\belowdisplayskip{3pt}
\lt\{\!\!\!
\begin{array}{ll} 
\ds\cP_h'(t)\!+\!\D_h\cP_h(t)\!+\!\cP_h(t)\D_h\!+\!\b^2\cP_h(t)\!+\!\mathds{1}_h\!-\!\cP_h^2(t)=0\qq t\in [0,T]\,,\\
\ns\ds\cP_h(T)=\a\mathds{1}_h\,,
\end{array}
\rt.
\ee
and
\bel{pde-h2}
\setlength\abovedisplayskip{3pt}
\setlength\belowdisplayskip{3pt}
\lt\{\!\!\!
\begin{array}{ll} 
\ds x_h'(s)=\D_h x_h(s)+\frac{\b^2}{2}x_h(s)+\Pi_h u_h(s)  \qq  s \in (t,T]\,,\\
\ns\ds x_h(t)=\Pi_hz \,.
\end{array}
\rt.
\ee

In this section, to simplify notations, we still apply some that are used in former parts.

Now we approximate 
system \rf{pde-h2} and the cost functional \rf{cost-t-h} as follows:
For any given $l= 0,1,\cds,N-1$, and $z\in \dbL^2$
\bel{pde-h-t}
\setlength\abovedisplayskip{3pt}
\setlength\belowdisplayskip{3pt}
\lt\{\!\!\!
\begin{array}{ll} 
\ds x_{n+1}=\big[1+\frac {\b^2\t} 2\big]A_0x_{n}+\t A_0{\Pi_h u_n}  \qq n=l, l+1,\cds,N-1\,,\\
\ns\ds x_l=  \Pi_hz\,,
\end{array}
\rt.
\ee
where $A_0=\big(\mathds{1}_h-\t\D_h\big)^{-1}$, and the cost functional $\cG_{h,\t}(t_l,z; \cd)$ is defined in \rf{cost-ht1}.
Also, we can propose a discretized version of the family of Problem {\bf (LQ)}$_{\tt aux}^{t;h}$ as that proposed in Section
\ref{dis-lq}.\\
\no{\bf Problem (LQ)$^{t_{l};h,\t}_{\tt aux}$.} For any  given $l=0,1,\cds,N-1$ and $z\in \dbL^2$, search for 
$u^*_\cd\in \dbU(t_l,T)$ such that
\begin{equation*}
\setlength\abovedisplayskip{3pt}
\setlength\belowdisplayskip{3pt}
\cG_{h,\t}(t_l,z; u^*_\cd)=\inf_{u_\cd \in \dbU(t_l,T)}\cG_{h,\t}(t_l,z; u_\cd)=:V_{h,\t}(t_l,z)\,.
\end{equation*}

Different from difference Riccati equation \rf{dis-Riccati1-h-t}, we introduce the following one:
\bel{dif-Riccati}
\setlength\abovedisplayskip{3pt}
\setlength\belowdisplayskip{3pt}
\lt\{\!\!\!
\begin{array}{ll} 
\ds \bar \cP_n\!=\!\big[1\!+\!\frac {\b^2\t} 2\big]^2A_0\bar \cP_{n+1}A_0\!+\!\t\mathds{1}_h \!-\!\t \bar \cH_n \bar \cK_n^{-1}\bar \cH_n 
\qq n=0,1,\cds,N-1\,,\\
\ns\ds \bar \cP_N=\a\mathds{1}_h\,,\\
\ns\ds \bar \cH_n=\big[1+\frac {\b^2\t} 2\big]A_0\bar \cP_{n+1}A_0\,,\\
\ns\ds \bar \cK_n=\mathds{1}_h+\t A_0\bar \cP_{n+1}A_0\,, \quad \mbox{with} \quad A_0=\big(\mathds{1}_h-\t \D_h\big)^{-1}\, .
\end{array}
\rt.
\ee 
As  in Lemma \ref{w1021l1}, we can deduce that
Problem {\bf (LQ)}$^{t_{l};h,\t}_{\tt aux}$
has a unique optimal control which enjoys a {\em discrete feedback form}
\bel{feedback-2}
\setlength\abovedisplayskip{3pt}
\setlength\belowdisplayskip{3pt}
u^*_n=- \bar \cK_n^{-1}\bar \cH_n x^*_n \qquad n=l,l+1,\cds,N-1\,,
\ee
and the {\em discrete value function} is  given by
\bel{val-2}
\setlength\abovedisplayskip{3pt}
\setlength\belowdisplayskip{3pt}
V_{h,\t}(t_l,z)=\frac 1 2 (\bar \cP_l\Pi_h z,z )_{\dbL^2} \qquad  \forall\, z \in {\mathbb L}^2\,.
\ee
By \rf{feedback-2}, or with the same argument as in \rf{w1026e5}, we deduce that the optimal control $u^*_\cd$ to Problem {\bf(LQ)$^{t_l;h,\t}_{\tt aux}$} is $\dbV_h$-valued.

We may
follow  the procedure in Section \ref{Riccati-dis-2} to derive the following results. 
We only provide proofs where the argumentation differs.

\bl{w1013l1}
For any $l=0,1,\cds,N-1$, $z\in\dbL^2$, $u_\cd\in\dbU(t_l,T)$, and any $n=l,l+1,\cds,N$, it holds that
\bel{w1013e1}
\setlength\abovedisplayskip{3pt}
\setlength\belowdisplayskip{3pt}
x_n(t_l,z,u_\cd)=\big[1+\frac {\b^2\t} 2\big]^{n-l}A_0^{n-l}\Pi_hz+\t\sum_{k=l}^{n-1}\big[1+\frac {\b^2\t} 2\big]^{n-k-1}A_0^{n-k}\Pi_h u_k\,,
\ee
or equivalently
\bel{w1013e2}
\setlength\abovedisplayskip{3pt}
\setlength\belowdisplayskip{3pt}
x_n(t_l,z,u_\cd)=A_0^{n-l}\Pi_hz+\t\sum_{k=l}^{n-1} A_0^{n-k}\big[\frac {\b^2} 2 x_k(t_l,z,u_\cd)+\Pi_h u_k\big] \,.
\ee
\el

\bl{w1013l2}
Let $\bar \cP_\cd\equiv \{\bar \cP_n\}_{n=0}^N\subset \dbS_+(\dbL^2|_{\dbV_h})$ be the solution to difference Riccati equation \rf{dif-Riccati}, and
for any $l=0,1,\cds,N-1$,  $z\in\dbL^2$ let $\big(x^*_\cd(t_l,z), u^*_\cd(t_l,z)\big)$ be the optimal pair of Problem {\bf (LQ)$^{t_l; h,\t}_{\tt aux}$}.
Then, there exists a constant $\cC$ independent of $h$ and $\t$ such that
\begin{subequations}\label{w1013e3}
    \begin{empheq}[left={\empheqlbrace\,}]{align}
      & \max_{0\leq l\leq N-1}\|\bar \cP_l\|_{\cL(\dbL^2|_{\dbV_h})}\leq \cC \,, \label{w1013e3a}\\
      & \max_{0\leq l\leq N-1}\max_{l\leq n\leq N-1} \big[\| x^*_n(t_l,z)\|  +\|u^*_n(t_l,z)\|\big] \leq \cC\|z\|\,. \label{w1013e3b}
    \end{empheq}
\end{subequations}
\el

The following result states the discrete version of Pontryagin's maximum  principle for Problem {\bf (LQ)$^{t_l; h,\t}_{\tt aux}$}; the proof is 
similar to that of Theorem \ref{MP}.

\bt{MP2}
For any $l=0,1,\cds,N-1$, $z\in\dbL^2$,
the unique optimal pair $\big(x^*_\cd(t_l,z), u^*_\cd(t_l,z)\big)\in \dbX(t_l,T) \times \dbU(t_l,T)$ of
Problem {\bf (LQ)$^{t_l; h,\t}_{\tt aux}$}  solves the following coupled  equations  for $n=l,l+1,\cds,N-1$:
\beq
\lt\{
\bal
&x^*_{n+1}(t_l,z)=\big[1+\frac {\b^2\t} 2\big]^{n+1-l}A_0^{n+1-l}\Pi_hz+\t\sum_{k=l}^{n}\big[1+\frac {\b^2\t} 2\big]^{n-k}A_0^{n+1-k}\Pi_h u^*_k(t_l,z)\,,\\
&y_n=-\t \sum_{k=n+1}^{N-1} \big[1+\frac {\b^2\t} 2\big]^{k-1-n} A_0^{k-n} x^*_k (t_l,z) 
- \big[1+\frac {\b^2\t} 2\big]^{N-1-n}  A_0^{N-n} \a x^*_N(t_l,z)\,, \\
&x^*_l(t_l,z)=\Pi_h z\, ,
\eal
\rt.
\eeq
 together with the discrete optimality condition
\bel{w1026e6}     
u^*_n(t_l,z)-y_n=0 \, .
\ee
\et

\br{w1106r1}
By Theorem \ref{MP2}, we can deduce that $y_\cd$ solves the following difference equation
\bel{bpde-h-t2}
\setlength\abovedisplayskip{3pt}
\setlength\belowdisplayskip{3pt}
\lt\{
\bal
&y_{n+1}-y_n=-\t\D_h y_n- \frac{\b^2\t}{2}y_{n+1}+\t x^*_{n+1}(t_l,z)\,,\\
&y_N=-\frac{\a-\t}{1+\frac{\b^2}{2}\t}x^*_N(t_l,z)\,.
\eal
\rt.
\ee
Note that since the terminal condition $y_N\neq y_h(T)$, $y_\cd$ is different from the semi-implicit Euler method for \rf{bpde-h}.
\er

\bl{w1013l4}
For any $l=0,1,\cds,N-1$, $z\in\dbL^2$, suppose that $u^*_\cd(t_l,z)$ is the optimal control of Problem
{\bf (LQ)$^{t_l; h,\t}_{\tt aux}$}. Then there exists a constant $\cC$ independent of $h\,,\t$ such that
\begin{equation*}
\setlength\abovedisplayskip{3pt}
\setlength\belowdisplayskip{3pt}
\max_{0\leq l\leq N-1}\Big[\max_{l\leq n\leq N}\|\nb u^*_n(t_l,z)\|^2 +\t\sum_{n=l}^{N-1}\|\D_h u^*_n(t_l,z)\|^2\Big]
\leq \cC \|z\|_{\dbH_0^1}^2\,.
\end{equation*}
\el

\begin{proof}
Testing \rf{bpde-h-t2} by $\D_h y_n$, applying discrete Gronwall's inequality and then taking summation, we deduce that
\begin{equation*}
\setlength\abovedisplayskip{3pt}
\setlength\belowdisplayskip{3pt}
\bal
 \max_{l\leq n\leq N}\|\nb y_n\|^2+\t\sum_{n=l}^{N-1}\|\D_h y_n\|^2
&\leq \cC\Big[ \|\nb y_N\|^2+\t\sum_{n=l+1}^N\|x^*_k(t_l,z)\|^2\Big]\\
&\leq \cC\big[ \|\nb x^*_N(t_l,z)\|^2+\max_{l\leq n\leq N}\|x^*_n(t_l,z)\|^2\big]\,.
\eal
\end{equation*}
Similarly if we test \rf{pde-h-t} with $\D_h x^*_{n+1}(t_l,z)$, then by \rf{w1013e3b} in Lemma \ref{w1013l2}
\begin{equation*}
\setlength\abovedisplayskip{3pt}
\setlength\belowdisplayskip{3pt}
\bal
&\max_{l\leq n\leq N}\|\nb x^*_n(t_l,z)\|^2+\t\sum_{n=l}^{N-1}\|\D_h x_n^*(t_l,z)\|^2\\
&\qq\leq \cC\Big[ \|\nb x^*_l(t_l,z)\|^2+\t\sum_{k=l+1}^N\|u^*_k(t_l,z)\|^2\Big]\\
&\qq\leq \cC \|z\|_{\dbH_0^1}^2\,.
\eal
\end{equation*}
By combining with these two estimates and the optimality condition \rf{w1026e6}, we conclude that
\begin{equation*}
\setlength\abovedisplayskip{3pt}
\setlength\belowdisplayskip{3pt}
\bal
&\max_{l\leq n\leq N}\|\nb u^*_n(t_l,z)\|^2+\t\sum_{n=l}^{N-1}\|\D_h u^*_n(t_l,z)\|^2
\leq \cC \|z\|_{\dbH_0^1}^2\,,
\eal
\end{equation*}
where $\cC$ is independent of $l$.
That settles the assertion.
\end{proof}

\bl{w1013l3}
For any $l=0,1,\cds,N-1$ and $z\in\dbH_0^1$,
let $\big(x^*_\cd(t_l,z), u^*_\cd(t_l,z)\big)$ be the optimal pair of Problem {\bf (LQ)$^{t_l;h,\t}_{\tt aux}$}.
Suppose that 
\bel{w1013e4}
\setlength\abovedisplayskip{3pt}
\setlength\belowdisplayskip{3pt}
\bar u_h(t;t_l,z)=u^*_n(t_l,z) \qq \forall\, t\in [t_n,t_{n+1})\,,n=l,l+1,\cds,N-1\,,
\ee
and
$\bar x_h(\cd;t_l,z) \deq x_h\big(\cd;t_l,z, \bar u_h(\cd;t_l,z)\big)$. There exists a constant $\cC$ independent of $h\,,\t$ 
such that
\begin{subequations}\label{w1013e5}
    \begin{empheq}[left={\empheqlbrace\,}]{align}
    & \max_{0\leq l\leq N-1}\Big[\sup_{t\in[t_l,T]}\|\bar x_h(t;t_l,z)\|^2_{\dbH_0^1}+\!\int_{t_l}^T\!\|\D_h \bar x_h(t;t_l,z)\|^2\rd t\Big] 
    \leq \cC\|z\|_{\dbH_0^1}\,,\label{w1013e5a}\\
      &\sum_{n=l}^{N-1}\!\int_{t_n}^{t_{n+1}}\!\big[\|\bar x_h(t;t_l,z)\!-\!\bar x_h(t_n;t_l,z)\|+\|\bar x_h(t;t_l,z)\!-\!\bar x_h(t_{n+1};t_l,z)\|\big]\rd t \notag\\
&\qq\qq\qq\qq\qq\leq \cC \t\|z\|_{\dbH_0^1} \,,\label{w1013e5b}\\
&\|\bar x_h (t_{n};t_l,z)-x^*_{n}(t_l,z)\|\leq \cC\frac{\t}{\sqrt{t_n-t_l}}\|z\|_{\dbH_0^1} \q n=l+1,l+2,\cds,N\,. \label{w1013e5c}
\end{empheq}
\end{subequations}
\el

\begin{proof}
The proof is similar to that of Lemma \ref{w207l5}.
For simplicity, we write $\bar x_h(\cd)$ for $\bar x_h(\cd;t_l,z)$ throughout the proof.

\ss

{\bf (1) Verification of \rf{w1013e5a}.} A standard stability estimate for the solution to equation \rf{pde-h2}, and 
Lemma \ref{w1013l2} lead to
\begin{equation*}
\setlength\abovedisplayskip{3pt}
\setlength\belowdisplayskip{3pt}
\bal
\sup_{t\in[t_l,T]}\|\bar x_h(t;t_l,z)\|^2_{\dbH_0^1}+\int_{t_l}^T\|\D_h \bar x_h(t;t_l,z)\|^2\rd t 
\leq \cC \Big[ \|\Pi_h z\|_{\dbH_0^1}^2+\int_{t_l}^T\|\bar u_h(t;t_l,z)\|^2\rd t\Big]
\leq \cC\|z\|_{\dbH_0^1}^2\,,
\eal
\end{equation*}
where $\cC$ is independent of $l$.
That settles the assertion.

\ss
{\bf (2) Verification of \rf{w1013e5b}.}  We only prove 
\begin{equation*}
\setlength\abovedisplayskip{3pt}
\setlength\belowdisplayskip{3pt}
\bal
I:=\sum_{n=l}^{N-1}\int_{t_n}^{t_{n+1}}\|\bar x_h(t)-\bar x_h(t_n)\|\rd t\leq \cC \t \|z\|_{\dbH_0^1}\,,
\eal
\end{equation*}
and the remaining part can be obtained similarly.
By \rf{pde-h2}, we find that
\begin{eqnarray}
&I &\leq \sum_{n=l}^{N-1}\int_{t_n}^{t_{n+1}}\big\|\big[E_h(t-t_l)-E_h(t_n-t_l) \big]\Pi_h z\big\|\rd t \nonumber\\
&&\q +\sum_{n=l}^{N-1}\int_{t_n}^{t_{n+1}}\Big\|\int_{t_l}^{t_n} \big[E_h(t-\th)-E_h(t_n-\th) \big] \big[\frac {\b^2} 2 \bar x_h(\th)+\bar u_h(\th; t_l,z)\big]\rd \th \Big\|\rd t\nonumber\\
&&\q +\sum_{n=l}^{N-1}\int_{t_n}^{t_{n+1}}\Big\|\int_{t_n}^{t} E_h(t-\th)\big[\frac {\b^2} 2 \bar x_h(\th)+u^*_n(t_l,z)\big]\rd \th\Big\|\rd t \nonumber\\
&&=: I_1+I_2+I_3\,. \label{w1013e6}
\end{eqnarray}

Now we estimate summands on the right-hand side of \rf{w1013e6}. For $I_1$, by virtue of $\nu(\cd)$ defined in \rf{w827e1}, Remark \ref{w1006r1}, and \rf{w320e01} with $\g=0$ 
as well as \rf{w228e2} with $\g=1/2$, \rf{w228e3} with $\g=1$ in Lemma \ref{w228l1}, it follows that
\bel{w1013e7}
\setlength\abovedisplayskip{3pt}
\setlength\belowdisplayskip{3pt}
I_{1} 
\leq \int_{t_{l}}^{t_{l+1}}\cC\|z\|\rd t
 +\int_{t_{l+1}}^T\cC \frac{t-\nu(t)}{\sqrt{t-\t-t_l}}\|z\|_{\dbH_0^1}\rd t
\leq \cC\t \|z\|_{\dbH_0^1}\,.
\ee
For $I_{2}$,   by Lemma \ref{w228l1}, Remark \ref{w1006r1}, assertion \rf{w1013e5a} and Lemma \ref{w1013l4}, it follows that
\bel{w1013e8}
\setlength\abovedisplayskip{3pt}
\setlength\belowdisplayskip{3pt}
I_2
\leq \cC\sum_{n=l}^{N-1}\int_{t_n}^{t_{n+1}}\int_{ t_l}^{t_n} \t \big[\|\D_h \bar x_h(\th)\|+\|\D_h \bar u_h(\th;t_l,z)\|\big]\rd \th \rd t
\leq \cC\t \|z\|_{\dbH_0^1}\,.
\ee
For $I_{3}$, by \rf{w1013e3b} and \rf{w1013e5a}, it is easy to see that
\bel{w1013e9}
\setlength\abovedisplayskip{3pt}
\setlength\belowdisplayskip{3pt}
I_{3}\leq \cC\t \Big[\sup_{t\in[t_l,T]}\|\bar x_h(t)\| +\max_{l\leq n\leq N-1}\| u^*_n(t_l,z)\| \Big]
\leq \cC\t \|z\|_{\dbH_0^1}\,.
\ee
Combining with \rf{w1013e6}--\rf{w1013e9}, we prove the second assertion.

\ss

{\bf (3) Verification of \rf{w1013e5c}.} By \rf{pde-h2} and \rf{pde-h-t}, we find that
\begin{equation*}
\setlength\abovedisplayskip{3pt}
\setlength\belowdisplayskip{3pt}
\bal
\|\bar x_h(t_{n+1})- x^*_{n+1}(t_l,z)\|  
&\leq \big\| E_h(t_{n+1}-t_l)\Pi_h z-A_0^{{n+1}-l}\Pi_hz \big\| \nonumber\\
&\q+\Big\| \sum_{j=l}^{n}\int_{t_j}^{t_{j+1}}\big[ E_h(t_{n+1}-t)-A_0^{n+1-j}\Pi_h\big]u^*_j(t_l,z)\rd t \Big\|\nonumber\\
&\q +\frac {\b^2} 2\Big\| \sum_{j=l}^{n} \int_{t_j}^{t_{j+1}} \big[ E_h(t_{n+1}-t)-A_0^{n+1-j}\Pi_h\big] \bar x_h(t) \rd t \Big\| \nonumber\\
&\q+\frac {\b^2} 2 \Big\| \sum_{j=l}^{n}\int_{t_j}^{t_{j+1}} A_0^{n+1-j}\Pi_h \big[\bar x_h(t)-\bar x_h (t_j)\big]  \rd t \Big\| \nonumber\\
&\q +\frac {\b^2} 2 \Big\| \sum_{j=l}^{n}\int_{t_j}^{t_{j+1}}A_0^{n+1-j}\Pi_h  \big[\bar x_h(t_j)- x^*_j(t_l,z)\big]\rd t \Big\| \nonumber\\
&=:\sum_{i=1}^5 J_i\,.
\eal
\end{equation*}

Now, we estimate these five terms one by one.
By \rf{w1115e2b} with $\g=2$ in Lemma \ref{w207l1} and \rf{w320e01} in Lemma \ref{w228l1} with $\g=1$,
following the vein to estimate $J_1$ in the proof of Lemma \ref{w207l5}, 
%
we conclude that
\begin{equation*}
\setlength\abovedisplayskip{3pt}
\setlength\belowdisplayskip{3pt}
J_1 \leq \cC\frac{\t}{\sqrt{t_{n+1}-t_l}}\|z\|_{\dbH_0^1}\,.
\end{equation*}
Also, by \rf{w1115e2b} with $\g=2$ in Lemma \ref{w207l1}  and Lemma \ref{w1013l4},
\begin{equation*}
\setlength\abovedisplayskip{3pt}
\setlength\belowdisplayskip{3pt}
J_2 
\leq \int_{t_l}^{t_{n+1}}\|G_\t(t_{n+1}-t)\bar u_h(t;t_l,z) \| \rd t
\leq  \cC\t \|z\|_{\dbH_0^1} \,.
\end{equation*}
By the same trick,  applying \rf{w1013e5a}, we can deduce that
\begin{equation*}
\setlength\abovedisplayskip{3pt}
\setlength\belowdisplayskip{3pt}
J_3\leq \cC\t \|z\|_{\dbH_0^1} \,.
\end{equation*}
By \rf{w1013e5b}, it is evident that 
\begin{equation*}
\setlength\abovedisplayskip{3pt}
\setlength\belowdisplayskip{3pt}
J_4 \leq \cC \sum_{j=l}^{n}\int_{t_j}^{t_{j+1}}\|\bar x_h(t)-\bar x_h(t_j)\|\rd t\leq \cC \t\|z\|_{\dbH_0^1}  \,.
\end{equation*}
For $J_5$, we can easily see that
\begin{equation*}
\setlength\abovedisplayskip{3pt}
\setlength\belowdisplayskip{3pt}
 J_5  
\leq \cC \t  \sum_{j=l}^{n} \|\bar x_h(t_j)- x^*_j(t_l,z)\|
=\cC \t  \sum_{j=l+1}^{n} \|\bar x_h(t_j)- x^*_j(t_l,z)\| \,.
\end{equation*}

By setting $e_n= \|\bar x_h(t_n)- x^*_n(t_l,z)\| $ and combining with above estimates, we have
\begin{equation*}
\setlength\abovedisplayskip{3pt}
\setlength\belowdisplayskip{3pt}
e_{n+1}\leq \cC\frac {\t}{\sqrt{t_{n+1}-t_l}}\|z\|_{\dbH_0^1}+ \cC \t  \sum_{j=l+1}^{n} e_j\,,
\end{equation*}
which, together with discrete Gronwall's inequality, yields 
\begin{equation*}
\setlength\abovedisplayskip{3pt}
\setlength\belowdisplayskip{3pt}
e_{n+1}
\leq \cC\frac {\t}{\sqrt{t_{n+1}-t_l}}\|z\|_{\dbH_0^1} +\cC\t \sum_{j=l+1}^{n}\frac{\t}{\sqrt{t_j-t_l}} \|z\|_{\dbH_0^1} 
\leq \cC\frac {\t}{\sqrt{t_{n+1}-t_l}}\|z\|_{\dbH_0^1} \,.
\end{equation*}
That completes the proof.
\end{proof}

The proof of the following regularity result on $u^*_h(\cd;t,z)$ can be derived in the same vein as that of Lemma \ref{w1013l4}, and we omit it. 
\bl{w1014l1}
For any $t\in[0,T)$, $z\in\dbL^2$, suppose that $u^*_h(\cd;t,z)$ is the optimal control of Problem
{\bf (LQ)$^{t; h}_{\tt aux}$}. Then there exists a constant $\cC$ independent of $h$ such that
\begin{equation*}
\setlength\abovedisplayskip{3pt}
\setlength\belowdisplayskip{3pt}
\sup_{s\in[t,T]}\|\nb u^*_h(s;t,z)\|^2 +\int_{t}^T\|\D_h u^*_h(s;t,z)\|^2\rd s
\leq \cC \|z\|_{\dbH_0^1}^2\,.
\end{equation*}
\el

The following result is similar to Lemma \ref{w207l6} and can be proved in a similar way as Lemma \ref{w1013l3}. 

\bl{w1014l2}
For any $l=0,1,\cds,N-1$ and $z\in\dbL^2$,
let $\big(x^*_h(\cd;t_l,z), u^*_h(\cd; t_l,z)\big)$ be the optimal pair of Problem {\bf (LQ)$^{t_l;h}_{\tt aux}$}.
Suppose that 
\bel{w1014e1}
\setlength\abovedisplayskip{3pt}
\setlength\belowdisplayskip{3pt}
\h u_n(t_l,z)=\frac 1 \t \int_{t_n}^{t_{n+1}}u^*_h(t;t_l,z)\rd t \qq \forall\,n=l,l+1,\cds,N-1\,,
\ee
and $\h x_\cd(t_l,z)\deq x_\cd\big(t_l,z,\h u_\cd(t_l,z)\big)$. Then
\begin{subequations}\label{w1014e2}
    \begin{empheq}[left={\empheqlbrace\,}]{align}
      &\max_{0\leq l\leq N-1}\Big[\max_{l\leq n\leq N}\|\nb \h x_n(t_l,z)\|^2 +\t\sum_{n=l}^{N-1}\|\D_h \h x_n(t_l,z)\|^2\Big]
\leq \cC \|z\|_{\dbH_0^1}^2\,, \label{w1014e2a}\\
     &\int_{t_l}^{T}\big[\big\| x^*_h(t;t_l,z)- x^*_h\big(\nu(t);t_l,z\big)\big\|+\big\| x^*_h(t;t_l,z)- x^*_h\big(\mu(t);t_l,z\big)\big\|\big]\rd t \notag\\
&\qq\qq\qq     \leq \cC \t \|z\|_{\dbH_0^1} \,, \label{w1014e2b}\\
& \|\h x_{n}(t_l,z)\!-\! x^*_h(t_{n};t_l,z)\|  
\!\leq \!\cC \frac \tau {\sqrt{t_n-t_l}}\|z\|_{\dbH_0^1} \qq n=l+1,l+2,\cds,N \,. \label{w1014e2c}
\end{empheq}
\end{subequations}
\el

\begin{proof}
{\bf (1) Verification of \rf{w1014e2a}.}  
Derived by Lemma \ref{w1014l1} and the same vein as in the proof of Lemma \ref{w1013l4}.

\ss

{\bf (2) Verification of \rf{w1014e2b} and \rf{w1014e2c}.}  
As done for \rf{w1013e5b} and \rf{w1013e5c}.
\end{proof}

The following result can be proved in the same way as that in the proof of Lemma \ref{w226l1}.
\bl{w1014l3}
For a given uniform partition $I_\t$ of size $\t\in(0,\t_0]\subset(0,1)$,
for any $l=0,1,\cds,N-1$ and $z\in \dbL^2$,
suppose that $u^*_h(\cd;t_l,z)$ is the optimal control of Problem {\bf (LQ)$^{t_l;h}_{\tt aux}$}. Then, for $s,s_0\in [t_n,t_{n+1}), n=l,l+1,\cds,N-1$, 
there exists  a constant $\cC$ independent of $\t$ such that
\bel{w1014e3}
\setlength\abovedisplayskip{3pt}
\setlength\belowdisplayskip{3pt}
\|u^*_h(s;t_l,z)-u^*_h(s_0;t_l,z)\|\leq \cC \t \frac{1}{\sqrt{T-s\vee s_0}} \|z\|_{\dbH_0^1} \,.
\ee
\el

Now we are in the position to bound the error of $\bar \cP_l-\cP_h(t_l)$.

\bt{Riccati-rate2}
Suppose that $\cP_h(\cd)$ and $\bar \cP_\cd$ are solutions to Riccati equations \rf{dis-Riccati1-h2} and \rf{dif-Riccati} respectively.
Then there exists a constant $\cC$ independent of $h$ and $\t$ such that
\bel{w1014e6}
\setlength\abovedisplayskip{3pt}
\setlength\belowdisplayskip{3pt}
\|\cP_h(t_l)\Pi_h- \bar \cP_l\Pi_h\|_{\cL(\dbH_0^1;\dbL^2)}
\leq \cC \t \Big[\frac{\a}{\sqrt{T-t_l}}+1\Big] \qq l=0,1,\cds,N-1\,.
\ee
\et

\begin{proof}
The proof is long and we divide it into two steps.

{\bf (1)} 
In this step, we show that there exists a constant $\cC$ such that for any $z\in\dbH_0^1$,  $l=0,1,\cds,N-1$,
\bel{w1014e4}
\setlength\abovedisplayskip{3pt}
\setlength\belowdisplayskip{3pt}
\big|\big(\bar \cP_l\Pi_hz,z\big)_{\dbL^2}- \big(\cP_h(t_l)\Pi_hz,z\big)_{\dbL^2}\big|
\leq \cC \t \Big[\frac{\a}{\sqrt{T-t_l}}+1\Big]\|z\| \|z\|_{\dbH_0^1} \,.
\ee
To do that, we proceed as in the proof of Theorem \ref{Riccati-rate}.
For given $l = 0, 1,\cds,N-1$, and $z\in\dbH_0^1$, we divide the proof into the following two cases.

\no {\bf Case (i)}. $\big(\bar \cP_l\Pi_hz,z\big)_{\dbL^2}\leq \big(\cP_h(t_l)\Pi_hz,z\big)_{\dbL^2}\,.$

By relying on Lemmata \ref{w1013l2} and \ref{w1013l3}, we can conclude that
\begin{equation*}
\setlength\abovedisplayskip{3pt}
\setlength\belowdisplayskip{3pt}
0\leq  \big( \cP_h(t_l)\Pi_hz,z \big)_{\dbL^2}-\big( \bar \cP_l\Pi_hz,z \big)_{\dbL^2}
\leq \cC\t\Big[\frac{\a}{\sqrt{T-t_l}}+1\Big]\|z\| \|z\|_{\dbH_0^1}\,.
\end{equation*}

\no {\bf Case (ii)}. $\big(\bar \cP_l\Pi_hz,z\big)_{\dbL^2}> \big(\cP_h(t_l)\Pi_hz,z\big)_{\dbL^2}.$

In the same vein as  in Case (ii) of Theorem \ref{Riccati-rate}, 
by Lemmata \ref{w1014l2} and \ref{w1014l3}, we can deduce that
\begin{equation*}
\setlength\abovedisplayskip{3pt}
\setlength\belowdisplayskip{3pt}
\bal
0 &\leq \big( \bar \cP_l\Pi_hz,z \big)_{\dbL^2}-\big( \cP_h(t_l)\Pi_h z,z \big)_{\dbL^2}\\
&\leq \cC\|z \|\bigg\{\sum_{n=l}^{N-1}\Big[\int_{t_n}^{t_{n+1}} \|\h x_{n}(t_l,z)- x^*_h(t_n;t_l,z)\| 
+ \| x^*_h(t_n; t_l,z)- x^*_h(t;t_l,z)\|  
  \\
&\q +\Big\| \frac 1 \t\int_{t_n}^{t_{n+1}}u^*_h(s;t_l,z)\rd s -u^*_h(t;t_l,z)   \Big\|\Big]\rd t +\a \|\h x_N(t_l,z)-x^*_h(T;t_l,z)\|\bigg\}\\
&=:\cC\|z\| \sum_{i=1}^4J_i\,.
\eal
\end{equation*}

For $J_1$, \rf{w1014e2c} in Lemma \ref{w1014l2} and the triangle inequality yield
\begin{equation*}
\setlength\abovedisplayskip{3pt}
\setlength\belowdisplayskip{3pt}
\bal
J_1\leq \cC \t\|z\|+ \cC\sum_{n=l+2}^{N-1}\int_{t_n}^{t_{n+1}} \frac {1}{\sqrt{t-\t-t_l}}\|z\|_{\dbH_0^1}  \rd t
\leq   \cC \t\|z\|_{\dbH_0^1} \,.
\eal
\end{equation*}
\rf{w1014e2b} and \rf{w1014e2c} in Lemma \ref{w1014l2} lead to
\begin{equation*}
\setlength\abovedisplayskip{3pt}
\setlength\belowdisplayskip{3pt}
\bal
J_2\leq  \cC \t\|z\|_{\dbH_0^1}\,, \qq J_4\leq \cC\t \frac{\a}{\sqrt{T-t_l}}\|z\|_{\dbH_0^1}\,.
\eal
\end{equation*}
Lemma \ref{w1014l3} implies that 
\begin{equation*}
\setlength\abovedisplayskip{3pt}
\setlength\belowdisplayskip{3pt}
\bal
J_3\leq  \frac 2 \t \sum_{n=l}^{N-1} \int_{t_n}^{t_{n+1}}\int_{t_n}^{s}\|u^*_h(s;t_l,z) -u^*_h(t;t_l,z) \| \rd t\rd  s
\leq \cC \t\|z\|_{\dbH_0^1}\,.
\eal
\end{equation*}
A combination of these estimates leads to
\begin{equation*}
\setlength\abovedisplayskip{3pt}
\setlength\belowdisplayskip{3pt}
\bal
0\leq \big( \bar \cP_l\Pi_hz,z \big)_{\dbL^2}-\big( \cP_h(t_l)\Pi_hz,z \big)_{\dbL^2}
\leq \cC\t\Big[\frac{\a}{\sqrt{T-t_l}}+1\Big]\|z\| \|z\|_{\dbH_0^1}\,.
\eal
\end{equation*}

Combining with these two cases, we get the desired assertion.

\ss

{\bf  (2)} 
In this step, we adopt a spectral decomposition technique to derive sharp estimates.
Based on the eigenvalues $\{\l_{h,i}\}_{i=1}^{\dim(\dbV_h)}$ resp.~($\dbL^2$-orthonormal) eigenfunctions $\{\f_{h,i}\}_{i=1}^{\dim(\dbV_h)}$,
of the symmetric operator $-\D_h$, for any $i=1,2,\cds,\dim(\dbV_h)$, we define a family of LQ problems.\\
\no {\bf Problem (LQ)$^{t;h,i}_{\tt aux}$.} For any  given $t\in[0,T)$ and $x\in \dbR$, search for 
$u^*_{h,i}(\cd)\in L^2(t,T)$ such that
\begin{equation*}
\setlength\abovedisplayskip{3pt}
\setlength\belowdisplayskip{3pt}
\bal
\cG_{h}^i\big(t,x; u^*_{h,i}(\cd)\big)=\inf_{u_h(\cd) \in L^2(t,T)}\cG_{h}^i \big(t,x; u_h(\cd)\big)=:V_{h}^i(t,x)\,,
\eal
\end{equation*}
where the cost functional is
\begin{equation*}
\setlength\abovedisplayskip{3pt}
\setlength\belowdisplayskip{3pt}
\bal
\cG_{h}^i\big(t,x; u_h(\cd)\big)=\frac 1 2 \int_t^T\big[ | x_{h,i}(s) |^2+ |u_h(s) |^2 \big] \rd s +\frac \a 2 |x_{h,i}(T)|^2\,,
\eal
\end{equation*}
and the state variable $x_{h,i}(\cd)$ satisfies
\begin{equation*}
\setlength\abovedisplayskip{3pt}
\setlength\belowdisplayskip{3pt}
\lt\{\!\!\!
\begin{array}{ll} 
\ds x'_{h,i}(s)=-\l_{h,i}x_{h,i}(s)+\frac{\b^2}{2} x_{h,i}(s)+u_h(s)  \qq  s \in (t,T]\,,\\
\ns\ds x_{h,i}(t)=x\in\dbR \,.
\end{array}
\rt.
\end{equation*}

By LQ theory, Problem {\bf (LQ)}$_{\tt aux}^{t;h,i}$ has a unique optimal pair $\big(x^*_{h,i}(\cd), u^*_{h,i}(\cd)\big)$,
satisfying  the following state feedback representation:
\begin{equation*}
\setlength\abovedisplayskip{5pt}
\setlength\belowdisplayskip{5pt}
u^*_{h,i}(s) =-p_{h,i}(s) x^*_{h,i}(s) \qquad  s \in [t,T]\,,
\end{equation*}
where $p_{h,i}(\cd)$ solves the following scalar Riccati equation
\begin{equation*}
\setlength\abovedisplayskip{3pt}
\setlength\belowdisplayskip{3pt}
\lt\{\!\!\!
\begin{array}{ll} 
\ds p'_{h,i}(t)-2\l_{h,i}  p_{h,i}(t) +\b^2 p_{h,i}(t)+1- p_{h,i}^2(t)=0\qq t\in [0,T]\,,\\
\ns\ds p_{h,i}(T)=\a\,.
\end{array}
\rt.
\end{equation*}

By the uniqueness of the optimal controls to Problem {\bf (LQ)$_{\tt aux}^{t;h}$} and Problem {\bf (LQ)$_{\tt aux}^{t;h,i}$}
we conclude for any $z\in\dbL^2$ with $\Pi_h z=\sum_{i=1}^{\dim(\dbV_h)}(\Pi_h z, \f_{h,i})_{\dbL^2}\f_{h,i}=: \sum_{i=1}^{\dim(\dbV_h)} a_i \f_{h,i}$ that
\begin{equation*}
\setlength\abovedisplayskip{3pt}
\setlength\belowdisplayskip{3pt}
x^*_h(\cd;t,z)=\sum_{i=1}^{\dim(\dbV_h)} x^*_{h,i}\big(\cd;t,a_i\big)\f_{h,i}\,, \qq
u^*_h(\cd;t,z)=\sum_{i=1}^{\dim(\dbV_h)} u^*_{h,i}\big(\cd;t,a_i\big)\f_{h,i} \,,
\end{equation*}
and
\begin{equation*}
\setlength\abovedisplayskip{3pt}
\setlength\belowdisplayskip{3pt}
\cP_h(\cd)\f_{h,i}=p_{h,i}(\cd)\f_{h,i} \qq i=1,2,\cds,\dim(\dbV_h)\,.
\end{equation*}

Similarly, we can define Problem {\bf (LQ)$_{\tt aux}^{t_l;h,\t,i}$}, derive $p_{\cd,i}$, and for the optimal pair 
$(x^*_{\cd,i}, u^*_{\cd,i})$
\begin{equation*}
\setlength\abovedisplayskip{3pt}
\setlength\belowdisplayskip{3pt}
x^*_\cd(t_l,z)=\sum_{i=1}^{\dim(\dbV_h)} x^*_{\cd,i}\big(t_l,a_i\big)\f_{h,i} \,,\qq
u^*_\cd(t_l,z)=\sum_{i=1}^{\dim(\dbV_h)} u^*_{\cd,i}\big(t_l,a_i\big)\f_{h,i} \,,
\end{equation*}
and
\begin{equation*}
\setlength\abovedisplayskip{3pt}
\setlength\belowdisplayskip{3pt}
\bar \cP_\cd \f_{h,i}=p_{\cd,i}\f_{h,i}\qq i=1,2,\cds,\dim(\dbV_h)\,.
\end{equation*}
as well as 
\bel{w1014e7}
\setlength\abovedisplayskip{3pt}
\setlength\belowdisplayskip{3pt}
\max_{0\leq l\leq N-1}\max_{l\leq n\leq N}|p_{n,i}|\leq \cC\,.
\ee

By taking $z=\f_{h,i}$ in \rf{w1014e4}, we have
\bel{w1014e5}
\setlength\abovedisplayskip{3pt}
\setlength\belowdisplayskip{3pt}
\bal
&|p_{l,i}-p_{h,i}(t_l)|=\big|(p_{l,i}\f_{h,i},\f_{h,i})_{\dbL^2}- \big(p_{h,i}(t_l)\f_{h,i},\f_{h,i}\big)_{\dbL^2}\big|\\
&\qq=\big|\big(\bar \cP_l\Pi_h\f_{h,i},\f_{h,i}\big)_{\dbL^2}- \big(\cP_h(t_l)\Pi_h\f_{h,i},\f_{h,i}\big)_{\dbL^2}\big|\\
&\qq\leq \cC \l_{h,i}^{1/2} \t \Big[\frac{\a}{\sqrt{T-t_l}}+1\Big]\,.
\eal
\ee
Thus for any $z\in\dbH_0^1$,
we can deduce that 
\begin{equation*}
\setlength\abovedisplayskip{3pt}
\setlength\belowdisplayskip{3pt}
\bal
&\big\|\big[\bar \cP_l\Pi_h-\cP_h(t_l)\Pi_h\big]z\big\|^2
=\Big\| \sum_{i=1}^{\dim(\dbV_h)} a_i \big[\bar \cP_l\Pi_h-\cP_h(t_l)\Pi_h\big]\f_{h,i} \Big\|^2\\
&\q\leq \cC \sum_{i=1}^{\dim(\dbV_h)} a_i^2 \l_{h,i} \t^2 \Big[\frac{\a}{\sqrt{T-t_l}}+1\Big]^2
\leq \cC\|z\|_{\dbH_0^1}^2 \t^2 \Big[\frac{\a}{\sqrt{T-t_l}}+1\Big]^2\,,
\eal
\end{equation*}
which settles the assertion \rf{w1014e6}.
\end{proof}

\br{w1015r1}
With the procedure in this Section, we can also deduce that, for any $l=0,1,\cds,N-1$,
\begin{equation*}
\setlength\abovedisplayskip{3pt}
\setlength\belowdisplayskip{3pt}
\bal
\|\bar \cP_l\Pi_h- \cP_h(t_l)\Pi_h\|_{\cL(\dot\dbH^2_h;\dbL^2)}
+\| \h \cP_l\Pi_h- \cP_h(t_l)\Pi_h\|_{\cL(\dot\dbH^2_h;\dbL^2)}
\leq \cC \t\,, 
\eal
\end{equation*}
which will be applied in Section \ref{SLQ-rate}.
We leave the proof to the interested reader. 
\er

\subsection{Riccati-based discretization of Problem {\bf (SLQ)} and rates }\label{SLQ-rate}

From Section \ref{se-openclosed}, we know that the optimal control $U^*(\cd)$ of Problem {\bf (SLQ)} 
has a state feedback representation
\bel{feedback-4}
\setlength\abovedisplayskip{3pt}
\setlength\belowdisplayskip{3pt}
U^*(t)=-\cP(t)X^*(t)-\eta(t) \qq t\in[0,T]\,,
\ee
where $\cP(\cd)$ solves Riccati equation \rf{Riccati}, and $\eta(\cd)$ solves a $\cP(\cd)$-dependent PDE \rf{varphi}.

In Section \ref{Riccati-dis}, we propose the spatial discretization $\cP_h$ solving \rf{dis-Riccati1-h2}, and different 
spatio-temporal
discretization schemes $\cP_\cd$ and $\bar \cP_\cd$. Inspired by \rf{feedback-4} and different 
 schemes  of the Riccati equation \rf{dis-Riccati1-h2},
we consider the following {\em discrete feedback law}
\bel{feedback-5}
\setlength\abovedisplayskip{3pt}
\setlength\belowdisplayskip{3pt}
U_n=-P_{n+1}X_n-\eta_n\qq n=0,1,\cds,N-1
\ee
to discretize Problem {\bf (SLQ)} as follows:
\bel{dis-state}
\setlength\abovedisplayskip{3pt}
\setlength\belowdisplayskip{3pt}
\lt\{\!\!\!
\begin{array}{ll}
\ds X_{n+1}=A_0X_{n}+\t A_0U_n+\big[\b A_0 X_{n}+A_0\Pi_h\si(t_{n})\big]\D_{n+1}W\\
\ns\ds\qq\qq\qq \q n=0,1,\cds,N-1\,,\\
\ns\ds X_0=\Pi_hx\,.
\end{array}
\rt.
\ee
In \rf{feedback-5}, $P_\cd$ is a spatio-temporal approximation  of $\cP(\cd)$, and in this section we choose
$P_\cd=\bar \cP_\cd$. 
$\eta_\cd$ is
a finite element method-based spatio-temporal discretization scheme  of \rf{varphi}, which is of the form 
\bel{w-dif-eta}
\setlength\abovedisplayskip{3pt}
\setlength\belowdisplayskip{3pt}
\lt\{\!\!\!
\begin{array}{ll}
\ds \eta_n\!=\!A_0\eta_{n+1}\!+\!\t A_0\big[\!-\!P_{n+1}\eta_{n+1}\!+\! \b  P_{n+1}\Pi_h\si(t_{n+1})\big]
\q  n\!=\!0,1,\cds,N-1\,,\\
\ns\ds \eta_N=0\,.
\end{array}
\rt.
\ee

\ss

In Section \ref{ch-open}, we make a spatial discretization scheme of Problem {\bf (SLQ)} and obtain Problem {\bf (SLQ)$_h$}.
Similar to Section \ref{se-openclosed}, Problem {\bf (SLQ)$_h$} is uniquely solvable, and admits a closed-loop optimal control
\bel{feedback-h}
\setlength\abovedisplayskip{3pt}
\setlength\belowdisplayskip{3pt}
U^*_h(t)=-\cP_h(t)X^*_h(t)-\eta_h(t) \qq t\in[0,T]\,,
\ee
where $\cP_h(\cd)$ solves \rf{dis-Riccati1-h2} and $\eta_h(\cd)$ solves
\bel{eta-h}
\setlength\abovedisplayskip{3pt}
\setlength\belowdisplayskip{3pt}
\lt\{\!\!\!
\begin{array}{ll}
\ds\eta_h'(t)=-\D_h \eta_h(t)+\cP_h(t)\eta_h(t)-\b \cP_h(t)\Pi_h\si(t) \qq t\in [0,T]\,,\\
\ns\ds\eta_h(T)=0\,.
\end{array}
\rt.
\ee

\ss

The following result is on the regularity of $\eta_h(\cd)$ in \rf{eta-h}.
\bl{w311l1}
Suppose that $\si(\cd)\in C\big([0,T]; \dbH_0^1\big)$, and $I_\t$ is a uniform partition of $[0,T]$ with mesh size $\t\in(0,1)$. 
Then  there
exists a constant $\cC$ independent of $h$ such that
\begin{subequations}\label{w311e2}
    \begin{empheq}[left={\empheqlbrace\,}]{align}
      & \sup_{t\in[0,T]}\|\eta_h(t)\|_{\dot\dbH^1_h}^2+\int_0^T\|\eta_h(t)\|_{\dot\dbH^2_h}^2\rd t\leq \cC\int_0^T\|\si(t)\|^2\rd t\,,\label{w311e2a}\\
      &\max_{0\leq k\leq N-1}\sup_{t\in[t_k,t_{k+1})}\|\eta_h(t)-\eta_h(t_k)\|\leq  \cC\t \|\si(\cd)\|_{C([0,T]; \dbH_0^1)} \,.\label{w311e2b}
\end{empheq}
\end{subequations}
\el


\begin{proof}
{\bf (1) Verification of \rf{w311e2a}.}  
By testing \rf{eta-h} with $\D_h \eta_n(t)$ and applying \rf{w1012e1a} in Lemma \ref{w1012l1}, \rf{w1029e1b} in Lemma \ref{w229l4}, we can arrive at
\begin{equation*}
\setlength\abovedisplayskip{3pt}
\setlength\belowdisplayskip{3pt}
\bal
\frac 1 2 \|\nb \eta_h(t)\|^2+\int_t^T \|\D_h \eta_h(s)\|^2\rd s
&=\int_t^T\big(\cP_h(s)\eta_h(s)-\b\cP_h(s)\Pi_h\si(s), \D_h \eta_h(s) \big)_{\dbL^2}\rd s\\
&\leq \frac 1 2 \int_t^T \|\D_h \eta_h(s)\|^2\rd s+\cC \int_t^T \|\eta_h(s)\|^2+\|\si(s)\|^2 \rd s\\
&\leq \frac 1 2 \int_t^T \|\D_h \eta_h(s)\|^2\rd s+\cC \int_t^T \|\si(s)\|^2 \rd s\,,
\eal
\end{equation*}
which settles assertion \rf{w311e2a}.

%

\ss

{\bf (2) Verification of \rf{w311e2b}.} In the same vein as  in the proof of  \rf{w1013e5b}, we can deduce that 
\begin{equation*}
\setlength\abovedisplayskip{3pt}
\setlength\belowdisplayskip{3pt}
\bal
\|\eta_h(t)-\eta_h(t_k)\|
 &\leq \int_{t_k}^t\big\|E_h(s-t_k)\cP_h(s)\big[\eta_h(s)+\b\Pi_h\si(s) \big] \big\|\rd s\\
& \q+\int_{t}^T \big\|\cP_h(s)\big[E_h(s-t_k)-E_h(s-t)\big]\big[\eta_h(s)-\b\Pi_h\si(s) \big] \big\|\rd s\\
& \leq \cC\t \sup_{t\in[0,T]}\big[\|\eta_h(t)\|+\|\si(t)\|\big]\\
&\q+\cC \int_t^T\frac {t-t_k}{\sqrt{s-t}}\sup_{s\in[0,T]}\big[\|\eta_h(s)\|_{\dot\dbH^1_h}+\|\Pi_h\si(s)\|_{\dot\dbH^1_h}\big]\rd s\\
& \leq \cC\t \|\si(\cd)\|_{C([0,T]; \dbH_0^1)}\,,
\eal
\end{equation*}
where $\cC$ is independent of $k$.
That completes the proof.
\end{proof}

The following is on improved stability bounds for the optimal pair of Problem {\bf(SLQ)$_h$}.
\bl{reg-x}
Let
$\big(X^*_h(\cd),U^*_h(\cd)\big)$ be the optimal pair of Problem {\bf (SLQ)$_h$}.
Then, $X^*_h(\cd),U^*_h(\cd)\in L^2_\dbF\big(\O; C([0,T];\dbH_0^1)\big)\cap L^2_\dbF(0,T;\dot \dbH^2_h)$ and
there exists a constant $\cC$ 
such that
\begin{eqnarray}
&& \me\Big[\sup_{t\in[0,T]}\|X^*_h(t)\|_{\dot \dbH ^\g_h}^2+\sup_{t\in[0,T]}\|U^*_h(t)\|_{\dot \dbH ^\g_h}^2 +\int_0^T \big[ \|X^*_h(s)\|_{\dot \dbH ^{\g+1}_h}^2+\|U^*_h(s)\|_{\dot \dbH ^{\g+1}_h}^2\big] \rd s\Big] \notag\\
 & &\qq\qq\qq \leq \cC\Big[\|x\|_{\dot \dbH ^\g}^2+\int_0^T\|\si(t)\|_{\dot \dbH ^\g}^2\rd t\Big] \qq\g=0,1,2\,. \label{w228e1}
\end{eqnarray}
Furthermore, if $\b=0$,  then
 for  a given 
uniform partition $I_\t$ in time with mesh size 
$\t\in(0,\t_0]\subset(0,1)$, there exists  $\cC$ independent of $\t$ such that,
\bel{reg}
\bal
\me\Big[\Big\| \int_0^{T} \big[X^*_h(t)-X^*_h\big(\nu(t)\big) \big]\rd t \Big\|^2\Big]
  \leq \cC \t^2 \Big[\|x\|_{\dbH_0^1}^2+\sup_{t\in[0,T]} \|\si(t)\|^2+ \int_0^{T}\|\si(t)\|_{\dbH_0^1\cap\dbH^2}^2\rd t \Big] \,. 
\eal
\ee
\el

\begin{proof}
{\bf (1) Verification of \rf{w228e1}.}  
Partial results have been derived in \rf{w1029e2a} and \rf{w1029e2b}. We can settle the remaining by
relying on the spectral Galerkin method (see {\em e.g.}~Step (2) in the proof of Theorem \ref{Riccati-rate2}, \cite[Section 7.1]{Evans98}, \cite{Wang20}), $U^*_h(\cd)$'s
feedback representation \rf{feedback-h}, Lemma \ref{w1012l1}, and the BDG inequality.


\ss
{\bf (2) Verification of \rf{reg}.}  
The proof is similar to that in Lemma \ref{w207l5}.
By equation on $X^*_h(\cd)$ (SDE \rf{spde-h} with $U_h(\cd)=U^*_h(\cd)$), we have
\begin{eqnarray*}
&&\me\Big[\Big\| \sum_{n=0}^{N-1}\int_{t_n}^{t_{n+1}}\big[ X^*_h(t)- X^*_h(t_n) \big]\rd t \Big\|^2\Big]
=\me\Big[\Big\| \int_0^{T}\big[ X^*_h(t)- X^*_h\big(\nu(t)\big) \big]\rd t \Big\|^2\Big]\\
&&\qq\leq 5\bigg\{\me\Big[\Big\|  \int_0^{T} \big[E_h(t)-E_h\big(\nu(t)\big) \big] \Pi_hx \rd t \Big\|^2\Big]\\
&&\qq\q+\me\Big[\Big\| \int_0^{T}\int_0^{\nu(t)}  \big[E_h(t-s)-E_h\big(\nu(t)-s\big) \big] U^*_h(s) \rd s \rd t \Big\|^2\Big]\\
&&\qq\q+\me\Big[\Big\|  \int_0^{T}\int_{\nu(t)}^t E_h(t-s) U^*_h(s)\rd s \rd t \Big\|^2\Big]\\
&&\qq\q+\me\Big[\Big\|  \int_0^{T}\int_{\nu(t)}^t E_h(t-s) \Pi_h\si(s)\rd W(s) \rd t \Big\|^2\Big]\\
&&\qq\q+\me\Big[\Big\| \int_0^{T}\int_0^{\nu(t)} \big[E_h(t-s)-E_h\big(\nu(t)-s\big) \big]  \Pi_h\si(s)\rd W(s) \rd t \Big\|^2\Big] \bigg\}\\
&&\qq=:5\sum_{i=1}^5 J_i\,.
\end{eqnarray*}
Applying Lemma \ref{w228l1} and \rf{w228e1}, we conclude that
\begin{equation*}
\setlength\abovedisplayskip{3pt}
\setlength\belowdisplayskip{3pt}
\bal
J_1\leq \Big[\int_{0}^{t_1}2\|x\|\rd t+ \cC \int_{t_1}^T \t \nu(t)^{-1/2}\|x\|_{\dbH_0^1} \rd t\Big]^2
\leq \cC\t^2\|x\|_{\dbH_0^1}^2\,,
\eal
\end{equation*}
\begin{equation*}
\setlength\abovedisplayskip{3pt}
\setlength\belowdisplayskip{3pt}
\bal
J_2&\leq \me \Big[\int_0^T\int_0^{\nu(t)}\big\|\big[E_h\big(t-\nu(t)\big)-\mathds{1}_h\big]E_h\big(\nu(t)-s\big)U^*_h(s)\big\|\rd s\rd t \Big]^2\\
&\leq \cC\, \me \Big[\int_0^T\int_0^{\nu(t)}\t \big(\nu(t)-s\big)^{-1/2}\sup_{s\in[0,T]}\|U^*_h(s)\big\|_{\dbH_0^1}\rd s\rd t \Big]^2\\
&\leq \cC \t^2
  \Big[\|x\|_{\dbH_0^1}^2+\int_0^T\|\si(t)\|_{\dbH_0^1}^2\rd t\Big]\,,
\eal
\end{equation*}
and
\begin{equation*}
\setlength\abovedisplayskip{3pt}
\setlength\belowdisplayskip{3pt}
\bal
J_3 
\leq \cC\t^2\me\Big[\sup_{t\in[0,T]}\|U^*_h(t)\|^2\Big]
\leq \cC \t^2  \Big[\|x\|^2+\int_0^T\|\si(t)\|^2\rd t \Big]\,.
\eal
\end{equation*}
For $J_4$, by the mutual independence of $\big\{\int_{t_k}^{t_{k+1}}\int_{t_k}^t E_h(t-\th)\Pi_h \si(\th)\rd W(\th) \rd t  \big\}_{k=0}^{N-1}$, and  the It\^o isometry,
we have
\begin{equation*}
\setlength\abovedisplayskip{3pt}
\setlength\belowdisplayskip{3pt}
\bal
J_4\leq  \t \sum_{n=0}^{N-1}\int_{t_n}^{t_{n+1}}\me\Big[\int_{t_n}^t \big\| E_h(t-s)\Pi_h \si(s)\big\|^2\rd s \Big]\rd t
\leq  \cC \t^2\sup_{t\in[0,T]} \|\si(t)\|^2\,.
\eal
\end{equation*}
For $J_5$, the It\^o isometry and  Lemmata \ref{w228l1}, \ref{w1022e1} imply
\begin{equation*}
\setlength\abovedisplayskip{3pt}
\setlength\belowdisplayskip{3pt}
\bal
J_5
&\leq \cC  \int_0^{T}\me\Big[\int_0^{\nu(t)} \big\| \big[E_h(t-s)-E_h(t_k-s) \big] \Pi_h\si(s)\big\|^2\rd s \Big]\rd t\\
&\leq \cC \int_0^T \int_0^{\nu(t)} \t^2\|\Pi_h\si(s)\|_{\dot\dbH^2_h}^2\rd s \rd t\\
&\leq \cC \t^2  \int_0^{T}\|\si(t)\|_{\dbH_0^1\cap\dbH^2}^2\rd t\,.
\eal
\end{equation*}

A combination of these estimates for  $J_1$ through $J_5$ then completes the proof.
\end{proof}

For any $l=0,1,\cds,N-1$,
similar to $\dbX(t_l,T)\,, \dbU(t_l,T)$ defined in \rf{w1015e4}, we now introduce `discrete spaces'
\begin{equation*}
\setlength\abovedisplayskip{3pt}
\setlength\belowdisplayskip{3pt}
\bal
 \dbX_\dbA &\deq \Big\{X_\cd\equiv \{X_n\}_{n=0}^{N-1} \,\Big| \, X_{n}\in L^2_{\mf_{t_n}}(\O;\dbV_h)\q   \forall\, n=0,1,\cds, \, N-1\,,\\
 &\qq\qq\qq\qq \mbox{and } \t \sum_{n=0}^{N-1}\me\big[\|X_n\|^2\big] <\infty \Big\}\, ,\\
\eal
\end{equation*} 
\begin{equation*}
\setlength\abovedisplayskip{3pt}
\setlength\belowdisplayskip{3pt}
\bal
  \dbU_\dbA &\deq \Big\{U_\cd\equiv \{U_n\}_{n=0}^{N-1} \,\Big| \, U_{n}\in L^2_{\mf_{t_n}}(\O;\dbV_h)\q  \forall\, n=0,1,\cds, \, N-1\,,\\
 &\qq\qq\qq\qq \mbox{and } \t \sum_{n=0}^{N-1}\me\big[\|U_n\|^2\big] <\infty \Big\}\,,
\eal
\end{equation*}
%
which we endow with the corresponding  norms
\begin{equation*}
\setlength\abovedisplayskip{3pt}
\setlength\belowdisplayskip{3pt}
\bal
\|X_\cd\|_{\dbX_\dbA}\deq \Big(\t \sum_{n=0}^{N-1}\me\big[\|X_n\|^2\big] \Big)^{1/2} \q
\mbox{and} \q
\|U_\cd\|_{\dbU_\dbA}\deq \Big(\t \sum_{n=0}^{N-1}\me\big[\|U_n\|^2\big]\Big)^{1/2}\,.
\eal
\end{equation*}

\ss

Similar to Lemma \ref{w207l2}, we may use
 \rf{w-dif-eta}, \rf{dis-state} and proceed by induction to have the following representations for $\eta_\cd$ and $X_\cd$.
\bl{w211l1}
Suppose that $\si(\cd)\in C([0,T],\dbL^2)$.
Then
$\eta_\cd\in \dbU_\dbA$, which is of the form
\bel{w311e1}
\setlength\abovedisplayskip{3pt}
\setlength\belowdisplayskip{3pt}
\eta_n=\t\sum_{k=n+1}^N A_0^{k-n}\big[- P_{k}\eta_{k}+ \b  P_{k}\Pi_h\si(t_{k})\big] \qq \forall\, n=0,1,\cds,N-1\,.
\ee
For any $U_\cd\in \dbU_\dbA$, then $X_\cd\in \dbX_\dbA$, and for any $n=0,1,\cds,N$,
\bel{w303e1}
\setlength\abovedisplayskip{3pt}
\setlength\belowdisplayskip{3pt}
X_n=A_0^{n}\Pi_hx+\t\sum_{k=0}^{n-1}A_0^{n-k} U_k+\sum_{k=0}^{n-1}A_0^{n-k}\big[\Pi_h\si(t_k)+ \b X_k\big] \D_{k+1}W\,. 
\ee
\el

\begin{proof}
{\bf (1)} Assertions
\rf{w311e1} and \rf{w303e1} can be derived inductively with the help of \rf{w-dif-eta} and \rf{dis-state}.

\ss

{\bf (2)}
Thanks to \rf{w1013e3a} in Lemma \ref{w1013l2},
we know that 
$\| P_\cd\|_{\cL(\dbL^2|_{\dbV_h})}$ is bounded. Then \rf{w311e1}, together with discrete Gronwall's inequality leads to
\begin{equation*}
\setlength\abovedisplayskip{3pt}
\setlength\belowdisplayskip{3pt}
\bal
\max_{0\leq n\leq N}\|\eta_n\|\leq \cC\sup_{t\in[0,T]}\|\si(t)\|\,,
\eal
\end{equation*}
by which we have
\begin{equation*}
\setlength\abovedisplayskip{3pt}
\setlength\belowdisplayskip{3pt}
\bal
\|\eta_\cd\|_{\dbU_\dbA}= \Big(\t \sum_{n=0}^{N-1}\me\big[\|\eta_n\|^2\big]\Big)^{1/2}
\leq \cC\sup_{t\in[0,T]}\|\si(t)\|\,,
\eal
\end{equation*}
 or equivalently $\eta_\cd\in \dbU_\dbA$. 

\ss

{\bf (3)}
Assertion $X_\cd\in\dbX_\dbA$ can be derived by
combining \rf{w303e1} with the It\^o isometry, 
and application of discrete Gronwall's inequality.
\end{proof}

We are now able to state the main result of this section.
\bt{SLQ-pair-rate}
Suppose that assumption {\rm \bf (A)} holds.
Let $\big(X^*_h(\cd),U^*_h(\cd)\big)$ be the optimal pair of Problem {\bf (SLQ)$_h$} and
$(X_\cd, U_\cd)$ be its approximation given by \rf{feedback-5}--\rf{w-dif-eta}.
Then there exists a constant $\cC$ independent of $h$ and $\t$ such that
\bel{slq-rate}
\setlength\abovedisplayskip{3pt}
\setlength\belowdisplayskip{3pt}
\bal
\max_{0\leq n\leq N}\me\big[\|X_h^*(t_n)\!-\!X_n\|^2\big]
 \!+\!\max_{0\leq n\leq N-1}\me\big[\|U_h^*(t_n)\!-\!U_n\|^2\big]
\!\leq\! \cC\big[\t^2\!+\!|\b| \t\big]\,.
\eal
\ee
\et

\begin{proof}

Firstly,  by  \rf{w303e1}, the It\^o isometry, feedback form \rf{feedback-5}, discrete Gronwall's 
inequality and the fact that $\eta_\cd\in\dbU_\dbA$, we can deduce that $X_\cd\in \dbX_\dbA$. Then we conclude that $U_\cd\in \dbU_\dbA$.

\ss

The remaining part is on assertion \rf{slq-rate}. 
Firstly we prove the error estimate \rf{slq-rate} for $X$-part.
By \rf{spde-h}, \rf{w303e1} and state feedback representations \rf{feedback-h}, \rf{feedback-5}, we have
\bel{w210e5}
\setlength\abovedisplayskip{3pt}
\setlength\belowdisplayskip{3pt}
\bal
&X^*_h(t_n)- X_n\\
&\q=\big[E_h(t_n)\Pi_h-A_0^n\Pi_h \big] x  
-\sum_{k=0}^{n-1}\int_{t_k}^{t_{k+1}}\Big\{E_h(t_n-t)\big[\cP_h(t)-\cP_h(t_{k+1})\big]X^*_h(t) \\
&\qq\q+\big[E_h(t_n-t)-A_0^{n-k} \big]  \big[\cP_h(t_{k+1})X^*_h(t)+\eta_h(t)\big]\\
&\qq\q +A_0^{n-k} \big[\cP_h(t_{k+1})- P_{k+1} \big]X^*_h(t)
+A_0^{n-k}  P_{k+1} \big[X^*_h(t)-X^*_h(t_k)\big]   \\
&\qq\q +A_0^{n-k} \big[\eta_h (t)-\eta_h (t_k)\big]
+A_0^{n-k} P_{k+1} \big[X^*_h(t_k)- X_k\big] 
+A_0^{n-k} \big[\eta_h (t_k)-\eta_k\big]
\Big\}\rd t  \\
&\qq+\sum_{k=0}^{n-1}\int_{t_k}^{t_{k+1}}\Big[ E_h(t_n-t)\Pi_h\si(t)-A_0^{n-k}\Pi_h\si(t_k) \Big]\rd W(t)\\
&\qq+\b \sum_{k=0}^{n-1}\int_{t_k}^{t_{k+1}}\Big[ E_h(t_n-t)X^*_h(t)-A_0^{n-k}\Pi_h X_k \Big]\rd W(t)  \\
&\q=: I_0+\sum_{i=1}^7  Leb_i+Ito_1+\b\times Ito_2\,.
\eal
\ee
Note that Lebesgue integrals on the right-hand side of \rf{w210e5} come from continuous and discrete feedback laws. 
Since feedback laws consist of $\cP_h(\cd), X^*_h(\cd), \eta_h(\cd)$ and their approximations $ P_\cd, X_\cd, \eta_\cd$, 
as well as
semigroup $E_h(\cd)$ and its approximation $E_{h,\t}(\cd)$ are involved in the Lebesgue integrals, we use seven terms to 
distinguish them and each of the Lebesgue integrals addresses specific errors in the scheme:
the terms $Leb_1, Leb_3$ capture the main error effects due to the discretization of 
the Riccati equation,
while the term $Leb_2$ addresses discretization of $E_h(\cd)$;
the terms $Leb_4, Leb_6$  describe discretization effects of the optimal state $X^*_h(\cd)$,
while the terms $Leb_5, Leb_7$ account for discretization effects introduced by \rf{w-dif-eta}.
Finally, the It\^o integral $Ito_2$ is only active in the presence of multiplicative noise when $\b\neq 0$.

\ss

Now we estimate every term on the right-hand sides of \rf{w210e5}.
By \rf{w1115e2a} in Lemma \ref{w207l1},
\bel{w228e7}
\setlength\abovedisplayskip{3pt}
\setlength\belowdisplayskip{3pt}
\me\big[\|I_0\|^2\big]\leq \cC \t^2\|x\|_{\dbH_0^1\cap\dbH^2}^2\,.
\ee

For $Leb_1$ which is on the regularity of $\cP_h(\cd)$, by $\mu(\cd)$ defined in \rf{w827e1} and \rf{Riccati-sol}, we can derive
\begin{equation*}
\setlength\abovedisplayskip{3pt}
\setlength\belowdisplayskip{3pt}
\bal
&\me\big[\|Leb_1\|^2\big]\leq \cC\, \me\Big[\Big\| \int_0^{t_n}  E_h(t_n-t)\big[\cP_h(t)-\cP_h\big(\mu(t)\big)\big]X^*_h(t)\rd t \Big\|^2\Big]\\
&\q\leq \cC\, \me\bigg[
\Big\| \int_0^{t_n}  E_h(t_n-t)  \big[ E_h\big(2(\mu(t)-t)\big)-\mathds{1}_h\big] E_h\big(2\big(T-\mu(t)\big)\big) X^*_h(t) \rd t \Big\|^2\\
&\qq+\Big\| \int_0^{t_n} E_h(t_n-t)   \int_{t}^{\mu(t)} E_h(s-t)\big(\b^2\cP_h(s)+\mathds{1}_h- \cP_h^2(s) \big)\\
&\qq\qq\times E_h(s-t) X^*_h(t) \rd s  \rd t \Big\|^2\\
&\qq+ \Big\| \int_0^{t_n}E_h(t_n-t)    \int_{\mu(t)}^T \big[E_h(s-t)-E_h\big(s-\mu(t)\big)\big]\\
&\qq\qq\times  \big(\b^2\cP_h(s)+\mathds{1}_h-\cP_h^2(s) \big)E_h(s-t) X^*_h(t)\rd s \rd t \Big\|^2 \\
&\qq+\Big\| \int_0^{t_n}E_h(t_n-t)     \int_{\mu(t)}^T E_h\big(s-\mu(t)\big) \big(\b^2\cP_h(s)+\mathds{1}_h- \cP_h^2(s) \big)\\
&\qq\qq\times \big[E_h(s-t)-E_h\big(s-\mu(t)\big)\big] X^*_h(t)\rd s  \rd t\Big\|^2 \bigg]\\
&\q=:\sum_{i=1}^4 Leb_{1,i} \,.
\eal
\end{equation*}
For $Leb_{1,1}$, by Lemma \ref{w207l1} and  \rf{w228e1} in Lemma \ref{reg-x}, it is easy to see
\begin{equation*}
\setlength\abovedisplayskip{3pt}
\setlength\belowdisplayskip{3pt}
\bal
Leb_{1,1}\leq \cC\t^2 \me\Big[\int_0^{t_n}\|X^*_h(t)\|_{\dot\dbH^2_h}^2\rd t\Big]\leq \cC\t^2 \Big[\|x\|_{\dbH_0^1}^2+\int_0^T\|\si(t)\|_{\dbH_0^1}^2\rd t\Big]\,.
\eal
\end{equation*}
By \rf{w228e1} in Lemma \ref{reg-x},
it is easy to check that
\begin{equation*}
\setlength\abovedisplayskip{3pt}
\setlength\belowdisplayskip{3pt}
\bal
Leb_{1,2}\leq \cC\t^2 \sup_{t\in[0,T]} \me\big[\|X^*_h(t)\|^2\big]\leq  \cC\t^2 \Big[\|x\|^2+\int_0^T\|\si(t)\|^2\rd t\Big]\,.
\eal
\end{equation*}
For terms $Leb_{1,3}$ and $Leb_{1,4}$, by applying spectral decomposition, we know that 
$E_h(t)\cP_h(s)=\cP_h(s)E_h(t)$ for 
$t,s\in[0,T]$; then similar to estimate for $Leb_{1,1}$, we have
\begin{equation*}
\setlength\abovedisplayskip{3pt}
\setlength\belowdisplayskip{3pt}
\bal
Leb_{1,3}+Leb_{1,4}
\leq  \cC\t^2 \me\Big[\int_0^{t_n}\|X^*_h(t)\|_{\dot\dbH^2_h}^2\rd t\Big]
\leq \cC\t^2 \Big[\|x\|_{\dbH_0^1}^2+\int_0^T\|\si(t)\|_{\dbH_0^1}^2\rd t\Big]\,.
\eal
\end{equation*}
Now, combining with estimates of $Leb_{1,1}$ through $Leb_{1,4}$, we arrive at
\bel{w228e8}
\setlength\abovedisplayskip{3pt}
\setlength\belowdisplayskip{3pt}
\me\big[\|Leb_1\|^2\big]\leq  \cC\t^2 \Big[\|x\|_{\dbH_0^1}^2+\int_0^T\|\si(t)\|_{\dbH_0^1}^2\rd t\Big]\,.
\ee
For $Leb_2$,
Lemma \ref{w207l1},  \rf{w311e2a} in Lemma \ref{w311l1} and \rf{w228e1} in Lemma \ref{reg-x} lead to
\begin{equation}\label{w112e1}
\setlength\abovedisplayskip{3pt}
\setlength\belowdisplayskip{3pt}
\bal
\me\big[\!\|Leb_2\|^2\!\big]
&\!\leq \!\cC\, \me\Big[\Big\|\! \int_0^{t_n} \!\cP_h\big(\mu(t)\big) G_{\t}(t_n\!-\!t)X^*_h(t)\!+\!G_{\t}(t_n\!-\!t)\eta_h(t)\big]\! \rd t \Big\|^2 \Big]\\
&\leq\cC\t^2 \me\Big[\int_0^{t_n}\|X^*_h(t)\|_{\dot\dbH^2_h}^2+\|\eta_h(t)\|_{\dot\dbH^2_h}^2\rd t\Big]\\
&\leq \cC\t^2 \Big[\|x\|_{\dbH_0^1}^2+\int_0^T\|\si(t)\|_{\dbH_0^1}^2\rd t\Big]\,.
\eal
\end{equation}
For $Leb_3$,  by Theorem \ref{Riccati-rate2} 
and  \rf{w228e1} in Lemma \ref{reg-x}, it follows that
\begin{equation}\label{w228e10}
\setlength\abovedisplayskip{3pt}
\setlength\belowdisplayskip{3pt}
\bal
\me\big[\|Leb_3\|^2\big]
&\leq \cC\, \me\Big[\int_0^{t_n} \big\| \cP_h\big(\mu(t)\big)\Pi_h- P_{\pi(t)}\Pi_h \big\|_{\cL(\dbH_0^1;\dbL^2)} \sup_{t\in[0,T]} \| X^*(t) \|_{\dbH_0^1}\rd t\Big]^2 \\
&\leq \cC\t^2 \Big[\|x\|_{\dbH_0^1}^2+\int_0^T\|\si(t)\|_{\dbH_0^1}^2\rd t\Big]\,.
\eal
\end{equation}
For $Leb_4$, by \rf{w1029e2d} (for general $\b$) or $L^2$-regularity \rf{reg}, we can deduce that
\bel{w228e6}
\setlength\abovedisplayskip{3pt}
\setlength\belowdisplayskip{3pt}
\bal
\me\big[\| Leb_4\|^2\big]
\leq 
\lt\{\!\!
\begin{array}{ll}
\ds \!\cC\t \big[\|x\|_{\dbH_0^1}^2\!+\!\|\si(\cd)\|^2_{C([0,T];\dbL^2) \cap L^2(0,T;\dbH_0^1)}\big] \qq\qq& \b\neq 0\,,\\
\nm\ds  \! \cC \t^2  \big[\|x\|_{\dbH_0^1}^2\!+\!\|\si(\cd)\|^2_{C([0,T];\dbL^2) \cap L^2(0,T;\dbH_0^1\cap\dbH^2)} \big]\,\,\,\,\,   \qq& \b=0\,.\\
\end{array}
\rt.
\eal
\ee
\rf{w311e2b} yields that
\bel{w311e3}
\setlength\abovedisplayskip{3pt}
\setlength\belowdisplayskip{3pt}
\me\big[\| Leb_5\|^2 \big] \leq  \cC\t^2 \|\si(\cd)\|^2_{C([0,T];\dbH_0^1)} \,.
\ee
For $Leb_6$, based on the fact that $\|P_\cd\|_{\cL(\dbL^2|_{\dbV_h})}$ is uniformly bounded, then it is easy to see that
\bel{w228e11}
\setlength\abovedisplayskip{3pt}
\setlength\belowdisplayskip{3pt}
\me\big[\| Leb_6\|^2\big]\leq \cC\t \sum_{k=0}^{n-1}\me\big[\|X^*_h(t_k)-X_k\|^2\big]\,.
\ee
For $Leb_7$, we follow the procedure as  in the estimates for $Leb_1$ through $Leb_6$ to get
\begin{equation*}
\setlength\abovedisplayskip{3pt}
\setlength\belowdisplayskip{3pt}
\bal
\|\eta_h(t_n)-\eta_n\|\leq \cC\t \|\si(\cd)\|_{C([0,T];\dbH_0^1)}
+\sum_{k=n+1}^N\|\eta_h(t_k)-\eta_k\|\,,
\eal
\end{equation*}
which, and discrete Gronwall's inequality, lead to 
\bel{w808e1}
\setlength\abovedisplayskip{3pt}
\setlength\belowdisplayskip{3pt}
\max_{0\leq n\leq N}\|\eta_h(t_n)-\eta_n\|\leq \cC\t \|\si(\cd)\|_{C([0,T];\dbH_0^1)}\,,
\ee
subsequently,
\bel{w311e4}
\setlength\abovedisplayskip{3pt}
\setlength\belowdisplayskip{3pt}
\me\big[\| Leb_7\|^2\big] \leq   \cC\t^2   \|\si(\cd)\|^2_{C([0,T];\dbH_0^1)} \,.
\ee

\ss
For $Ito_1$, we use the It\^o isometry, the
 triangular inequality and Lemma \ref{w207l1} to estimate
\bel{w1015e6}
\setlength\abovedisplayskip{3pt}
\setlength\belowdisplayskip{3pt}
\bal
\me\big[\|Ito_1\|^2\big]
&\leq \cC\Big[ \int_0^{t_n}\big\|E_{h,\t}(t_n-t)\Pi_h\big[ \si\big(\nu(t)\big)-\si(t)\big]\big\|^2\rd t 
+\int_0^{t_n}\|G_{\t}(t_n-t)\si(t) \|^2\rd t\Big]\\
&\leq \cC\int_0^{t_n}\big\|\si\big(\nu(t)\big)-\si(t)\big\|^2\rd t
+ \cC\int_0^{t_n} \t^2\|\Pi_h\si(t)\|_{\dot \dbH ^2_h}^2\rd t  \\
&\leq \cC\t^2\Big[L_{\si,1}^2 +\int_0^T\|\si(t)\|_{\dbH_0^1\cap\dbH^2}^2\rd t\Big]\,.
\eal
\ee

For $Ito_2$, by virtue of $\nu(\cd), \pi(\cd)$ defined in \rf{w827e1} and the It\^o isometry, we have
\begin{equation*}
\setlength\abovedisplayskip{3pt}
\setlength\belowdisplayskip{3pt}
\bal
\me\big[\|Ito_2\|^2\big]
&\leq \cC\me\Big[\int_0^{t_n}\big\|E_{h,\t}(t_n-t)\big[ X^*_h\big(\nu(t)\big)-X^*_h(t)\big]\big\|^2\rd t 
 +\int_0^{t_n}\|G_{\t}(t_n-t) X^*_h(t)\|^2\rd t\\
&\qq +\int_0^{t_n} \big\|E_{h,\t}(t_n-t)\big[ X_{\pi(t)-1}-X^*_h\big(\nu(t)\big)\big]\big\|^2 \rd t  \Big]\\
&=: \sum_{i=1}^3 Ito_{2,i}\,.
\eal
\end{equation*}
By \rf{w1029e2d}, it follows that
\begin{equation*}
\setlength\abovedisplayskip{3pt}
\setlength\belowdisplayskip{3pt}
\bal
Ito_{2,1}\leq
 \cC\t \big[\|x\|_{\dbH_0^1}^2+\|\si(\cd)\|^2_{C([0,T];\dbL^2) \cap L^2(0,T;\dbH_0^1)} \big]\,.
\eal
\end{equation*}
By \rf{w228e1} and Lemma \ref{w207l1} with  $\g=1$, we have
\begin{equation*}
\setlength\abovedisplayskip{3pt}
\setlength\belowdisplayskip{3pt}
\bal
Ito_{2,2}
 \leq \cC\t \int_0^{t_n}\me\big[\|X^*_h(t)\|_{\dbH_0^1}^2\big] \rd t  
 \leq \cC \t \Big[\|x\|^2+\int_0^T\|\si(t)\|^2\rd t\Big]\,.
\eal
\end{equation*}
It is easy to see that 
\begin{equation*}
\setlength\abovedisplayskip{3pt}
\setlength\belowdisplayskip{3pt}
\bal
Ito_{2,3}\leq
 \cC\t \sum_{k=0}^{n-1} \me\big[\|X^*_h(t_k)-X_k\|^2 \big]\,.
\eal
\end{equation*}
Hence, for $Ito_2$, we conclude that
\bel{w228e4}
\setlength\abovedisplayskip{3pt}
\setlength\belowdisplayskip{3pt}
\bal
\me\big[\|Ito_2\|^2\big]
\leq  \cC\t \big[\|x\|_{\dbH_0^1}^2+\|\si(\cd)\|^2_{C([0,T];\dbL^2) \cap L^2(0,T;\dbH_0^1)} \big]
+\cC\t \sum_{k=0}^{n-1} \me\big[\|X^*_h(t_k)-X_k\|^2 \big]\,.
\eal
\ee

Finally,
combining with \rf{w210e5}--\rf{w228e4}, we have
\begin{equation*}
\setlength\abovedisplayskip{3pt}
\setlength\belowdisplayskip{3pt}
\bal
 \|X^*_h(t_n)-X_n\|^2_{L^2_{\mf_{t_n}}(\O;\dbL^2)}
 \leq \cC\t^2+\cC\b^2\t+\cC\sum_{k=0}^{n-1}\t \|X^*_h(t_k)-X_k\|^2_{L^2_{\mf_{t_k}}(\O;\dbL^2)}\,.
\eal
\end{equation*}
Then relying on discrete Gronwall's inequality, we can derive the desired convergence rate for $X$-part.

\ss
Finally, we apply feedback law to prove the convergence rate for $U$-part. Relying on \rf{feedback-5}, \rf{feedback-h},
we can see that
\begin{equation*}
\setlength\abovedisplayskip{3pt}
\setlength\belowdisplayskip{3pt}
\bal
U^*_h(t_n)-U_n
&=-\big[\cP_h(t_n)-\cP_h(t_{n+1})\big]X^*_h(t_n)
-\big[\cP_h(t_{n+1})\Pi_h- P_{n+1}\Pi_h\big]X^*_h(t_n)\\
&\q- P_{n+1}\Pi_h\big[X^*_h(t_n)-X_n\big]
-\big[\eta_h(t_n)-\eta_n\big]\\
&=:\sum_{i=1}^4 J_i\,.
\eal
\end{equation*}
With the similar procedure as  in the estimate for $Leb_1$, by applying  \rf{w228e1}, we have
\begin{equation*}
\setlength\abovedisplayskip{3pt}
\setlength\belowdisplayskip{3pt}
\bal
\me\big[\|J_1\|^2\big]
\leq \cC\t^2 \Big[\|x\|_{\dbH_0^1\cap\dbH^2}^2+\int_0^T\|\si(t)\|_{\dbH_0^1\cap\dbH^2}^2\rd t\Big]\,.
\eal
\end{equation*}
Remark \ref{w1015r1} and \rf{w228e1} lead to
\begin{equation*}
\setlength\abovedisplayskip{3pt}
\setlength\belowdisplayskip{3pt}
\bal
\me\big[\|J_2\|^2\big]
\leq \cC\t^2 \me\big[\|X^*_h(t_n)\|_{\dot \dbH ^2_h}^2\big]
\leq \cC\t^2 \Big[\|x\|_{\dbH_0^1\cap\dbH^2}^2+\int_0^T\|\si(t)\|_{\dbH_0^1\cap\dbH^2}^2\rd t\Big]\,.
\eal
\end{equation*}
Convergence rate \rf{slq-rate} for the $X$-part and the fact  that $\| P_\cd\|_{\cL(\dbL^2|_{\dbV_h})}$ is uniformly bounded  imply that
\begin{equation*}
\setlength\abovedisplayskip{3pt}
\setlength\belowdisplayskip{3pt}
\bal
\me\big[\|J_3\|^2\big]\leq \cC\, \me\big[\|X^*_h(t_n)-X_n\|^2\big]
\leq \cC[\t^2+|\b| \t \big]\,.
\eal
\end{equation*}
\rf{w808e1} yields that
\begin{equation*}
\setlength\abovedisplayskip{3pt}
\setlength\belowdisplayskip{3pt}
\bal
\me\big[\|J_4\|^2\big]
\leq \cC\t^2\,.
\eal
\end{equation*}
Now, combining with estimates of $J_1$ through $J_4$,  we derive the rate \rf{slq-rate} for $U$-part.
That completes the proof.
\end{proof}


\br{w1023r1}
Relying on Theorem \ref{Riccati-rate-h} and following the procedure in the proof of Theorem \ref{SLQ-pair-rate}, we can prove that
\begin{equation*}
\setlength\abovedisplayskip{3pt}
\setlength\belowdisplayskip{3pt}
\bal
\sup_{t\in[0,T]}\me\big[\|X^*(t)-X^*_h(t)\|^2+\|U^*(t)-U^*_h(t)\|^2\big] 
\leq \cC h^4\ln{\frac 1 h}\,.
\eal
\end{equation*}

\er

\br{w1106r2}
 Compared with the open-loop method, the closed-loop method avoids computing conditional expectations
and iterating derived by the gradient descent method (see Algorithm \ref{alg1}); hence it is easier to compute.
%


The result in Theorem \ref{SLQ-pair-rate} shows an improved rate of convergence in case the driving noise in SPDE 
\rf{spde} is additive; we do not have a corresponding improvement for the `open-loop approach' in Section \ref{ch-open}.

\er

{\small


\begin{thebibliography}{10}

\bibitem{Zhou02}
{\sc M.~Ait~Rami, X.~Chen, and X.~Y. Zhou}, {\em Discrete-time indefinite {LQ}
  control with state and control dependent noises}, J. Global Optim., 23
  (2002), pp.~245--265.

\bibitem{Al-Hussein05}
{\sc A.~R. Al-Hussein}, {\em Strong, mild and weak solutions of backward
  stochastic evolution equations}, Random Oper. Stochastic Equations, 13
  (2005), pp.~129--138.

\bibitem{Benner-Stillfjord-Trautwein22}
{\sc P.~Benner, T.~Stillfjord, and C.~Trautwein}, {\em A linear implicit
  {E}uler method for the finite element discretization of a controlled
  stochastic heat equation}, IMA J. Numer. Anal., 42 (2022), pp.~2118--2150.

\bibitem{Brenner-Scott08}
{\sc S.~C. Brenner and L.~R. Scott}, {\em The mathematical theory of finite
  element methods}, vol.~15 of Texts in Applied Mathematics, Springer, New
  York, third~ed., 2008.

\bibitem{C-M-P-W24}
{\sc A.~Chaudhary, F.~Merle, A.~Prohl, and Y.~Wang}, {\em An efficient
  discretization to simulate the solution of linear-quadratic stochastic
  boundary control problem}, Preprint,  (2024).

\bibitem{Chow15}
{\sc P.-L. Chow}, {\em Stochastic partial differential equations}, Advances in
  Applied Mathematics, CRC Press, Boca Raton, FL, second~ed., 2015.

\bibitem{DaProto-Zabczyk92}
{\sc G.~Da~Prato and J.~Zabczyk}, {\em Stochastic equations in infinite
  dimensions}, vol.~44 of Encyclopedia of Mathematics and its Applications,
  Cambridge University Press, Cambridge, 1992.

\bibitem{Dou-Lv19}
{\sc F.~Dou and Q.~L\"u}, {\em Partial approximate controllability for linear
  stochastic control systems}, SIAM J. Control Optim., 57 (2019),
  pp.~1209--1229.

\bibitem{Dunst-Prohl16}
{\sc T.~Dunst and A.~Prohl}, {\em The forward-backward stochastic heat
  equation: numerical analysis and simulation}, SIAM J. Sci. Comput., 38
  (2016), pp.~A2725--A2755.

\bibitem{ElKaroui-Peng-Quenez97}
{\sc N.~El~Karoui, S.~Peng, and M.~C. Quenez}, {\em Backward stochastic
  differential equations in finance}, Math. Finance, 7 (1997), pp.~1--71.

\bibitem{Evans98}
{\sc L.~C. Evans}, {\em Partial differential equations}, vol.~19 of Graduate
  Studies in Mathematics, American Mathematical Society, Providence, RI, 1998.

\bibitem{Fujita-Suzuki91}
{\sc H.~Fujita and T.~Suzuki}, {\em Evolution problems}, in Finite element
  methods. Part 1, P.~G. Ciarlet and J.-L. Lions, eds., vol.~II of Handbook of
  Numerical Analysis, North-Holland, Amsterdam, 1991, pp.~789--928.

\bibitem{Gobet-Turkedjiev16}
{\sc E.~Gobet and P.~Turkedjiev}, {\em Linear regression {MDP} scheme for
  discrete backward stochastic differential equations under general
  conditions}, Math. Comp., 85 (2016), pp.~1359--1391.

\bibitem{Gyorfi-Kohler-Krzyzak-Walk02}
{\sc L.~Gy\"{o}rfi, M.~Kohler, A.~Krzy\.{z}ak, and H.~Walk}, {\em A
  distribution-free theory of nonparametric regression}, Springer Series in
  Statistics, Springer-Verlag, New York, 2002.

\bibitem{Heij-Ran-Schagen21}
{\sc C.~Heij, A.~C.~M. Ran, and F.~van Schagen}, {\em Introduction to
  mathematical systems theory---discrete time linear systems, control and
  identification}, Birkh\"{a}user/Springer, Cham, 2021.

\bibitem{Hinze-Pinnau-Ulbrich-Ulbrich09}
{\sc M.~Hinze, R.~Pinnau, M.~Ulbrich, and S.~Ulbrich}, {\em Optimization with
  {PDE} constraints}, vol.~23 of Mathematical Modelling: Theory and
  Applications, Springer, New York, 2009.

\bibitem{Hu-Nualart-Song11}
{\sc Y.~Hu, D.~Nualart, and X.~Song}, {\em Malliavin calculus for backward
  stochastic differential equations and application to numerical solutions},
  Ann. Appl. Probab., 21 (2011), pp.~2379--2423.

\bibitem{Kabanikhin12}
{\sc S.~I. Kabanikhin}, {\em Inverse and ill-posed problems}, vol.~55 of
  Inverse and Ill-posed Problems Series, Walter de Gruyter GmbH \& Co. KG,
  Berlin, 2012.

\bibitem{Kroller-Kunisch91}
{\sc M.~Kroller and K.~Kunisch}, {\em Convergence rates for the feedback
  operators arising in the linear quadratic regulator problem governed by
  parabolic equations}, SIAM J. Numer. Anal., 28 (1991), pp.~1350--1385.

\bibitem{Kruse14}
{\sc R.~Kruse}, {\em Optimal error estimates of {G}alerkin finite element
  methods for stochastic partial differential equations with multiplicative
  noise}, IMA J. Numer. Anal., 34 (2014), pp.~217--251.

\bibitem{Li-Zhou21}
{\sc B.~Li and Q.~Zhou}, {\em Discretization of a distributed optimal control
  problem with a stochastic parabolic equation driven by multiplicative noise},
  J. Sci. Comput., 87 (2021), pp.~Paper No. 68, 37.

\bibitem{Lv19}
{\sc Q.~L\"{u}}, {\em Well-posedness of stochastic {R}iccati equations and
  closed-loop solvability for stochastic linear quadratic optimal control
  problems}, J. Differential Equations, 267 (2019), pp.~180--227.

\bibitem{Lv-Wang-Wang-Zhang22}
{\sc Q.~L\"{u}, P.~Wang, Y.~Wang, and X.~Zhang}, {\em Chapter 6 - numerics for
  stochastic distributed parameter control systems: a finite transposition
  method}, in Numerical Control: Part A, E.~Tr{\'e}lat and E.~Zuazua, eds.,
  vol.~23 of Handbook of Numerical Analysis, Elsevier, 2022, pp.~201--232.

\bibitem{Lv-Zhang21}
{\sc Q.~L\"{u} and X.~Zhang}, {\em Mathematical control theory for stochastic
  partial differential equations}, vol.~101 of Probability Theory and
  Stochastic Modelling, Springer, Cham, 2021.

\bibitem{Stillfjord18}
{\sc A.~M\aa~lqvist, A.~Persson, and T.~Stillfjord}, {\em Multiscale
  differential {R}iccati equations for linear quadratic regulator problems},
  SIAM J. Sci. Comput., 40 (2018), pp.~A2406--A2426.

\bibitem{Mukam-Tambue18}
{\sc J.~D. Mukam and A.~Tambue}, {\em Strong convergence analysis of the
  stochastic exponential {R}osenbrock scheme for the finite element
  discretization of semilinear {SPDE}s driven by multiplicative and additive
  noise}, J. Sci. Comput., 74 (2018), pp.~937--978.

\bibitem{Nesterov04}
{\sc Y.~Nesterov}, {\em Introductory lectures on convex optimization}, vol.~87
  of Applied Optimization, Kluwer Academic Publishers, Boston, MA, 2004.

\bibitem{Nualart06}
{\sc D.~Nualart}, {\em The {M}alliavin calculus and related topics},
  Probability and its Applications (New York), Springer-Verlag, Berlin,
  second~ed., 2006.

\bibitem{Pardoux-Peng90}
{\sc {\'E}.~Pardoux and S.~Peng}, {\em Adapted solution of a backward
  stochastic differential equation}, Systems Control Lett., 14 (1990),
  pp.~55--61.

\bibitem{Pazy83}
{\sc A.~Pazy}, {\em Semigroups of linear operators and applications to partial
  differential equations}, vol.~44 of Applied Mathematical Sciences,
  Springer-Verlag, New York, 1983.

\bibitem{Prevot-Rockner07}
{\sc C.~Pr\'{e}v\^{o}t and M.~R\"{o}ckner}, {\em A concise course on stochastic
  partial differential equations}, vol.~1905 of Lecture Notes in Mathematics,
  Springer, Berlin, 2007.

\bibitem{Prohl-Wang21}
{\sc A.~Prohl and Y.~Wang}, {\em Strong rates of convergence for a space-time
  discretization of the backward stochastic heat equation, and of a
  linear-quadratic control problem for the stochastic heat equation}, ESAIM
  Control Optim. Calc. Var., 27 (2021), pp.~Paper No. 54, 30.

\bibitem{Prohl-Wang22}
\leavevmode\vrule height 2pt depth -1.6pt width 23pt, {\em Strong error
  estimates for a space-time discretization of the linear-quadratic control
  problem with the stochastic heat equation with linear noise}, IMA J. Numer.
  Anal., 42 (2022), pp.~3386--3429.

\bibitem{Prohl-Wang24}
\leavevmode\vrule height 2pt depth -1.6pt width 23pt, {\em {Convergence with
  rates for a Riccati-based discretization of SLQ problems with SPDEs}}, IMA
  Journal of Numerical Analysis,  (accepted), p.~drad097.

\bibitem{Stillfjord18-2}
{\sc T.~Stillfjord}, {\em Adaptive high-order splitting schemes for large-scale
  differential {R}iccati equations}, Numer. Algorithms, 78 (2018),
  pp.~1129--1151.

\bibitem{Tambue-Mukam19}
{\sc A.~Tambue and J.~D. Mukam}, {\em Strong convergence of the linear implicit
  {E}uler method for the finite element discretization of semilinear {SPDE}s
  driven by multiplicative or additive noise}, Appl. Math. Comput., 346 (2019),
  pp.~23--40.

\bibitem{Thomee06}
{\sc V.~Thom\'{e}e}, {\em Galerkin finite element methods for parabolic
  problems}, vol.~25 of Springer Series in Computational Mathematics,
  Springer-Verlag, Berlin, second~ed., 2006.

\bibitem{WangX17}
{\sc X.~Wang}, {\em Strong convergence rates of the linear implicit {E}uler
  method for the finite element discretization of {SPDE}s with additive noise},
  IMA J. Numer. Anal., 37 (2017), pp.~965--984.

\bibitem{Wang16}
{\sc Y.~Wang}, {\em A semidiscrete {G}alerkin scheme for backward stochastic
  parabolic differential equations}, Math. Control Relat. Fields, 6 (2016),
  pp.~489--515.

\bibitem{Wang20}
\leavevmode\vrule height 2pt depth -1.6pt width 23pt, {\em {$L^2$}-regularity
  of solutions to linear backward stochastic heat equations, and a numerical
  application}, J. Math. Anal. Appl., 486 (2020), pp.~123870, 18.

\bibitem{Wang23}
\leavevmode\vrule height 2pt depth -1.6pt width 23pt, {\em Error analysis of
  the feedback controls arising in the stochastic linear quadratic control
  problems}, J. Syst. Sci. Complex., 36 (2023), pp.~1540--1559.

\bibitem{Yong-Zhou99}
{\sc J.~Yong and X.~Y. Zhou}, {\em Stochastic controls: Hamiltonian systems and
  HJB equations}, vol.~43 of Applications of Mathematics (New York),
  Springer-Verlag, New York, 1999.

\bibitem{Zabczyk20}
{\sc J.~Zabczyk}, {\em Mathematical control theory---an introduction},
  Birkh\"{a}user/Springer, Cham, second~ed., 2020.

\bibitem{ZhangJF04}
{\sc J.~Zhang}, {\em A numerical scheme for {BSDE}s}, Ann. Appl. Probab., 14
  (2004), pp.~459--488.

\end{thebibliography}
}

\end{document}